\documentclass[11pt]{amsart}
\usepackage{amssymb}
\usepackage{amsfonts}
\usepackage{amsmath}
\usepackage{graphicx}
\usepackage{xcolor}
\usepackage{amsmath}
\usepackage{mathrsfs}
\usepackage{stmaryrd}
\usepackage{epsfig,color}
\usepackage{blindtext}
\usepackage{enumerate}
\usepackage{hyperref}
\usepackage{url}
\usepackage{bbm}
\usepackage{nicefrac,mathtools}
\usepackage{bm}   
\usepackage{tabularx}
\DeclareGraphicsExtensions{.pdf,.jpeg,.png}
\usepackage{epstopdf}
\usepackage{cancel} 
\usepackage[normalem]{ulem} 
\usepackage{verbatim} 
\usepackage{scalerel}
\usepackage{enumitem}

\usepackage{cases}

\usepackage{subcaption}
\usepackage{soul}

\usepackage{color}
\usepackage[msc-links, lite]{amsrefs}
\usepackage{geometry}
\usepackage{stackengine,scalerel}

\newtheorem{theorem}{Theorem}[section]

\newtheorem{proposition}[theorem]{Proposition}
\newtheorem{lemma}[theorem]{Lemma}
\newtheorem{corollary}[theorem]{Corollary}
\newtheorem{question}[theorem]{Question}

\newtheorem{claim}[]{Claim}
\newtheorem*{acknowledgements}{Acknowledgements}
\theoremstyle{definition}
\newtheorem{definition}[theorem]{Definition}
\theoremstyle{remark}
\newtheorem{remark}[theorem]{Remark}
\newtheorem{example}[]{Example}

\numberwithin{equation}{section}

\newcommand{\mf}{\mathbf}
\newcommand{\mb}{\mathbb}
\newcommand{\mc}{\mathcal}
\newcommand{\ms}{\mathscr}
\newcommand{\mk}{\mathfrak}

\newcommand{\wti}{\widetilde}
\newcommand{\res}{\scaleobj{1.75}{\llcorner}}

\newcommand{\dist}{\operatorname{dist}}
\newcommand{\inj}{\operatorname{inj}}
\newcommand{\codim}{\operatorname{codim}}
\newcommand{\Cohom}{\operatorname{Cohom}}
\newcommand{\bd}{\partial}

\newcommand{\rom}[1]{\expandafter\romannumeral #1}
\newcommand{\Rom}[1]{\uppercase\expandafter{\romannumeral #1}}

\DeclareMathOperator{\genus}{genus}
\DeclareMathOperator{\Ric}{Ric}

\DeclareMathOperator{\spt}{spt}
\DeclareMathOperator{\interior}{int}
\DeclareMathOperator{\closure}{Clos}

\DeclareMathOperator{\Graph}{Graph}

\newcommand{\cohom}{\operatorname{Cohom}}
\newcommand{\bmu}{\boldsymbol \mu}
\newcommand{\btau}{\boldsymbol \tau}
\newcommand{\bleta}{\boldsymbol \eta}
\newcommand{\whC}{\widehat{C}}
\newcommand{\whD}{\widehat{D}}

\newcommand{\whH}{\widehat{H}}
\newcommand{\whU}{\widehat{U}}
\newcommand{\cDG}{\mathcal{D}^G}
\newcommand{\cKG}{\mathcal{K}^G}
\newcommand{\cSG}{\mathcal{S}^G}
\newcommand{\cCG}{\mathcal{C}^G}
\newcommand{\cAG}{\mathcal{A}^G}
\newcommand{\cBG}{\mathcal{B}^G}
\newcommand{\whcD}{\widehat{\mathcal{D}}}
\newcommand{\whcK}{\widehat{\mathcal{K}}}
\newcommand{\whcC}{\widehat{\mathcal{C}}}
\newcommand{\whcS}{\widehat{\mathcal{S}}}
\newcommand{\whcA}{\widehat{\mathcal{A}}}
\newcommand{\whcB}{\widehat{\mathcal{B}}}

\newcommand{\sreduc}{\underset{(G,\gamma)}{<}}

\newcommand{\red}{\color{red}}
\newcommand{\blue}{\color{blue}}

\title{Equivariant min-max theory and the spherical Bernstein problem in $\mathbb{S}^4$}

\author{Tongrui Wang}
\address{School of Mathematical Sciences, Shanghai Jiao Tong University, 800 Dongchuan RD, Minhang District, Shanghai, 200240, China}
\email{wangtongrui@sjtu.edu.cn}

\author{Zhichao Wang}
\address{Shanghai Center for Mathematical Sciences, 2005 Songhu Road, Fudan University, Shanghai, 200438, China}
\email{zhichao@fudan.edu.cn}

\author{Xin Zhou}
\address{Department of Mathematics, 531 Malott Hall, Cornell University, Ithaca, NY 14853, USA}
\email{xinzhou@cornell.edu}

\begin{document}

\begin{abstract}
	We construct an embedded non-equatorial minimal hypersphere in the unit $4$-sphere $\mathbb{S}^4$, which provides a new resolution of Chern's spherical Bernstein problem in $\mathbb{S}^4$. 
	The construction is based on our equivariant min-max theory for $G$-invariant minimal hypersurfaces with reduced genus bound, where $G$ is a compact Lie group acting by isometries on a closed Riemannian manifold with $3$-dimensional orbit space. 
    This confirms an assertion made by Pitts-Rubinstein in 1986. 
	We also establish the regularity of solutions to the $G$-equivariant Plateau problem and the $G$-equivariant isotopy area minimization problem. 
\end{abstract}

\maketitle

\section{Introduction}

The classical Bernstein problem concerns whether every entire solution of the minimal equation in $\mb R^n$ is necessarily a linear function, which has been affirmed for dimensions $n\leq 7$ by De Giorgi \cite{degiorgi1965bernstein} ($n=3$), Almgren \cite{almgren1966bernstein} ($n=4$), and Simons \cite{simons1968minimal} ($n\leq 7$). 
For dimensions $n\geq 8$, the answer was proven to be negative by Bombieri, De Giorgi, and Giusti \cite{bombieri1969minimal}. 
Recently, there have been several breakthroughs (\cite{chodosh2024stable}\cite{chodosh2024stableR5}\cite{mazet2024stable}) on the stable Bernstein problem, which investigates whether the rigidity holds under the condition of stability.  

The investigation of the Bernstein problem is closely related to the theory of minimal cones and the analysis of singularities in minimal hypersurfaces. 
Since the link of a minimal hypercone in $\mb R^{n+1}$ defines a minimal hypersurface in the unit sphere $\mb S^{n}$, it is rather natural to study minimal hypersurfaces in $\mb S^{n}$ that are diffeomorphic to $\mb S^{n-1}$ (the link of a hyperplane). 
In 1969, Professor S. S. Chern \cite{chern1969survey} proposed the following outstanding problem in differential geometry, and later highlighted it at the 1970 International Congress of Mathematicians in Nice \cite{chern1971differential}. 

\begin{question}[Spherical Bernstein Problem]
	Let the $(n-1)$-sphere be embedded as a minimal hypersurface in the unit $n$-sphere $\mb S^{n}$. Is it (necessarily) an equator?
\end{question}

For the case $n=3$, the answer was proven to be affirmative by a theorem of Almgren \cite{almgren1966bernstein} and Calabi \cite{calabi1967minimal} under the weaker immersed assumption. 
However, in contrast to the Bernstein problem in Euclidean spaces, the spherical rigidity fails at a much lower dimension. 
In a series of celebrated works, Hsiang \cite{hsiang1983sphericalI}\cite{hsiang1983sphericalII} constructed infinitely many distinct new examples of non-equatorial minimal hyperspheres in $\mb S^n$ for $n\in \{4,5,6,7,8,10,12,14\}$, thereby settling the spherical Bernstein problem in the negative for low dimensions. 
Subsequently, Tomter \cite{tomter1987spherical} significantly extended the existence of non-equatorial minimal hyperspheres in $\mb S^n$ for all even dimensions $n=2m\geq 4$.
Moreover, if the condition is relaxed from embedding to immersion, then the answer to the spherical Bernstein problem is negative in $\mb S^{n}$ for every $n\geq 4$ by Hsiang-Sterling \cite{hsiang1986sphericalIII}, who also constructed embedded non-equatorial minimal hyperspheres as links of stable minimal hyper-cones. 

The constructions of Hsiang and Tomter rely on the framework of {\em equivariant differential geometry}. 
This approach enables the reduction of the underlying non-linear parametric problem to a more tractable analysis of global ordinary differential equations.
Namely, by imposing suitable orthogonal group actions on $\mb S^n$, the problem of constructing symmetric minimal hyperspheres is transformed into finding geodesic curves in the associated orbit space. 
Consequently, the existence of such geodesics is established via a dynamical systems approach using the oscillatory behavior.

In this paper, we instead use an equivariant version of min-max theory to give a new construction of the embedded non-equatorial minimal hypersphere in $\mb S^4$, which also shows the existence of a minimal hypertorus.   
\begin{theorem}\label{Main thm: spherical bernstein in S4}
	Let $S^1$ act on $\mb S^4=\{(x,z_1,z_2)\in \mb R\times\mb C^2 : |x|^2+|z_1|^2+|z_2|^2=1 \}$ by the suspension of the Hopf action, i.e. $e^{i\theta}\cdot (x,z_1,z_2):=(x, e^{i\theta}z_1, e^{i\theta}z_2)$ for all $e^{i\theta}\in S^1$ and $(x,z_1,z_2)\in \mb S^4$. 
	Then there exist an embedded non-equatorial minimal hypersphere $S^3$,  an embedded minimal hypertorus $T^3=S^1\times S^1\times S^1$, and an embedded minimal $S^1\times S^2$ in $\mb S^4$ that are invariant under the $S^1$-action. 
\end{theorem}

In the constructions, we indeed used the extensions $G\cong (S^1\rtimes \mb Z_4) / \mb Z_2$ and $G'=S^1\rtimes \mb Z_2$ of $\mb Z_2$ by $S^1$ as our isometric group actions (Lemma \ref{Lem: extend action}, \eqref{Eq: extension of S1 by double reflection}). 
	Hence, our minimal hypersurfaces have more symmetries. 
	
\begin{remark}
	While Hsiang's dynamical approach relies heavily on the round metric of the unit sphere, our min-max variational method offers greater flexibility regarding the ambient geometry. 
	In particular, our methods imply the existence of at least four distinct minimal $S^1$-invariant hyperspheres for a family of $S^1$-symmetric Riemannian $S^4$ with positive Ricci curvature (Remark \ref{Rem: valid for other metric}). 
	Moreover, in light of the generic resolution of Yau's conjecture on the existence of four minimal $S^2$ in $S^3$ \cite{wangzc2023four-sphere} (see also \cite{haslhofer2019sphere}\cite{li2025RP2} for minimal $RP^2$ in $RP^3$), we would expect to have four distinct $S^1$-invariant minimal hyperspheres in $S^4$ under a generic $S^1$-invariant metric. 
\end{remark}

\begin{remark}
    In \cite{carlotto2023hypertori}, Carlotto-Schulz also constructed an embedded minimal $S^{n-1}\times S^{n-1}\times S^1$ in $ \mb S^{2n}$ by a method similar to Hsiang's \cite{hsiang1983sphericalI}, which relies on the round metric.  
    In contrast, our construction remains valid for a family of $S^1$-invariant Riemannian $S^4$ of positive Ricci curvature to produce at least four distinct minimal $T^3$ (possibly congruent in the round metric) and one minimal $S^1\times S^2$ (see Remark \ref{Rem: valid for other metric}, \ref{Rem: valid of S1 times S2 for other metrics}). 
\end{remark}

The proof of Theorem \ref{Main thm: spherical bernstein in S4} 
is mainly based on our newly established equivariant min-max theory for minimal hypersurfaces with reduced topological control. 
As intermediate products, we also show the regularity of the solutions to the equivariant Plateau problem and the equivariant isotopy area minimization problems. 
We will discuss these aspects in more detail below.

\subsection{Equivariant min-max theory}

The strategy of min-max theory traces back to Birkhoff, who first utilized a min-max method to find simple closed geodesics on spheres. 
In higher-dimensional closed manifolds, the min-max theory was established to construct minimal hypersurfaces by Almgren \cite{almgren1962homotopy}\cite{almgren1965theory} and Pitts \cite{pitts2014existence} (see also \cite{schoen1981regularity}). 
Although the celebrated Almgren-Pitts theory is a general and powerful existence mechanism, it generally provides no topological information about the min-max minimal hypersurfaces. 
Nevertheless, in $3$-dimensional manifolds, Simon and Smith \cite{smith1983existence} refined the min-max approach to control the genus of the minimal surface, and proved the existence of an embedded minimal $2$-sphere in any Riemannian 3-sphere. 
We refer to \cite{colding2002minmax}\cite{delellis2010genus}\cite{ketover2019genus} for the Simon-Smith min-max theory in any closed Riemannian $3$-manifold and the (weighted) genus upper bound for min-max minimal surfaces.

The Simon-Smith min-max theory relies on the regularity of solutions to Plateau's problem and minimizers in isotopy classes (cf. \cite{almgren79plateau}\cite{meeks82exotic}). 
These regularity results are restricted to surfaces in three-manifolds. 
Consequently, the theory cannot be directly generalized to higher dimensions to yield minimal hypersurfaces with topological control. 
Against this backdrop, using symmetry to perform dimensional reduction is an effective method.

By imposing symmetry constraints on the ambient manifold, Pitts and Rubinstein \cite{pitts1987applications}\cite{pitts1988equivariant} first proposed the {\em equivariant min-max theory}. 
 In the context of finite group actions on 3-manifolds, they stated that this approach could construct minimal surfaces with estimates on the index and topology. 
The first rigorous min-max construction of such equivariant minimal surfaces was achieved by Ketover \cite{ketover2016equivariant} for a finite group acting by orientation-preserving isometries (see also \cite{ketover2016free} for the free boundary version). 
We also refer to \cite{franz2021equivariant}\cite{franz2023equivariant} for related results on equivariant min-max theory of finite group actions.

In this paper, one of our main results is the following equivariant min-max theory for symmetric minimal hypersurfaces with reduced topological control in a symmetric closed Riemannian manifold with three-dimensional orbit space. 
We mention that there is a mild assumption on the $G$-hypersurfaces (abbreviation for $G$-invariant hypersurfaces) in our variational constructions. 
Namely, the $G$-hypersurfaces in sweepouts are {\em locally $G$-boundary-type} in the sense that $\Sigma$ locally lies in the reduced boundary of some $G$-invariant Caccioppoli set (see Definition \ref{Def: local G-boundary}). 
This condition is analogous to the fact that every embedded hypersurface is locally a boundary by the simple connectedness of geodesic balls. 

\begin{theorem}[Equivariant Min-max Theory]\label{Main Thm: equivariant min-max}
	Let $(M^{n+1}, g_{_M})$ be a closed $(n+1)$-dimensional Riemannian manifold, and $G$ be a compact Lie group acting isometrically on $M$ so that $M$ contains no special exceptional orbit (cf. Definition \ref{Def: special exceptional orbit} or \cite{bredon1972introduction}*{P. 185}), and 
	\begin{align}\label{Eq: main thm - dimension assumption}
		3=\min_{x\in M}\codim(G\cdot x) \leq \max_{x\in M}\codim(G\cdot x) \leq 7. 
	\end{align}
	Suppose $\Sigma_0\subset M$ is an embedded closed $G$-hypersurface of locally $G$-boundary-type so that $\Sigma_0/G$ is an orientable surface of genus $\mk g_0$. 
	Then for any homotopy class of smooth $G$-equivariant sweepouts of $\Sigma_0$, there exists an associated min-max $G$-invariant varifold $V\in\mc V^G_n(M)$ so that 
	\[V=\sum_{j=1}^R m_j|\Gamma_j|, \]
	for some disjoint, closed, smoothly embedded, minimal $G$-hypersurfaces $\{\Gamma_j\}_{j=1}^R$ in $M$, ($R\in\mb N$), 
    with multiplicities $\{m_j\}_{j=1}^R\subset \mb N $. 
	Moreover, denote by $(M^{prin})^*$ the regular part of $M/G$, and by $\mk g(\Gamma_j):= \genus(\Gamma_j/G\cap (M^{prin})^*)$ the reduced genus of $\Gamma_j$. 
	Then, 
	\begin{align}\label{Eq: main thm - min-max genus control}
		\sum_{j\in\mc O} m_j \mk g(\Gamma_j) + \sum_{j\in\mc U}\frac{1}{2} m_j(\mk g (\Gamma_j)-1)   \leq \mk g_0,
	\end{align}
	where $\mc O$ (resp. $\mc U$) is the set of $j\in \{1,\dots, R\}$ so that $\Gamma_j/G$ is orientable (resp. non-orientable). 
\end{theorem}

\begin{remark}\label{Rem: main thm - min-max}
	We make the following comments and explanations for the above theorem. 
	\begin{itemize}
		\item[(1)] In the dimension assumption \eqref{Eq: main thm - dimension assumption}, $\min_{x\in M}\codim(G\cdot x)$ is also known as the {\em cohomogeneity} $\cohom(G)$ of the $G$-action on $M$. In particular, $\cohom(G)=3$ indicates that the orbit space $M/G$ is a topological $3$-manifold (possibly with boundary) away from at most finitely many points $\mc S_{n.m.}/G$ (\cite{bredon1972introduction}*{P.187, Theorem 4.3}), which is used to obtain the topological control. 
        Additionally, the assumption $\codim(G\cdot x)\leq 7$ in \eqref{Eq: main thm - dimension assumption} is also required due to the regularity theory for minimal hypersurfaces (cf. \cite{schoen1981regularity}). 
        \item[(2)] An orbit $G\cdot p$ in $M$ is said to be a {\em special exceptional orbit} if $\codim(G\cdot p)=\Cohom(G)$ and there is a codimension-one subspace $P$ in the normal vector space $N_p(G\cdot p)$ of $G\cdot p$ so that $P$ is the fixed point set under the action of $G_p:=\{g\in G: g\cdot p=p\}$. 
        Roughly speaking, $G$ locally looks like a reflection near a special exceptional orbit (\cite{bredon1972introduction}*{P. 185}). 
		The assumption on the non-existence of special exceptional orbits is due to the regularity theory for $G$-isotopy area minimizing problems, which will be illustrated further after Theorem \ref{Main thm: isotopy minimizing}. 
        In particular, for orientation-preserving finite $G$-actions on $M^3$, there is no special exceptional orbit, and $\codim(G\cdot p)=3$ for all $p\in M$. 
        Hence, Theorem \ref{Main Thm: equivariant min-max} can construct smoothly embedded minimal $G$-surfaces in $M^3$ as in Ketover \cite{ketover2016equivariant}. 
		\item[(3)] For an embedded closed $G$-hypersurface $\Sigma\subset M$, it follows from \cite{bredon1972introduction}*{P.187, Lemma 4.1} that $\Sigma/G$ is a topological $2$-manifold (possibly with boundary). 
		Hence, the genus of the surface $\Sigma/G$ is naturally defined. 
		In particular, for an orientable (resp. non-orientable) closed surface, the genus is defined by $(2-\chi)/2$ (resp. $2-\chi$), where $\chi$ denotes the Euler characteristic. 
		If the surface is compact with boundary, then the genus is defined by the genus of its natural compactification. 
		\item[(4)] It is well known that there is an open dense subset $M^{prin}\subset M$ (called the {\em principal orbit type stratum}) so that $\pi: M^{prin}\to  (M^{prin})^*:=M^{prin}/G$ is a smooth submersion. 
			Additionally, since removing the boundary and finitely many points does not affect the genus of a surface, we have $\mk g(\Sigma)=\genus(\Sigma/G)$ provided that $\Sigma\cap M^{prin}\neq\emptyset$; and $\mk g(\Sigma)=0$ if $\Sigma\subset M\setminus M^{prin}$. 
            In particular, the non-existence of special exceptional orbits implies $\dim(M\setminus M^{prin})\leq n-1$, and a $G$-hypersurface cannot be contained in $M\setminus M^{prin}$. 
		\item[(5)] Although $\Gamma_j/G$ may be a surface with boundary, we do not have an upper bound on the number of its boundary components in analogy with \cite{franz2023equivariant}. 
		This is mainly because the reduced metric (see \eqref{Eq: weighted metric in M/G}) on the boundary of $M/G$ may be degenerate. 
	\end{itemize}
\end{remark}

In \cite{pitts1987applications}*{Theorem 6}, Pitts and Rubinstein first announced a similar equivariant min-max theory adopting the same definition for the reduced genus $\mk g(\Gamma)$. 
Notably, their $G$-actions are assumed to be orientation-preserving so that there is no special exceptional orbit (cf. \cite{pitts1987applications}*{P.157}). 
Although a detailed proof was not provided in \cite{pitts1987applications}, relying on this assertion, they proposed a construction for minimal hypersurfaces $M_i\subset \mb S^n$ diffeomorphic to $\#_{2i}S^1\times S^{n-2}$ for sufficiently large $i\in\mb N$. 
The topological characterization of $M_i$ in \cite{pitts1987applications}*{\S 5} is only valid for $i$ large enough, for reasons similar to Remark \ref{Rem: main thm - min-max}(5) (cf. \cite{pitts1987applications}*{P.163}). 
Apart from the claim concerning regularity up to a singular set of codimension $7$, our work (Theorem \ref{Main Thm: equivariant min-max}) provides a detailed proof and confirms the assertion of Pitts-Rubinstein \cite{pitts1987applications}*{Theorem 6} in the most general setting. 

During the preparation of this manuscript, we became aware of the recent work by Ko \cite{ko2025}, who developed an equivariant min-max theory in simply connected closed Riemannian manifolds (e.g., $S^n$) under the action of the compact product group $G_c\times G_f$, where $G_c$ is a continuous Lie group with connected principal orbits and $G_f$ is a finite group acting by orientation-preserving isometries on $M/G_c$ with certain assumptions on singular loci. 
There are notable differences between our construction and that of Ko \cite{ko2025}. 
Specifically, Ko's variational construction is carried out on the orbit space, which requires additional assumptions to ensure the manifold structure of $M/G$. 
In contrast, our variational problem is formulated directly on the ambient manifold subject to equivariant constraints with only certain topological arguments applied in $M/G$. 
Additionally, we also introduce a novel technique to address the analytical challenges arising near the non-manifold singularities of the orbit space (cf. {\bf Step 2, 3} in Theorem \ref{Thm: G-isotopy minimizer}). 
In particular, we can use the convergence result \cite{meeks82exotic}*{Remark 3.27} to locally upgrade the $G$-stability of the $G$-isotopy area minimizer to stability, and extend the regularity to the orbits where $M/G$ fails to be a $C^0$ $3$-manifold by \cite{schoen1981regularity}. 
In the regularity theory for our equivariant min-max, we also use the non-existence of special exceptional orbits and the boundary regularity result \cite{delellis2010genus}*{Lemma 8.1} to upgrade the $G$-stability of constrained $G$-isotopy area minimizers to stability (cf. Proposition \ref{Prop: constrained minimizer}). 
These enable us to establish the theory in full generality. 
In particular, our result accommodates more general non-product group actions (e.g., $G=(S^1\rtimes \mb Z_4)/\mb Z_2$ in Theorem \ref{Main thm: spherical bernstein in S4}) and is used to produce minimal hypersurfaces with simple topological types.

In \cite{wangzc2023four-sphere}, the second and the third authors showed a multiplicity one theorem for Simon-Smith min-max theory, and applied it to confirm a conjecture of Yau in the generic sense. 
Their multiplicity one theorem is based on a Simon-Smith min-max theory for surfaces with prescribed mean curvature (see also \cite{zhou2020multiplicity}\cite{zhou2018existence}). 
We expect a similar result to hold for Theorem \ref{Main Thm: equivariant min-max}.

Under the weaker assumption $3\leq\Cohom(G)$, Liu \cite{liu2021existence} developed an equivariant min-max construction for compact connected Lie group actions in the setting of \cite{de2013existence}. 
In the Almgren-Pitts setting, the multi-parameter min-max theory was adapted to an equivariant version by the first author in \cite{wang2022min}\cite{wang2023free}, without requiring the connectedness of $G$.

In recent years, we have witnessed the tremendous breakthroughs in Almgren-Pitts min-max theory, including the resolution of the Willmore conjecture \cite{marques2014min}, Yau's abundance conjecture on the existence of infinitely many minimal surfaces \cite{marques2017existence}\cite{song2018existence} and the progress on the spatial distribution of minimal hypersurfaces \cite{liokumovich2018weyl}\cite{irie2018density}\cite{marques2019equidistribution}\cite{song2021scarring}. 
Moreover, based on the third author's resolution \cite{zhou2020multiplicity} (relying on \cite{zhou2019constant}\cite{zhou2018existence}) of the multiplicity one conjecture, Marques and Neves \cite{marques2016morse}\cite{marques2021morse} complete their program on the Morse theory for the area functional; see also the local version \cite{marques2023morseinequalities}. We refer to the survey article \cite{Zhou-ICM22} for more detailed histories.

\subsection{Equivariant isotopy area minimizing}

Given a uniformly convex smooth open set $B\subset \mb R^3$, the classical {\em Plateau problem} is to find an area-minimizing disk spanning a fixed boundary $\Gamma\subset\bd B$ in $\mb R^3$, where $\Gamma$ is a simple closed curve. 
This celebrated problem was rigorously solved independently by Douglas and Rad{\' o} in 1930 by mapping methods; the solutions are shown to be immersions by Osserman \cite{osserman1970plateau} and Gulliver \cite{gulliver1973regularity}. 
In 1976, Almgren-Simon \cite{almgren79plateau} and Meeks-Yau \cite{meeks1982plateau} constructed embedded solutions to Plateau's problem, which can be generalized to minimal surfaces of high topological type. 

As a key ingredient, we also show the following regularity result for the {\em equivariant Plateau problem}. 
Specifically, for $r_0>0$ small enough, suppose $A:=B_{r}(G\cdot a)\subset M$ is the $r$-neighborhood of $G\cdot a$ with $r\in (0,r_0)$. 
Let $\Gamma\subset \bd A$ be an $(n-1)$-dimensional $G$-submanifold so that $\Gamma/G$ is connected, and $(\bd A\setminus\Gamma)/G$ has two components. 
Denote by $\mc M^G_{0,\Gamma}$ the set of locally $G$-boundary-type (Definition \ref{Def: local G-boundary}) $G$-hypersurfaces in $B_{r_0}(G\cdot a)$ with $\genus(\Sigma/G)=0$ and $\bd \Sigma=\Gamma$. 
Assume $\mc M^G_{0,\Gamma}\neq\emptyset$ and consider the area minimizing sequence $\{\Sigma_k\}_{k\in\mb N}\subset \mc M^G_{0,\Gamma}$ with 
\[ \mc H^n(\Sigma_k) \leq  \inf_{\Sigma'\in \mc M^G_{0,\Gamma}} \mc H^n(\Sigma') + \epsilon_k ,\]
for some positive numbers $\epsilon_k\to 0$. 
Then we have the following interior regularity.

\begin{theorem}\label{Main thm: plateau problem}
	Let $(M^{n+1}, g_{_M})$ be a closed $(n+1)$-dimensional Riemannian manifold, and $G$ be a compact Lie group acting isometrically on $M$ so that \eqref{Eq: main thm - dimension assumption} is satisfied.
    Then, using the above notations,
    the minimizing sequence $\Sigma_k\in \mc M^G_{0,\Gamma}$ converges in the varifold sense to  $V\in \mc V_n^G(M)$ so that for any $p_0\in \spt(\|V\|)\setminus \Gamma$, 
    there exist $m\in\mb N$, $\rho>0$, and an embedded minimal $G$-hypersurface $\Sigma\subset M$ satisfying 
	\[V\llcorner B_\rho(G\cdot p_0)=m |\Sigma|.\]
\end{theorem}

\begin{remark}
    Indeed, we will show that $\Sigma$ is a minimal $G$-hypersurface that is of {\em locally $G$-boundary-type} (Definition \ref{Def: local G-boundary}) except possibly at finitely many orbits. 
    This implies that $\Sigma/G$ can only meet $\bd (M/G)$ orthogonally (Lemma \ref{Lem: hypersurface position}), which is similar to the {\em proper} embeddedness of locally area minimizing free boundary minimal hypersurfaces. 
    This property also helps in the regularity theory (cf. Section \ref{Subsec: plateau} {\bf Step 2, 3}). 
    In addition, note that the non-existence of special exceptional orbits is not required in this theorem. 
    Moreover, 
    the boundary regularity is valid on $\Gamma\cap M^{prin}$ by Almgren-Simon \cite{almgren79plateau} and the equivariant reduction of Hsiang-Lawson \cite{hsiang71cohom}. 
\end{remark}

In closed manifolds, a natural analogue is to minimize area subject to topological constraints for closed (hyper)surfaces. 
This leads to the {\em isotopy area minimization problem}, which seeks a surface of least area within a fixed isotopy class.
Based on the embedded resolution of Plateau's problem, Meeks-Simon-Yau \cite{meeks82exotic} established the regularity and topological control for isotopy minimizers. In this paper, we present an equivariant generalization of their results.

\begin{theorem}\label{Main thm: isotopy minimizing}
    Let $M^{n+1}$ and $G$ be given as in Theorem \ref{Main Thm: equivariant min-max} so that \eqref{Eq: main thm - dimension assumption} is satisfied and $M$ has no special exceptional orbit (Definition \ref{Def: special exceptional orbit}). 
    Suppose $\{\Sigma_k:=\varphi^k_1(\Sigma)\}_{k\in \mb N}$ with $\{\{\varphi^k_t\}_{t\in [0,1]}\}_{k\in\mb N}\subset\mk {Is}^G(M)$ is a minimizing sequence for the minimization problem $(\Sigma, \mk {Is}^G(M))$, i.e. 
    \[ \lim_{k\to\infty} \mc H^n(\Sigma_k) = \inf_{\{\phi_t\}_{t\in [0,1]}\in \mk {Is}^G(M)} \mc H^n(\phi_1(\Sigma)).  \] 
	Then, up to a subsequence, $\Sigma_k$ converges to a $G$-varifold $V\in \mc V^G_n(M)$ so that 
	\[
		V=\sum_{j=1}^R m_j |\Sigma^{(j    
        )}|, 
	\]
	where $R, m_1,\dots,m_R\in\mb N$, and $\Sigma^{(j)}$ are pairwise disjoint, smoothly embedded, closed, minimal $G$-hypersurfaces that are stable with respect to $G$-equivariant variations (i.e. $G$-stable). 
	Additionally, if the initial $\Sigma/G$ is orientable, then for all $k\in\mb N$ sufficiently large,
	\begin{align}\label{Eq: main thm - isotopy regularity - genus}
		 \sum_{j\in\mc O} m_j \mk g(\Sigma^{(j)})  + \sum_{j\in\mc U} \frac{1}{2}m_j(\mk g(\Sigma^{(j)})-1)\leq \genus(\Sigma_k/G),
	\end{align}
	where $\mc O$ (resp. $\mc U$) is the set of $j\in \{1,\dots, R\}$ so that $\Sigma^{(j)}/G$ is orientable (resp. non-orientable), and $\mk g(\Sigma^{(j)}):=\genus(\Sigma^{(j)}/G\cap (M^{prin})^*)$ is defined as in \eqref{Eq: main thm - min-max genus control}. 
\end{theorem}

\begin{remark}
    In the above theorem, the multiplicity $m_j$ is even whenever $j\in \mc U$. 
    Additionally, we will use \cite{meeks82exotic}*{Remark 3.27} and the non-existence of special exceptional orbits to extend the regularity of $V$ to the orbits where $M/G$ fails to be a $C^0$ $3$-manifold. 
\end{remark}

In Theorem \ref{Main thm: isotopy minimizing}, the assumption that special exceptional orbits do not exist is for technical reasons. 
Note that, roughly speaking, $G$ locally acts by reflection near a special exceptional orbit (cf. Proposition \ref{Prop: local structure of Sigma/G 1}(ii), \cite{bredon1972introduction}*{P.185}). 
Hence, one can show that the quotient of the union of all special exceptional orbits forms a minimal boundary portion $A\subset \bd (M/G)$. 
A natural idea is to transform the equivariant area minimization problem in $M$ into a non-equivariant free boundary version in $M/G$ near $A$ by the theory of Hsiang-Lawson \cite{hsiang71cohom} (cf. \eqref{Eq: area in orbit space}). 
However, near $A\subset M/G$, one cannot directly apply the regularity theory for isotopy minimizing/min-max minimal surfaces with free boundary established by Li \cite{li2015general}. 
This is because the isotopies in \cite{li2015general} can move points across the boundary $A$ to the outside of $M/G$ (but not into $M/G$), while such an outer isotopy is not allowed in the equivariant setting since the equivariant isotopy cannot move $p\in M^{prin}$ to $M\setminus M^{prin}$ (cf. \cite{schwarz1980lifting}*{Corollary 1.7}). 
Nevertheless, if $M$ admits no special exceptional orbit, then $\dim(M\setminus M^{prin})\leq n-1$ so that the weighted metric on $M/G$ given by Hsiang-Lawson \cite{hsiang71cohom} (cf. \ref{Eq: weighted metric in M/G}) is degenerate on $\bd (M/G)\subset (M\setminus M^{prin})/G$, and the area of a $G$-hypersurface tends to zero by approximating $M\setminus M^{prin}$ through $G$-isotopies. 
This is analogous to pushing a surface out of the manifold via outer isotopies to contribute no area. 

As in Theorem \ref{Main thm: plateau problem}, every $\Sigma^{(j)}$ in the above theorem is indeed locally $G$-boundary-type (possibly except at finitely many orbits). 
This phenomenon plays an important role in our equivariant min-max theory. 
Specifically, if a minimal $G$-hypersurface $\Sigma$ is stable with respect to all $G$-equivariant variations (i.e. $G$-stable), it may not be stable with respect to non-equivariant variations. 
Hence, we cannot directly obtain the regularity and curvature estimates of $\Sigma$ from its $G$-stability using \cite{schoen1981regularity}. 
Nevertheless, it has been noticed that if $\Sigma$ admits a $G$-invariant unit normal, then the $G$-stability of $\Sigma$ is equivalent to the classical stability (Lemma \ref{Lem: stability and G-stability}). 
In particular, the locally $G$-boundary-type property of $\Sigma$ indicates the existence of a $G$-invariant unit normal in any open $G$-set $U\subset M$ with simply connected $U/G$. 
Hence, we would have curvature estimates and a compactness theorem for equivariant isotopy area minimizers (Corollary \ref{Cor: stability of minimizers}), which are crucial in our equivariant min-max (cf. Proposition \ref{Prop: constrained minimizer}).

\subsection{Outline}
We now describe the outline of the paper and also the outline of the proof. 

In Section \ref{Sec: preliminary}, we first collect some notations and review fundamental aspects of Lie group actions, including the orbit types and the stratification. 
To provide a concrete geometric intuition, we also give a detailed description of the local structures of the orbit space. 
Specifically, we completely classify the local behaviors for the ambient quotient $M/G$ into $5$ distinct types (Proposition \ref{Prop: local structure of M/G}), and show $6$ types of local structures for the quotient hypersurface $\Sigma/G$ (Proposition \ref{Prop: local structure of Sigma/G 1}, \ref{Prop: local structure of Sigma/G 2}). 
Subsequently, we utilize the coarea formula to perform the equivariant reduction of the area functional. 
Another important technical step here is the stratified uniform radius (Lemma \ref{Lem: uniform constants}), which is designed to overcome the analytical difficulty posed by the lack of a uniform positive lower bound for the orbits' injectivity radius (i.e. $\inf_{p\in M}\inj(G\cdot p)=0$ in general). 
Finally, we collect necessary notations from geometric measure theory and explain the distinction/relation between equivariant variations and general variations.

\medskip
As a key ingredient of the regularity theory for isotopy minimizers, Meeks-Simon-Yau \cite{meeks82exotic} introduced the method of {\em $\gamma$-reduction}, which utilizes surgery to locally transform the isotopy minimization problem into Plateau's problem. 
In the free boundary setting, Li \cite{li2015general} adapted this technique by extending the Riemannian manifold $M$ with boundary into a closed one $\wti M$ and employing the concept of {\em outer isotopy}, i.e. isotopies $\{\varphi_t\}$ in $\wti M$ with $M\subset \varphi_t(M)$. 

To establish our results, we develop an {\em equivariant $\gamma$-reduction} in Section \ref{Sec: gamma reduction}. 
Roughly speaking, using the local classification of the orbit space (Proposition \ref{Prop: local structure of M/G}, \ref{Prop: local structure of Sigma/G 1}, and \ref{Prop: local structure of Sigma/G 2}), we perform the necessary topological constructions within the quotient space $M/G$, while the area computations are performed in $M$ with the equivariant reduction formula in Section \ref{Subsec: preliminary - equivariant reduction}. 
One of the primary challenges here is that the orbit space $M/G$ may possess a boundary. In this context, the outer isotopy construction used by Li \cite{li2015general}*{\S 7} is inapplicable for two reasons: first, there is no natural closed extension of the singular orbit space; second, equivariant isotopies are constrained to preserve the orbit type stratification (\cite{schwarz1980lifting}*{Corollary 1.7}), thus precluding the boundary-contacting deformations characteristic of outer isotopies. 
Nevertheless, we overcome this obstacle by invoking isotopy lifting theory \cite{schwarz1980lifting}. 
Through a delicate construction, we first have an orbit-type-preserving isotopy in $M/G$, and lift it by \cite{schwarz1980lifting} to obtain the desired equivariant isotopy approximation (Lemma \ref{Lem: G-isotopy approximation}) in analogy to \cite{li2015general}*{Lemma 7.11}.
Note that the orbit space $M/G$ has a nice manifold structure away from a finite set $\mc S^*$ (Proposition \ref{Prop: local structure of M/G}). Consequently, our implementation of the $(G,\gamma)$-reduction is confined to the manifold locus $(M/G)\setminus \mc S^*$.

Another subtlety lies in the `isoperimetric disk selection' in the constructions of Meeks-Simon-Yau \cite{meeks82exotic} (and also in \cite{almgren79plateau}), that is, to split a sphere by a closed curve into two disks and choose the one with smaller area. 
In our equivariant setting, however, the presence of a boundary in $M/G$ introduces a new scenario: a closed curve within a hemisphere may divide the surface into a disk and a cylindrical region.
Critically, near special exceptional orbits, the equivariant isotopy approximation lemma (Lemma \ref{Lem: G-isotopy approximation}) may fail if the cylindrical region is chosen, which is different from Li \cite{li2015general}*{Lemma 7.11} due to the orbit type preserving condition for equivariant isotopies. 
To overcome this, we refined the selection criteria to explicitly mandate the choice of the disk component in such cases (see Theorem \ref{Thm: MSY thm2 gamma reduction} \eqref{Eq: MSY thm2 assumption - disk instead of cylinder}). 
Nevertheless, our analysis of the local structure confirms that this additional constraint is always achievable (Remark \ref{Rem: compare disk instead cylinder}). 

\medskip
Section \ref{Sec: plateau} is dedicated to the regularity theory of the equivariant Plateau problem. 
While our main strategy parallels the work of Almgren-Simon \cite{almgren79plateau}, the intricate local structure of $M/G$ requires a delicate case-by-case treatment. 
Additionally, the issue of `isoperimetric disk selection' also occurs in this context; for instance, in the Filigree Lemma \ref{Lem: filigree}, we impose a selection constraint \eqref{Eq: compare to disk} which can be relaxed in the absence of restrictions on the number of boundary components (see Remark \ref{Rem: filigree lemma statements}).
Subsequently, provided that the limit varifold admits a planar tangent cone, we establish the First Regularity Theorem \ref{Thm: first regularity} analogous to \cite{almgren79plateau}*{Theorem 2}. 
This relies on the splitting properties of equivariant varifold tangent cones; in particular, via equivariant blow-up analysis, we also obtain equivariant `tangent bundles' (Claim \ref{Claim: tangent bundle}). 
However, due to the potential degeneration of the metric along the boundary of $M/G$, certain technical conclusions from Almgren-Simon (e.g., \cite{almgren79plateau}*{(5.10)}) fail in our setting. 
To address this, we introduce a novel topological construction by adjoining a small neighborhood of the degenerate boundary to the cylindrical region (see $\cCG_{\rho,\sigma,\tau}$ in Theorem \ref{Thm: first regularity} {\bf Case I}).

Moreover, we emphasize that the assumption of dimension three in \cite{almgren79plateau} is critical not only for the topological arguments but also for the area comparison estimates in the characterization of tangent cones (cf. \cite{almgren79plateau}*{Theorem 3 {\it Case III} and Corollary 2}). 
Consequently, our construction yields the characterization of tangent cones (and thus the regularity of $V$ in Theorem \ref{Main thm: plateau problem}) near an orbit of codimension three (Theorem \ref{Thm: interior regularity k0=n-2}). 
For higher codimensional orbits, we employ a dimension reduction argument to overcome these difficulties (Theorem \ref{Thm: interior regularity k_0<n-2}).

\medskip
In Section \ref{Sec: G-isotopy minimizing}, we establish the regularity theory and topological control for equivariant isotopy minimizers. 
In addition to the difficulties shared with the equivariant Plateau problem, one should also note that our proof is initially confined to the manifold part of $M/G$ since the topological arguments are all based on the manifold structure. 
Noting that $M/G$ fails to be a manifold at only finitely many points, we can thus extend the regularity across these orbits by the regularity theory for stable minimal hypersurfaces (\cite{schoen1981regularity}). Indeed, this analysis relies on the convergence result of Meeks-Simon-Yau (\cite{meeks82exotic}*{Remark 3.27, (3.29)}), which illustrates the relation between the minimizing sequence obtained via surgery and the resulting limit varifolds. 
Combined with our previous constructions, this convergence result can be used to derive the stability of the limit minimal $G$-hypersurface (Theorem \ref{Thm: G-isotopy minimizer} {\bf Step 2, 3}) and allow us to apply the regularity result in \cite{schoen1981regularity}. 
Note that the codimension assumption \eqref{Eq: main thm - dimension assumption} is also used, and a similar trick is also applied in Section \ref{Subsec: plateau}.

\medskip
Section \ref{Sec: G-min-max} is devoted to the equivariant min-max theory. 
We provide a self-contained exposition including all essential details. 
In the regularity theory (Section \ref{Subsec: regularity min-max}), we introduce the {\em $G$-replacement chain property} (Definition \ref{Def: G-replacement chain}) analogous to \cite{wangzc2023four-sphere}*{Definition 3.6}, which is also very similar to the good $G$-replacement property in \cite{wang2024index}*{Definition 4.14}. 
However, the construction of $G$-replacements requires delicate handling (see Proposition \ref{Prop: constrained minimizer}, \ref{Prop: amv and replacement}). 
A primary technical obstacle is the absence of a general compactness theorem for $G$-stable minimal $G$-hypersurfaces. 
As we mentioned after Theorem \ref{Main Thm: equivariant min-max}, we shall use the non-existence of special exceptional orbits and the boundary regularity result \cite{delellis2010genus}*{Lemma 8.1} to derive the stability from $G$-stability of the constrained $G$-isotopy area minimizers (Proposition \ref{Prop: constrained minimizer}). 

\medskip
Given the theoretical framework, we present in Section \ref{Sec: spherical Bernstein problem} a new solution to the Spherical Bernstein Problem in $\mb S^4$. 
As previously noted, controlling the number of boundary components of the quotient $\Sigma/G$ for a min-max $G$-hypersurface $\Sigma$ currently poses a technical challenge. 
Nevertheless, we can combine the $\mb S^3$-suspension structure of $\mb S^4$ and the Hopf fibration of $\mb S^3$ to construct the suspended Hopf $S^1$-action, which yields a boundaryless orbit space $\mb S^4/S^1$. 
Indeed, $\mb S^4/S^1$ looks like an American football, i.e. $S^3$ with two singular points. 
Next, by a group extension using the antipodal symmetry on $\mb S^4/S^1$, we obtain a non-product Lie group $G=(S^1\rtimes \mb Z_4)/\mb Z_2$ that further reduces the orbit space to $RP^3=\mb S^4/G$. 
In this setting, we find that there exists exactly one $G$-invariant equatorial hypersphere $E_0$ as the unique area minimizer among all embedded $G$-hyperspheres $\Sigma$ with $\Sigma/G\cong RP^2$. 
Additionally, we conclude from \cite{laudenbach1974topologie} that a $G$-hypersurface $\Gamma$ is a hypersphere provided that $\Gamma/G\cong RP^2$ is embedded in the regular part of $\mb S^4/G$.
Combined with the min-max construction for minimal $RP^2$ in $RP^3$ (\cite{haslhofer2019sphere}\cite{li2025RP2}), we can apply the equivariant min-max theory (Theorem \ref{Main Thm: equivariant min-max}) to obtain a minimal $G$-hypersphere distinct from $E_0$. 
By a similar strategy, we also find a minimal hypertorus in $\mb S^4$. 
The group extension from $S^1$ to $G$ is very helpful since the topology of an $S^1$-hypersurface $\Sigma\subset \mb S^4$ not only depends on $\genus(\Sigma/S^1)$ but also is related to whether $\Sigma/S^1$ separates the two singular points in $\mb S^4/S^1$. 
Given this observation, we can have another group extension $G'=S^1\rtimes \mb Z_2$ so that $\mb S^4/G' = (\mb S^4/S^1)/\mb Z_2$ is a topological $3$-sphere, and the two singular points of $\mb S^4/S^1$ are glued together in $(\mb S^4/S^1)/\mb Z_2$. 
Combined with the min-max construction for minimal $S^2$ in $S^3$ (\cite{smith1983existence}), we can apply Theorem \ref{Main Thm: equivariant min-max} to have a minimal $S^1\times S^2$ in $\mb S^4$. 

\begin{acknowledgements}
    T.W. thanks Zhihan Wang and Xingzhe Li for some helpful discussions. 
    Part of this work was done when T.W. visited Prof. Xin Zhou at Cornell University; he thank them for their hospitality. 
    T.W. is supported by the National Natural Science Foundation of China 12501076 and the Natural Science Foundation of Shanghai 25ZR1402252. 
    X.Z. acknowledges the support by NSF grant DMS-2506717 and a grant from the Simons Foundation.
\end{acknowledgements}

\section{Preliminary}\label{Sec: preliminary}

\subsection{Group actions}

Let $(M,g_{_M})$ be a closed connected Riemannian $(n+1)$-manifold ($n+1\geq 3$) and let $G$ be a compact Lie group acting isometrically on $M$. 
Let $\mu$ be a bi-invariant Haar measure on $G$ with $\mu(G)=1$. 
Without loss of generality, we always assume the $G$-action is effective on $M$. 
We now collect some basic definitions and notations about the action of Lie groups and refer \cite{berndt2016submanifolds}\cite{bredon1972introduction}\cite{wall2016differential} for more details. 

By \cite{moore1980equivariant}, $M$ can be $G$-equivariantly isometrically embedded into some $\mb R^L$. 
Namely, there exist an orthogonal representation $\rho:G\to O(L)$ and an isometric embedding $i$ from $M$ to $\mb R^L$, which is $G$-equivariant, i.e. $i(g\cdot x) = \rho(g)\cdot i(x) $. 
For simplicity, we always denote the action of $g$ on $p$ by $g\cdot p:=\rho(g)(p)$ for every $p\in\mb R^L,g\in G$. 

For any $p\in M$, we define the {\em isotropy group} of $p$ in $G$ by 
\[G_p := \{g\in G: g\cdot p = p\},\] 
and denote by $(G_p)=\{g\cdot G_p\cdot g^{-1}:g\in G\}$ the conjugacy class of $G_p$ in $G$. 
Then we say $p\in M$ has the {\em $(G_p)$ orbit type}. 
Let 
\[M_{(G_p)}:= \{q\in M: (G_q)=(G_p)\}\]
be the union of points with $(G_p)$ orbit type, which is a disjoint union of smooth embedded submanifolds of $M$ (\cite{bredon1972introduction}*{Chapter 6, Corollary 2.5}). 
Additionally, there exists a (unique) minimal conjugacy class of isotropy groups $(P)$ such that 
\[M^{prin}:=M_{(P)}\] forms an open dense submanifold of $M$, which is known as the {\em principal orbit type}. 
The {\em cohomogeneity} ${\rm Cohom}(G)$ of $G$ is defined as the codimension of a principal orbit, i.e. 
\begin{align}\label{Eq: cohomogeneity}
	\Cohom(G):=\min\{\codim(G\cdot p): p\in M \}=\codim(G\cdot p_0)\quad\mbox{for $p_0\in M^{prin}$}.
\end{align}
Given an orbit $G\cdot p\subset M\setminus M^{prin}$, we say 
\begin{itemize}
	\item $G\cdot p$ is an {\em exceptional orbit} if $\dim(G\cdot p)=n+1-\cohom(G)=\dim(G\cdot p_0)$ for $p_0\in M^{prin}$; 
	\item $G\cdot p$ is a {\em singular orbit} if $\dim(G\cdot p)<n+1-\cohom(G)=\dim(G\cdot p_0)$ for $p_0\in M^{prin}$. 
\end{itemize}
An orbit type $(H)$ is said to be exceptional (resp. singular) if $G\cdot p\subset M_{(H)}$ is exceptional (resp. singular). 
For simplicity, we also introduce the following definition. 

\begin{definition}\label{Def: G-connected}
	For any $G$-invariant subset $U\subset M$ with connected components $\{U_i\}_{i=1}^I$, we say $U$ is {\em $G$-connected} if for any $i,j\in\{1,\dots,I\}$, there is $g_{i,j}\in G$ with $g_{i,j}\cdot U_j=U_i$. 
	We say $U$ is a {\em $G$-component} of $V$ if $U$ is a $G$-connected union of components of $V$. 
\end{definition}

Let $\pi: M\to M/G$ be the natural quotient map. 
In the rest of this paper, we always assume that the $G$-action on $M$ has
\begin{align}\label{Eq: cohomogeneity assumption}
	\cohom(G) = 3 \qquad{\rm and}\qquad 3\leq \codim(G\cdot p)\leq 7,~\forall p\in M.
\end{align}
We also use the following notations in this paper:
\begin{itemize}
    \item $\mb B^k,\mb B^k_+$: $k$-dimensional Euclidean ball and half-ball (containing the flat boundary);
    \item $\mb D^2=\mb B^2, \mb D^2_+=\mb B^2_+$; 
    \item $T_p(G\cdot p), N_p(G\cdot p)$: the space of vectors in $T_pM$ that are tangent/normal to $G\cdot p$;  
    \item $T(G\cdot p), N(G\cdot p)$: the tangent/normal vector bundles over $G\cdot p$;  
    \item $\exp_{G\cdot p}^\perp$: the normal exponential map from $N(G\cdot p)$ to $M$; 
    \item $\inj(G\cdot p)$: the injectivity radius of $\exp_{G\cdot p}^\perp$;
    \item $S_r(p):=\exp_{G\cdot p}^\perp(\mb B_r(0)\cap N_p(G\cdot p))$: the (geodesic) $r$-ball in the slice of $G\cdot p$ at $p$; 
    \item $B_r(G\cdot p)$: the geodesic $r$-tube centered at $G\cdot p$;
    \item $A_{s,t}(G\cdot p):=B_{t}(G\cdot p)\setminus \closure(B_{s}(G\cdot p))$ for $0<s<t$;
    \item $T(p,s,t):=\{x\in B_t(G\cdot p): \dist_{G\cdot p}(\mf n(x), p)<s \}$ is the part of tube $B_t(G\cdot p)$ centered at an $s$-neighborhood of $p$ in $G\cdot p$, where $\mf n: B_t(G\cdot p) \to G\cdot p$ is the nearest projection;
    \item $[p]:=\pi(G\cdot p)$;
    \item  $B_r^*([p]):=\pi(B_r(G\cdot p))$;
    \item $A^*:=\pi(A)$ for any $G$-subset $A\subset M$. 
\end{itemize}
In any open $G$-set $U\subset M$, we denote by 
\begin{itemize}
	\item $\mk X(U)$ the space of smooth vector fields supported in $U$;
	\item $\mk X^G(U):=\{X\in \mk X(U): dg(X)=X,~\forall g\in G\}$ the space of $G$-vector fields in $U$;
	\item $\mk {Is}(U)$ the space of isotopies supported in $U$;
	\item $\mk {Is}^G(U):=\{\{\varphi_t\}_{t\in [0,1]}\in \mk {Is}(U): \varphi_t\circ g=g\circ \varphi_t ~\mbox{for all } t\in [0,1] \mbox{ and } g\in G\}$ the space of $G$-isotopies in $U$;
	\item $J^G_U(\Sigma):=\{\varphi_1(\Sigma): \{\varphi_t\}_{t\in [0,1]}\in \mk {Is}^G(U)\}$ for any $G$-hypersurface $\Sigma$.
\end{itemize}
We also have the following definition and basic results for Lie group actions.

\begin{definition}[Special exceptional orbits \cite{bredon1972introduction}*{P. 185}]\label{Def: special exceptional orbit}
    An exceptional orbit $G\cdot p\subset M$ is said to be a {\em special exceptional orbit} in $M$ if the space of fixed vectors $\{v\in N_p(G\cdot p): G_p\cdot v=\{v\}\} $ has codimension one in $N_p(G\cdot p)$, where $g\in G_p$ acts on $v\in N_p(G\cdot p)$ by $g\cdot v:= dg(v)$. 
\end{definition}

\begin{lemma}\label{Lem: prin orbit stratum}
	For the slice $S_r(p)$, $S_r(p)/G_p$ is homeomorphic to $B_r(G\cdot p)/G$. 
	Moreover,
    \begin{itemize}
        \item[(i)] $M^{prin}$ (resp. $\pi(M^{prin})$) is an open dense subset in $M$ (resp. $\pi(M)$).
        \item[(ii)] For each orbit type stratum $M_{(H)}$, $\pi:M_{(H)}\to\pi(M_{(H)})$ is a smooth submersion. 
        \item[(iii)] $\pi(M\setminus M^{prin})$ does not locally disconnect the orbit space $\pi(M)=M/G$, i.e. a connected neighborhood of $[p]\in M/G$ remains connected after removing all $\pi(M\setminus M^{prin})$. 
        \item[(iv)] If ${\cohom}(G)=2$, then $\pi(M)$ is a topological $2$-manifold (possibly with boundary); if $\cohom(G)=3$, then $M^*$ is a topological $3$-manifold (possibly with boundary) away from at most finitely many points. 
    \end{itemize}
\end{lemma}
\begin{proof}
	The first statement comes from \cite{bredon1972introduction}*{P. 84, Proposition 4.7}. 
	(i) and (ii) follow from \cite{bredon1972introduction}*{P. 179, Theorem 3.1} and \cite{bredon1972introduction}*{P.182, Theorem 3.3}. 
    (iv) follows from \cite{bredon1972introduction}*{Page 187, Lemma 4.1, Theorem 4.3}. 
    The statement in (iii) follows easily from an induction argument. 
    Firstly, note that $B_r(G\cdot p)/G$ is homeomorphic to $S_r(p)/G_p$ and $S_r(p)\cong N_p(G\cdot p)$ with $G_p$ acting orthogonally. 
    Hence, it is sufficient to show $ (N_p(G\cdot p)\setminus (N_p(G\cdot p))^{prin} )/G_p$ cannot locally disconnect $(N_p(G\cdot p))/G_p$. 
    If $N_p(G\cdot p)\cong \mb R$, then the only non-trivial choice of $G_p$ is $O(1)=\mb Z_2$. 
    Hence, $(\mb R\setminus \mb R^{prin})/\mb Z_2=\{0\}$ cannot disconnect $\mb R/\mb Z_2=[0,\infty)$. 
    Suppose the claim has been verified in all dimensions at most $n$. 
    Then for $N_p(G\cdot p)\cong \mb R^{n+1}$, since $v$ and $tv$ with $t>0$ have the same orbit type, it is sufficient to show that $(\mb S^n_1 \setminus (\mb S^n_1)^{prin} )/G_p$ cannot locally disconnect $\mb S^n_1/G_p$, which is true by the induction assumption. 
\end{proof}

\subsection{Local structure of $M/G$}\label{Subsect: preliminary - M/G}

In this subsection, we classify the local structure of the orbit space $M/G$ using the tangent cones. 
Recall that $N_p(G\cdot p):=\{v\in T_pM: v\perp T_p(G\cdot p)\}$ is the normal vector space of $G\cdot p$ at $p$. 
Let 
\begin{align}\label{Eq: P}
	P(p):=N_p(G\cdot p)\cap T_pM_{(G_p)}.
\end{align}
One can easily verify (by the slice theorem \cite{bredon1972introduction}*{P.82, Theorem 4.2}) that
\begin{align}\label{Eq: P is fixed vectors}
	P(p)=\{v\in N_p(G\cdot p): dg(v)=v,~ \forall g\in G_p \}.
\end{align}
Then,  we can write
\begin{align}\label{Eq: P^perp}
	N_p(G\cdot p)=P(p)\oplus P^\perp(p)
\end{align}
where $P^\perp(p)$ is the orthogonal complement of $P(p)\subset N_p(G\cdot p)$. 
For $p\in M$ and $0<r<\inj(G\cdot p)$, $\exp_{G\cdot p}^\perp$ is a $G$-equivariant diffeomorphism in $B_r(G\cdot p)$, which indicates the homeomorphisms:
\begin{align}\label{Eq: product represent orbit space}
    B_r^*([p])=B_r(G\cdot p)/G ~\overset{\rm homeo}{\cong}~ N_p(G\cdot p)/G_p ~\overset{\rm homeo}{\cong}~ P(p) \times (P^\perp(p)/G_p).
\end{align}
Hence, we can use the cone $P^\perp(p)/G_p$ to classify the local structure $B_r^*([p])$ in $M/G$. 

\begin{remark}\label{Rem: actions in P^perp}
    We also make the following remarks on the $G_p$-action on $P^\perp(p)$:
    \begin{itemize}
        \item[(i)] $0\in N_p(G\cdot p)$ is the only $G_p$-fixed point in $P^\perp(p)$;
        \item[(ii)] for any $x\in N_p(G\cdot p)$ and $t>0$, $x$ and $tx$ have the same isotropy group in $G_p$, and thus $P^\perp(p)/G_p$ is a family of rays;
        \item[(iii)] if $\dim(N_p(G\cdot p))=3$ and $\dim(P(p))=2$, then $G\cdot p$ is a special exceptional orbit in $M$. 
    \end{itemize}
\end{remark}

\begin{proposition}\label{Prop: local structure of M/G}
	Given a $G$-manifold $M$ with $\cohom(G)=3$, let $p\in M$, $0<r<\inj(G\cdot p)$, $P(p)$, and $P^\perp(p)$ be defined as above. Then, we can classify $M/G$ in $B_r^*([p])$ by four cases:
	\begin{itemize}
		\item[(0)] if $\dim(P(p))=3$, then $M/G$ is a smooth $3$-manifold in $B_r^*([p])$;
		\item[(1)] if $\dim(P(p))=2$, then $M/G$ is a $3$-manifold with boundary $M_{(G_p)}^*$ in $B_r^*([p])$ so that $B_r^*([p])\setminus M_{(G_p)}^*\subset (M^{prin})^*$;
		\item[(2)] if $\dim(P(p))=1$, then either
			\begin{itemize}
				\item[(a)] $B_r^*([p])$ is homeomorphic to a $3$-dimensional Euclidean wedge with edge $\bd^EB_r^*([p])$ given by $ M_{(G_p)}^* \cap B_r^*([p])$ and two faces $\bd^F_\pm B_r^*([p])$ given by $M_\pm^*\cap B_r^*([p])$ respectively so that $B_r^*([p])\setminus (\bd^E B_r^*([p])\cup \bd^F B_r^*([p]))\subset (M^{prin})^*$, where $M_\pm\subset B_{r}(G\cdot p)\setminus (M^{prin}\cup M_{(G_p)})$ are two $G$-components of non-principal orbit type strata; or
				\item[(b)] $B_r^*([p])$ is homeomorphic to a $3$-dimensional closed cone with a $1$-dimensional spine given by $M_{(G_p)}^*$ so that $B_r^*([p])\setminus M_{(G_p)}^*\subset (M^{prin})^*$, and thus $B_r^*([p])$ is homeomorphic to $\mb R^3$; 
                moreover, $G\cdot p$ is an exceptional orbit in $M$, i.e. $\codim(G\cdot p)=3 $; 
			\end{itemize}
		\item[(3)] if $\dim(P(p))=0$, 
        then $B_r^*([p])$ is homeomorphic to either $\mb B^3$, $\mb B^3_+$, or a $3$-dimensional cone over $\mb {RP}^2$.
	\end{itemize}
	In particular, $M^*=\pi(M)$ is a topological $3$-manifold away from a finite set $\{[p_i]\}$ where $p_i$ satisfies (3) and $B_r^*([p_i])$ is a cone over $\mb {RP}^2$. 
\end{proposition}
\begin{proof}
	By Lemma \ref{Lem: prin orbit stratum}, $\dim(P(p))\in\{0,1,2,3\}$. 
	If $\dim(P(p))=3$, then we have $p\in M^{prin}$ by $\Cohom(G)=3$. 
	Thus, $P(p)=N_p(G\cdot p)$ and $\pi:B_r(G\cdot p)\to B_r^*([p])$ is a Riemannian submersion. 
	
	If $\dim(P(p))=2$, then by Remark \ref{Rem: actions in P^perp}(ii), $P^\perp(p)/G_p$ is formed by a family of rays $\{l_i\}_{i\in I}$. 
	Suppose there are two different rays $l_1,l_2$ with  $(l_1\cup l_2)\setminus\{0\}\subset (P^\perp(p))^{prin}/G_p$. Then, $v_1\in l_1\setminus\{0\}$ and $v_2\in l_2\setminus\{0\}$ can be connected by a curve in $(P^\perp(p))^{prin}/G_p$ by Lemma \ref{Lem: prin orbit stratum}(iii), which implies $\dim(P^\perp(p)/G_p)\geq 2$. 
	This contradicts the assumption 
	\[3={\rm Cohom}(G) = \dim((N_p(G\cdot p))^{prin}/G_p) = \dim(P(p))+ \dim((P^\perp(p))^{prin}/G_p)\geq 4 . \]
	Therefore, $(P^\perp(p))^{prin}/G_p\cong (0,\infty)$ is a single (open) ray. 
	By Lemma \ref{Lem: prin orbit stratum}(i) and (\ref{Eq: product represent orbit space}),  
	\begin{align}\label{Eq: local model 1}
	    B_r(G\cdot p)/G\overset{\rm homeo}{\cong} P(p)\times [0,\infty),
	\end{align}
	i.e. in a neighborhood of $G\cdot p$, $M_{(G_p)}^*$ is the boundary of the topological manifold $M^*$. 
    In particular, for any $v\in P^\perp(p)\setminus\{0\}$, $(0,v)\in P(p)\times P^\perp(p)$ is contained in $(N_p(G\cdot p))^{prin}$. 
    Noting $dg(u,v)=(u,dg(v))$ for all $g\in G_p$ and $(u,v)\in P(p)\times P^\perp(p)$, we conclude that $(u,v)\in (N_p(G\cdot p))^{prin}$ whenever $v\neq 0$, and thus $B_r^*([p])\setminus M_{(G_p)}^*\subset (M^{prin})^*$. 
	
	If $\dim(P(p))=1$, then by a similar argument to the previous case, we know $P^\perp(p)/G_p$ is either a cone over an arc (i.e. a $2$-dimensional Euclidean wedge) or a closed cone over a circle (i.e. $S^1\times[0,\infty)/(p,0)\sim(q,0)$ $\overset{\rm homeo}{\cong} \mb R^2 $), which implies (2). 
	One can also refer to  \cite{bredon1972introduction}*{P.206-207, Proposition 8.2, 8.3} for the classification of $P^\perp(p)/G_p$. 
    In particular, in Case (2.b), $P^\perp(p)/G_p$ is a cone over a circle $SP^\perp(p)/G_p$, where $SP^\perp(p)$ is the unit sphere in $P^\perp(p)$. 
    By Lemma \ref{Lem: prin orbit stratum}(iii), $SP^\perp(p)$ contains only principal $G_p$-orbits, which implies $SP^\perp(p)$ is a fiber bundle over $S^1\cong SP^\perp(p)/G_p$. 
    Since none of $\mb S^{n\geq 2}$ can fiber over a circle, we know $SP^\perp(p)=\mb S^1$, and thus $\codim(G\cdot p)=\dim(P(p))+ \dim(P^\perp(p))=1+(1+\dim(SP^\perp(p)))=3$. 
    Indeed, $G_p$ is a cyclic group of rotations in this case by the arguments before \cite{bredon1972introduction}*{P. 207, Theorem 8.2}. 
	
	If $\dim(P(p))=0$, 
    then $B_r^*([p])$ is homeomorphic to a $3$-dimensional cone over $SN_p(G\cdot p)/G_p$, where $SN_p(G\cdot p)$ is the Euclidean unit sphere in $N_p(G\cdot p)$. 
    Since the $G_p$ action on $SN_p(G\cdot p)$ has cohomogeneity $2$, we know $SN_p(G\cdot p)/G_p$ is a topological $2$-surface (Lemma \ref{Lem: prin orbit stratum}). 
    Additionally, by \cite{Burago_1992}*{P.16 Corollary}, $SN_p(G\cdot p)/G_p$ is also an Alexandrov space with curvature $\geq 1$.  If $SN_p(G\cdot p)/G_p$ has boundary, one can obtain a positively curved Alexandrov space without boundary by the doubling theorem \cite{Burago_1992}*{P.54 Theorem d)}.  
    Hence, its universal cover is a positively curved closed surface, which can only be $S^2$. 
    In particular, $SN_p(G\cdot p)/G_p$ is homeomorphic to either $\mb S^2$, $\mb D^2$, or $\mb {RP}^2$. 
    This implies (3) and the last statement (cf. \cite{burago2001course}*{Exercise 10.10.4}). 
\end{proof}

\begin{example}
    We list some examples of the situations described in Proposition \ref{Prop: local structure of M/G}.
    \begin{itemize}
        \item For $m\geq 2$, if $g\in SO(m)$ acts on $(x,y)\in \mb R^2\times \mb R^{m}$ by $g\cdot (x,y)=(x, g\cdot y)$, then every $p\in \{(x,0)\in \mb R^{2+m}\}$ satisfies Proposition \ref{Prop: local structure of M/G}(1). 
        \item For $k,l\geq 2$, if $(g,h)\in SO(k)\times SO(l)$ acts on $(x,y,z)\in \mb R \times \mb R^{k}\times \mb R^l$ by $(g,h)\cdot (x,y,z):=(x,g\cdot y, h\cdot z)$, then every $p\in \{(x,0,0)\in \mb R^{1+k+l}\}$ satisfies Proposition \ref{Prop: local structure of M/G}(2.a). 
        \item For $k\geq 2$, if the generator $[1]\in \mb Z_k$ acts on $(x,z)\in \mb R\times \mb C$ by $[1]\cdot (x,z)=(x, e^{2\pi i/k}\cdot z)$, then every $p\in \mb R\times \{0\}$ satisfies Proposition \ref{Prop: local structure of M/G}(2.b). 
        \item If $\mb Z_2$ acts on $\mb R^3$ by $[0]=id$ and $[1]\cdot x=-x$, then $0\in \mb R^3$ satisfies Proposition \ref{Prop: local structure of M/G}(3). 
        Note that $\mb R^3/\mb Z_2$ is a cone over $\mb {RP}^2$, which fails to be a $C^0$ manifold at the vertex. 
        \item If $S^1$ acts on $\mb R^4$ by the Hopf action on each $\mb S^3_r(0)$, then $0\in \mb R^4$ satisfies Proposition \ref{Prop: local structure of M/G}(3). 
        Now, $\mb R^4/S^1$ is a cone over $S^2$, which is a $C^0$ manifold at the vertex. 
    \end{itemize}
    In particular, if $p\in M$ satisfies Proposition \ref{Prop: local structure of M/G}(3), $M^*$ may still be a $C^0$ manifold at $[p]$.
\end{example}

\begin{definition}\label{Def: isolated orbits}
    An orbit $G\cdot p\subset M$ is said to be an {\em isolated non-principal orbit} if $N_p(G\cdot p)\cap T_pM_{(G_p)}=\{0\}$, i.e. $p$ satisfies Proposition \ref{Prop: local structure of M/G}(3). 
    Denote by 
    \begin{itemize}
        \item $\mc S\subset M$ the (finite) union of isolated non-principal orbits;
        \item $\mc S_{n.m.}\subseteq \mc S$ the (finite) union of isolated non-principal orbits $G\cdot p$ so that $M^*$ fails to be a topological $3$-manifold at $[p]$, i.e. $B_r^*([p])$ is homeomorphic to a cone over $\mb {RP}^2$. 
    \end{itemize}
\end{definition}
Although $M^*$ is a $C^0$ manifold at $[p]\in \mc S^*\setminus \mc S_{n.m.}^*$, the local structure of $M$ near $G\cdot p$ is generally hard to analyze. 
Hence, we will often first work in $M\setminus \mc S$, and then consider the orbits in $\mc S$. 
Additionally, in our applications, we will choose group actions with $\mc S_{n.m.}=\emptyset$.

\begin{lemma}\label{Lem: local structure of M/G with codim 3}
	Suppose $B_r(G\cdot p)$ satisfies Proposition \ref{Prop: local structure of M/G}(1) or (2.a). 
	Suppose also that there exists $x\in B_r(G\cdot p)\setminus M^{prin}$ with $\codim(G\cdot x)=3=\Cohom(G)$. 
	Then $N_p(G\cdot p)$ is equally separated by $N_{n.p.}$ into $k$ chambers for some $k\in \{2,3,...\}$, where $N_{n.p.}$ is the union of non-principal $G_p$-orbits in $N_p(G\cdot p)$. 
	In particular, in the slice $S_r(p)$ with $G_p$-actions, $(S_r(p))^{prin}$ has $k\geq 2$ connected components. 
\end{lemma}
\begin{proof}
	Since we are in the case of Proposition \ref{Prop: local structure of M/G}(1) or (2.a), the codimension assumption implies the existence of a plane $W^*$ or half-plane $W_+^*$ on the boundary of $N_p(G\cdot p)/G_p$ so that $\codim(G_p\cdot x)=3=\Cohom(G_p)$ in $N_p(G\cdot p)$ whenever $[x]\in W^*$ or $W^*_+$. 
	Let $\widehat W, \widehat W_+\subset N_p(G\cdot p)$ be the pre-images  of $W^*,W_+^*$, respectively. 
	Then $\dim(\widehat W)=\dim(W^*)+\dim(G_p\cdot x)=2+(n+1-\dim(G\cdot p)-\Cohom(G_p))=n-\dim(G\cdot p)$, and similarly $\dim(\widehat{W}_+ )=n-\dim(G\cdot p)$, so $\widehat W$ and $\widehat W_+$ have codimension $1$ in $N_p(G\cdot p)$. 
	Note that $\widehat W,\widehat W_+$ are the unions of some subspaces in $N_p(G\cdot p)$ since $G_p$ acts orthogonally on $N_p(G\cdot p)$.  
	Therefore, $\widehat W$ and $\widehat W_+$ contain at least a hyperplane in $N_p(G\cdot p)$, and thus $N_p(G\cdot p)\setminus N_{n.p.}$ has $k$ components for some $k\geq 2$. 
	Clearly, these components are isometric to each other by the orthogonal $G_p$-actions. 
\end{proof}

\subsection{Local structure of $\Sigma/G$ and notations for $G$-hypersurfaces}\label{Subsec: preliminary - Sigma/G}
Given $p\in M$ and $0<r<\inj(G\cdot p)$, let $\Sigma$ be an embedded $G$-hypersurface in $B_r(G\cdot p)$ with $\bd\Sigma\cap B_r(G\cdot p)=\emptyset$. 
We classify the local structure of $\Sigma^*=\Sigma/G$ near $[p]$ using the tangent cones. 
If $p\in M^{prin}$, then $\Sigma^*$ is a smooth surface in the smooth $3$-manifold $B_r^*([p])$ by Lemma \ref{Lem: prin orbit stratum}(i). 
Hence, we only consider $p\in M\setminus M^{prin}$. 
To begin with, the following lemma classifies the intersection between $\Sigma$ and the non-principal orbit type stratum $M_{(G_p)}$. 

\begin{lemma}\label{Lem: hypersurface position}
    Given $p\in M\setminus M^{prin}$ and $\Sigma\subset B_r(G\cdot p)$ as above, if $\Sigma\cap M_{(G_p)}\neq \emptyset$, then either
    \begin{itemize}
        \item $\Sigma$ meets $M_{(G_p)}\cap B_r(G\cdot p)$ orthogonally, or
        \item $M_{(G_p)}\cap B_r(G\cdot p) \subset \Sigma$. 
    \end{itemize}
\end{lemma}
\begin{proof}
	Suppose $\Sigma$ does not meet $M_{(G_p)}$ orthogonally at $G\cdot q\subset \Sigma\cap M_{(G_p)}$. 
	Then we have $\nu\notin T_qM_{(G_p)}$, where $\nu$ is a unit normal of $\Sigma$ at $q$. 
	By combining \eqref{Eq: P is fixed vectors} with the $G$-invariance of $\Sigma$, there exists $g_0\in G_q$ so that $\nu\neq dg_0(\nu) = -\nu $. 
	Decompose $\nu=\nu_\parallel +\nu_\perp$ with $\nu_\parallel\in T_qM_{(G_p)}$ and $\nu_\perp\in (T_q M_{(G_p)})^\perp $. 
	Then $dg_0(\nu_\perp)=dg_0(\nu-\nu_\parallel)=-\nu-\nu_\parallel$. 
	Together with $|\nu_\perp|=|dg_0(\nu_\perp)|$, we know $\langle\nu, \nu_\parallel\rangle=0$, and thus $\nu=\nu_\perp\in (T_q M_{(G_p)})^\perp$, which implies $\Sigma$ meets $M_{(G_p)}$ tangentially along $G\cdot q$. 
	In particular, $T_qM_{(G_p)}\subseteq T_q\Sigma$. 
	
	For $r>0$ small enough, let $S_q:=S_r(q)$ be the slice of $G\cdot q$ at $q$, and let $\Sigma_q:=\Sigma\cap S_q$, $\Sigma_q':=\exp_q^{-1}(\Sigma_q)$, $A_q:=M_{(G_p)}\cap S_q$, $A_q':=\exp_q^{-1}(A_q)$. 
	Since $A_q'\subset   T_qS_q=N_q(G\cdot q)$ is the $G_q$-fixed-point set, we know 
	\[A_q'\subset P(q)= T_qM_{(G_p)}\cap N_q(G\cdot q)\subseteq T_0\Sigma_q'.\] 
	Additionally, $\Sigma_q'=\Graph(f):=\{x+f(x)\nu: x\in T_0\Sigma_q'\}$ near $0\in N_q(G\cdot q)$ for some $f\in C^\infty(T_0\Sigma_q')$. 
	Hence, by the $G_q$-invariance of $\Sigma_q'$, $dg_0(x+f(x)\nu)= x- f(x)\nu\in \Sigma_q'$ for any $x\in A_q'$, which implies $f\llcorner A_q'=0$ and $A_q'\subset \Sigma_q'$ in a neighborhood of $0$. 
	Therefore, $A_q\subset \Sigma_q$ in a neighborhood of $q$, and $M_{(G_p)}\cap \Sigma$ is open in $M_{(G_p)}$. 
	Note that $M_{(G_p)}\cap \Sigma$ is also closed in $M_{(G_p)}$ by continuity, and $M_{(G_p)}\cap B_r(G\cdot p)$ is $G$-connected. 
	Hence, $M_{(G_p)}\cap B_r(G\cdot p)\subset \Sigma$. 
\end{proof}

\begin{proposition}\label{Prop: local structure of Sigma/G 1}
	Given $p\in M\setminus M^{prin}$ with $P(p)$ and $P^\perp(p)$ defined in \eqref{Eq: P}\eqref{Eq: P^perp} respectively, let $\Sigma\subset B_r(G\cdot p)$ be a $G$-connected embedded $G$-hypersurface and $G\cdot q\subset \Sigma\cap M_{(G_p)}$. 
	Suppose \[\dim(P(p))=2.\]
	Then, in a small $\epsilon$-neighborhood $B_\epsilon^*([q])$ of $[q]$, $\Sigma^*=\Sigma/G$ is a surface in $B_\epsilon^*([q])$ so that either
	\begin{itemize}
		\item[(i)] $\Sigma^*$ has free boundary on $M_{(G_p)}^*\cap B_\epsilon^*([q])$; or 
		\item[(ii)] $\Sigma^*=M_{(G_p)}^* \cap B_\epsilon^*([q])$.  
	\end{itemize}
	In case (ii), we have $\dim(P^\perp(p))=1$, $\dim(G\cdot p)=n+1-\cohom(G)=n-2$, $\Sigma\cap B_\epsilon(G\cdot q)=M_{(G_p)}\cap B_\epsilon(G\cdot q)$ is a smooth minimal $G$-hypersurface, and $G_q$ acts on $N_q(G\cdot q)$ by reflection across $P(q)$, i.e. $G_q\cdot (u,v)=\{(u,v), (u,-v)\}$, where $u\in P(q)\cong \mb R^2$ and $v\in P^\perp(q)\cong\mb R$. 
\end{proposition}

In case (ii), $G\cdot q$ and thus $G\cdot p$ are {\em special exceptional orbits} of $M$ in Definition \ref{Def: special exceptional orbit}.

\begin{proof}
	The result in (i) corresponds to the cases in which $\Sigma$ meets $M_{(G_p)}$ orthogonally. 
	By Lemma \ref{Lem: hypersurface position}, the remaining situation is $M_{(G_p)}\cap B_r(G\cdot p)\subset\Sigma$. 
	Consider the slice $S_{\epsilon}(q)\cong \mb B_{\epsilon}^{\codim(G\cdot q)}(0)$ of $G\cdot q$ at $q$. 
	Since $(M\setminus M^{prin})\cap B_r(G\cdot p)=M_{(G_p)}\cap B_r(G\cdot p)$, we know $\Sigma_{q}:=\Sigma\cap S_{\epsilon}(q)$ has a unit normal $\nu$ at $q$ which points into $S_\epsilon(q)\cap M^{prin}$. 
	Therefore, $1=\#(G_{q}\cdot q)< \#(G_{q}\cdot \nu)$. 
	Since $G_{q}\cdot \nu\subset \{\pm\nu\} $, we know $\#(G_{q}\cdot x)=2$ for any $x\in S_{\epsilon}(q)\cap M^{prin}$, and thus $\dim(G\cdot x)=\dim(G\cdot q)$ for $x\in M^{prin}$, i.e. $\dim(G\cdot q)=\dim(G\cdot p)=n+1-\cohom(G)$. 
	
	In addition, $(\exp_{G\cdot q}^\perp)^{-1}(M_{(G_p)}\cap S_\epsilon(q))=\{v\in N_{q}(G\cdot q): G_{q}\cdot v=\{v\}, |v|<\epsilon\}$ is contained in the $2$-plane $P(q)\subset N_{q}(G\cdot q)$ (since $\dim(P(q))=2$). 
	Hence, $\dim(M_{(G_p)})=\dim(G\cdot q)+ \dim(M_{(G_p)}\cap S_\epsilon(q)) = \dim(G\cdot q) + 2 = n=\dim(\Sigma)$, which implies $\dim(P^\perp(p))=\dim(P^\perp(q))=1$ and $\Sigma$ coincides with $M_{(G_p)}$ near $G\cdot q$. 
	Additionally, $M_{(G_p)}$ is minimal due to \cite[Corollary 1.1]{hsiang71cohom}. 
	Finally, since $P^\perp(q)\cong \mb R$ and $G_{q}\cdot \nu=\{\pm\nu\} $, we know the orthogonal $G_q$-action on $N_q(G\cdot q)$ is given by $G_q\cdot (u,v)=\{(u,v), (u,-v)\}$, where $u\in P(q)\cong \mb R^2$ and $v\in P^\perp(q)\cong\mb R$. 
\end{proof}

\begin{proposition}\label{Prop: local structure of Sigma/G 2}
	Given $p\in M\setminus M^{prin}$ with $P(p),P^\perp(p)$ defined in \eqref{Eq: P}\eqref{Eq: P^perp} respectively, let $\Sigma\subset B_r(G\cdot p)$ be a $G$-connected embedded $G$-hypersurface and $G\cdot q\subset \Sigma\cap M_{(G_p)}$. 
	Suppose \[\dim(P(p))=1.\] 
	\begin{itemize}
	    \item If $B_r^*([p])$ is homeomorphic to a $3$-dimensional Euclidean wedge (i.e. Proposition \ref{Prop: local structure of M/G}(2.a)), then in a small $\epsilon$-neighborhood $B_\epsilon^*([q])$ of $[q]$, either 
	    \begin{itemize}
		  \item[(1)] $\Sigma^*$ is a $2$-dimensional wedge with edge $\bd^E\Sigma^* =[q]=\Sigma^*\cap M_{(G_p)}^*\cap B_\epsilon^*([q])$ and two faces $\bd^F_\pm\Sigma^*=\bd^F_\pm B_\epsilon^*([q])\cap\Sigma^*$; or
		  \item[(2)] $\Sigma^*$ coincides with one face (say $\bd^F_-  B_\epsilon^*([q])$) of the wedge $B_\epsilon^*([q])$, and $\Sigma=\closure(M_-)$ in $B_\epsilon(G\cdot q)$ is a minimal $G$-hypersurface,
	   \end{itemize}
	   where $\bd^F_\pm B_\epsilon^*([q])$ and $M_\pm$ are defined as in Proposition \ref{Prop: local structure of M/G}(2.a) for $B_\epsilon^*([q])$. 
	   Additionally, in case (2), we have $\dim(G\cdot x)= \dim(G\cdot q)$ for all $x\in M_+$, $T_qM_+\cap P^\perp(q)$ is orthogonal to $\Sigma$ at $q$, and $M_-$ satisfies Proposition \ref{Prop: local structure of Sigma/G 1}(ii). 
    \item If $B_r^*([p])$ is homeomorphic to a $3$-dimensional cone with a $1$-dimensional spine (i.e. Proposition \ref{Prop: local structure of M/G}(2.b)), 
	   then in a small $\epsilon$-neighborhood $B_\epsilon^*([q])$ of $[q]$, either 
	   \begin{itemize}
		  \item[(3)] $\Sigma^*$ is a $2$-dimensional cone with vertex $[q]=\Sigma^*\cap M_{(G_p)}^* \cap B_\epsilon^*([q])$; or 
		  \item[(4)] $\Sigma^*$ is a $2$-dimensional half-plane with boundary $\Sigma^*\cap M_{(G_{p})}^*\cap B_\epsilon^*([q])$.
	   \end{itemize}
	   In case (4), we have $\dim(G\cdot q)=n-2$, and $G_q$ acts on $N_q(G\cdot q)$ by $G_q\cdot (u,v)=\{(u,v), (u,-v)\}$, where $u\in P(q)\cong \mb R$ and $v\in P^\perp(q)\cong\mb R^2$. 
	\end{itemize}
\end{proposition}

In case (4), $G\cdot p$ is a special exceptional orbit as a $G$-orbit in $\Sigma$ (not $M$). 
Note that $M_{(G_p)}^*\cap B_\epsilon^*([q])$ lies in the boundary of the $2$-manifold $\Sigma^*$ in cases (2)(4) (\cite{bredon1972introduction}*{P.207, Proposition 8.3}).

\begin{proof}
	Cases (1) and (3) correspond to the cases in which $\Sigma$ meets $M_{(G_p)}$ orthogonally at $G\cdot q$. 
	By Lemma \ref{Lem: hypersurface position}, the remaining situation is $M_{(G_p)}\cap B_r(G\cdot p)\subset\Sigma$. 
	In this case, the unit normal $\nu\in P^\perp(q)$ of $\Sigma$ at $q$ satisfies $G_q\cdot \nu\subset \{\nu,-\nu\}$. 
	
	If $B_\epsilon^*([q])$ is a $3$-dimensional wedge, then $\nu$ must point into $M\setminus M^{prin}$. (Otherwise, $\#(G_q\cdot \nu)\leq 2$ implies $(M\setminus M^{prin})\cap B_r(G\cdot p)= M_{(G_p)}\cap B_r(G\cdot p)$, which contradicts the wedge-shape assumption on $B_\epsilon^*([q])$.) 
	Hence, we know $G_q\cdot \nu=\{\nu,-\nu\}$ and $\nu$ is in $T_q\closure(M_\pm)$. 
	Suppose $\nu\in T_q\closure(M_+)$. 
    Then by $\dim(G_q\cdot \nu)=0$, we know $\dim(G\cdot x)=\dim(G\cdot q)$ for all $x\in M_+$. 
	 In addition, since $G_q$ acts orthogonally on $N_q(G\cdot q)$, every principal $G_q$-orbit in $N_q(G\cdot q)$ has two parts lying in $\{u:u\cdot \nu >0\}$ and $\{u:u\cdot \nu <0\}$ respectively. 
	 Hence, $\{u\in N_q(G\cdot q) : u\cdot \nu=0 \}=T_q\Sigma\cap N_q(G\cdot q) $ is in the set of non-principal $G_q$-orbits, which implies $T_q\Sigma=T_q\closure( M_-)$.
     Thus, $\Sigma=\closure( M_-)$ near $q$ by arguments similar to those in the proofs of Lemma \ref{Lem: hypersurface position} and Proposition \ref{Prop: local structure of Sigma/G 1}. 
	 This shows (2). 
	 
	 If $B_\epsilon^*([q])$ is a closed $3$-dimensional cone, 
	 then the set of non-principal $G_q$-orbits in $N_q(G\cdot q)$ is $P(q)$. 
	 Thus, $G_q\cdot\nu$ is principal in $N_q(G\cdot q)$, which implies every principal $G_q$-orbit in $N_q(G\cdot q)$ has $2$ points. 
	 In particular, we know $G\cdot q$ has the same dimension $n-2$ as the principal orbits, and $\dim(P^\perp(q))=2$. 
	 Note that $G_q$ acts on $\mb R^2\cong P^\perp(q)$ orthogonally with a single fixed point $0$ and $\#(G_q\cdot v)=2$ for all $v\in P^\perp(q)\setminus \{0\}$. 
	 Hence, $G_q$ must be the antipodal action on $\mb R^2\cong P^\perp(q)$. 
	 Combined with Lemma \ref{Lem: prin orbit stratum}(iv), we obtained $(4)$ and the last statements. 
\end{proof}

For any open $G$-set $U\subset M$ and a $G$-hypersurface $\Sigma\subset U$ with $\bd\Sigma\subset\bd U$, we define the following notations that are mainly used in $M\setminus \mc S$: 
\begin{itemize}
    \item $\bd_0U^* := \pi(U\setminus M^{prin})$, which is $\bd M^*\cap U^*$ in Proposition \ref{Prop: local structure of M/G}(1)(2.a);
    \item $\bd_1U^*:= \pi(\bd U\cap M^{prin})$, which is the relative boundary of $U^*$ in Proposition \ref{Prop: local structure of M/G}(1)(2.a);
    \item $\bd_0\Sigma^*:= \pi(\Sigma\setminus\Sigma^{prin})$;
    \item $\bd_1\Sigma^* := \pi(\bd\Sigma\cap\Sigma^{prin})$;
    \item $\bd_0\mb B^k_+:=\{x=(x_1,\dots,x_k)\in\mb R^k: |x|\leq 1,x_k=0\}$ is the flat boundary portion;
    \item $\bd_1\mb B^k_+:=\{x=(x_1,\dots,x_k)\in\mb R^k: |x|=1,x_k>0\}$ is the round boundary portion. 
\end{itemize}
We mention that $\Sigma^{prin}$ may not be in $M^{prin}$, but we still have
\[\bd_0 \Sigma^*\subset \bd_0 U^*,\quad{\rm and}\quad\bd_1 \Sigma^*\subset\closure(\bd_1 U^*).\]
Note that if $\dim(\bd_0B_r^*([p]))=2$ (resp. $= 1$), then $[p]=\pi(G\cdot p)$ is a boundary point (resp. interior point) of the topological manifold $M^*\setminus\mc S^*$, where $\mc S$ is defined in Definition \ref{Def: isolated orbits}. 
Hence, we define 
\begin{itemize}
    \item the {\em ball-type $G$-neighborhood $U$ of $p$}: \\
    a $G$-neighborhood of $G\cdot p$ in $M$ so that $\dim(\bd_0 U^*)\leq 1$ and $(U^*, \closure(\bd_1 U^*))$ is homeomorphic to $(\mb B^3, \bd\mb B^3)$;
    \item the {\em half-ball-type $G$-neighborhood $U$ of $p$}: \\
    a $G$-neighborhood of $G\cdot p$ in $M$ so that $\dim(\bd_0 U^*)=2$ and $(U^*, \bd_0 U^* ,\bd_1 U^*)$ is homeomorphic to $(\mb B^3_+, \bd_0\mb B^3_+, \bd_1\mb B^3_+)$;
    \item a {\em disk-type $G$-hypersurface}:\\
    an embedded $G$-hypersurface $\Sigma\subset M$ so that  $\dim(\bd_0\Sigma^*)=0$ or $\bd_0\Sigma^*=\emptyset$, and $(\Sigma^*,\closure(\bd_1\Sigma^*))$ is homeomorphic to $(\mb D^2, \bd\mb D^2)$;
    \item a {\em half-disk-type $G$-hypersurface}:\\
    an embedded $G$-hypersurface $\Sigma\subset M$ so that  $\dim(\bd_0\Sigma^*)=1$ and $(\Sigma^*, \bd_0\Sigma^*, \bd_1\Sigma^*)$ is homeomorphic to $(\mb D^2_+, \bd_0\mb D^2_+, \bd_1\mb D^2_+)$.
\end{itemize}

\begin{remark}
	We mention that for $U=B_r(G\cdot p)$, $U$ is a ball-type (resp. half-ball-type) $G$-neighborhood of $G\cdot p$ if it satisfies Proposition \ref{Prop: local structure of M/G} (0)(2.b) (resp. (1)(2.a)). 
	Additionally, a $G$-hypersurface $\Sigma\subset B_r(G\cdot p)$ that intersects $M_{(G_p)}$ at $G\cdot q$ is a disk-type (resp. half-disk-type) $G$-hypersurface near $G\cdot q$ if it satisfies Proposition \ref{Prop: local structure of Sigma/G 1}(ii) or Proposition \ref{Prop: local structure of Sigma/G 2}(3) (resp. Proposition \ref{Prop: local structure of Sigma/G 1}(i) or Proposition \ref{Prop: local structure of Sigma/G 2}(1)(2)(4)). 
\end{remark}

We also make the following definitions.
\begin{definition}\label{Def: local G-boundary}
	Let $\Sigma\subset M$ be an embedded $G$-hypersurface. 
	If for any $p\in M$, there exists $r\in(\inj(G\cdot p)/2, \inj(G\cdot p))$ so that $\Sigma\llcorner B_r(G\cdot p)\subset\bd\Omega$ for some $G$-invariant Caccioppoli set $\Omega\subset B_r(G\cdot p)$ with reduced boundary $\bd\Omega$, then we say $\Sigma$ is {\em locally $G$-boundary-type} in $M$. 
	Denote by $\mc {LB}^G$ the set of locally $G$-boundary-type compact embedded $G$-hypersurfaces in $M$.  
\end{definition}

\begin{remark}
	The locally $G$-boundary-type condition is a mild assumption, which is analogous to the fact that every embedded hypersurface is locally a boundary. 
	Under this condition, the local structures in Proposition \ref{Prop: local structure of Sigma/G 1}(ii) and Proposition \ref{Prop: local structure of Sigma/G 2}(2)(4) cannot occur. 
	However, after taking the varifold limit, the limit varifold may not be locally $G$-boundary-type even when it is a smooth $G$-hypersurface. 
	Nevertheless, we will show that the locally $G$-boundary-type condition remains valid at least in $M\setminus \mc S_{n.m.}$ for certain area minimizers (Theorem \ref{Thm: plateau problem}, \ref{Thm: G-isotopy minimizer}). 
    Also, Definition \ref{Def: local G-boundary} can be localized to any open $G$-set $U\subset M$ by using $U$ in place of $M$.
\end{remark}

By Lemma \ref{Lem: prin orbit stratum}(iv), we can define the genus of $\Sigma^*$ in the usual way for any $G$-hypersurface $\Sigma$. 
 Given any open $G$-set $U\subset M$, we also use the following notations for simplicity:
\begin{itemize}
	\item $\mc M^G_0:=\{\Sigma\in \mc {LB}^G: ~\genus(\Sigma^*)=0\}$; 
	\item $  \mc M^G_0(U)$ is the space of $\Sigma\in \mc M^G_0$ with $\Sigma\subset U, \bd\Sigma\subset\bd U$; 
	\item $\wti {\mc D}^G:= \mc D^G\cup \mc D^G_+\subset \mc M^G_0$, where $\mc D^G$ and $\mc D^G_+$ are the sets of disk-type and half-disk-type $G$-hypersurfaces $\Sigma\in\mc {LB}^G$ respectively;
	\item $\wti{\mc D}^G(U):=\{\Sigma\in \wti{\mc D}^G: \mbox{$\Sigma\subset U, \bd\Sigma\subset\bd U$}\}$, and $\mc D^G(U),\mc D^G_+(U)$ are defined similarly. 
\end{itemize} 
Note that if $U$ is a half-ball-type open $G$-set, then for $G$-connected $\Sigma\in \mc M^G_0(U)$ with $G$-connected $\bd\Sigma\neq\emptyset$, $\Sigma^*$ is homeomorphic to a disk or a half-disk with finitely many holes, i.e. $\bd_0\Sigma^*$ may have multiple (but finitely many) components in $ \bd_0 U^*$. 
\subsection{Equivariant reduction of area}\label{Subsec: preliminary - equivariant reduction}

Given $G\cdot p\subset M$, take $r\in (0,\inj(G\cdot p))$ and set 
\[k_0:=\dim(G\cdot p).\] 
Then the nearest projection $\mf n:B_r(G\cdot p)\to G\cdot p$ is a smooth submersion. 
For any $G$-hypersurface $\Sigma$ and open $G$-set $U$ in $B_r(G\cdot p)$, it follows from the co-area formula that 
\begin{align}
    \mc H^n(\Sigma)&= \int_{G\cdot p}\int_{\mf n^{-1}(p')\cap\Sigma} \frac{1}{J^{\Sigma,*}_{\mf n}(q)} d\mc H^{n-k_0}(q) d\mc H^{k_0}(p') \label{Eq: area in slice}
    \\ &= \mc H^{k_0}(G\cdot p)\int_{S_r(p)\cap\Sigma} \vartheta_\Sigma(q)d\mc H^{n-k_0}(q), \nonumber
    \\
    \mc H^{n+1}(U)&= \mc H^{k_0}(G\cdot p)\int_{S_r(p)\cap U} \vartheta(q)d\mc H^{n+1-k_0}(q), \label{Eq: volume in slice}
\end{align}
where $\vartheta_\Sigma(q)=1/J^{\Sigma,*}_{\mf n}(q)$ and $\vartheta(q)=1/J^{*}_{\mf n}(q)$. 
Since $d\exp_{G\cdot p}^\perp$ is the identity on the $0$-section of $N(G\cdot p)$, we know $\vartheta_\Sigma$ and $\vartheta$ tend to $1$ uniformly as $r\to 0$ (independent of $q$ and $\Sigma$). 

Hence, for $r>0$ sufficiently small, we have $\vartheta_\Sigma,\vartheta \in [1/2, 2]$. 
Let $\nu_r(p), C_r(p)$ be the constants in the $\mf M$-isoperimetric lemma with respect to $S_r(p)$ (see \cite{almgren1962homotopy}*{Proposition 1.11}). 
Then for any $G$-hypersurface $\Sigma\subset B_r(G\cdot p)$ with $\bd\Sigma\subset \bd B_r(G\cdot p)$, 
\begin{align}\label{Eq: isoperimetric in slice}
    \mc H^n(\Sigma)\leq \frac{\nu_r(p)}{2\mc H^{k_0}(G\cdot p)} ~\Rightarrow~\exists U\subset B_{r}(G\cdot p) {\rm ~s.t.~} 
    \left\{
    \begin{array}{l}
        \bd_{rel}U=\Sigma, \\
        \mc H^{n+1}(U)\leq C(p,r) \left(\mc H^n(\Sigma)\right)^{\frac{n+1-k_0}{n-k_0}},
    \end{array}
    \right.
\end{align}
where $C(p,r):=C_r(p)\cdot 2^{\frac{2n+1-2k_0}{n-k_0}}\cdot (\mc H^{k_0}(G\cdot p))^{-\frac{1}{n-k_0}}$. 

Note that $\inj(G\cdot p)$ generally does not have a uniform positive lower bound. 
Nevertheless, we can still obtain some uniform constants for a compact set in an orbit type stratum. 
\begin{lemma}\label{Lem: uniform constants}
    Fix an orbit type stratum $M_{(H)}\subset M$ and a compact subset $K_0\subset M_{(H)}$. 
    Let $k_0:=\dim(G\cdot p)$ for $p\in M_{(H)}$. 
    Then we have uniform constants $\rho_0,\alpha_0>0,C_0>1$ depending only on $K_0$, $M$, and $G$, so that for any $p\in K_0$, 
    \begin{itemize}
        \item[(i)] $\rho_0<\frac{1}{4}\inj(G\cdot p)$, and for any $r\in (0,4\rho_0]$, $\bd B_{r}(G\cdot p)$ is strictly mean convex, and $S_{r}(p)$ is uniformly convex;
        \item[(ii)] $\frac{1}{2}\mc H^{k_0}(G\cdot p)\leq \mc H^{k_0}(G\cdot q)\leq 2 \mc H^{k_0}(G\cdot p)$, for all $q\in B_{4\rho_0}(G\cdot p)\cap M_{(H)}$; 
        \item[(iii)] $|d\exp_{G\cdot p}^\perp|, \vartheta_{\Sigma},\vartheta\in [\frac{1}{2}, 2]$ in $B_{4\rho_0}(G\cdot p)$;
        \item[(iv)] for any $r\in (0,4\rho_0)$, $\nu_r(p)\geq 2\mc H^{k_0}(G\cdot p)\cdot \alpha_0 r^{n-k_0}$ and $C(p,r)\leq C_0$; in particular, if $\Sigma\subset B_r(G\cdot p)$ is a $G$-hypersurface with $\bd\Sigma\subset \bd B_r(G\cdot p)$ and $\mc H^n(\Sigma)<\alpha_0 r^{n-k_0}$, then there exists an open $G$-set $U\subset B_r(G\cdot p)$ so that $\bd_{rel} U=\Sigma$ and $\mc H^{n+1}(U)\leq C_0(\mc H^n(\Sigma))^{\frac{n+1-k_0}{n-k_0}}$;
        \item[(v)] $1/C_0\leq \inf_{p\in K_0}\mc H^{k_0}(G\cdot p)\leq \sup_{p\in K_0}\mc H^{k_0}(G\cdot p)\leq C_0$.
    \end{itemize}
\end{lemma}
\begin{proof}
	The lemma follows easily from a compactness argument. 
\end{proof}

In \cite{hsiang71cohom}, Hsiang-Lawson introduced a weighted metric on $M^*=M/G$. 
Specifically, by Lemma \ref{Lem: prin orbit stratum}, one can define a Riemannian metric $g_{_{M^*}}$ on each orbit type stratum $M_{(H)}^*$ so that $\pi: M_{(H)}\to M_{(H)}^*$ is a Riemannian submersion. 
Consider the orbit-volume function $V$ defined by 
\begin{align}\label{Eq: orbits volume}
	V([x])=\left\{\begin{array}{ll}\operatorname{Vol}\left(\pi^{-1}([x])\right), & \text { if } \pi^{-1}([x]) \text { is a principal orbit } \\ m \cdot \operatorname{Vol}\left(\pi^{-1}([x])\right), & \text { if } \pi^{-1}([x]) \text { is an exceptional orbit } \\ 0, & \text { otherwise }\end{array}\right. ,\quad [x]\in M^*,
\end{align}
where $m=[G_x:G_{p_0}]$ for $x\in\pi^{-1}([x])$ and an appropriate $p_0\in M^{prin}$ with $G_{p_0}\subset G_x$. 
Then,
\begin{align}\label{Eq: weighted metric in M/G}
	\wti g_{_{M^*}}:=Vg_{_{M^*}}
\end{align}
is a continuous metric on $M^*$ that is smooth and non-degenerate in $(M^{prin})^*$. 

Using the rescaled metric $\wti g_{_{M^*}}$ on $M^*$, we can express the area of a $G$-hypersurface in the orbit space. 
Specifically, given an embedded $G$-connected $G$-hypersurface $\Sigma$ with $\Sigma\cap M^{prin}\neq \emptyset$, it follows from Proposition \ref{Prop: local structure of Sigma/G 1} and \ref{Prop: local structure of Sigma/G 2} that $\mc H^n(\Sigma)=\mc H^n(\Sigma\cap M^{prin})$. 
By the co-area formula, 
\begin{align}\label{Eq: area in orbit space}
	\mc H^n(\Sigma)=\int_{\Sigma^*\cap (M^{prin})^*} \mc H^{n-2}(\pi^{-1}([q])) d\mc H^2_{g_{_{M^*}}}([q]) = \mc H^{2}_{\wti g_{_{M^*}}}(\Sigma^*). 
\end{align}
Combined with Lemma \ref{Lem: prin orbit stratum}(ii), we directly have the following result.
\begin{lemma}[\cite{hsiang71cohom}*{Theorem 2}]\label{Lem: minimal in M^prin}
	Suppose $\Sigma$ is an embedded $G$-hypersurface with $\Sigma\cap M^{prin}\neq \emptyset$, and $g_{_{M^*}}$ is the metric on $(M^{prin})^*$ induced by the Riemannian submersion $\pi$. 
	Then $\Sigma$ is minimal if and only if $\Sigma^*\llcorner(M^{prin})^*$ is minimal with respect to the rescaled metric $\wti g_{_{M^*}}$. 
\end{lemma}

\subsection{Geometric measure theory}

In this subsection, we introduce some notations for geometric measure theory and also some useful lemmas.  

Let $\mc V_k(M)$ be the closure (in the weak topology) of rectifiable $k$-varifolds supported in $M$. 
Define the $\mf F$-metric on $\mc V_k(M)$ as in  \cite{pitts2014existence}*{2.1(19)}. 
Then $\mf F$ induces the weak topology on any mass-bounded subset of $\mc V_k(M)$. 
Given any $V \in \mc V_k(M)$, we denote by $\|V\| $ the weight measure induced by $V$ on $M$. 
Additionally, for any hypersurface $\Sigma$, we denote by $|\Sigma|$  the integer rectifiable codimension-one varifold induced by $\Sigma$. 

Next, we say $V \in \mc V_k(M)$ is $G$-invariant if $g_\#V= V$ for all $g \in  G$. 
Denote by 
\begin{itemize}
	\item $\mc V^G_k (M)$ the space of $G$-invariant $k$-varifolds in $M$. 
\end{itemize}
Given a varifold $V\in\mc V_n(M)$, we say $V$ is {\em stationary} in an open set $U\subset M$ if 
\[ \delta V(X):=\left.\frac{d}{d t}\right|_{t=0}\left\|\left(\phi_{t}\right)_{\#} V\right\|(M)=\int \operatorname{div}_{S} X(x) d V(x, S)=0,\]
for all $X\in \mk X(U)$, where $\{\phi_t\}$ is the flow generated by $X$. 
Similarly, if $V\in \mc V^G_n(M)$ and $\delta V(X)=0$ with respect to all $G$-invariant vector fields supported in an open $G$-set $U$, then we say $V$ is {\em $G$-stationary} in $U$.  
The following lemma indicates that it is sufficient to consider $G$-invariant vector fields for the first variation of $G$-varifolds. 

\begin{lemma}[\cite{liu2021existence}*{Lemma 2.2}]\label{Lem: first variation and G-variation}
	For any $V\in\mc V^G_n(M)$ and an open $G$-set $U\subset M$, $V$ is stationary in $U$ if and only if $V$ is $G$-stationary in $U$. 
	In particular, $\delta V(X)=\delta V(X_G)$ for any $X\in\mk X(U)$ and $X_G:=\int_G dg^{-1}(X) d\mu(g)$. 
\end{lemma}

Suppose $\Sigma\subset M$ is a closed minimal hypersurface, i.e. $|\Sigma|$ is stationary in $M$. 
Then for any $X\in\mk X(M)$, let $X^\perp$ be the normal component of $X\llcorner\Sigma$. The second variation of $\Sigma$ is given by 
\[ \delta^2\Sigma(X):=\left.\frac{d^{2}}{d t^{2}}\right|_{t=0} \operatorname{Area}\left(\phi_{t}(\Sigma)\right)=-\int_{\Sigma}\left\langle L_{\Sigma} X^\perp, X^\perp\right\rangle, \]
where $L_\Sigma$ is the Jacobi operator of $\Sigma$. 
Then, we say $\Sigma$ is {\em stable} in an open set $U\subset M$ if $\delta^2\Sigma(X)\geq 0$ for all $X\in \mk X(U)$. 
Similarly, if $\Sigma$ is a minimal $G$-hypersurface and $U$ is an open $G$-set, then we say $\Sigma$ is {\em $G$-stable} in $U$ provided that $\delta^2\Sigma(X)\geq 0$ for all $X\in \mk X^G(U)$. 

In general, the stability is stronger than the $G$-stability since $\mk X^G(U)\subset \mk X(U)$. 
Nevertheless, the following lemma indicates that the stability is equivalent to the $G$-stability given the existence of a unit $G$-normal. 

\begin{lemma}[\cite{wang2022min}*{Lemma 7}]\label{Lem: stability and G-stability}
	Let $\Sigma\subset M$ be a compact smooth embedded minimal $G$-hypersurface and $U\subset M$ be an open $G$-set. 
	If $\Sigma$ admits a $G$-invariant unit normal in $U$, then $\Sigma$ is $G$-stable in $U$ if and only if it is stable in $U$. 
\end{lemma}

\section{Equivariant $\gamma$-reduction}\label{Sec: gamma reduction}

In this section, we show the equivariant $\gamma$-reduction generalizing \cite{meeks82exotic}*{\S 3}.

To begin with, we have the following lemma generalizing \cite{meeks82exotic}*{Lemma 1}. 
One should note that there is no positive lower bound for $\inj(G\cdot p)$ and $M^*=M/G$ is a $3$-manifold away from a finite set $\mc S^*$. 
Hence, some estimates and topological arguments in \cite{meeks82exotic}*{Lemma 1} cannot be directly generalized to the whole manifold. 
Nevertheless, we can use Lemma \ref{Lem: uniform constants} in each orbit type stratum in place of \cite{meeks82exotic}*{(1.1)(1.2)}. 
Recall that $\mc S$ is defined in Definition \ref{Def: isolated orbits}. 

\begin{lemma}\label{Lem: isoperimetric lemma in MSY}
	For any $\epsilon>0$, there are $\delta=\delta(M,G,\epsilon)\in (0,1)$ and $\rho_0=\rho_0(M,G,\epsilon) \in (0,1)$ so that if $\Sigma\subset M$ is a $G$-hypersurface with $\bd\Sigma\cap (M\setminus B_{\epsilon}(\mc S))=\emptyset$ and
	\begin{align}\label{Eq: MSY lemma1 2.1}
		\mc H^n(\Sigma\cap B_{\rho_0}(x_0))<(\delta\rho_0)^{n-\dim(G\cdot x_0)} \qquad \forall x_0\in M\setminus B_{\epsilon/2}(\mc S) , 
	\end{align}
	then there is a unique compact $G$-set $K_\Sigma\subset M\setminus B_\epsilon(\mc S)$ satisfying $\bd K_\Sigma=\Sigma$ in $M\setminus B_\epsilon(\mc S)$, 
	\begin{align}\label{Eq: MSY lemma1 2.2}
		\mc H^{n+1}(K_\Sigma\cap B_{\rho_0}(x_0)) \leq \delta^{n-\dim(G\cdot x_0)}\rho_0^{n+1-\dim(G\cdot x_0)} \qquad \forall x_0\in M\setminus B_{\epsilon}(\mc S), 
	\end{align}
	and 
	\begin{align}\label{Eq: MSY lemma1 2.3}
		\mc H^{n+1}(K_\Sigma) \leq c (\mc H^n(\Sigma))^{\frac{n+1}{n}},
	\end{align}
	where $c=c(M,G)>0$. 
\end{lemma}
\begin{proof}
	Without loss of generality, assume $r_0=\epsilon$ is the radius $\rho_0$ associated with $\mc S$ given in Lemma \ref{Lem: uniform constants}. 
	Let $\{M_{(H_i)}\}_{i=1}^I$ be the finite set of orbit type strata in $M\setminus\mc S$. 
	Take the most singular stratum, say $M_{(H_1)}$, in the sense that there is no $i\geq 2$ with $(H_1)\leq (H_i)$. 
	Then, every $G$-component of $M_{(H_1)}$ is a $G$-submanifold that is either closed or has boundary in $\Omega_0:=\mc S$. 
	For the compact $G$-set $\Omega_1= M_{(H_1)}\setminus B_{r_0/2}(\Omega_0)= M_{(H_1)}\setminus B_{\epsilon/2}(\mc S)$, let $r_1>0$ be the radius $\rho_0$ given in Lemma \ref{Lem: uniform constants} w.r.t. $\Omega_1$. 
	In $\{M_{(H_i)}\}_{i=2}^I$, let $M_{(H_2)}$ be the most singular stratum so that there is no $i\geq 3$ with $(H_2)\leq (H_i)$. 
	Similarly, every $G$-component of $M_{(H_2)}$ is a $G$-submanifold that is either closed or has boundary in $M_{(H_1)}\cup\mc S$. 
	Let $r_2>0$ be the radius $\rho_0$ given in Lemma \ref{Lem: uniform constants} associated with $\Omega_2:=M_{(H_2)}\setminus (\cup_{j=0}^1 B_{r_j/2}(\Omega_j))$. 
	After repeating this construction $I$ times, we obtain compact $G$-sets $\Omega_0$ and $\{\Omega_i\subset M_{(H_i)}\}_{i=1}^I$, together with constants $r_0>0$ and $\{r_i>0\}_{i=1}^I$ so that $B_{r_i}(G\cdot p)$ satisfies the estimates in Lemma \ref{Lem: uniform constants} for all $p\in\Omega_i$ and $M\subset \cup_{i=0}^IB_{r_i/2}(\Omega_i)$. 
	Then we take $0<\rho_0< \min\{r_i:0\leq i\leq I\}$. 
	
	Let $\mc K^*$ be a Lipschitz triangulation of the orbit space $M^*$ (under the induced metric), which generates a $G$-equivariant triangulation $\mc K$ of $M$ in the sense of \cite{illman1983equivariant}. 
	By subdivision, we can assume that for any $K,K'\in \mc K$ with a common cell, there exist $0\leq i\leq I$ and $p\in \Omega_{i}$ so that $K\cup K' \subset B_{\rho_0}(G\cdot p)$. 
	We can further assume that $\mc K$ also gives an equivariant triangulation on $M\setminus B_{\epsilon/2}(\mc S)$. 
	Then for any $K\in\mc K$ with $K\subset B_{\rho_0}(G\cdot p)$, we can\ use \eqref{Eq: area in slice}\eqref{Eq: MSY lemma1 2.1} in place of \cite{meeks82exotic}*{(2.1)}, and apply the proof of \cite{meeks82exotic}*{Lemma 1} in the slice $S_{\rho_0}(p)$ to show $\mc K$ can be perturbed so that $\Sigma^*\setminus B_{\epsilon}(\mc S^*)$ does not intersect the $1$-skeleton of $\mc K^*$. 
	Using the same topological argument of \cite{meeks82exotic}*{Lemma 1} in the $3$-manifold $M^*\setminus B_\epsilon(\mc S^*)$, we know there is a compact $G$-set $K_\Sigma\subset M\setminus B_\epsilon(\mc S)$ bounded by $\Sigma \setminus B_\epsilon(\mc S)$. 
	Then, using the isoperimetric inequality in the slice $S_{\rho_0}(p)$, the proof in \cite{meeks82exotic}*{P. 627} would carry over in $S_{\rho_0}(p)$ to find a small $\delta>0$ so that \cite{meeks82exotic}*{(2.9)(2.10)} are valid with $(\delta\rho_0)^{n+1-\dim(G\cdot x_0)}$ and $(\mc H^n(\Sigma\cap K))^{\frac{n+1}{n}}$ in place of $(\delta\rho_0)^3$ and $|\Sigma\cap K|^{\frac{3}{2}}$ respectively.  
	This proves the lemma. 
\end{proof}

In the rest of this section, we fix $\epsilon>0$ and take $\delta,\rho_0>0$ as in Lemma \ref{Lem: isoperimetric lemma in MSY}. 
Let 
\[0<\gamma< (\delta\rho_0)^n/9.\] 

\begin{definition}\label{Def: gamma reduction}
	Given two closed $G$-hypersurfaces $\Sigma_1,\Sigma_2\in \mc {LB}^G$ in $M$ and an open $G$-set $U\subset M\setminus B_\epsilon(\mc S)$, we say $\Sigma_2$ is a {\em $(G,\gamma)$-reduction} of $\Sigma_1$ in $U$, denoted by 
	\[\Sigma_2 \underset{(G,\gamma)}{\ll} \Sigma_1 \qquad\mbox{in $U$}, \]
	if the following conditions are satisfied:
	\begin{enumerate}[label=(\roman*)]
		\item $A=\closure(\Sigma_1\setminus\Sigma_2)$ and $D_1\sqcup D_2=\closure(\Sigma_2\setminus\Sigma_1)$ are $G$-hypersurfaces in $U$ so that either
			\begin{itemize}
				\item[(a)] $A^*$ is homeomorphic to the closed annulus $\{x\in \mb R^2: \frac{1}{2}\leq |x|\leq 1\}$, and $D_1,D_2\in \mc D^G(U)$ are disk-type ( $D_1^*,D_2^*$ are homeomorphic to $\mb D^2$); or
				\item[(b)] $A^*$ is homeomorphic to the half-annulus $\{x\in \mb R^2: \frac{1}{2}\leq |x|\leq 1, x_1\geq 0\}$, and $D_1,D_2\in \mc D_+^G(U)$ are half-disk-type ($D_1^*,D_2^*$ are homeomorphic to $\mb D^2_+$); 
			\end{itemize}
		\item $\bd A=\bd D_1\cup \bd D_2$, and $\mc H^n(A)+\mc H^n(D_1)+\mc H^n(D_2)<2\gamma$;
		\item there exists an open $G$-set $Y\subset U$ with $\bd Y=A\cup D_1\cup D_2=\Sigma_1\triangle \Sigma_2$ so that $Y$ is ball-type in case (i.a) and half-ball-type in case (i.b);
		\item for the $G$-component $\wti\Sigma_1$ of $\Sigma_1$ containing $A$, if $\wti\Sigma_1\setminus A$ is not $G$-connected, then each $G$-component $\Gamma$ of $\wti\Sigma_1\setminus A$ satisfies either 
			\begin{itemize}
				\item $\Gamma\in \mc M^G_0(U)$ (i.e. $\Gamma^*$ is a $0$-genus surface in $U^*$) with $\mc H^n(\Gamma)\geq (\delta\rho_0)^n/2$, or
				\item $\Gamma\notin \mc M^G_0(U)$, 
			\end{itemize}
			where $\mc M^G_0(U)$ is replaced by $\mc D^G(U)$ if $U$ is a ball-type open $G$-set. 
	\end{enumerate}
	We say $\Sigma$ is $(G,\gamma)$-irreducible in $U$ if there is no $\wti\Sigma$ with $\wti\Sigma \underset{(G,\gamma)}{\ll} \Sigma$ in $U$. 
\end{definition}

Note that the above definition generalizes the definitions in \cite{meeks82exotic}*{\S 3}\cite{li2015general}*{Definition 7.6}. 

\begin{remark}\label{Rem: MSY 3.6 irreducible}
	By definition, a $G$-hypersurface $\Sigma$ is $(G,\gamma)$-irreducible in $U$ if and only if for any $\Delta\in \wti{\mc D}^G$ with $\Delta\subset U$, $\bd\Delta=\Delta\cap\Sigma$, and $\mc H^n(\Delta)<\gamma$, there exists a $G$-hypersurface $D\subset \Sigma$ with $D\in \mc M^G_0(U)$ (or $D\in \mc D^G(U)$ if $U$ is ball-type) and $\bd D=\bd\Delta$ so that $\mc H^n(D)<(\delta\rho_0)^n/2$. 
\end{remark}

Similarly, we also define the strong $(G,\gamma)$-reduction in $U$ as follows:
\begin{definition}\label{Def: gamma reduction strong}
	Given two closed $G$-hypersurfaces $\Sigma_1,\Sigma_2\in \mc {LB}^G$ in $M$ and an open $G$-set $U\subset M\setminus B_\epsilon(\mc S)$, we say that $\Sigma_2$ is a {\em strong $(G,\gamma)$-reduction} of $\Sigma_1$ in $U$ and write
	\[ \Sigma_2 \underset{(G,\gamma)}{<} \Sigma_1\qquad \mbox{in $U$}, \]
	if there exists a $G$-equivariant isotopy $\{\psi_t\}_{t\in [0,1]}\in \mk {Is}^G(U)$ so that 
	\begin{itemize}
		\item $\Sigma_2$ is a $(G,\gamma)$-reduction of $\psi_1(\Sigma_1)$ in $U$;
		\item $\Sigma_2=\Sigma_1$ in $M\setminus U$;
		\item $\mc H^n(\psi_1(\Sigma_1)\triangle\Sigma_2)<\gamma$. 
	\end{itemize}
	We say $\Sigma$ is {\em strongly $(G,\gamma)$-irreducible} in $U$ if there is no $\wti\Sigma$ with $\wti\Sigma \underset{(G,\gamma)}{<} \Sigma$ in $U$. 
\end{definition}

The following lemma generalizes \cite{meeks82exotic}*{Remark 3.1}. 
\begin{lemma}\label{Lem: MSY 3.1 finite reduction}
	Given a closed $G$-hypersurface $\Sigma\in\mc {LB}^G$ in $M$ and an open $G$-set $U\subset M\setminus B_\epsilon(\mc S)$, there exists a sequence of closed $G$-hypersurfaces $\{\Sigma_i\}_{i=1}^k\subset \mc {LB}^G$ so that 
	\[\Sigma_k \sreduc \Sigma_{k-1} \sreduc\cdots\sreduc\Sigma_1=\Sigma \qquad\mbox{in $U$},\]
	and $\Sigma_k$ is strongly $(G,\gamma)$-irreducible in $U$. 
	Moreover, there is a constant $c$ depending only on the genus of $\Sigma^*$ in $(M\setminus B_\epsilon(\mc S))^*$ and $\mc H^n(\Sigma)/(\delta\rho_0)^n$ so that $k\leq c$ and 
	\[\mc H^n(\Sigma\triangle\Sigma_k) \leq 3c\gamma. \]
\end{lemma}
\begin{proof}
	The proof is the same as \cite{meeks82exotic}*{Remark 3.1} using our $G$-equivariant objects.
\end{proof}

For any closed $G$-hypersurface $\Sigma\in\mc {LB}^G$ and open $G$-set $U\subset M\setminus B_\epsilon(\mc S)$, define
\begin{align}\label{Eq: area error}
	E(\Sigma):=\mc H^n(\Sigma) - \inf_{\Sigma'\in J_U^G(\Sigma)} \mc H^n(\Sigma'),
\end{align}
where $J^G_U(\Sigma):=\{\varphi_1(\Sigma): \{\varphi_t\}_{t\in [0,1]}\in \mk {Is}^G(U)\}$. 
Additionally, denote by 
\begin{align}\label{Eq: Sigma0}
	\Sigma_0\subset\Sigma
\end{align}
the union of all $G$-components $\Lambda\subset \Sigma\cap\closure(U)$ so that there is some ball-type or half-ball-type open $G$-set $K_\Lambda\subset U$ with $\Lambda\subset K_\Lambda$ and $\bd K_\Lambda\cap\Sigma=\emptyset$. 
The following result extends \cite{meeks82exotic}*{Theorem 2} to our equivariant setting.

\begin{theorem}\label{Thm: MSY thm2 gamma reduction}
	Let $U:=B_{\rho}(G\cdot p)\subset\subset M\setminus B_\epsilon(\mc S)$ so that the estimates in Lemma \ref{Lem: uniform constants} are satisfied, and let $B:=\closure B_{r}(G\cdot p)\subset U$ be a compact ball-type or half-ball-type $G$-neighborhood of $G\cdot p$. 
	Suppose $\Sigma\in\mc {LB}^G$ is a closed $G$-hypersurface in $M$ so that 
	\begin{enumerate}[label=(\roman*)]
		\item $\Sigma$ intersects $\bd B$ transversally;
		\item $E(\Sigma)\leq \gamma/4$, and $\Sigma$ is strongly $(G,\gamma)$-irreducible in $U$;
		\item for any $G$-component $\Gamma$ of $\Sigma\cap\bd B$, there exists a $G$-component $F_\Gamma$ of $\bd B\setminus \Gamma$ with $\bd F_\Gamma=\Gamma$ so that 
			\begin{align}\label{Eq: MSY thm2 assumption - isoperimetric}
				\mc H^{n}(F_\Gamma)\leq C\cdot \left( \mc H^{n-1}(\Gamma) \right)^{\frac{n-\dim(G\cdot p)}{n-\dim(G\cdot p)-1}}
			\end{align}
			for some $C=C(M,G,p)>0$; moreover, if $\Gamma^*\cap \bd_0 B^*=\emptyset$ (i.e. $\Gamma^*\subset \bd_1 B^*$ is a simple closed curve) and $\codim(G\cdot x)=3$ for some $x\in U\setminus M^{prin}$, then $F_\Gamma$ is disk-type, i.e. 
			\begin{align}\label{Eq: MSY thm2 assumption - disk instead of cylinder}
				(F_\Gamma^* , \Gamma^*)\overset{\rm homeo}{\cong} (\mb D^2, \bd \mb D^2) ;
			\end{align}
		\item $\sum_{j=1}^J\mc H^n(F_j)\leq \gamma/8$, where $F_j=F_{\Gamma_j}$, and $\{\Gamma_j\}_{j=1}^J$ are the $G$-components of $\Sigma\cap\bd B$. 
	\end{enumerate}
	Then $\mc H^n(\Sigma_0)\leq \mc H^n(\Sigma)$, and there are pairwise disjoint, $G$-connected, compact $G$-hypersurfaces $\{D_i\}_{i=1}^I\subset \mc M^G_0$ with $\genus(D_i^*)=0$ so that 
	\begin{enumerate}
		\item $D_i\subset (\Sigma\setminus\Sigma_0)\cap U$, $\bd D_i\subset \bd B$, and $(\cup_{i=1}^I D_i)\cap B=(\Sigma\setminus\Sigma_0)\cap B$;
		\item $\sum_{i=1}^I\mc H^n(D_i) \leq \sum_{j=1}^J\mc H^n(F_j) + E(\Sigma)$;
		\item for any given $\alpha>0$, we have 
		\[ \mc H^n\left( \left( \bigcup_{i=1}^I ( \varphi_1(D_i)\setminus \bd D_i ) \right)  \setminus \interior(B) \right) <\alpha \]
		for some $G$-equivariant isotopy $\{\varphi_t\}_{t\in [0,1]}\in \mk{Is}^G(U)$ depending on $\alpha$ that is the identity in a small $G$-neighborhood of $(\Sigma\setminus\Sigma_0)\setminus \cup_{i=1}^I(D_i\setminus \bd D_i)$. 
	\end{enumerate}
	Moreover, if $B,U$ are ball-type, then we can further choose $\{D_i\}_{i=1}^I\subset \mc D^G$ to be disk-type.  
\end{theorem}

\begin{remark}\label{Rem: compare disk instead cylinder}
	Note that $\Gamma^*\subset \closure(\bd_1 B^*)$ can either be a Jordan arc with endpoints on $\bd_0 B^*$, or be a simple closed curve in $\bd_1 B^*$. 
	In the first case, both of the $G$-components of $\bd B\setminus \Gamma$ (and also $F_\Gamma$) are half-disk-type. 
	In the latter case, $\bd B\setminus \Gamma$ has two $G$-components $F_1,F_2$ so that 
	\begin{itemize}
		\item if $B$ is ball-type, then $F_1,F_2$ are disk-type;
		\item if $B$ is half-ball-type, then $F_1^*,F_2^*$ are homeomorphic to a disk and an annulus respectively. 
	\end{itemize}
	Additionally, combining \eqref{Eq: area in slice} with the isoperimetric inequality in the slice $\bd B\cap S_\rho(p)$, we know \eqref{Eq: MSY thm2 assumption - isoperimetric} is valid by choosing $F_\Gamma$ to satisfy the following condition (where $k_0:=\dim(G\cdot p)$):
	\begin{align}\label{Eq: MSY thm2 assumption - small area}
		\mc H^{n-k_0}(F_\Gamma\cap S_\rho(p)) = \min\{\mc H^{n-k_0}(F_\Gamma\cap S_\rho(p)), \mc H^{n-k_0}((\bd B\setminus F_\Gamma)\cap S_\rho(p) )\}. 
	\end{align} 
	Moreover, if $\Gamma^*\subset \bd_1 B^*$ is a simple closed curve and $\codim(G\cdot x)=3$ for some $x\in U\setminus M^{prin}$, then we can always pick the disk-type $G$-component $F_\Gamma$ of $\bd B\setminus\Gamma$ satisfying \eqref{Eq: MSY thm2 assumption - isoperimetric}. 
	Indeed, if $B$ is ball-type, we simply take $F_\Gamma$ satisfying \eqref{Eq: MSY thm2 assumption - small area}. 
	If $B$ is half-ball-type, then $p$ and $B$ satisfy either Proposition \ref{Prop: local structure of M/G}(1) or (2.a). 
	\begin{itemize}
		\item[(i)] In the case of Proposition \ref{Prop: local structure of M/G}(1), one can take $N_+ :=\{v=(v_1,v_2,v_3)\in N_{p}(G\cdot p)\cong\mb R^3: v_3\geq 0\} $ as a fundamental domain of $G_p$-actions, and identify $\Gamma^*, B^*, F_\Gamma^*$ with $\widehat{\Gamma}^+:= \exp_{G\cdot p}^{\perp,-1}(\Gamma)\cap N_+$, $\widehat{B}^+:=\exp_{G\cdot p}^{\perp,-1}(B)\cap N_+$, $\widehat{F}_\Gamma^+:= \exp_{G\cdot p}^{\perp,-1}(F_\Gamma)\cap N_+$ respectively. 
			Then $\widehat{\Gamma}^+$ is a simple closed curve on the relative boundary $\bd_1\widehat{B}^+$ of $\widehat{B}^+$. 
			By Lemma \ref{Lem: local structure of M/G with codim 3} and the last statement in Proposition \ref{Prop: local structure of Sigma/G 1}, we can reflect $\bd_1\widehat{B}^+\cong \mb S^2_+$ across $P=\{v: v_3=0\}$ to get $\bd\widehat{B}:=\exp_{G\cdot p}^{\perp,-1}(\bd B)\cap N_{p}(G\cdot p) \cong \mb S^2$. 
			Hence, the region $\widehat F_\Gamma^+\subset \widehat B^+$ enclosed by $\widehat \Gamma^+$ is contained in one half of $\bd \widehat B$, and thus $\mc H^2(\widehat F_\Gamma^+) = \min\{ \mc H^2(\widehat F_\Gamma^+), \mc H^2(\bd\widehat B \setminus \widehat F_\Gamma^+ )\}$. 
			Hence, for $\widehat F_\Gamma:=\exp_{G\cdot p}^{\perp,-1}(F_\Gamma)\cap N_{p}(G\cdot p)$, we can use the isoperimetric inequality on $ \bd \widehat F_\Gamma^+\subset \bd\widehat B$ to get $\mc H^2(\widehat F_\Gamma) = 2\mc H^2(\widehat F_\Gamma^+)\leq 2 c_{iso}(\mc H^{1}(\bd \widehat F_\Gamma^+))^{2} \leq c_{iso}(\mc H^{1}(\bd \widehat F_\Gamma))^{2}$ for some $c_{iso}>0$. 
		\item[(ii)] In the case of Proposition \ref{Prop: local structure of M/G}(2.a), let $\widehat B_{n.p.}$ be the union of non-principal $G_p$-orbits in $N_p(G\cdot p)$, and thus $\bd_0 \widehat B^*\subset (\widehat B_{n.p.})^*$, where $\widehat{B}:=\exp_{G\cdot p}^{\perp,-1}(B)\cap N_{p}(G\cdot p)$. 
			Then by Lemma \ref{Lem: local structure of M/G with codim 3}, the dimension assumption implies that $N_p(G\cdot p)$ can be equally divided by $\widehat B_{n.p.}$ into $k$ parts $\{\Omega_i\}_{i=1}^k$ for some $k\in \{2,3,...\}$. 
			Since $\Gamma^*\cap \bd _0 B^*=\emptyset$ and the disk $F_\Gamma^*$ satisfies $F_\Gamma^*\cap \bd_0 B^*=\emptyset$, we know each $\interior(\Omega_i)$ contains a component $\widehat \Gamma_i$ of $\widehat \Gamma:= \exp_{G\cdot p}^{\perp,-1}(\Gamma)\cap N_p(G\cdot p)$, and $\widehat{F}_\Gamma:= \exp_{G\cdot p}^{\perp,-1}(F_\Gamma)\cap N_p(G\cdot p)$ has a component $\widehat F_\Gamma^i\subset \interior(\Omega_i)$ enclosed by $\widehat \Gamma_i$. 
			In particular, $\frac{1}{k}\cdot\mc H^{n-\dim(G\cdot p)}(\widehat F_\Gamma) =\mc H^{n-\dim(G\cdot p)}(\widehat F_\Gamma^1)\leq \mc H^{n-\dim(G\cdot p)}(\bd \widehat B\setminus  \widehat F_\Gamma^1)$, where $\bd\widehat{B}:=\exp_{G\cdot p}^{\perp,-1}(\bd B)\cap N_{p}(G\cdot p) \cong \mb S^{n-\dim(G\cdot p)}$. 
			Thus, the isoperimetric inequality can be applied to $\widehat F_\Gamma^1$ in $\bd\widehat{B}$. 
	\end{itemize}
	Therefore, in both cases, we can further use \eqref{Eq: area in slice} to find a constant $C=C(M,G,p)>0$ so that \eqref{Eq: MSY thm2 assumption - isoperimetric} and \eqref{Eq: MSY thm2 assumption - disk instead of cylinder} are simultaneously valid. 
\end{remark}

Before we prove Theorem \ref{Thm: MSY thm2 gamma reduction}, we need the following lemma, which is analogous to \cite{li2015general}*{Lemma 7.11}, to construct the approximating $G$-isotopy. 
The constructions in \cite{li2015general} cannot be applied directly in $M^*$ since the $G$-isotopy cannot be reduced to an outward isotopy (cf. \cite{li2015general}).

\begin{lemma}\label{Lem: G-isotopy approximation}
	Let $G\cdot p, U,B$ be as in Theorem \ref{Thm: MSY thm2 gamma reduction}, and $\Gamma$ be a $G$-hypersurface in $\bd B$ so that $\Gamma^*$ is a simple curve on $\closure(\bd_1 B^*)$ that is either closed or has endpoints on $\bd_0 B^*$. 
	Suppose $\bd B\setminus \Gamma$ has two $G$-components, and $F\subset \bd B$ is a $G$-component of $\bd B\setminus\Gamma$ satisfying Theorem \ref{Thm: MSY thm2 gamma reduction}(iii). 
	Suppose $D\in \mc M^G_0$ is a $G$-connected $G$-hypersurface in $U$ with $\genus(D^*)=0$, $\bd D=\bd F=\Gamma$, and $D\cap F=\emptyset$. 
	Assume also that $D\cup F$ bounds a unique compact $G$-set $K\subset M$, i.e. $\bd K=D\cup F$. 
	
	Then for any $\alpha>0$, there exists a $G$-isotopy $\{\varphi_t\}_{t\in [0,1]}\in \mk{Is}^G(U)$ supported in a $G$-neighborhood of $K$ so that every $\varphi_t$ is the identity on $\Gamma$ and 
	\begin{align}\label{Eq: MSY thm2 - isotopy approximate}
		\mc H^n(\varphi_1(D) \triangle F ) <\alpha.
	\end{align}
\end{lemma}

\begin{proof}
	{\bf Step 1.}
	If $G\cdot p\subset M^{prin}$, then $D^*$ and $F^*$ are diffeomorphic to disks, and $K^*$ is a closed $3$-ball in the smooth manifold $U^*\subset (M^{prin})^*$. 
	Hence, $D^*$ and $F^*$ are smooth isotopic in $U^*$ with $\Gamma^*$ fixed, which can be lifted to a $G$-isotopy in $U$ by \cite{schwarz1980lifting}*{Corollary 2.4}. 
    
	Suppose next $p\in M\setminus M^{prin}$. 
	
	\medskip
	{\bf Step 2.}
	If $U$ and $B$ are ball-type (Proposition \ref{Prop: local structure of M/G}(2.b)), 
	then $P(p)\cong \mb R$, and $F^*$ is a closed disk. 
	After changing coordinates on the topological manifold $U^*$, we can assume that 
	\begin{itemize}
		\item $U^*$ is an open ball of radius $2$ in $\mb R^3$ at the origin;
		\item $B^*\subset U^*$ is the closed unit ball at the origin;
		\item the $x_3$-axis is the non-principal orbit type stratum in $U^*$ (since $P(p)\cong \mb R$); 
	\end{itemize}
	Note that in this case $\Gamma^*$ contains no point on $(M\setminus M^{prin})^*$ since otherwise Proposition \ref{Prop: local structure of Sigma/G 2}(4) would imply that $\Gamma$ does not separate $\bd B$ into two $G$-components, yielding a contradiction. 
	Since $D^*\subset U^*$ has genus zero and a connected boundary $\Gamma^*$, we see from Proposition \ref{Prop: local structure of Sigma/G 2} that $D^*$ is also a disk that intersects the $x_3$-axis orthogonally at finitely many points. 
	In particular, $K^*$ is a closed $3$-ball. 
	Therefore, we can construct a smooth isotopy $\{\phi_t\}_{t\in [0,1]}$ (in the sense of \cite{schwarz1980lifting}) that is the identity in a $\delta$-neighborhood of $\Gamma^*\cup \{x_1=x_2=0\}$ such that $\pi^{-1}(\phi_1(D^*))\triangle F$ restricted to $U\setminus B_{2\delta}(\Gamma\cup(M\setminus M^{prin})) $ has area less than $\alpha/2$. 
	Note that $\mc H^n(\bd B_{2\delta}(\Gamma\cup(M\setminus M^{prin})))\to 0 $ as $\delta\to 0$. 
	Hence, by taking $\delta>0$ small enough and applying \cite{schwarz1980lifting}*{Corollary 2.4}, $\{\phi_t\}$ can be lifted to a $G$-isotopy $\{\varphi_t\}$ in $U$ with $\Gamma $ fixed and \eqref{Eq: MSY thm2 - isotopy approximate} satisfied. 
	
	\medskip
	Suppose next that $U$ and $B$ are half-ball-type (Proposition \ref{Prop: local structure of M/G}(1) and (2.a)). 
	
	{\bf Step 3.}
	We first assume that $\dim(P(p))=2$, i.e. $U,B$ satisfy Proposition \ref{Prop: local structure of M/G}(1), and thus $P(p)\cong \mb R^2$. 
	We separate the proof into two cases.
	
	\medskip
	{\it Case 1.} $\Gamma^*\subset \bd_1 B^*$ is a simple closed curve. 
	
	In this case, $F^*$ is either a disk or an annulus (Remark \ref{Rem: compare disk instead cylinder}). 
	
	Suppose first that $F^*$ is a disk. 
	Then by changing coordinates on the topological manifold $U^*$, we can assume that 
	\begin{itemize}
		\item $U^*=\{x\in \mb R^3: |x|<2, x_3\geq 0\}$ is a relative open half-ball of radius $2$ at the origin;
		\item $B^*=\{x\in \mb R^3: |x|\leq 1, x_3\geq 0\}$ is the closed unit half-ball at the origin;
		\item the $\{x_3=0\}$-plane is the non-principal orbit type stratum in $U^*$ (since $P(p)\cong \mb R^2$); 
		\item $\Gamma^*=\{(x_1,x_2,x_3)\in\mb R^3: ~x_3=1/2, ~x_1^2+x_2^2=3/4\}$;
		\item $F^*=\{(x_1,x_2,x_3)\in \mb R^3: ~|x|=1, ~x_3\geq 1/2\}$. 
	\end{itemize}
	Note that $D^*$ is a $0$-genus surface in $U^*$ with boundary $\bd D^*=\bd_1 D^*\cup \bd_0 D^*$, where $\bd_1 D^*= \bd D^*\cap (M^{prin})^* =\Gamma^*$, and $\bd_0 D^*\subset \{x_3=0\}$ is a disjoint union of simple closed curves $\{\Gamma^*_i\}_{i=1}^m$. 
	Let $D_i^*\subset \{x_3=0\}$ be the disk enclosed by $\Gamma_i^*$. 
	Then $D^*\cup F^*\cup(D_1^*\cup\cdots\cup D_m^*)$ is homeomorphic to a $2$-sphere, and thus $K^*$ is homeomorphic to a $3$-ball with $\bd_1 K^*=D^*\cup F^*$ and $\bd_0 K^*=\cup_{i=1}^m D_i^*$. 
	Hence, up to a smooth isotopy in $U^*$ (in the sense of \cite{schwarz1980lifting}) that preserves every orbit type stratum, we can make $\cup_{i=1}^m D_i^*\subset \mb B^2_{1/2}(0)\subset \{x_3=0\}$, and take $D^*$ as the union of $m$ cylinders $\cup_{i=1}^m(\bd D_i^*\times [0, 1/2])$ and a disk with $m$ holes $\{|x|\leq 1, x_3=1/2\}\setminus (\cup_{i=1}^m D_i^*\times \{1/2\})$. 
	Using smooth isotopies in $U^*$ (in the sense of \cite{schwarz1980lifting}), we then further shrink the cylinder $\cup_{i=1}^m(D_i^*\times [0, 1/2])$ sufficiently thin, and push the disk $\{|x|\leq 1, x_3=1/2\}$ onto $F^*$. 
	Note that all these isotopies can be the identity near $\Gamma^*$. 
	Hence, we can apply \cite{schwarz1980lifting}*{Corollary 2.4} to obtain the desired $G$-isotopy with \eqref{Eq: MSY thm2 - isotopy approximate} satisfied. 
	
	Consider next that $F^*$ is a cylinder, i.e. $F^*=\{x\in \mb R^3: |x|=1, 0\leq x_3\leq 1/2\}$ in the above coordinates. 
	Then $D^*$ is still a $0$-genus surface with $\bd_1 D^*:= \bd D^*\cap (M^{prin})^* =\Gamma^*$, and $\bd_0 D^*\subset \{x_3=0\}$ is a disjoint union of simple closed curves $\{\Gamma^*_i\}_{i=1}^m$. 
	As in \cite{li2015general}*{Lemma 7.11}, all except possibly one $\Gamma_i^*$ bound disks in $K^*$. 
	Using the neck-shrinking $G$-isotopy as before, we can eliminate all the $\Gamma_i^*$ that bound a disk, and assume w.l.o.g. that $m=1$, and $\Gamma_1^*$ together with $\Gamma_0^*:=\bd_0F^*$ bounds a connected genus zero surface $\Lambda^*\subset \{x_3=0\}= (M\setminus M^{prin} )^*$. 
	After a change of coordinates, we can assume that $K^*=\Lambda^*\times [0, 1/2]$, $F^*=\Gamma_0^*\times [0, 1/2]$, and $D^*$ is the union of  $(\bd \Lambda^*\setminus\Gamma_0^*)\times [0,1/2]$ (a part of a cylinder) and $\Lambda^* \times \{1/2\}$ (a disk with holes). 
	For a small $\delta>0$, let $\{\phi_t\}_{t\in [0,1]}$ be the vertical contraction 
    \[(x_1,x_2,x_3)\mapsto (x_1,x_2, (1-(1-\delta)t)x_3)\] 
    with a cut-off at $\bd U^*\cup (\Gamma_0^*\times [0,\infty))$ so that $\{\phi_t\}_{t\in [0,1]}$ is a smooth isotopy in $U^*$ (in the sense of \cite{schwarz1980lifting}) and is the identity on $F^*\subset  (\Gamma_0^*\times [0,\infty))$ (see \cite{li2015general}*{P. 311} for the cut-off construction).  
	Note that $\phi_t$ cannot push the points in $\{x_3>0\}$ into $\{x_3=0\}$ since the smooth isotopy in $M^*$ preserves orbit-type strata (cf. \cite{schwarz1980lifting}*{Corollary 1.7(2)}). 
	Nevertheless, by our assumption in Theorem \ref{Thm: MSY thm2 gamma reduction}(iii) for \eqref{Eq: MSY thm2 assumption - disk instead of cylinder}, we know $\dim (\pi^{-1}((x_1,x_2,0)) ) < n+1-\Cohom(G)=n-2$, which implies $\dim(U\setminus M^{prin})=2+\dim (\pi^{-1}(x_1,x_2,0) ) <n$. 
	Since $\Lambda^*\times \{\delta/2\}$ can be seen as a subset of $( \bd B_{\delta/2}(U\setminus M^{prin}) )^*$, we know $\mc H^n(\pi^{-1}(\Lambda^*\times \{\delta/2\})) < \alpha/2$ for $\delta>0$ small enough. 
	After lifting $\{\phi_t\}$ to a $G$-isotopy by \cite{schwarz1980lifting} with $\delta>0$ small, we can $G$-isotope $D$ arbitrarily close to $F$ as in \cite{li2015general}*{Lemma 7.11}. 
	
	\medskip
	{\it Case 2.} $\Gamma^*\subset \closure(\bd_1 B^*)$ is a Jordan arc with free boundary on $\bd_0 B^*$.  
	
	In the coordinates in Case 1, we can assume for convenience that (up to a smooth isotopy in $U^*$), $\Gamma^*=\{x\in \mb R^3: |x|=1, x_1=0, x_3\geq 0\}\subset \closure(\bd_1 B^*)=\{x: |x|=1, x_3\geq 0\}$ and $F^*= \{x\in \closure(\bd_1 B^*): x_1\geq 0\}$. 
	Let $\Gamma_0^*=F^*\cap \{x_3=0\}$. 
	Then $D^*\cap \{x_3=0\}$ is the disjoint union of a Jordan arc $\Gamma_1^*$ and a (possibly empty) finite collection of simple closed curves $\{\Gamma_i^*\}_{i=2}^m$ so that $\bd \Gamma_1^*=\bd\Gamma_0^*$. 
	By the assumptions, there is a $0$-genus surface $\Lambda^*\subset \{x_3=0\}$ with $\bd\Lambda^*=\cup_{i=0}^m\Gamma_i^*$ so that $\bd K^*=D^*\cup F^*\cup \Lambda^*$. 
	For $2\leq i\leq m$, if $\Gamma_i^*$ encloses a disk in $\Lambda^*$, then we can apply the neck-shrinking $G$-isotopy as in Case 1, and thus we assume without loss of generality that $\Lambda^*$ is connected. 
	Denote by $D_1^*$ and $\{D_i^*\}_{i=2}^m$ the disks in $\{x_3=0\}$ enclosed by $\Gamma_0^*\cup\Gamma_1^*$ and $\{\Gamma_i^*\}_{i=2}^m$ respectively. 
	
	If $\Gamma_0^*\cup\Gamma_1^*$ is the outermost boundary of $\Lambda^*$, i.e. $\Lambda^*\subset D_1^*$. 
	Then we apply the neck-shrinking $G$-isotopy as in Case 1 to every other $\Gamma_i$, $i\geq 2$, and thus assume without loss of generality that $\Lambda^*=D_1^*$. 
	By changing coordinates, we can rewrite $F^*=\{x\in\mb R^3: x_1=0, x_3\geq 0, |x|\leq 1\}$ and $\Lambda^*=\{x\in \mb R^3: x_3=0, x_1\geq 0, |x|\leq 1\}$ so that $D^*$ is a graph over $\Lambda^*$. 
	For a small $\delta_1>0$, we apply the vertical contraction $\{\phi^{(1)}_t\}_{t\in[0,1]}$ as in Case 1 with $\bd U^*\cup F^*$ fixed so that
	\begin{itemize}
		\item in $\{x_3\geq \delta_1\}$, $\phi^{(1)}_1(D^*)$ is sufficiently close to $F^*$, and
		\item in $\{0\leq x_3\leq \delta_1\}$, $\phi^{(1)}_1(D^*)$ is the union of a cylinder $E^*$ over $\bd\Lambda^*$ with height $\leq \delta_1$, and a graph over $\Lambda^*$ with height $\leq \delta_1$.
	\end{itemize}
	Then for $\delta_2>0$ to be determined later, let $\{\phi^{(2)}_t\}_{t\in [0,1]}$ be the horizontal contraction 
    \[(x_1,x_2,x_3)\mapsto (x_1, (1-(1-\delta_2)t) x_2,x_3) \]
    along the $x_2$-axis with a cut-off so that $\{\phi^{(2)}_t\}$ is supported in the relative open set $V^*:=\{ x\in \mb R^3: 0\leq x_3<2\delta_1, x_1> 2\delta_1 \}$. 
	Hence, $\{\phi^{(2)}_t\}$ can shrink $\Lambda^*\cap V^*$ and $\phi^{(1)}_1(D^*)\cap V^*$ into the $\delta_2$-neighborhood of the $x_1$-axis. 
	Since the pre-image of the $x_1$-axis in $M$ has codimension at least $2$, we know $\pi^{-1}(\phi^{(2)}_1\circ \phi^{(1)}_1(D^*)\cap V^*)$ has area $\leq \alpha/3$ by taking $\delta_2$ small enough. 
	After taking $\delta_1$ even smaller, the pre-image of $\phi^{(2)}_1\circ \phi^{(1)}_1(D^*)\cap \{x\in \mb R^3: 0\leq x_3\leq 2\delta_1, x_1\leq 2\delta_1\}$ also has area $\leq \alpha/3$. 
	Together, $\pi^{-1}(\phi^{(2)}_1\circ \phi^{(1)}_1(D^*)\cap \{0\leq x_3\leq 2\delta_1\})$ has area $\leq 2\alpha/3$. 
	Hence, we can concatenate $\phi^{(1)}_t$ and $\phi^{(2)}_t$, and lift the resulting isotopy to a desired $G$-isotopy in $U$ by \cite{schwarz1980lifting}. 
	
	If $\Gamma_0^*\cup\Gamma_1^*$ is not the outermost boundary of $\Lambda^*$, i.e. $D_1^*\subset \Lambda^*$. 
	Suppose $\Gamma_2^*$ is the outermost boundary. 
	Then, using the neck-shrinking $G$-isotopy as in Case 1 to every other $\Gamma_i$, $i\geq 3$, we can assume without loss of generality that $\Lambda^*=D_2^*\setminus \interior(D_1^*)$. 
	By changing coordinates, we can rewrite $F^*=\{x\in\mb R^3: x_1=0, x_3\geq 0, |x|\leq 1\}$ and $\Lambda^*=\{x\in \mb R^3: x_3=0, |x|\leq 3/2\}\setminus \{x\in \mb R^3: x_3=0, x_1\geq 0, |x|\leq 1\}$ so that $D^*$ is a graph over $\Lambda^*$. 
	As above, we first apply the isotopy $\{\phi^{(1)}_t\}$ of vertical contraction along the $x_3$-axis (with $\delta_1>0$) that is the identity near $F^*$, and then apply the isotopy $\{\phi^{(2)}_t\}$ of horizontal contraction along the $x_2$-axis (with $\delta_2>0$) that is supported in $B_{2\delta_1}(\{x_3=0\})\setminus B_{2\delta_1}(F^*)$. 
	The desired $G$-isotopy is obtained by taking $\delta_1,\delta_2>0$ small enough and applying \cite{schwarz1980lifting} to the concatenated isotopy $\phi^{(2)}* \phi^{(1)}$. 
	
	\medskip
	{\bf Step 4.} 
	Finally, consider the case that $\dim(P(p))=1$, i.e. $U,B$ satisfy Proposition \ref{Prop: local structure of M/G}(2.a).  
	Thus, $P(p)\cong \mb R$, and by changing coordinates on the topological manifold $U^*$, we assume 
	\begin{itemize}
		\item $U^*=\{x\in \mb R^3: |x|<2, x_1\geq 0, x_3\geq 0\}$ is a relative open wedge in $\mb B^3_2(0)$;
		\item $B^*=\{x\in \mb R^3: |x|\leq 1, x_1\geq 0, x_3\geq 0\}$ is the closed wedge in $\closure(\mb B^3_1(0))$;
		\item the $x_2$-axis is the $(G_p)$-orbit type stratum in $U^*$ (since $P(p)\cong \mb R$); 
		\item $\bd_- B^*=\{x\in \mb R^3: |x|\leq 1, x_1= 0, x_3> 0\}$;
		\item  $\bd_+ B^*= \{x\in \mb R^3: |x|\leq 1, x_1> 0, x_3= 0\}$.
	\end{itemize}
	Note that $\bd_0 B^*=\bd_+ B^*\cup\bd_-B^*\cup \{x_1=x_3=0\}$ is the union of non-principal orbit type strata in $B^*$. 
	Since the wedge $B^*$ is homeomorphic to the half-ball $\{x\in \mb R^3: |x|\leq 1, x_3\geq 0\}$ with $\bd_0 B^*$ homeomorphic to $\{x: |x|\leq 1,x_3=0\}$, the proof in {\bf Step 3} can be adapted to this case. 
	
	To be exact, if $\Gamma^*$ is a simple closed curve and $F^*$ is a disk, then the neck-shrinking construction in {\bf Step 3} Case 1 would carry over since every neck meets $\bd_0 B^*$ orthogonally (Proposition \ref{Prop: local structure of Sigma/G 2}). 
	If $\Gamma^*$ is a simple closed curve and $F^*$ is an annulus, we can use the radial contraction $(x_1,x_2,x_3)\mapsto ((1-(1-\delta_1)t)x_1, x_2, (1-(1-\delta_1)t)x_3)$ (with an appropriate cut-off) in place of the vertical contraction. 
	Combined with the dimension assumption in Theorem \ref{Thm: MSY thm2 gamma reduction} before \eqref{Eq: MSY thm2 assumption - disk instead of cylinder}, the proof in {\bf Step 3} Case 1 would also carry over. 
	
	If $\Gamma^*$ is a Jordan arc with free boundary on $\bd_0 B^*$, then the endpoints of $\Gamma^*$ either lie in two different $\bd_\pm B^*$, or both lie in one of $\bd_\pm B^*$. 
	In the first case, we use a smooth isotopy in $U^*$ to write $F^*=\{x\in B^*: x_2=0\}$. 
	Note that the vertical contraction and the horizontal contraction in {\bf Step 3} Case 2 can be combined into our radial contraction in the previous paragraph. 
	Hence, the proof in {\bf Step 3} Case 2 would carry over. 
	In the second case, we change the coordinates on $U^*$ into the form in {\bf Step 3}, i.e. $U^*=\{|x|<2, x_3\geq 0\}$ and $B^*=\{|x|\leq 1, x_3\geq 0\}$, so that 
	\begin{itemize}
		\item the $x_2$-axis is the $(G_p)$-orbit type stratum in $U^*$ (since $P(p)\cong \mb R$); 
		\item $\bd_- B^*=\{x\in \mb R^3: |x|\leq 1, x_1< 0, x_3= 0\}$;
		\item  $\bd_+ B^*= \{x\in \mb R^3: |x|\leq 1, x_1> 0, x_3= 0\}$.
	\end{itemize}
	Then, by further changing coordinates, we can rewrite $F^*=\{x\in B^*: x_1=c\}$ for some $c\in (-1,1)\setminus\{0\}$. 
	Note that the vertical/horizontal contractions in {\bf Step 3} Case 2 still preserve the orbit type strata and give us a smooth isotopy in $U^*$ (\cite{schwarz1980lifting}*{Proposition 3.1}). 
	Therefore, the proof in {\bf Step 3} Case 2 would carry over again. 
\end{proof}

\begin{remark}\label{Rem: no special and approximation}
    We mention that the assumption \eqref{Eq: MSY thm2 assumption - disk instead of cylinder} is used in {\bf Step 3} and {\bf Step 4}, where a disk may not isotopically approximate an annulus in $B_\rho^*([p])$ if $G\cdot p$ is a special exceptional orbit in $M$. 
    One can remove the assumption \eqref{Eq: MSY thm2 assumption - disk instead of cylinder} if there is no special exceptional orbit in $M$. 
\end{remark}

We can now prove Theorem \ref{Thm: MSY thm2 gamma reduction}. 

\begin{proof}[Proof of Theorem \ref{Thm: MSY thm2 gamma reduction}]
	The proof is almost the same as in \cite{meeks82exotic}*{Theorem 2}\cite{li2015general}*{Theorem 7.10}, and we only point out the modifications. 
	Firstly, for the same reason as in \cite{meeks82exotic}*{Remark 3.14}, we can assume $\Sigma_0=\emptyset$. 
	Then we can argue by induction on $J$. 
	In the inductive argument, since $\delta,\rho_0\in (0,1)$ and $(\delta\rho_0)^n \leq (\delta\rho_0)^{n-\dim(G\cdot p)}$, we can use Remark \ref{Rem: MSY 3.6 irreducible} and Lemma \ref{Lem: isoperimetric lemma in MSY} in place of \cite{meeks82exotic}*{Remark 3.6, Lemma 1} respectively. 
	Using $G$-isotopies, we can construct a $G$-hypersurface $\widehat \Sigma_*$ as in \cite{meeks82exotic}*{P. 632} so that the statements in \cite{meeks82exotic}*{P. 632, (i)-(iii)} remain valid except for `$\widehat \Sigma_*\in J_U^G(\Sigma)$'. 
	Nevertheless, we can use Lemma \ref{Lem: G-isotopy approximation} as in \cite{li2015general}*{(7.3)} to obtain $E(\widehat \Sigma_* ) < E(\Sigma) +\mc H^n(F_J) - \mc H^n(D) +\epsilon_0$, i.e. \cite{meeks82exotic}*{P. 632, (iii)'}, where we replaced the notations $q,\epsilon$ by $J,\epsilon_0$ respectively. 
	Finally, using Lemma \ref{Lem: G-isotopy approximation} in place of \cite{li2015general}*{Lemma 7.11}, the proof in \cite{li2015general}*{P. 313-314} can be taken almost verbatim for our $G$-equivariant objects. 
	Indeed, we only need to take $\bd M=\emptyset$, compute the genus in $M^*$, and replace the notations $\mk {Is}_{out},\mc H^2, p,q,\epsilon,A$ by $\mk {Is}^G, \mc H^n, I,J,\epsilon_0, B$ respectively. 
\end{proof}

At the end of this section, we show the following lemma that is analogous to Almgren-Simon \cite{almgren79plateau}*{Lemma 2, Corollary 1}, Meeks-Simon-Yau \cite{meeks82exotic}*{Lemma 2}, and Li \cite{li2015general}*{Lemma 7.12}. 
\begin{lemma}\label{Lem: switch lemma}
    Let $p\in M\setminus \mc S$ and $B=\closure(B_r(G\cdot p))$ be a closed $G$-neighborhood of $G\cdot p$ with $r\in (0,\inj(G\cdot p))$. 
    Suppose $\Sigma_1,\dots,\Sigma_R\in\mc {LB}^G$ are $G$-connected $G$-hypersurfaces in $B$ with $\genus(\Sigma_i^*)=0$, $\Sigma_i\setminus\bd\Sigma_i\subset B\setminus\bd B$, $\bd\Sigma_i\subset\bd B$, so that for any $i\neq j \in \{1,\dots,R\}$, $\bd\Sigma_i\cap\bd\Sigma_j=\emptyset$, and either $\Sigma_i\cap\Sigma_j=\emptyset$ or $\Sigma_i$ is transversal to $\Sigma_j$. 
    
    Then there exist pairwise disjoint $\wti\Sigma_1,\dots,\wti\Sigma_R\in\mc {LB}^G$ in $B$ with $\genus(\wti\Sigma_i^*)=0$, $\wti\Sigma_i\setminus\bd\wti\Sigma_i\subset B\setminus\bd B$, $\bd\wti\Sigma_i=\bd\Sigma_i$ and $\mc H^n(\wti\Sigma_i)\leq \mc H^n(\Sigma_i)$, $i=1,\dots,R$.  
    Moreover, if $\{\Sigma_i\}_{i=1}^R\subset\mc D^G$, then $\{\wti\Sigma_i\}_{i=1}^R\subset\mc D^G$. 
\end{lemma}
\begin{proof}
	The proof is essentially the same as in \cite{li2015general}*{Lemma 7.12} and \cite{meeks82exotic}*{Lemma 2}. 
	Indeed, let $U:=B_\rho(G\cdot p)$ with $r<\rho<\inj(G\cdot p)$. 
	If $U$ is ball-type, i.e. $G\cdot p$ satisfies Proposition \ref{Prop: local structure of M/G}(0) or (2.b). 
	Then we can adapt the constructions in \cite{meeks82exotic}*{Lemma 2} to the orbit space $U^*\cong \mb B^3$ for $\Sigma_i^*\subset B^*$. 
	If $U$ is half-ball-type, i.e. $G\cdot p$ satisfies Proposition \ref{Prop: local structure of M/G}(1) or (2.a). 
	Then we can apply the constructions of \cite{li2015general}*{Lemma 7.12} to $\Sigma_i^*\subset B^*$ in the orbit space $U^*$ with $M$ and $\bd M$ replaced by $U^*$ and $\bd_0 U^*$ respectively. 
	Additionally, although the topological constructions are made in $U^*$, the area is always computed in $U$ by $\mc H^n(\pi^{-1}(\cdot ))$. 
	Note also that the local smoothing and perturbing constructions can both be $G$-equivariant. 
	
	Moreover, we mention that the locally $G$-boundary-type assumption is used to ensure that each $\Sigma_i$ can only meet $M\setminus M^{prin}$ orthogonally (cf. Proposition \ref{Prop: local structure of Sigma/G 1}, \ref{Prop: local structure of Sigma/G 2}), and thus $(\bd\Sigma_i)^*= \closure(\bd_1\Sigma_i^* ) = \closure(\bd \Sigma_i^*\setminus \bd_0 U^*)$. 
	In particular, if $G\cdot p$ satisfies Proposition \ref{Prop: local structure of M/G}(2.b), then the curve $(\bd\Sigma_i)^*$ is simple and closed so that \cite{meeks82exotic}*{Lemma 2} can be adapted.  
	Finally, by the inductive arguments and the cross-joining constructions in \cite{meeks82exotic}*{Lemma 2}\cite{li2015general}*{Lemma 7.12}, we know every $\wti \Sigma_i$ still meets $M\setminus (M^{prin}\cup\mc S)$ orthogonally, and thus is locally $G$-boundary-type.  
\end{proof}

\section{Interior regularity of equivariant Plateau's problem}\label{Sec: plateau}

In this section, we generalize the resolution of the embedded Plateau problem (cf. \cite{almgren79plateau}\cite{meeks82exotic}) to the equivariant case. 
We mainly focus on the interior regularity for our purpose. 

Given $B_{r_0}(G\cdot a)\subset M$ satisfying the estimates in Lemma \ref{Lem: uniform constants}, let $A:=B_{r}(G\cdot a)$ be an open $G$-set with $r\in (0,r_0)$, and $\Gamma\subset \bd A$ be an $(n-1)$-dimensional $G$-connected $G$-submanifold so that $\bd A\setminus\Gamma$ has two $G$-components. 
Consider
\begin{align}
    \mc M^G_{0,\Gamma} :=\{ \Sigma\in\mc M^G_0 : ~\bd\Sigma = \Gamma, ~\Sigma\subset B_{r_0}(G\cdot a) \}. 
\end{align}
Then a sequence $\{\Sigma_k\}_{k\in\mb N}\subset \mc M^G_{0,\Gamma}$ is said to be a {\em minimizing sequence for $\mc M^G_{0,\Gamma}$} if
\begin{align}\label{Eq: plateau minimizing sequence}
    \mc H^n(\Sigma_k) \leq  \inf_{\Sigma'\in \mc M^G_{0,\Gamma}} \mc H^n(\Sigma') + \epsilon_k,
\end{align}
for some positive numbers $\epsilon_k\to 0$. 
Our main result in this section is the following theorem:

\begin{theorem}\label{Thm: plateau problem}
	Using the above notations, suppose \eqref{Eq: cohomogeneity assumption} is satisfied, $\{\Sigma_k\}_{k\in\mb N}\subset \mc M^G_{0,\Gamma}$ is a minimizing sequence for $\mc M^G_{0,\Gamma}$, and $V=\lim_{k\to\infty}|\Sigma_k|\in \mc V^G_n(M)$. 
	Then for any $p_0\in \spt(\|V\|)\setminus \Gamma$, 
    there exist an embedded $G$-hypersurface $\Sigma\subset M$ and $m\in\mb N$ so that 
	\[V\llcorner B_\rho(G\cdot p_0)=m |\Sigma|,\]
    for some $\rho>0$. 
	Moreover, for any $G$-connected component $N\subset M\setminus (M^{prin}\cup\mc S)$ of a non-principal orbit type stratum, $\Sigma$ can only meet $N$ orthogonally. 
\end{theorem}

\begin{remark}\label{Rem: local G-boundary for plateau}
	By the last statement in Theorem \ref{Thm: plateau problem} and the simple connectedness of $B_{\inj(G\cdot p)}^*([p])$ for $p\in M\setminus\mc S_{n.m.}$, we can show that $\spt(\|V\|)$ is locally $G$-boundary-type in $ M\setminus(\mc S_{n.m.}\cup\Gamma)$. 
    Furthermore, 
    the boundary regularity of $V$ is also valid at $\Gamma\cap M^{prin}$. 
\end{remark}

Note that $\pi: M^{prin}\to (M^{prin})^*$ induces a smooth Riemannian submersion. 
By \eqref{Eq: area in orbit space}, Lemma \ref{Lem: minimal in M^prin}, and the classical regularity result in \cite{almgren79plateau}\cite{meeks82exotic}, we know the above regularity theory holds for $p_0\in M^{prin}$. 
In the rest of this section, we will show the regularity and the orthogonal meeting result of $V$ near any $p_0\in M\setminus (M^{prin}\cup \mc S)$ by Theorem \ref{Thm: regularity part1} and \ref{Thm: regularity part2}, where $\mc S$ is defined in Definition \ref{Def: isolated orbits}. 
The proof for $p_0\in \mc S$ will be left to Section \ref{Subsec: plateau}.

\subsection{Regularity near $p\in M\setminus M^{prin}$ with $\dim(P(p))=2$}
Fix a non-principal orbit type stratum $M_{(G_{p_0})}$ and a compact subset $K_0\subset M_{(G_{p_0})}$ with 
\begin{align}\label{Eq: Case dimP=2 assumption}
	\dim(K_0^*)=2.
\end{align}
Set 
\[k_0:=\dim(G\cdot p) \qquad \mbox{for $p\in K_0$,}\] 
and let $\rho_0>0$ be as in Lemma \ref{Lem: uniform constants}. 
In this subsection, all the results are established near some $G\cdot p\subset K_0$. 
Hence, for simplicity, we also denote by
\begin{align}\label{Eq: non-principal part}
	K_1:= (M\setminus M^{prin}) \cap B_{\rho_0}(K_0)
\end{align} 
so that $K_1^*$ is the boundary of $M^*$ in $B_{\rho_0}^*([p])$ for $p\in K_0$. 
Note that $K_0 = K_1$ in $B_{\rho_0}(G\cdot p) $ only in this subsection. 
For simplicity, given any $G$-hypersurface $\Sigma\in\mc {LB}^G$, we denote by 
\begin{align}\label{Eq: boundary components number}
	\mk b(\Sigma):=\#\{\mbox{$G$-connected components of $\Sigma\cap K_1$}\}.
\end{align}
In particular, $\mk b(\Sigma)$ is the number of boundary components of $\Sigma^*$ that intersect $K_1^*$ (a part of the boundary of $M^*$).

Using the notations in Lemma \ref{Lem: uniform constants}, we have the following lemma generalizing \cite{meeks82exotic}*{Lemma 3}. 

\begin{lemma}\label{Lem: convex}
    Let $p_0\in K_0$ and $U\subset B_{\rho_0/2}(G\cdot p_0)$ be a half-ball-type $G$-neighborhood of $G\cdot p_0$ with $C^2$-boundary. 
    Let 
    \[d_U:=\dist(\cdot, U)\qquad{\rm  and}\qquad U(s):=\{p\in M: d_U(p)< s\},\]
    for $s>0$. 
    Suppose 
    \begin{itemize}
        \item $\Delta d_U\geq 0$ in $U(\theta\rho_0)$ for some $\theta\in (0,\frac{1}{2})$ (e.g., $U(s)$ is mean convex for all $s\in (0,\theta\rho_0)$);
        \item there exists $\beta\geq 1$ so that 
        \[\min\{\mc H^n(E), \mc H^n(\bd U(s)\setminus E)\}\leq \beta(\mc H^{n-1}(\bd E))^{\frac{n-k_0}{n-1-k_0}}\]
        for any $s\in (\theta\rho_0/2,\theta\rho_0)$ and $G$-invariant $C^2$ region $E$ in $\bd U(s)$. 
    \end{itemize}
    Then there are $C_1=C_1(K_0,M,G)>1$ and $\delta=\delta(C_1,\beta,\theta,K_0,M,G)>0$ satisfying the following property. 
    
    Given any $\Sigma\in\mc M^G_0$ so that $\bd\Sigma\subset M\setminus U$, $\Sigma$ is transversal to $\bd U$, and 
    \begin{align}\label{Eq: convex lem - small area}
        \mc H^n(\bd U) + \mc H^n(\Sigma)\leq (\delta \rho_0)^{n-k_0}/C_1,
    \end{align}
    if $\Lambda$ is a $G$-component of $\Sigma\setminus U$ with $\bd \Sigma\cap\Lambda=\emptyset$, then there is a unique compact $G$-set $K_\Lambda\subset M\setminus U$ so that 
    \begin{align}\label{Eq: convex lem - volume bound}
        \mc H^{n+1}(K_\Lambda\cap B_{\rho_0}(G\cdot p_0))\leq C_0(\delta\rho_0)^{n+1-k_0}, \quad \bd K_\Lambda = \Lambda\cup F , 
    \end{align}
    where $C_0>1$ is given in Lemma \ref{Lem: uniform constants}, $F\subset \bd U$ is a compact $G$-hypersurface with $\bd F=\bd \Lambda$ and 
    \begin{align}\label{Eq: convex lem - area bound}
        \mc H^n(F) < \mc H^n (\Lambda\cap U(\theta\rho_0)).
    \end{align}
\end{lemma}

\begin{remark}\label{Rem: replacement statement}
	We will see in the proof that if $\Lambda\cap K_1 = \emptyset$, then $F\cap K_1=\emptyset$. 
	Thus, we would not increase the number of boundary components on $K_1^*$ in the replacement constructions (Theorem \ref{Thm: replacement}). 
\end{remark}


\begin{proof}
    By the locally $G$-boundary-type condition on $\Sigma$, we know either
    \begin{itemize}
        \item $\Sigma\cap K_0=\emptyset$ (then $\bd_1(\Sigma\cap U)^*$ are simple closed curves), or
        \item $\Sigma$ meets $K_0$ orthogonally (then $\bd_1(\Sigma\cap U)^*$ are simple curves with (possibly empty) free boundary in $\bd_0 U^*$). 
    \end{itemize}

    Given $\Lambda$ as above, $\bd\Lambda\subset \Sigma\cap\bd U$ is a closed $G$-invariant $(n-1)$-manifold so that $\bd_1\Lambda^*\subset \bd_1 U^*$ is a union of (properly embedded) simple curves that are either closed in (the interior of) $\bd_1U^*$ or have free boundary on $\bd_0U^*$. 
    Since $\bd_1 U^*$ is homeomorphic to a half-sphere $\bd_1\mb B^3_+$, there is a closed $F_0^*\subset\closure(\bd_1 U^*)$ with $\bd F_0^* \cap \bd_1 U^*=\bd_1\Lambda^*$. 
    Hence, $F_0:=\pi^{-1}(F_0^*)\subset \bd U$ is a $G$-hypersurface with $\bd F_0 = \bd \Lambda$. 
    Since $\mc H^n(F_0)+\mc H^n(\Lambda)\leq \mc H^n(\bd U) + \mc H^n(\Sigma)\leq (\delta\rho_0)^{n-k_0}/C_1$, we can shrink $\delta<1$ and take $C_1$ sufficiently large so that there is a unique compact $G$-set $W$ with $\bd W = F_0\cup\Lambda$ and $\mc H^{n+1}(W)\leq C_{iso}(\mc H^n(F_0)+\mc H^n(\Lambda))^{\frac{n+1}{n}}$ due to the isoperimetric inequality in $M$. 
    By the uniqueness of the relative isoperimetric choice in Lemma \ref{Lem: uniform constants}(iv), we can take $C_1>1$ even greater so that 
    \[\bd_{rel} (W\res B_{\rho_0}(G\cdot p_0))=(F_0\cup\Lambda)\cap B_{\rho_0}(G\cdot p_0), ~ ~ and ~ ~ \mc H^{n+1}(W\res B_{\rho_0}(G\cdot p_0))\leq C_0 (\delta\rho_0)^{n+1-k_0}.\]
    Then we set $K_\Lambda:=W\setminus U$ and \eqref{Eq: convex lem - volume bound} holds with either $F=F_0$ or $F=\bd U\setminus F_0$. 
    The uniqueness of $K_\Lambda$ follows from the uniqueness of the isoperimetric choice. 

    Next, we aim to show \eqref{Eq: convex lem - area bound} with $C_1>1$ fixed as above and some $\delta=\delta(C_1,\beta,\theta,K_0,M,G)>0$. 
    For $t\geq 0$, define 
    \[F_t:=K_\Lambda\cap\{x:d_U(x)=t\},\qquad E_t:=\{x\in\Lambda:d_U(x)<t\}. \]
    Since $\Delta d_U\geq 0$ in $U(\theta\rho_0)$, we know
    \[ \mc H^n(F_{t_1}) - \mc H^n(F_{t_2})\leq \int_{\closure(E_{t_2})\setminus E_{t_1}} |\langle \nu, \nabla d_U\rangle|, \qquad \forall 0\leq t_1<t_2\leq \theta\rho_0,\]
    where $\nu$ is the unit normal of $\Lambda$ pointing out of $K_{\Lambda}$. 
    Hence, 
    \begin{align}\label{Eq: convex lem - 4.7}
        \mc H^n(F_{t_1}) - \mc H^n(F_{t_2}) \leq \mc H^n(\closure(E_{t_2})) - \mc H^n(E_{t_1}), \qquad\forall 0\leq t_1<t_2\leq \theta\rho_0,
    \end{align}
    where the strict inequality holds provided $E_{t_1}\neq E_{t_2}$. 

    By \eqref{Eq: convex lem - volume bound} and the co-area formula, we see 
    \[\int_0^{\theta\rho_0} \mc H^n(F_s) ds \leq \mc H^{n+1}(K_\Lambda\cap U(\theta\rho_0))\leq C_0(\delta\rho_0)^{n+1-k_0},\]
    and thus $\mc H^n(F_t)\leq 4C_0\theta^{-1}\delta^{n+1-k_0}\rho_0^{n-k_0}$ for a set of $t\in [0,\theta\rho_0]$ of Lebesgue measure at least $\frac{3}{4}\theta\rho_0$. 
    By shrinking $\delta<\theta/(4C_0C_1)$, we see from \eqref{Eq: convex lem - small area}\eqref{Eq: convex lem - 4.7} that 
    \begin{align}\label{Eq: convex lem - 4.8}
        \mc H^n(F_t)\leq \frac{2}{C_1} \cdot \delta^{n-k_0}\rho_0^{n-k_0} \quad {\rm for~all~}t\in [0,3\theta\rho_0/4]. 
    \end{align}
    In particular, \eqref{Eq: convex lem - 4.7} implies 
    \begin{align}\label{Eq: convex lem - 4.9}
        \mc H^n(F)-\mc H^n(\closure(E_t)) < \mc H^n(F_t), \quad t\in (0,\theta\rho_0]. 
    \end{align}
    Hence, if $\mc H^n(F_{\theta\rho_0})=0$, we have \eqref{Eq: convex lem - area bound} by setting $t=\theta\rho_0$. 

    Suppose $\mc H^n(F_{\theta\rho_0})> 0$. 
    Then for any $t\in [\theta\rho_0/2,\theta\rho_0]$, we can use \eqref{Eq: volume in slice} and Lemma \ref{Lem: uniform constants} to show:
    \begin{align*}
        \mc H^n(\bd U(t)) &\geq \frac{1}{2}\mc H^{k_0}(G\cdot p_0) \mc H^{n-k_0}(\bd U(t)\cap S_{\rho_0}(p_0)) \\
        &\geq \frac{1}{2}\mc H^{k_0}(G\cdot p_0)\frac{1}{4^{n-k_0}} \left(\mc H^{n+1-k_0}(U(t)\cap S_{\rho_0}(p_0)) \right)^{\frac{n-k_0}{n+1-k_0}} \\
        &\geq \frac{1}{2^{2n-2k_0+1}}\left(\mc H^{k_0}(G\cdot p_0)\right)^{\frac{1}{n+1-k_0}} \left(\frac{1}{2}\mc H^{n+1}(B_{\theta\rho_0/2}(G\cdot x)) \right)^{\frac{n-k_0}{n+1-k_0}}\\
        &\geq \frac{1}{2^{2(n-k_0+1)+\frac{n-k_0}{n+1-k_0}}}\left(\mc H^{k_0}(G\cdot p_0)\right)^{\frac{1}{n+1-k_0}}\left(\mc H^{k_0}(G\cdot x) \cdot \mc H^{n+1-k_0}(S_{\frac{\theta\rho_0}{2}}(x))\right)^{\frac{n-k_0}{n+1-k_0}}\\
        &\geq \frac{\inf_{y\in K_0} \mc H^{k_0}(G\cdot y)}{2^{2n-2k_0+4}} \left( \mc H^{n+1-k_0}(\mb B^{n+1-k_0}_{\theta\rho_0/2}(0)) \right)^{\frac{n-k_0}{n+1-k_0}}\\
        &\geq \frac{|\mb B^{n+1-k_0}_1(0)|}{C_0 2^{3n-3k_0+4}} (\theta\rho_0)^{n-k_0},
    \end{align*}
    where $x\in U\cap K_0$, $B_{\theta\rho_0/2}(G\cdot x)\subset U(\theta\rho_0/2)$, and we used the isoperimetric inequality (in $\mb B^{n+1-k_0}\cong (\exp_{G\cdot p_0}^\perp)^{-1}(S_{\rho_0}(p_0))$ with Lipschitz constant $\in [1/2,2]$) in the second line. 
    Combining this with \eqref{Eq: convex lem - 4.8}, we can shrink $\delta<\theta\cdot (\frac{C_1\cdot|\mb B^{n+1-k_0}_1(0)|}{C_0 \cdot 2^{3n-3k_0+6}})^{\frac{1}{n-k_0}}$ so that 
    \[\mc H^n(F_t)\leq \frac{1}{2}\mc H^n(\bd U(t))\quad \mbox{for all } t\in [\theta\rho_0/2,3\theta\rho_0/4]. \]
    By the choice of $\beta$ in the second bullet, we can apply the co-area formula to show 
    \begin{align}\label{Eq: convex lem - 4.10}
        \mc H^n(F_t)\leq \beta \left(\mc H^{n-1}(\bd F_t)\right)^{\frac{n-k_0}{n-1-k_0}} = \beta \left(\mc H^{n-1}(\bd E_t)\right)^{\frac{n-k_0}{n-1-k_0}}\leq \beta\left(\frac{d}{dt} \mc H^n(E_t)\right)^{\frac{n-k_0}{n-1-k_0}}
    \end{align}
    for almost all $t\in [\theta\rho_0/2,3\theta\rho_0/4]$. 
    Thus, \eqref{Eq: convex lem - 4.9} implies that for almost all $t\in [\theta\rho_0/2,3\theta\rho_0/4]$,
    \begin{align}\label{Eq: convex lem - 4.11}
        \mc H^n(F) - \mc H^n(E_t)\leq \mc H^n(F_t)\leq \beta\left(\frac{d}{dt} [- (\mc H^n(F)-\mc H^n(E_t))] \right)^{\frac{n-k_0}{n-1-k_0}} .
    \end{align}
    If $\mc H^n(F)>\mc H^n(E_{3\theta\rho_0/4})$, then $A_t:=(\mc H^n(F)-\mc H^n(E_t))^{\frac{1}{n-k_0}} > 0$ for $t\in [\theta\rho_0/2,3\theta\rho_0/4]$. 
    Hence, 
    \[A_t^{n-1-k_0}\leq \beta^{\frac{n-1-k_0}{n-k_0}}\frac{d}{dt}(-A_t^{n-k_0}) = \beta^{\frac{n-1-k_0}{n-k_0}} (n-k_0)\cdot A_t^{n-1-k_0}\frac{d}{dt}(-A_t)\]
    and $\frac{d}{dt}(-A_t) \geq \frac{1}{n-k_0}\beta^{-\frac{n-1-k_0}{n-k_0}} $.
    After integrating over $[\theta\rho_0/2,3\theta\rho_0/4]$,
    \[A_{\theta\rho_0/2}-A_{3\theta\rho_0/4}\geq \frac{1}{n-k_0}\beta^{-\frac{n-1-k_0}{n-k_0}} \frac{\theta\rho_0}{4}.\]
    However, $A_{\theta\rho_0/2}<\left(\mc H^n(F)\right)^{\frac{1}{n-k_0}}< \delta\rho_0\cdot (\frac{2}{C_1})^{\frac{1}{n-k_0}}$, which gives a contradiction by shrinking $\delta< (\frac{C_1}{2})^{\frac{1}{n-k_0}}\beta^{-\frac{n-1-k_0}{n-k_0}}\frac{\theta}{4(n-k_0)} $ small enough. 
    Therefore, $\mc H^n(F)\leq \mc H^n(E_{3\theta\rho_0/4})<\mc H^n(E_{\theta\rho_0})$ since $\mc H^n(F_{\theta\rho_0})\neq 0$. 
\end{proof}

In the following theorem, we use the notation 
\[U_\eta:=\{x\in U: \dist(x,\bd U)>\eta\},\]
where $U$ is a $C^2$ open $G$-set and $\eta>0$. 

\begin{theorem}[Replacement Theorem]\label{Thm: replacement}
Let $p_0,\rho_0, U, \theta, \beta, C_1, \delta$ be the same as in Lemma \ref{Lem: convex}. Suppose $\eta>0$ is given and 
\begin{enumerate}[label=(\roman*)]
    \item $\Sigma\in \mc M^G_0$, $\partial \Sigma\subset M\setminus U$, and $\partial \Sigma$ does not intersect $K_\Lambda$, where $\Lambda$ and $K_\Lambda$ are as in Lemma \ref{Lem: convex};
    \item $\Sigma$ and $U$ satisfy \eqref{Eq: convex lem - small area};
    \item \label{item: transversality} $\Sigma$ intersects $\partial U$ transversally; in case $\partial \Sigma\cap \partial U\neq \emptyset$, we will always take this to mean that there is a $C^2$ (open) $G$-hypersurface $N$ with $\Sigma\subset N$, with $(N\setminus\Sigma)\cap U=\emptyset$ and with $N$ intersecting $\partial U$ transversally.
\end{enumerate}
Then there is $\wti \Sigma\in \mc M^G_0$ with $\mk b(\wti\Sigma)\leq \mk b(\Sigma)$ such that 
\begin{enumerate}[label=(\roman*), resume]
    \item $\partial \Sigma=\partial \wti\Sigma$, $\wti \Sigma\setminus U\subset \Sigma\setminus U$, $\wti \Sigma\cap U_\eta\subset \Sigma\cap U_\eta$;
    \item $\wti \Sigma$ intersects $\partial U$ transversally (in the same sense as in \ref{item: transversality});
    \item $\mc H^n(\wti \Sigma)+\mc H^n((\Sigma\setminus \wti \Sigma)\cap U_\eta)\leq \mc H^n(\Sigma)$;
    \item $\wti\Sigma\cap \closure (U)$ is a disjoint union $\cup_{j=1}^k\Gamma_j$ of elements $\Gamma_j\in \mc M^G_0$ (and consequently $\wti \Sigma\subset \closure (U)$ in case $\partial \Sigma\subset \partial U$).
\end{enumerate}
If, in addition, we have 
\begin{enumerate}[label=(\roman*), resume]
    \item $\mc H^n(\Sigma)\leq \mc H^n(P)+\epsilon$ for every $P\in \mc M^G_0$ satisfying $\partial P=\partial \Sigma$ and $\mk b(P) \leq \mk b(\Sigma)$,
\end{enumerate}
then there are nonnegative numbers $\epsilon_1,\epsilon_2,\cdots,\epsilon_k$ with $\sum_{j=1}^k\epsilon_j\leq \epsilon$ so that 
\begin{enumerate}
    \item[(ix)] $\mc H^n(\Gamma_j)\leq \mc H^n(P)+\epsilon_j$ for every $P\in \mc M^G_0$ satisfying $\partial P=\partial \Gamma_j$ and $\mk b(P) \leq \mk b(\Gamma_j)$, $j=1,\cdots, k$.
\end{enumerate}
Moreover, if (viii) is valid for all $P\in \mc M^G_0$ satisfying $\partial P=\partial \Sigma$, then (ix) is also valid for all $P\in \mc M^G_0$ satisfying $\partial P=\partial \Gamma_j$. 
\end{theorem}

Note that if $\Sigma$ is a disk-type $G$-hypersurface, then each $\Gamma_j$ is disk-type.

\begin{proof}
Note that the above assumptions have been modified as in \cite{meeks82exotic}*{\S 4}. 
Using Lemma \ref{Lem: switch lemma} in place of \cite{almgren79plateau}*{Corollary 1}, the proof of \cite{almgren79plateau}*{Theorem 1} (see also \cite{jost1986existenceI}*{Lemma 4.4} for the free boundary version) would carry over in our setting. 
\end{proof}

\begin{lemma}[Filigree Lemma]\label{Lem: filigree}
    Let $\{Y_t\}_{t\in[0,1]}$ be an increasing family of half-ball-type $G$-sets satisfying the assumptions for $U$ in Lemma \ref{Lem: convex} (with constants $\beta,\theta$ independent of $t$). 
    Suppose $f$ is a $G$-invariant non-negative function on $M$ so that $f$ is $C^2$ on $M\setminus\closure(Y_0)$, $\nabla f\neq 0$ on $Y_1\setminus \closure(Y_0)$, 
    \[Y_t=\{x\in M: f(x)<t\} {\rm ~for~all~} t\in (0,1],\quad{\rm and}\quad \sup_{\closure(Y_1)\setminus\closure(Y_0)}|\nabla f|\leq c_1\]
    for some $c_1>0$. 
    Suppose also that there is a constant $c_2<\infty$ so that, whenever $\gamma$ is a $G$-connected $G$-invariant $C^2$ $(n-1)$-manifold on $\bd Y_t$ such that $\gamma\not\subset K_1\cap \bd Y_t$ and 
    $\gamma^*$ is a simple curve in $\closure(\bd_1Y_t^*)$ with (possibly empty) boundary on $\closure(\bd_1Y_t^*)\cap \bd_0 Y_t^*$, there is a $G$-invariant region $E\subset \bd Y_t$ with $\bd E=\gamma$ (i.e. $\bd_1 E^*=\gamma^*\cap\bd_1 Y_t^*$), satisfying 
    \begin{align}\label{Eq: compare to disk}
    	E^*\cap \bd_0Y_t^*=\emptyset \qquad \mbox{whenever $\gamma^*\cap \bd_0 Y_t^*=\emptyset$ and $n+1-k_0=\cohom(G)=3$,}
    \end{align}
    and 
    \begin{align}\label{Eq: filigree - isoperimetric}
        \mc H^{n}(E)\leq c_2\left(\mc H^{n-1}(\bd E)\right)^{\frac{n-k_0}{n-k_0-1}}.
    \end{align}
    Suppose $\Sigma\in\mc M^G_0$ satisfies the assumptions in Theorem \ref{Thm: replacement}(i) with $Y_t$ in place of $U$ for every $t\in (0,1)$, 
    \begin{align}\label{Eq: filigree - not all non-principal assumption}
    	\Sigma\not\subset K_1\qquad {\rm in~}Y_t,
    \end{align}
    and there exists $\epsilon>0$ so that 
    \begin{align}\label{Eq: filigree - almost minimizing}
        \mc H^n(\Sigma)\leq \mc H^n(\Gamma') + \epsilon,\quad\forall \Gamma'\in\mc M^G_0 {\rm ~with~}\bd\Gamma'=\bd\Sigma {\rm ~and~} \mk b(\Gamma')\leq\mk b(\Sigma). 
    \end{align}
    Then
    \begin{align}\label{Eq: filigree}
        \mc H^n(\Sigma\cap Y_t)\leq 2\epsilon, \quad \forall 0\leq t \leq 1-(n-k_0)\cdot c_1\cdot c_2^{\frac{n-k_0-1}{n-k_0}}\left(\mc H^n(\Sigma\cap Y_1)\right)^{\frac{1}{n-k_0}} .
    \end{align}
\end{lemma}

\begin{remark}\label{Rem: filigree lemma statements}
	We make the following comments on the assumptions in the above lemma. 
	\begin{itemize}
		\item[(i)] The assumption \eqref{Eq: compare to disk} is necessary due to the assumption on $\mk b$ in \eqref{Eq: filigree - almost minimizing}. 
			One notices that if $\Sigma^*$ is a disk $D$ in the half-ball $Y_1^*\cong \mb B^3_+$ with $\bd D\subset \interior(\mb S^2_+)\subset \bd_1 \mb B^3_+$, 
			then $\mk b(\Sigma)=0$, and $\interior(\mb S^2_+)\setminus \bd D$ is the union of a spherical region $E_s^*$ and a cylindrical region $E_c^*$. 
			Hence, using \eqref{Eq: filigree - almost minimizing}, one can compare the area of $\Sigma=\pi^{-1}(D)$ with the area of $E_s$ instead of the area of $E_c$, since $\mk b(E_c)=1$. 
			Thus, we need \eqref{Eq: compare to disk} to guarantee that $E^*=E_s^*$ is a disk in certain cases.  
			Nevertheless, if we remove the requirement $\mk b(\Gamma')\leq \mk b(\Sigma)$ in \eqref{Eq: filigree - almost minimizing}, then \eqref{Eq: compare to disk} can be removed too. 
            A similar obstacle was also considered in \cite{jost1986existenceI}*{Theorem 5.2} with an additional curvature assumption on the boundary. 
		\item[(ii)] In our applications, we will take $Y_t\subset B_{2\rho_0}(G\cdot p_0)$ so that in the slice, $\bd Y_t \cap S_{2\rho_0}(p_0)$ is a portion of $(n-k_0)$-dimensional spheres, or a cylinder with $(n-k_0-1)$-dimensional circular cross-section, which have such a constant $c_2$ satisfying \eqref{Eq: filigree - isoperimetric}.  
			Moreover, similar to Remark \ref{Rem: compare disk instead cylinder}, we can also find a constant $c_2$ so that \eqref{Eq: compare to disk} and \eqref{Eq: filigree - isoperimetric} are simultaneously valid. 
		\item[(iii)] Note that \eqref{Eq: filigree - not all non-principal assumption} is listed just for clarity, and is directly guaranteed by the locally $G$-boundary-type assumption. 
        Compared with the Filigree Lemma in \cite{almgren79plateau}, the additional assumptions \eqref{Eq: compare to disk} and \eqref{Eq: filigree - not all non-principal assumption} are required because $Y_t^*$ are relative open sets near the boundary of $M^*$. 
	\end{itemize}	
\end{remark}

\begin{proof}
    By Sard's theorem, $\Sigma$ is transversal to $\bd Y_t$ for almost all $t\in (0,1)$ so that Theorem \ref{Thm: replacement} is applicable. 
    Hence, for almost all $t\in (0,1)$, there exists $\wti\Sigma\in \mc M^G_0$ with $\bd\wti\Sigma = \bd\Sigma$, $\mk b(\wti\Sigma)\leq\mk b(\Sigma)$, $\mc H^n(\wti\Sigma) \leq \mc H^n (\Sigma)$, $\wti\Sigma\cap \bd Y_t\subset\Sigma\cap\bd Y_t$, and $\wti\Sigma\cap \closure(Y_t) = \cup_{j=1}^k \Gamma_j$ for some $\{\Gamma_j\}\subset\mc M^G_0$ so that 
    \begin{align}\label{Eq: filigree - 4.4}
        \mc H^n(\Gamma_j)\leq\mc H^n(\Gamma') +\epsilon_j, \quad\forall\Gamma'\in\mc M^G_0 {\rm ~with~}\bd\Gamma'=\bd\Gamma_j {\rm ~and~} \mk b(\Gamma')\leq\mk b(\Gamma_j),
    \end{align}
    where $\sum_{j=1}^k\epsilon_j\leq\epsilon$. 
	
	\begin{claim}\label{Claim: compare to disk}
		For each $j\in\{1,\dots, k\}$, we have $\mc H^n(\Gamma_j)\leq c_2\left(\mc H^{n-1}(\Gamma_j\cap\bd Y_t)\right)^{\frac{n-k_0}{n-k_0-1}} + \epsilon_j$.
	\end{claim}
	\begin{proof}[Proof of Claim \ref{Claim: compare to disk}]
		Let $j\in\{1,\dots, k\}$ and $\gamma=\bd \Gamma_j$. By hypothesis, we have a $G$-set $E\subset \bd Y_t$ with $\bd E=\gamma$ and \eqref{Eq: compare to disk}\eqref{Eq: filigree - isoperimetric}. 
		By Lemma \ref{Lem: hypersurface position} and Proposition \ref{Prop: local structure of Sigma/G 1}, there are three possibilities:
		\begin{itemize}
			\item[(a)] $\gamma^*\subset\bd_0Y_t^*$, then $\Sigma=\wti\Sigma=\Gamma_j\subset K_1$ in $Y_t$; this case has been ruled out by \eqref{Eq: filigree - not all non-principal assumption};
			\item[(b)] $\gamma^*\subset \closure(\bd_1Y_t^*)$ is a Jordan arc with free boundary on $\bd_0Y_t^*$;
			\item[(c)] $\gamma^*\subset \bd_1Y_t^*$ is a simple closed curve with $\gamma^*\cap \bd_0Y_t^*=\emptyset$. 
		\end{itemize}
		In case (b), we know $E$ is a half-disk-type $G$-hypersurface with $\bd E=\bd \Gamma_j$ and $\mk b(E)=1\leq \mk b(\Gamma_j)$. 
		Hence, it follows from \eqref{Eq: filigree - 4.4}\eqref{Eq: filigree - isoperimetric} that 
		\[\mc H^n(\Gamma_j)~\leq~\mc H^n(E)+\epsilon_j ~\leq~ c_2\left(\mc H^{n-1}(\bd E)\right)^{\frac{n-k_0}{n-k_0-1}} + \epsilon_j ~=~ c_2\left(\mc H^{n-1}(\Gamma_j\cap\bd Y_t)\right)^{\frac{n-k_0}{n-k_0-1}} + \epsilon_j. \]
		In case (c), we know 
		$\gamma^*$ separates $\bd_1Y_t^*$ into a disk region $E_s^*$ (with $\mk b(E_s)=0$) 
		and a cylindrical region $E_c^*$ (with $\mk b(E_c)=1$). 
		By \eqref{Eq: compare to disk}, $E=\pi^{-1}(E_s^*)$ is a disk-type $G$-hypersurface satisfying \eqref{Eq: filigree - isoperimetric} provided $n+1-k_0=\cohom(G)=3$, and thus the above inequality is still valid. 
		If $n+1-k_0>\cohom(G)=3$ and $E=\pi^{-1}(E_s^*)$, the proof is the same. 
		However, if $n+1-k_0>\cohom(G)=3$ and $E=\pi^{-1}(E_c^*)$, we cannot compare $\mc H^n(\Gamma_j)$ with $\mc H^n(E)$ directly since it may happen that $\mk b(E)=1>\mk b(\Gamma_j)$. 
		Nevertheless, since $E^*=E_c^*$ is the region bounded between $\gamma^*$ and $\bd_0Y^*_t\cap \closure(\bd_1Y^*_t)$, we can choose $\tau>0$ arbitrarily small so that $\bd B_\tau(K_1)$ is transversal to $\bd Y_t$ 
        and $\bd_1 B_\tau^*(K_1) \cap  Y_t^* $ is a disk. 
        Then the union $F^*$ of $\bd_1 B_\tau^*(K_1)\cap Y_t^*$ and the cylinder $E_c^*\setminus B_\tau^*(K_1)$ is a disk with $\bd F^*=\gamma^*$ so that 
		\[ \mc H^n(\Gamma_j)\leq \mc H^n(\pi^{-1}(F^*)) +\epsilon_j \leq \mc H^n(E) + \mc H^n(\bd B_\tau(K_1)\cap Y_t) +\epsilon_j. \]
		Note that the dimension assumption in this case indicates $\dim(K_1\cap B_{\rho_0}(G\cdot p_0))<n$, and thus $\mc H^n(\bd B_\tau(K_1)\cap Y_t)\to 0$ as $\tau\to 0$. 
		Therefore, after taking $\tau \to 0$, we still have $\mc H^n(\Gamma_j)\leq \mc H^n(E) +\epsilon_j$ to apply an argument similar to that used in the previous cases. 
	\end{proof}
	
    Next, summing the inequality in Claim \ref{Claim: compare to disk} over $j$, we obtain
    \[\mc H^n(\wti\Sigma\cap Y_t)~\leq~ c_2\Big[\sum_{j=1}^k\left(\mc H^{n-1}(\Gamma_j\cap\bd Y_t)\right)^{\frac{n-k_0}{n-k_0-1}}\Big] + \epsilon ~\leq~ c_2 \left(\mc H^{n-1}(\wti\Sigma\cap\bd Y_t)\right)^{\frac{n-k_0}{n-k_0-1}} + \epsilon .\]
    Noting $\wti\Sigma\cap\bd Y_t\subset \Sigma\cap \bd Y_t$, we see 
    \begin{align}\label{Eq: filigree - 4.5}
        \mc H^n(\wti\Sigma\cap Y_t)~\leq~ c_2 \left(\mc H^{n-1}(\Sigma\cap\bd Y_t)\right)^{\frac{n-k_0}{n-k_0-1}} + \epsilon. 
    \end{align}
    Combined with \eqref{Eq: filigree - almost minimizing} and the fact that $\wti\Sigma\setminus Y_t\subset\Sigma\setminus Y_t$, we conclude 
    \begin{align}\label{Eq: filigree - 4.6}
        \mc H^n(\Sigma\cap Y_t) ~\leq~ \mc H^n(\wti\Sigma\cap Y_t) + \epsilon ~\leq~ c_2 \left(\mc H^{n-1}(\Sigma\cap\bd Y_t)\right)^{\frac{n-k_0}{n-k_0-1}} + 2\epsilon . 
    \end{align}
    Without loss of generality, we assume $\mc H^n(\Sigma\cap Y_1)\geq 2\epsilon$ and set $t_0:=\inf\{t:\mc H^n(\Sigma\cap Y_t)>2\epsilon\}$. 
    Define $F(t):=\mc H^n(\Sigma\cap Y_t)-2\epsilon> 0$ for $t\in (t_0,1]$. 
    By the co-area formula and \eqref{Eq: filigree - 4.6}, we see $F'(t)=\int_{\{x\in\Sigma:f(x)=t\}}\frac{1}{|\nabla f|} \geq \frac{1}{c_1} \mc H^{n-1}(\Sigma\cap\bd Y_t)$ for almost all $t\in (t_0,1]$, and thus 
    \[F(t)\leq c_2 \left(\mc H^{n-1}(\Sigma\cap\bd Y_t)\right)^{\frac{n-k_0}{n-k_0-1}}\leq c_2\left( c_1 F'(t) \right)^{\frac{n-k_0}{n-k_0-1}}\quad {\rm a.e.~}t\in(t_0,1].\]
    For $A(t):=(F(t))^{\frac{1}{n-k_0}}>0$ on $(t_0,1]$, the above inequality implies $ A'(t)\geq \frac{1}{n-k_0}c_1^{-1}c_2^{-\frac{n-k_0-1}{n-k_0}}$ for almost all $t\in (t_0,1]$. 
    By integration, we see 
    \[0=A(t_0)\leq A(1) - \frac{1-t_0}{n-k_0}c_1^{-1}c_2^{-\frac{n-k_0-1}{n-k_0}}.\]
    Combined with $A(1)=(F(1))^{\frac{1}{n-k_0}}<(\mc H^n(\Sigma\cap Y_1))^{\frac{1}{n-k_0}}$, we deduce 
    \[1-t_0<(n-k_0)c_1c_2^{\frac{n-k_0-1}{n-k_0}}(\mc H^n(\Sigma\cap Y_1))^{\frac{1}{n-k_0}},\]
    which gives \eqref{Eq: filigree}. 
\end{proof}

We now show the following regularity theorem, which is analogous to \cite{almgren79plateau}*{Theorem 2}.

\begin{theorem}[First Regularity]\label{Thm: first regularity}
    Let $K_0,K_1,k_0,\rho_0$ be defined as above, $p_0\in K_0$, and $B:= B_{2\rho_0}(G\cdot p_0)$. 
    Suppose $\{\Sigma_k\}_{k\in\mb N}$ is a sequence in $\mc M^G_0$ with $\bd\Sigma_k\subset \bd B$, $\Sigma_k\setminus\bd\Sigma_k\subset B$, and 
    \[\mc H^n(\Sigma_k)\leq \mc H^n(\Gamma') + \epsilon_k,\qquad\forall\Gamma'\in\mc M^G_0 {\rm ~with~}\bd\Gamma'=\bd\Sigma_k,\]
    where $\epsilon_k\to 0$ as $k\to\infty$. 
    Suppose further that $W:=\lim_{k\to\infty}|\Sigma_k|$ exists in $\mc V^G_n(M)$. 

    Then $W$ is a stationary integral varifold in $B$ so that if $W$ has a tangent varifold $C$ at $x_0\in\spt(\|W\|)\cap B$ with $\spt(\|C\|)\subset H$ for some $n$-plane $H\subset T_{x_0}M$, then 
    \begin{itemize}
    	\item $H$ must be orthogonal to $T_{x_0}K_0$ whenever $x_0\in K_0$, and
    	\item there exists $\rho>0$ such that 
    \begin{align}\label{Eq: first regularity - regular}
        W\res G_n(B_{\rho}(G\cdot x_0)) = m |\Sigma|,
    \end{align}
    \end{itemize}
    where $m\in\mb N$ and $\Sigma$ is an embedded $G$-connected minimal $G$-hypersurface containing $G\cdot x_0$. 
\end{theorem}

\begin{remark}\label{Rem: regularity in Mprin}
	For any $x_0\in B\cap M^{prin}$ with $B_\rho(G\cdot x_0)\subset B$, we have a Riemannian submersion $\pi: (B_\rho(G\cdot x_0), g_{_M})\to (B^*_\rho([x_0]), g_{_{M^*}})$. 
    Then, using \eqref{Eq: area in orbit space} and Lemma \ref{Lem: minimal in M^prin}, the regularity of $W$ in $B_\rho(G\cdot x_0)\subset B\setminus K_0$ follows directly from \cite{almgren79plateau}*{Theorem 2, 4} in $(B_\rho^*([x_0]), \tilde{g}_{_{M^*}})$. 
    
    In addition to the regularity of $W$ in $B\setminus K_0$, the construction of $W$ also implies $\spt(\|W\|)$ is a $G$-stable minimal $G$-hypersurface in $B\setminus K_0$. 
    Moreover, by shrinking $B\setminus K_0$ slightly, we can obtain an open $G$-set $U\subset\subset M^{prin}$ so that $U^*$ is a ball and $\bd U$ is transversal to $\spt(\|W\|)$. 
    Thus, every component of $(\spt(\|W\|))^*\cap U^*$ separates $U^*$ into two components, which indicates that every $G$-component of $\spt(\|W\|)\cap U$ separates $U$ into two $G$-components. 
    Therefore, $\spt(\|W\|)\cap U$ admits a $G$-invariant unit normal. 
    By Lemma \ref{Lem: stability and G-stability} and the arbitrariness of $U$, we know $\spt(\|W\|)\cap (B\setminus K_0)$ is a {\em stable} minimal $G$-hypersurface with integer multiplicity. 
\end{remark}

\begin{proof}
    One easily verifies that $W$ is a $G$-varifold and is $G$-stationary in $B$, which implies $W$ is stationary in $B$ (Lemma \ref{Lem: first variation and G-variation}). 
    By Remark \ref{Rem: regularity in Mprin}, we mainly focus on $x_0\in \spt(\|W\|)\cap B\setminus M^{prin} \subset K_0$. 
    By the locally $G$-boundary-type assumption, every $\Sigma_k$ either does not intersect $K_1$ or is orthogonal to $K_1$ by Lemma \ref{Lem: hypersurface position}. 

    \begin{claim}\label{Claim: first regularity - 5.2}
        There is a constant $c>0$ so that $\Theta^n(\|W\|,x_1)\geq c$ for all $x_1\in\spt(\|W\|)\cap B$, and thus $W\res G_n(B)$ is rectifiable. 
    \end{claim}
    \begin{proof}[Proof of Claim \ref{Claim: first regularity - 5.2}]
        By the regularity results in $B\cap M^{prin}$, we see $\Theta^n(\|W\|,x_1)\geq 1$ for all $x_1\in\spt(\|W\|)\cap B\cap M^{prin}$. 
        Next, consider $x_1\in \spt(\|W\|)\cap B\setminus M^{prin}\subset K_1$ and $B_\rho(G\cdot x_1)\subset B$. 

        Combining the isoperimetric inequality in Euclidean spheres with \eqref{Eq: area in slice} and Lemma \ref{Lem: uniform constants}(iii), we know $Y_t=B_{t\rho}(G\cdot x_1)$ satisfies the assumptions in the Filigree Lemma \ref{Lem: filigree}. 
        In particular, Remark \ref{Rem: compare disk instead cylinder} indicates that \eqref{Eq: compare to disk} and \eqref{Eq: filigree - isoperimetric} can be valid simultaneously for $Y_t=B_{t\rho}(G\cdot x_1)$. 
        Therefore, whenever $\mc H^n(\Sigma_k\cap B_\rho(G\cdot x_1))\leq (\frac{\rho}{2(n-k_0) c_2})^{n-k_0}$, we can apply the Filigree Lemma \ref{Lem: filigree} with $f(x)=\frac{1}{\rho}\dist(x,G\cdot x_1)$, $c_1=\frac{1}{\rho}$, $Y_t=B_{t\rho}(G\cdot x_1)$ to get $\mc H^n(\Sigma_k\cap B_{\rho/2}(G\cdot x_1)) < 2\epsilon_k$. 
        Since $\epsilon_k\to 0$ and $x_1\in\spt(\|W\|)$, we must have 
        \begin{align}\label{Eq: first regularity - 5.3}
            \mc H^n(\Sigma_k\cap B_\rho(G\cdot x_1)) \geq \frac{\rho^{n-k_0}}{(2(n-k_0))^{n-k_0} c_2^{n-k_0}} \quad \mbox{for every $k$ large enough,}
        \end{align}
        which implies 
        \begin{align}\label{Eq: first regularity - 5.4}
            \|W\|(B_\rho(G\cdot x_1))\geq \frac{\rho^{n-k_0}}{(2(n-k_0))^{n-k_0} c_2^{n-k_0}}. 
        \end{align}
        For $s,t>0$, denote by $\mf n:B_\rho(G\cdot x_1)\to G\cdot x_1$ the nearest projection, and recall that
        \begin{align}\label{Eq: part of tube}
            T(x_1,s,t):=\{x\in B_t(G\cdot x_1): \dist_{G\cdot x_1}(\mf n(x), x_1)<s \}
        \end{align}
        is the part of the tube $B_t(G\cdot x_1)$ centered at an $s$-neighborhood of $x_1$ in $G\cdot x_1$. 
        By \eqref{Eq: first regularity - 5.4},
        \begin{align}\label{Eq: first regularity - 5.4.2}
            \|W\|(B_{2\rho}(x_1)) \geq \|W\|(T(x_1,\rho,\rho)) = \|W\|(B_{\rho}(G\cdot x_1)) \frac{\mc H^{k_0}(B^{G\cdot x_1}_\rho(x_1))}{\mc H^{k_0}(G\cdot x_1)} \geq c\rho^n,
        \end{align}
        where $B^{G\cdot x_1}_\rho(x_1)$ is the ($k_0$-dimensional) geodesic $\rho$-ball at $x_1$ in $G\cdot x_1$ with area at least $c_{k_0}\rho^{k_0}$ for some dimensional constant $c_{k_0}>0$ and every $\rho>0$ small enough, and $c>0$ is a uniform constant depending only on $c_2,n,k_0,K_0$. 
    \end{proof}

    Next, given $x_0\in \spt(\|W\|)\cap B$, we can use $(\exp_{G\cdot x_0}^{\perp})^{-1}$ to pull everything into the normal bundle $(\exp_{G\cdot x_0}^{\perp})^{-1}(B)\subset N(G\cdot x_0)$. 
    By abuse of notation, we often identify $G\cdot x_0$ with the zero section of $N(G\cdot x_0)$ and regard the ambient region $B$ as a $G$-neighborhood in $N(G\cdot x_0)$. 
    Then for any $r>0$ and $q=(g\cdot x_0, v)\in B\subset N(G\cdot x_0)$, define 
    \begin{itemize}
        \item $\btau_{g\cdot x_0} : \mb R^L\to \mb R^L $, $\btau_{g\cdot x_0}(q):=q- g\cdot x_0$ to be the translation;
        \item $\bmu_r: \mb R^L\to \mb R^L $, $\bmu_r(q)=q/r$;
        \item $\bleta_{r}= \bmu_r\circ\btau_{x_0} : B\subset \mb R^L\to \mb R^L$, $\bleta_{r}(g\cdot x_0, v):=\frac{(g\cdot x_0, v) - (x_0,0)}{r}$;
        \item $\bleta^G_{r}: N(G\cdot x_0)\to N(G\cdot x_0)$, $\bleta^G_{r}(g\cdot x_0, v):=(g\cdot x_0, v/r)$;
        \item $g_{_r}$ is the metric $(\exp_{G\cdot x_0}^\perp\circ (\bleta^G_r)^{-1})^* g_{_M}$ with a $\frac{1}{r^2}$-rescaling on the tangent space of each fiber $N_{g\cdot x_0}(G\cdot x_0)$, and $g_{_0}:=\lim_{r\to 0} g_{_r}$. 
    \end{itemize}
    Note that $(N(G\cdot x_0), g_{_0})$ is locally isometric to the product of $G\cdot x_0$ with fiber $\mb R^{n+1-k_0}$, and 
    \begin{align}\label{Eq: uniform ambient blowup}
        \bleta_{r_i}(B)\to T_{x_0}B,\quad \bleta_{r_i}(G\cdot x_0)\to T_{x_0}(G\cdot x_0), \quad \bleta_{r_i}(N_{x_0}(G\cdot x_0))\to N_{x_0}(G\cdot x_0)
    \end{align}
    locally uniformly as $r_i\to 0$. 
    
    Consider $C:=\lim_{r_i\to 0}\bleta_{r_i\#}W$ with $\spt(\|C\|)\subset H$ for some $n$-plane $H\subset T_{x_0}B$. 
    By \cite{allard1972first}*{6.5} and the proof of Claim \ref{Claim: first regularity - 5.2}, $C$ is a rectifiable cone, which is also $G_{x_0}$-invariant. 

    \begin{claim}\label{Claim: tangent bundle}
        There exists $\widehat{C}\in \mc V_{n-k_0}^{G_{x_0}}(N_{x_0}(G\cdot x_0))$ so that 
        \begin{align}\label{Eq: splitting tangent}
            C=T_{x_0}(G\cdot x_0) \times \widehat{C}\quad{\rm and}\quad \spt(\|\widehat{C}\|)\subset \widehat{H}, 
        \end{align}
        where $\widehat{H}$ is a $G_{x_0}$-invariant subspace of $N_{x_0}(G\cdot x_0)$ with $H=T_{x_0}(G\cdot x_0)\times \widehat{H}$. 
        Moreover, 
        \begin{align}\label{Eq: splitting tangent bundle}
            \lim_{r_i\to 0}\bleta^G_{r_i\#}W = G\cdot \widehat{C}:= \mf v( G\cdot \widehat H,  \theta_{G\cdot \widehat{C}}), 
        \end{align}
        where $G\cdot \widehat{C}$ denotes the rectifiable varifold supported on the bundle $G\cdot \widehat{H}$ (over $G\cdot x_0$ with fiber $\widehat{H}$) with $G$-invariant multiplicity $ \theta_{G\cdot \widehat{C}}(g\cdot v) := \theta_{\widehat{C}}(v)$.  
    \end{claim}
    \begin{proof}[Proof of Claim \ref{Claim: tangent bundle}]
        Given $w\in T_{x_0}(G\cdot x_0)$, we have a curve $g(t)$ in $ G$ with $g(0)=e$ and $w=\frac{d}{dt}\big\vert_{t=0} g(t)\cdot x_0$.
        In addition, we have $A(t)$ in the space of $L\times L$-matrices with $A(0)\cdot x_0 =w$ so that $g(t)=I_L+tA(t)$, where $I_L$ is the identity. 
        Therefore
        \begin{eqnarray}\label{Eq: splitting blowup map}
        (\bleta_{r_i}\circ g(r_i))(y) 
            &=& \frac{g(r_i)\cdot y - g(r_i)\cdot x_0}{r_i}+\frac{g(r_i)\cdot x_0 -  x_0}{r_i} \nonumber
            \\
            &=& (\btau_{-A(r_i)\cdot x_0}\circ g(r_i)\circ \bleta_{r_i})(y).
        \end{eqnarray}
        Noting $W$ is $G$-invariant, $A(r_i)\cdot x_0\to w$, and $g(r_i)\to e$, we have $C=(\btau_{-w})_\# C$, which implies $C$ (as well as $H$) is invariant under the translation along $T_{x_0}(G\cdot x_0)$. 
        Since $T_{x_0}M=T_{x_0}(G\cdot x_0)\times N_{x_0}(G\cdot x_0)$ and $C$ is rectifiable, we obtain \eqref{Eq: splitting tangent}.  
        
        Next, to show \eqref{Eq: splitting tangent bundle}, we notice that $W$ is rectifiable and $G$-invariant, and thus can be written as $W\res B=\mf v(N,\theta)$ for some $G$-invariant $n$-rectifiable set $N\subset B$ and $G$-invariant integrable non-negative function $\theta$. 
        Let $\widehat{N}:=N\cap N_{x_0}(G\cdot x_0)$. 
        Then for any $f\in C^0_c(N(G\cdot x_0))$, let $f_G(v):=\int_G f(g\cdot v) d\mu(g)$ be the averaged $G$-invariant function. 
        Noting $G$ acts by isometries in $(N(G\cdot x_0), g_{_0})$, we see
        \begin{align*}
            \int_{\bleta^G_{r_i}(N)} f\cdot \theta\circ (\bleta^G_{r_i})^{-1}~d\mc H^n_0  
            &= \int_{\bleta^G_{r_i}(N)} f_G\cdot \theta\circ (\bleta^G_{r_i})^{-1}~d\mc H^n_0 
            \\
            &= \mc H^{k_0}_0(G\cdot x_0) \int_{\bleta^G_{r_i}(\widehat{N})} \hat{f}_G(v)\cdot \theta\circ (\bleta^G_{r_i})^{-1}(v)~d\mc H^{n-k_0}_0(v) 
            \\
            &= \int_{\bmu_{r_i}(\mb B^{k_0}_s(0))\times\bmu_{r_i}(\widehat{N}\cap \mb B_t(0))} \hat{f}_G(v)\cdot  h(u) \cdot \theta\circ \bmu_{r_i}^{-1}(0,v)~d\mc H^{n}_{\mb E}(u,v),
        \end{align*}
        where the co-area formula is used as in \eqref{Eq: area in slice}, $\hat{f}_G=f_G\res N_{x_0}(G\cdot x_0)$ is supported in $\mb B^{n+1-k_0}_t(0)$, $\mc H^*_0,\mc H^*_{\mb E}$ are the Hausdorff measures under $g_{_0},g_{_{Eucl}}$ respectively, and $h\in C^0_c(\mb B^{k_0}_s(0))$ satisfies $\int_{T_{x_0}(G\cdot x_0)} h = \mc H^{k_0}_0(G\cdot x_0)$.  
        Let $F:\mb B^{k_0}_s(0)\times N_{x_0}(G\cdot x_0)\to N(G\cdot x_0)$ be a local trivialization so that $dF(0)=id$ and $F(\mb B^{k_0}_s(0)\times \{0\})$ is the geodesic $s$-ball in $G\cdot x_0$. 
        Then 
        \begin{align*}
        		\bmu_{r_i}(\mb B^{k_0}_s(0))\times\bmu_{r_i}(\widehat{N}\cap \mb B_t(0)) &= \bmu_{r_i}(\mb B^{k_0}_s(0)\times(\widehat{N}\cap \mb B_t(0)))
        		\\&=\bmu_{r_i}\circ F^{-1}(N\cap T(x_0,s,t))
        		\\
        		&=\bmu_{r_i}\circ F^{-1}\circ\bleta_{r_i}^{-1}\circ\bleta_{r_i}(N\cap T(x_0,s,t)).
        \end{align*}
        Hence, for $F_i:=\bmu_{r_i}\circ F^{-1}\circ\bleta_{r_i}^{-1}$ with Jacobian $J_{F_i}$, we see 
        \begin{align*}
            \int_{\bleta^G_{r_i}(N)} f\cdot \theta\circ (\bleta^G_{r_i})^{-1}~d\mc H^n_0  
            &= \int_{\bleta_{r_i}(N\cap T(x_0,s,t))} J_{F_i}(x) \cdot (\hat{f}_G  h)(F_i(x)) \cdot \theta\circ F^{-1}\circ \bleta_{r_i}^{-1}(x)~d\mc H^{n}(x).
        \end{align*}
        By \eqref{Eq: uniform ambient blowup}, we can identify $\bleta_{r_i}(B)$, $\bleta_{r_i}(G\cdot x_0)$ and $\bleta_{r_i}(N_{x_0}(G\cdot x_0))$ with $T_{x_0}B$, $T_{x_0}(G\cdot x_0)$, and $N_{x_0}(G\cdot x_0)$ respectively for $i$ large enough, which further shows that $|J_{F_i}(x) \cdot (\hat{f}_G  h)(F_i(x))-(\hat{f}_G  h)(x)|\to 0 $ locally uniformly as $i\to\infty$. 
        Hence, 
        \begin{align*}
            \lim_{r_i\to 0}\int_{\bleta^G_{r_i}(N)} f\cdot \theta\circ (\bleta^G_{r_i})^{-1}~d\mc H^n_0  
            &= \lim_{r_i\to 0}\int_{\bleta_{r_i}(N\cap T(x_0,s,t))} (\hat{f}_G  h)(x) \cdot \theta\circ F^{-1}\circ\bleta_{r_i}^{-1}(x)~d\mc H^{n}(x) 
            \\&= \int_{T_{x_0}(G\cdot x_0)\times \widehat{H}} (\hat{f}_G  h)(x) \cdot \theta_{C}(x)~d\mc H^{n}_{\mb E}(x)
            \\&= \mc H^{k_0}(G\cdot x_0) \int_{\widehat{H}} \hat{f}_G   \cdot \theta_{\widehat{C}}~d\mc H^{n-k_0}_{\mb E}
            \\&= \int_{G\cdot \widehat{H}} f_G   \cdot \theta_{G\cdot \widehat{C}}~d\mc H^{n}_{0} = \int_{G\cdot \widehat{H}} f   \cdot \theta_{G\cdot \widehat{C}}~d\mc H^{n}_{0},
        \end{align*}
        where $\theta_{G\cdot \widehat{C}}(g\cdot v):=\theta_{\widehat C}(v)$ for $v\in N_{x_0}(G\cdot x_0)$ and $g\in G$. 
        By the $G_{x_0}$-invariance of $\theta_{\widehat C}(v)$, $\theta_{G\cdot \widehat{C}}$ is well defined. 
        Finally, the above equations imply that $\bleta^G_{r_i\#}W\to G\cdot \widehat{C}$ as $r_i\to 0$. 
    \end{proof}

    For simplicity, suppose that $x_0 = 0$ and $(0,\dots,0,1)$ is the unit normal of $\widehat{H}$ in $N_{x_0}(G\cdot x_0)\cong \mb R^{n+1-k_0}$, and write $r$ in place of $r_i$. 
    Since $\widehat{H}$ is a $G_{x_0}$-invariant codimension-one subspace in $N_{x_0}(G\cdot x_0)$, we see from Lemma \ref{Lem: hypersurface position} that $\widehat{H}$ is either orthogonal to $P(x_0):=T_{x_0} K_0\cap N_{x_0}(G\cdot x_0)$ (see \eqref{Eq: P}) or contains $P(x_0)$.  
    We now aim to show that $\Theta^n(\|W\|,x_0)$ is a positive integer, and the latter case cannot happen. 
    
    
	
	\medskip
    {\bf Case I: $\widehat{H}$ is orthogonal to $P:=P(x_0)=T_{x_0} K_0\cap N_{x_0}(G\cdot x_0)$.}
	
	Recall that $\dim(P(x_0))=2$ in this section. 
    Hence, we can take 
    \[P:=P(x_0)=\{(0,\dots,0,x_{n-k_0},x_{n+1-k_0})\}\subset\mb R^{n+1-k_0}=N_{x_0}(G\cdot x_0),\] 
    and use the set $\{x\in\mb R^{n+1-k_0}: x_{1}=\cdots=x_{n-2-k_0}=0, x_{n-1-k_0}\geq 0 \}\cong [0,\infty)\times P(x_0)$ as a fundamental domain of the $G_{x_0}$-action (resp. $G$-action) in $N_{x_0}(G\cdot x_0)$ (resp. $N(G\cdot x_0)$). 
    
    For any $\rho,\sigma,\tau>0$, set $r(\rho,\sigma):=\sqrt{\rho^2+\sigma^2}$ and define the following sets in Euclidean space $(N_{x_0}(G\cdot x_0) , g_{_0}= g(x_0))$:
    \begin{itemize}
        \item $\whcD_\rho:=\{x\in\whH:|x|\leq\rho\}$ is a closed $\rho$-ball in $\whH$;
        \item $\whcK_{\rho,\sigma}:= \{x\in (\whcD_\rho\setminus\bd\whcD_\rho) \times \mb R : |x| < \sqrt{\rho^2+\sigma^2}  \} \subset N_{x_0}(G\cdot x_0)$ is the $\rho$-radius $2\sigma$-height pillar region with two half-ball hats;
        \item $\bd_c\whcK_{\rho,\sigma} := \bd\whcD_\rho \times [-\sigma,\sigma]$ is the cylindrical boundary part; 
        \item $\bd_s\whcK_{\rho,\sigma} := (\whcD_\rho\times \mb R)\cap \bd\mb B_{r(\rho,\sigma)}(0)$ is the spherical boundary part;
        \item $\bd_e\whcK_{\rho,\sigma}:=\bd_c\whcK_{\rho,\sigma}\cap \bd_s\whcK_{\rho,\sigma}$ is the edge part; 
        \item $\whcS_{\rho,\sigma,\tau} := \bd_s\whcK_{\rho,\sigma} \setminus B_{\tau}(T_{x_0}K_1)=\{x\in \bd_s\whcK_{\rho,\sigma}: \dist(x,T_{x_0}K_1)\geq \tau\}$;
        \item $\whcC_{\rho,\sigma,\tau}:=\bd\whcK_{\rho,\sigma}\setminus \whcS_{\rho,\sigma,\tau}=\bd_c\whcK_{\rho,\sigma}\cup (\bd_s\whcK_{\rho,\sigma}\cap B_\tau(T_{x_0}K_1))$;
        \item $\whcB^{\pm}_{\rho,\sigma,\tau}:= (\whcD_{\rho}\times \{\pm\sigma\})\setminus B_\tau(T_{x_0}K_1)$; 
        \item $\whcA^{\pm}_{\rho,\sigma,\tau}:= \bd B_{\tau}(T_{x_0}K_1)\cap \{x\in \mb B_{r(\rho,\sigma)}(0): \pm x_{n+1-k_0}> \sigma\}$.
    \end{itemize}
    Note the above sets are $G_{x_0}$-invariant in $N_{x_0}(G\cdot x_0)$. 
    Hence, we can define the following $G$-sets in $N(G\cdot x_0)$ (where $K_1$ is identified with $(\exp^{\perp}_{G\cdot x_0})^{-1}(K_1)=G\cdot (T_{x_0}K_1\cap N_{x_0}(G\cdot x_0))$): 
    \begin{itemize}
        \item $\cDG_\rho:= G\cdot \whcD_\rho$ the bundle over $G\cdot x_0$ with fiber $\whcD_\rho$;
        \item $\cKG_{\rho,\sigma}:=G\cdot \whcK_{\rho,\sigma}$ the bundle over $G\cdot x_0$ with fiber $\whcK_{\rho,\sigma}$;
        \item $\bd_c \cKG_{\rho,\sigma}:=G\cdot \bd_c\whcK_{\rho,\sigma}$, $\bd_s\cKG_{\rho,\sigma}:=G\cdot \bd_s\whcK_{\rho,\sigma}$ the cylindrical/spherical part of $\bd \cKG_{\rho,\sigma}$;
        \item $\bd_e \cKG_{\rho,\sigma}:=G\cdot \bd_e\whcK_{\rho,\sigma}$ the edge part of $\bd \cKG_{\rho,\sigma}$;
        \item $\cSG_{\rho,\sigma,\tau}:= G\cdot \whcS_{\rho,\sigma,\tau} = \bd_s\cKG_{\rho,\sigma}\setminus B_\tau(K_1)$;
        \item $\cCG_{\rho,\sigma,\tau}:=G\cdot \whcC_{\rho,\sigma,\tau} =\bd \cKG_{\rho,\sigma}\setminus\cSG_{\rho,\sigma,\tau} = \bd_c \cKG_{\rho,\sigma}\cup (\bd_s\cKG_{\rho,\sigma}\cap B_\tau(K_1)) $;
        \item $\mc B^{G,\pm}_{\rho,\sigma,\tau}:= G\cdot \whcB^{\pm}_{\rho,\sigma,\tau}$;  
        \item $\mc A^{G,\pm}_{\rho,\sigma,\tau}:= G\cdot \whcA^{\pm}_{\rho,\sigma,\tau}$. 
    \end{itemize}
    One can see the shape of the above $G$-sets in the orbit space (as well as in the fundamental domain $P\times [0,\infty)$) by the following figures. 
    \begin{figure}[h]
        \centering
        \begin{subfigure}{0.22\linewidth}
            \centering
            \includegraphics[width=1.3in]{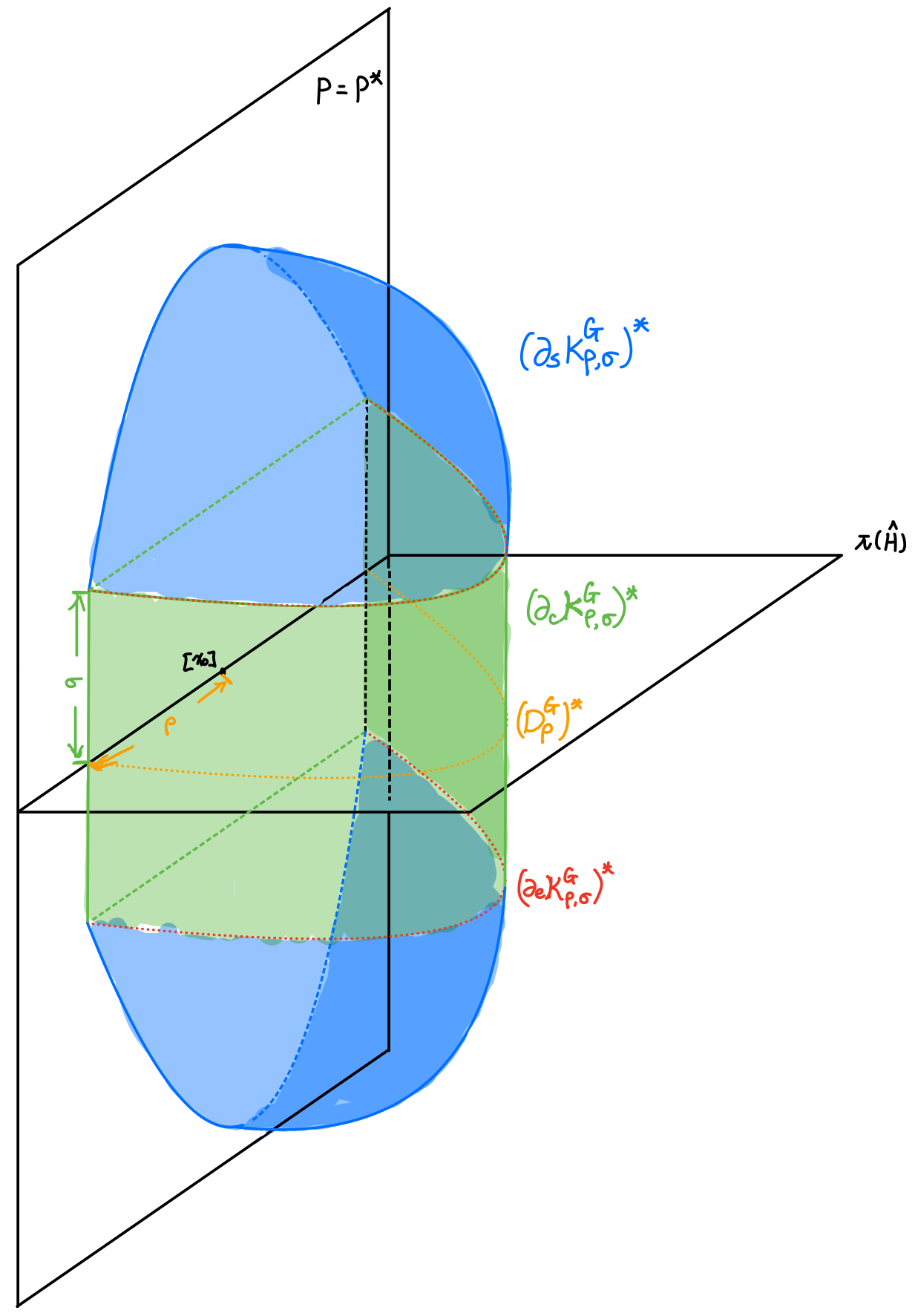} 
            \caption{$(\bd \cKG_{\rho,\sigma})^*$}
        \end{subfigure}
        \hspace{0.1cm}
        \begin{subfigure}{0.22\linewidth}
            \centering
            \includegraphics[width=1.3in]{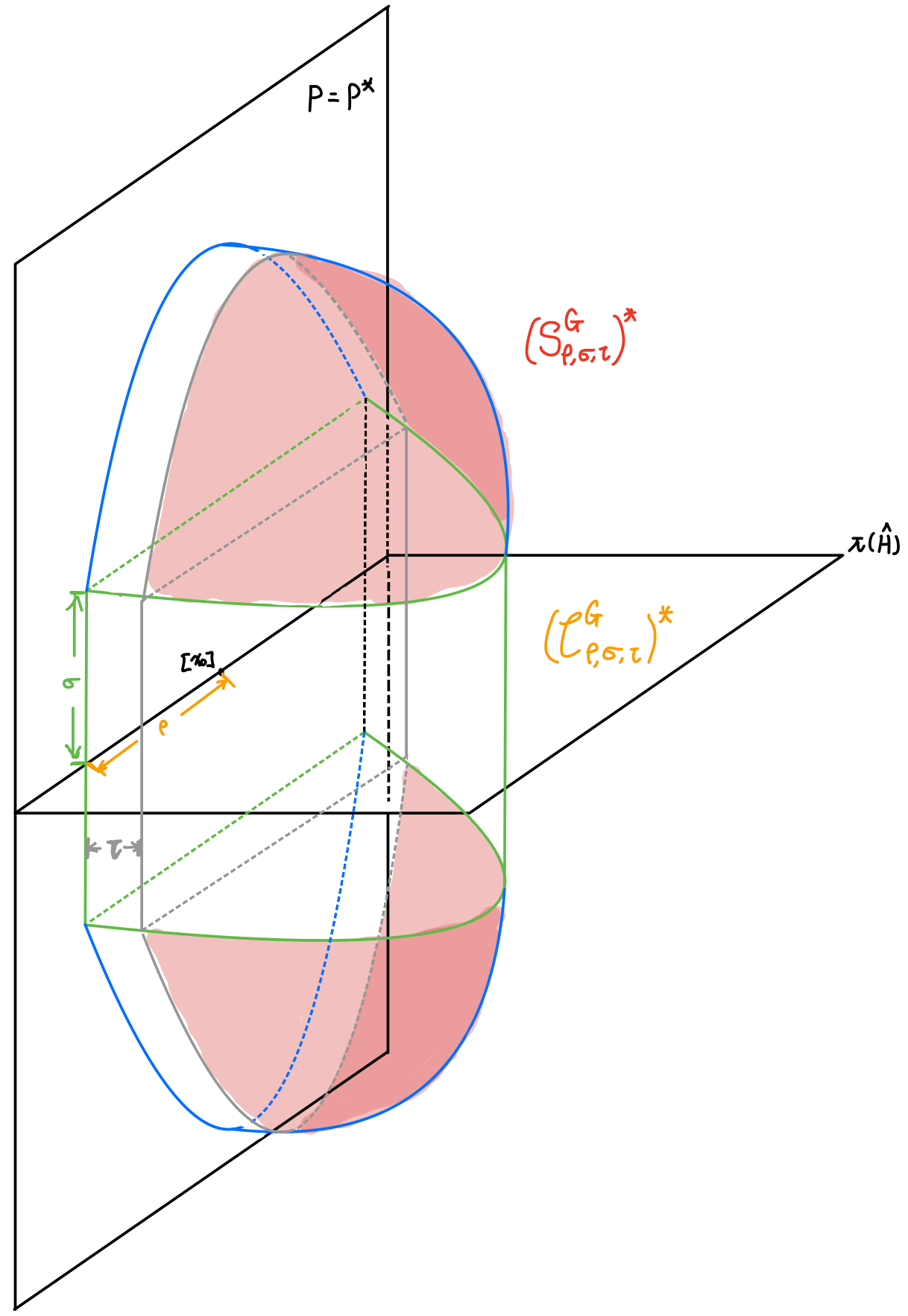}
            \caption{$(\cSG_{\rho,\sigma,\tau})^*$}
        \end{subfigure}
        \hspace{.1cm}
        \begin{subfigure}{0.22\linewidth}
            \centering
            \includegraphics[width=1.3in]{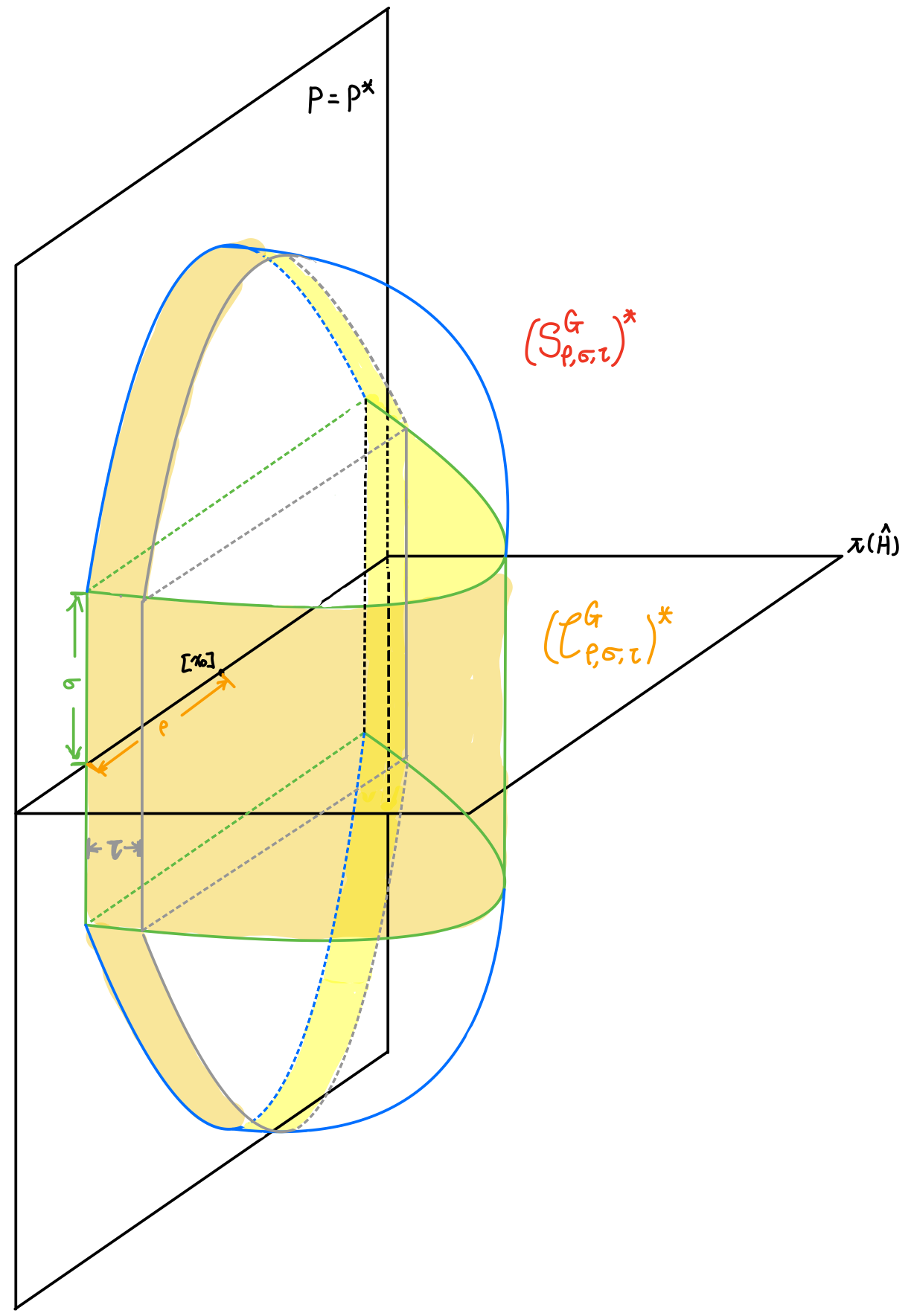}
            \caption{$(\cCG_{\rho,\sigma,\tau})^*$}
        \end{subfigure}
        \hspace{.1cm} 
        \begin{subfigure}{0.22\linewidth}
            \centering
            \includegraphics[width=1.3in]{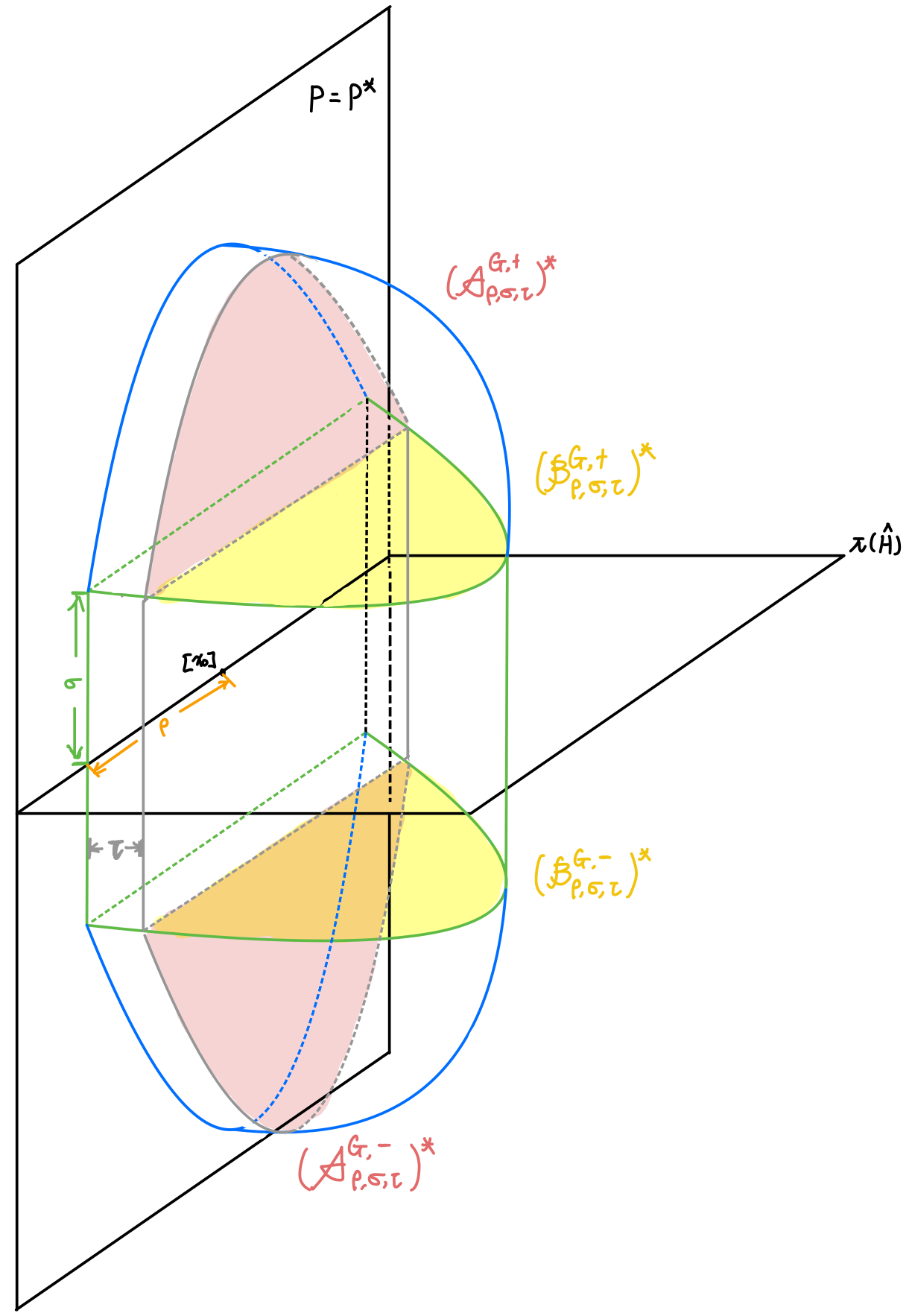}
            \caption{$(\mc A^{G,\pm}_{\rho,\sigma,\tau})^*,(\mc B^{G,\pm}_{\rho,\sigma,\tau})^*$}
        \end{subfigure}
        \caption{}
    \end{figure}
    
    Note that the metric $g_{_r}$ tends as $r\to 0$ to the metric $g_{_0}$ so that $N(G\cdot x_0)$ is locally isometric to the product of $(G\cdot x_0,g_{_M})$ with the fiber $(\mb R^{n+1-k_0},g_{_{Eucl}})$. 
    Thus, by slightly smoothing $\bd\cKG_{\rho,\sigma}$ near $\bd_e \cKG_{\rho,\sigma}$, we can assume that for all $\rho,\sigma\in (0,1]$, $\cKG_{\rho,\sigma}$ satisfies the assumptions imposed on $U$ in Lemma \ref{Lem: convex} with respect to the metric $g_{_r}$ for $r$ sufficiently small. 

    We now proceed as \cite{almgren79plateau}*{Page 464}. 
    Since $\spt(\|\bleta_{r_i\#}W\|)$ and $\spt(\|\bleta_{r_i\#}W\|)\cap N_{x_0}(G\cdot x_0)$ converge to $H$ and $\widehat{H}$ in the Hausdorff sense, we see from \eqref{Eq: splitting tangent bundle} that for any $\sigma_0\in (0,1/100)$, there exists $r>0$ small so that 
    \begin{align}\label{Eq: first regularity - 5.6}
        \cKG_{1,1}\cap \spt(\|\bleta^G_{r\#}W\|) \subset \cKG_{1,\sigma_0/2}.
    \end{align}
    We also choose $r>0$ so that 
    \begin{align}\label{Eq: first regularity - 5.7}
        \mc H^n\left( \spt(\|\bleta^G_{r\#}W\|)\cap (\bd \cDG_{\frac{1}{2}}\times \mb R) \right) = 0.
    \end{align}
    Henceforth, we will suppose that $r$ has been thus chosen, and compute under the $g_{_r}$-metric. 

    By the assumptions and Lemma \ref{Lem: uniform constants}, for any $\Gamma\in\mc M^G_0$ with $\bd\Gamma = \bd \bleta^G_r(\Sigma_k)$, we have 
    \begin{align}\label{Eq: first regularity - 5.8}
        \mc H^n(\bleta^G_r(\Sigma_k))\leq \mc H^n(\Gamma) + c\cdot r^{-(n-k_0)}\epsilon_k,
    \end{align}
    where $c>1$ is a uniform constant independent of $k$ and $r$. 
    Moreover,  we claim the following result similar to \cite{almgren79plateau}*{(5.9)(5.10)}.
    The difference is that we cannot avoid the whole edge $\bd_e\cKG$ since the rescaled metric \eqref{Eq: weighted metric in M/G} may degenerate at $\bd_0 (\cKG)^*\subset K_1^*$. 
    \begin{claim}\label{Claim: first regularity - not on edge}
        Fix any small $\tau>0$. We can take $\eta>0$ small enough so that for each sufficiently large $k$, there are $\sigma_k\in (\frac{3}{4}\sigma_0,\sigma_0)$, $\rho_k\in (\sqrt{1-\sigma_0^2}, 1)\subset (\frac{3}{4},1)$, and $\tau_k\in [\tau,2\tau]$ satisfying 
        \begin{align}\label{Eq: first regularity - 5.9}
            \mc H^{n-1}(\bleta^G_r(\Sigma_k)\cap \bd_s \cKG_{1,\sigma_k}) < \eta,  
         \end{align}
         and 
         \begin{align}\label{Eq: first regularity - 5.10}
            \bleta^G_r(\Sigma_k) \cap \bd \cSG_{\rho_k,\sigma_k',\tau_k} = \emptyset,\quad{\rm where}~ \sigma_k':=\sqrt{1+\sigma_k^2-\rho_k^2}< \sqrt{2}\sigma_0.
        \end{align}
        In particular, if $\codim(G\cdot x)=\cohom(G)=3$ for $x\in K_0$, then $\tau$ and $\tau_k$ can be set equal to $0$. 
    \end{claim}
    Note $\bd \cSG_{\rho_k,\sigma_k',\tau_k}=\bd \cCG_{\rho_k,\sigma_k',\tau_k}$, and the choice of $\sigma_k'$ makes $\cSG_{\rho_k,\sigma_k',\tau_k}\subset \bd_s\cKG_{\rho_k,\sigma_k'}\subset \bd_s\cKG_{1,\sigma_k}$. 
    Moreover, for $\tau=\tau_k=0$, we regard $\cSG_{\rho_k,\sigma_k',0}$ as $\bd_s\cKG_{\rho_k,\sigma_k'}$, and regard $\cCG_{\rho_k,\sigma_k',0}$ as $\bd_c\cKG_{\rho_k,\sigma_k'}$. 
    \begin{proof}[Proof of Claim \ref{Claim: first regularity - not on edge}]
        By \eqref{Eq: first regularity - 5.6} and the co-area formula, we see that for almost all $\sigma\in (\sigma_0/2,1)$
        \[\mc H^{n-1}(\bleta^G_r(\Sigma_k)\cap \bd_s \cKG_{1,\sigma}) \to 0 \qquad\mbox{as }k\to \infty . \]
        Therefore, for any given $\eta>0$ and each sufficiently large $k$, there exists $\sigma_k\in (\frac{3}{4}\sigma_0,\sigma_0)$ such that \eqref{Eq: first regularity - 5.9} holds. 
    
        Meanwhile, for any $\tau>0$ small enough, we have 
        \begin{align}\label{Eq: first regularity - orbit volume inf}
            0<\epsilon_\tau := \inf\{ \mc H^{n-2}(G\cdot x): x\in B, \dist(x,K_1)\geq \tau \}.
        \end{align}
        By the Riemannian submersion $\pi:B\setminus B_\tau(K_1) \subset M^{prin} \to (B\setminus B_\tau(K_1))^*$ and \eqref{Eq: area in orbit space}, we get 
        \[ \epsilon_\tau \cdot  \mc H^1\left( (\bleta^G_r(\Sigma_k)\cap \bd_s \cKG_{1,\sigma}))^* \setminus (B_\tau(K_1))^* \right) \leq \mc H^{n-1}\left( (\bleta^G_r(\Sigma_k)\cap \bd_s \cKG_{1,\sigma})\setminus B_\tau(K_1) \right) , \]
        	for almost all $\sigma\in (\sigma_0/2,1)$. 
        Hence, for any small $\tau>0$, we can take $\eta>0$ small enough (e.g., $\frac{\eta}{\epsilon_\tau}< 1-\sqrt{1-\sigma_0^2} < \frac{1}{4}$ and $\frac{\eta}{\epsilon_\tau}<\tau$) so that the above inequality together with \eqref{Eq: first regularity - 5.9} indicates 
        \[\bleta^G_r(\Sigma_k) \cap \left(\bd_e \cKG_{\rho_k,\sigma_k'} \setminus B_\tau(K_1) \right) = \emptyset \quad{\rm and}\quad \bleta^G_r(\Sigma_k) \cap \left(\bd_s\cKG_{\rho_k,\sigma_k'} \cap \bd B_{\tau_k}(K_1)\right) =\emptyset, \]
        for some $\sigma_k\in (\frac{3}{4}\sigma_0,\sigma_0)$, $\rho_k\in (\sqrt{1-\sigma_0^2}, 1)\subset (\frac{3}{4},1)$, $\sigma_k':=\sqrt{1+\sigma_k^2-\rho_k^2}< \sqrt{2}\sigma_0$, and $\tau_k\in [\tau,2\tau]$, which is exactly \eqref{Eq: first regularity - 5.10}.  

        Finally, if $\codim(G\cdot x)=\cohom(G)=3$ for $x\in K_1$, then \eqref{Eq: first regularity - orbit volume inf} also holds for $\tau=0$, and thus \eqref{Eq: first regularity - 5.10} is valid with $\tau=\tau_k=0$. 
    \end{proof}
    
    Combining the proof of Claim \ref{Claim: first regularity - not on edge} with Sard's theorem, we can further assume $\bleta^G_r(\Sigma_k)$ is transversal to both $\bd_s \cKG_{\rho_k,\sigma_k'}$ and $\bd \cDG_{\rho_k}\times [-1,1]$. 
    Moreover, after smoothing $\bd \cKG_{\rho_k,\sigma_k'}$ near $\bd_e\cKG_{\rho_k,\sigma_k'}$, we can further make $\cKG_{\rho_k,\sigma_k'}$ satisfy the assumptions for $U$ in Theorem \ref{Thm: replacement}. 

    After applying Theorem \ref{Thm: replacement} to $\bleta^G_r(\Sigma_k)$ in $\cKG_{\rho_k,\sigma_k'}$, we obtain $P_k^1,\dots,P_k^{R_k},P_k^{R_k+1},\dots,P_k^{R_k'}\in\mc M^G_0$ with $R_k,R_k'\in\mb N$ so that 
    \begin{itemize}
        \item[(1)] there are $\epsilon_{k,i}>0$, $i=1,\dots ,R_k'$, satisfying $\sum_{i=1}^{R_k'}\epsilon_{k,i}\leq c\cdot r^{-(n-k_0)}\epsilon_k$ and 
        \begin{align}\label{Eq: first regularity - 5.11}
            \mc H^n(P^i_k)\leq \mc H^n(P') + \epsilon_{k,i}, \qquad\forall P'\in\mc M^G_0 \mbox{ with }\bd P'=\bd P^i_k;
        \end{align}
        \item[(2)] by \eqref{Eq: first regularity - 5.7} and \cite{allard1972first}*{2.6(2) (d)},
        \begin{align}\label{Eq: first regularity - 5.12}
            \left(\bleta^G_{r\#}W \right)\res G_n(\cKG_{\frac{1}{2},1}) = \lim_{k\to\infty} \sum_{i=1}^{R_k'} |P_k^i\cap \cKG_{\frac{1}{2},1}|;
        \end{align}
        \item[(3)] for each $i\in\{1,\dots,R_k\}$, $\bd_1 P_k^{i*}$ is a curve that is non-trivial in the relative homotopy group $\pi_1((\cCG_{\rho_k,\sigma_k',\tau_k})^*, \bd_0(\cCG_{\rho_k,\sigma_k',\tau_k})^* )$, i.e. $\bd_1 P_k^{i*}$ cannot be homotopic in $(\cCG_{\rho_k,\sigma_k',\tau_k})^*$ to a curve in $\bd_0(\cCG_{\rho_k,\sigma_k',\tau_k})^* $. 
        \item[(4)] for each $i\in \{R_k+1,\dots,R_k'\}$, 
        the curve $\bd_1 P_k^{i*}$ is either 
        	\begin{itemize}
        		\item closed and null-homotopic in $(\cSG_{\rho_k,\sigma_k',\tau_k})^*$, or
        		\item homotopic (relative to its boundary) in $(\cSG_{\rho_k,\sigma_k',0})^*$ to a curve in $\partial_0(\cSG_{\rho_k,\sigma_k',0})^*$ (in this case $\tau=\tau_k=0$ and $\codim(G\cdot x)=3$ for $x\in K_1$), or
        		\item homotopic in $(\cCG_{\rho_k,\sigma_k',\tau_k})^*$ to a curve in $\bd_0(\cCG_{\rho_k,\sigma_k',\tau_k})^*$. 
        	\end{itemize}
    \end{itemize}
    We mention that \eqref{Eq: first regularity - 5.10} is used to get (3)(4). 

    \begin{claim}\label{Claim: discard trivial curve}
        $P^i_k$, $i=R_k+1,\dots,R_k'$, can be discarded without changing the varifold limit in \eqref{Eq: first regularity - 5.12}: 
        \begin{align}\label{Eq: first regularity - 5.13}
             \left(\bleta^G_{r\#}W \right)\res G_n(\cKG_{\frac{1}{2},1}) = \lim_{k\to\infty} \sum_{i=1}^{R_k} |P_k^i\cap \cKG_{\frac{1}{2},1}|.
        \end{align}
    \end{claim}
    \begin{proof}[Proof of Claim \ref{Claim: discard trivial curve}]
        Denote by $C>1$ a uniform constant depending only on $n,k_0,K_0,K_1$. 
        
        If $\bd P_k^{i} \subset \cSG_{\rho_k,\sigma_k',\tau_k} $, then $\bd_1 P_k^{i*}$ encloses a region $A_k^{i*}\subset (\cSG_{\rho_k,\sigma_k',\tau_k})^*$ so that
        \[ \mc H^n(A_k^i) \leq C \eta^{\frac{n-k_0}{n-1-k_0}}, \]
        where $A_k^i:=\pi^{-1}(A_k^{i*})\subset \cSG_{\rho_k,\sigma_k',\tau_k}$. Note that \eqref{Eq: first regularity - 5.9}, the isoperimetric inequality (in $\whcS_{\rho_k,\sigma_k',\tau_k}\subset \bd_s\whcK_{\rho_k,\sigma_k'}$), and \eqref{Eq: area in slice} are used in the above inequality. 
        In particular, $A_k^i\in\mc M^G_0$ and $\bd A_k^i=\bd P^i_k$. 

        If $\bd P_k^{i} \subset \cCG_{\rho_k,\sigma_k',\tau_k} $ and $i\geq R_k+1$, then $\bd_1 P_k^{i*}$ encloses a region $B_k^{i*}\subset (\cCG_{\rho_k,\sigma_k',\tau_k})^*$ so that $B_k^{i*}\cap \bd_1(\cCG_{\rho_k,\sigma_k',\tau_k})^*=\emptyset$ (i.e. $B_k^{i*}$ is also the region enclosed by $\bd_1 P_k^{i*}$ in $\bd_1(\cKG_{\rho_k,\sigma_k'})^*$) and 
        \begin{align*}
            \mc H^n(B_k^i) \leq \mc H^n(\cCG_{\rho_k,\sigma_k',\tau_k}) 
            &\leq \mc H^{n}(\bd_c \cKG_{\rho_k,\sigma_k'}) + \mc H^n(\bd \cKG_{\rho_k,\sigma_k'}\cap B_{\tau_k}(K_1)) \\&\leq  C \rho_k^{n-1-k_0}\cdot \sigma_0 + C(\rho_k^2+\sigma_k'^2)^\frac{1}{2} \cdot \tau^{n-k_0-1},  
        \end{align*}
        where $B_k^i:=\pi^{-1}(B_k^{i*})\subset \bd_c \cKG_{\rho_k,\sigma_k'}$. 
        If $\bd_1 P_k^{i*}$ has free boundary on $\partial_0(\cCG_{\rho_k,\sigma_k',\tau_k})^*$ or $\bd_1 P_k^{i*}$ is a closed curve null-homotopic in $(\cCG_{\rho_k,\sigma_k',\tau_k})^*$, then we can take $A_k^i=B_k^i$ so that $A_k^i\in\mc M^G_0$ and $\bd A_k^i=\bd P^i_k$. 
        	If $\bd_1 P_k^{i*}$ is closed and homotopic to $\bd_0(\cCG_{\rho_k,\sigma_k',\tau_k})^*$ in $(\cCG_{\rho_k,\sigma_k',\tau_k})^*$, then we must have $\tau_k>0$ and $\codim(G\cdot x)=n+1-k_0>3$ for $x\in K_1$. 
        	Thus, we can take $\tau'\in(0,\tau_k)$ small enough so that $(B_{\tau'}(K_1))^*\cap \bd_1 P_k^{i*} = \emptyset$, and take a $G$-subset $\tilde{B}_k^i= B_k^i\setminus B_{\tau'}(K_1)$. 
        	 Now, we set $A_k^i= \tilde{B}_k^{i} \cup (\bd B_{\tau'}(K_1)\cap \cKG_{\rho_k,\sigma_k'})$, which still satisfies $A_k^i\in\mc M^G_0$, $\bd A_k^i=\bd P^i_k$, and 
        	 \begin{align}\label{Eq: claim5}
        	 	\mc H^n(A_k^i) \leq \mc H^n(\tilde{B}_k^i) + \mc H^n(\bd B_{\tau'}(K_1)\cap \cKG_{\rho_k,\sigma_k'}) \leq \mc H^n(\tilde{B}_k^i) + C\tau^{n-k_0-2} .
        	 \end{align}
        	 
        	 Therefore, for each $i\geq R_k+1$, we have $A_k^i\in\mc M^G_0$ with $\bd A_k^i=\bd P^i_k$ and $\mc H^n(A_k^i)$ bounded from above by a constant depending on $\eta,\sigma_0,\tau$. 
		

        To apply the filigree lemma (Lemma \ref{Lem: filigree}), we consider the $G_{x_0}$-invariant function $\hat{f}\in C^0(N_{x_0}(G\cdot x_0))$ given by 
        \[
        \hat{f} (y) = 
        \left\{
            \begin{array}{ll}
                \frac{|y'|}{\rho_k} & {\rm if~}\frac{|y_{n+1-k_0}|}{|y'|}\leq \frac{\sigma_k'}{\rho_k}, \\
                \frac{|y|}{\sqrt{\rho_k^2+\sigma_k'^2}} & {\rm if~}\frac{|y_{n+1-k_0}|}{|y'|} > \frac{\sigma_k'}{\rho_k} ,
            \end{array}
        \right.
        \quad\forall y=(y',y_{n+1-k_0})\in\mb R^{n-k_0}\times\mb R\cong N_{x_0}(G\cdot x_0).
        \]
        The $G_{x_0}$-invariance of $\hat{f}$ indicates $v=g\cdot y\mapsto\hat{f}(y)$ is a well-defined $G$-function in $N(G\cdot x_0)$. 
        After equivariantly smoothing, we obtain a $G$-invariant function $f$ so that
        \begin{itemize}
            \item $\cKG_{\rho_k,\sigma_k'}=Y(1):=\{v: f(v)<1\}$; 
            \item $c_1:=\sup |\nabla f|\leq 1+ \max\{\frac{1}{\rho_k},\frac{1}{\sqrt{\rho_k^2+\sigma_k'^2}}\} \leq 3$;
            \item there exists $c_2>0$ so that \eqref{Eq: compare to disk} and \eqref{Eq: filigree - isoperimetric} are valid for $Y_t=\{v: f(v)<t\}$ with $t>0$ (by an argument similar to Remark \ref{Rem: compare disk instead cylinder}). 
        \end{itemize}
        Therefore, after shrinking $\sigma_0,\tau,\eta>0$ successively, we can use \eqref{Eq: first regularity - 5.11} to make
        \[(\mc H^n(P_k^i))^\frac{1}{n-k_0}\leq (\mc H^n(A_k^i) + \epsilon_{k,i})^\frac{1}{n-k_0} \leq \frac{1}{4}\cdot \frac{1}{(n-k_0)}c_1^{-1}c_2^{-\frac{n-k_0-1}{n-k_0}} \] 
        for $k$ large enough, and then apply the Filigree Lemma \ref{Lem: filigree} to show
        \[\mc H^n(P_k^i\cap \cKG_{\frac{1}{2},\frac{\sigma_0}{2}}) < \mc H^n(P^i_k\cap \cKG_{\frac{2}{3}\rho_k,\frac{2}{3}\sigma_k'} ) \leq \mc H^n(P_k^i\cap Y(\frac{3}{4})) \leq 2\epsilon_{k,i}, \]
        which together with \eqref{Eq: first regularity - 5.6}\eqref{Eq: first regularity - 5.12} indicates \eqref{Eq: first regularity - 5.13}. 
    \end{proof}
    
    Let $\omega_{k}:=\mc H^k_{\mb E}(\mb B^k_1(0))$, and let $C>1$ be a (varying) uniform constant depending only on $n,k_0,K_0,K_1$. 
    Then for each $i\in\{1,\dots,R_k\}$, we have $\mc H^n ( G\cdot (\widehat{H}\cap \mb B_{\rho}(0)) \setminus B_\tau(K_1) ) \leq \mc H^n (P_k^i\cap \cKG_{\rho,1})$ in the metric $g_{_0}$, and thus, in the metric $g_{_r}$ for $r>0$ small:
    \begin{align}\label{Eq: first regularity - 5.15}
        0<\mc H^n \left( G\cdot (\widehat{H}\cap \mb B_{\rho}(0))\right) - C\tau^{n-1-k_0} - \epsilon_1(r) ~\leq~ \mc H^n \left(P_k^i\cap \cKG_{\rho,1}\right), \qquad\forall \rho\in (\tau,\rho_k],
    \end{align}
    where $\epsilon_1(r)\geq 0$ and $\lim_{r\to 0}\epsilon_1(r)=0$. 
    Meanwhile, \eqref{Eq: first regularity - 5.11} indicates that for all $k$ large enough, 
    \begin{align}\label{Eq: first regularity - 5.16}
        \mc H^n \left(P_k^i\right) \leq \mc H^n \left( G\cdot (\widehat{H}\cap \mb B_{\rho_k}(0))\right) + C\sigma_0  + \epsilon_2(\tau) + \epsilon_1(r) + \epsilon_{k,i} , 
    \end{align}
    where $\epsilon_2(\tau)\geq 0$ with $\lim_{\tau\to 0}\epsilon_2(\tau)=0$. 
    Indeed, one notices that $\bd_1 P_k^{i*}$ separates $(\cCG_{\rho_k,\sigma_k',\tau_k})^*$ into two components $N_+^*,N_-^*$ containing the two components $\bd_1(\cCG_{\rho_k,\sigma_k',\tau_k})^*\cap (\whcD_1\times \{\pm x_{n+1-k_0}\geq 0\})^*$ respectively. 
    Take 
    \[N := \pi^{-1}(N_-^*) \cup \mc B^{G,-}_{\rho_k,\sigma_k',\tau_k} \cup \mc A^{G,-}_{\rho_k,\sigma_k',\tau_k}. \]
    Then $N\in \mc M^G_0$ with $\bd N = \bd P_k^i$ (see Figures \ref{fig: N-} and \ref{fig: A- and B-}). 
    \begin{figure}[h]
        \centering
        \begin{subfigure}{0.4\linewidth}
            \centering
            \includegraphics[width=1.6in]{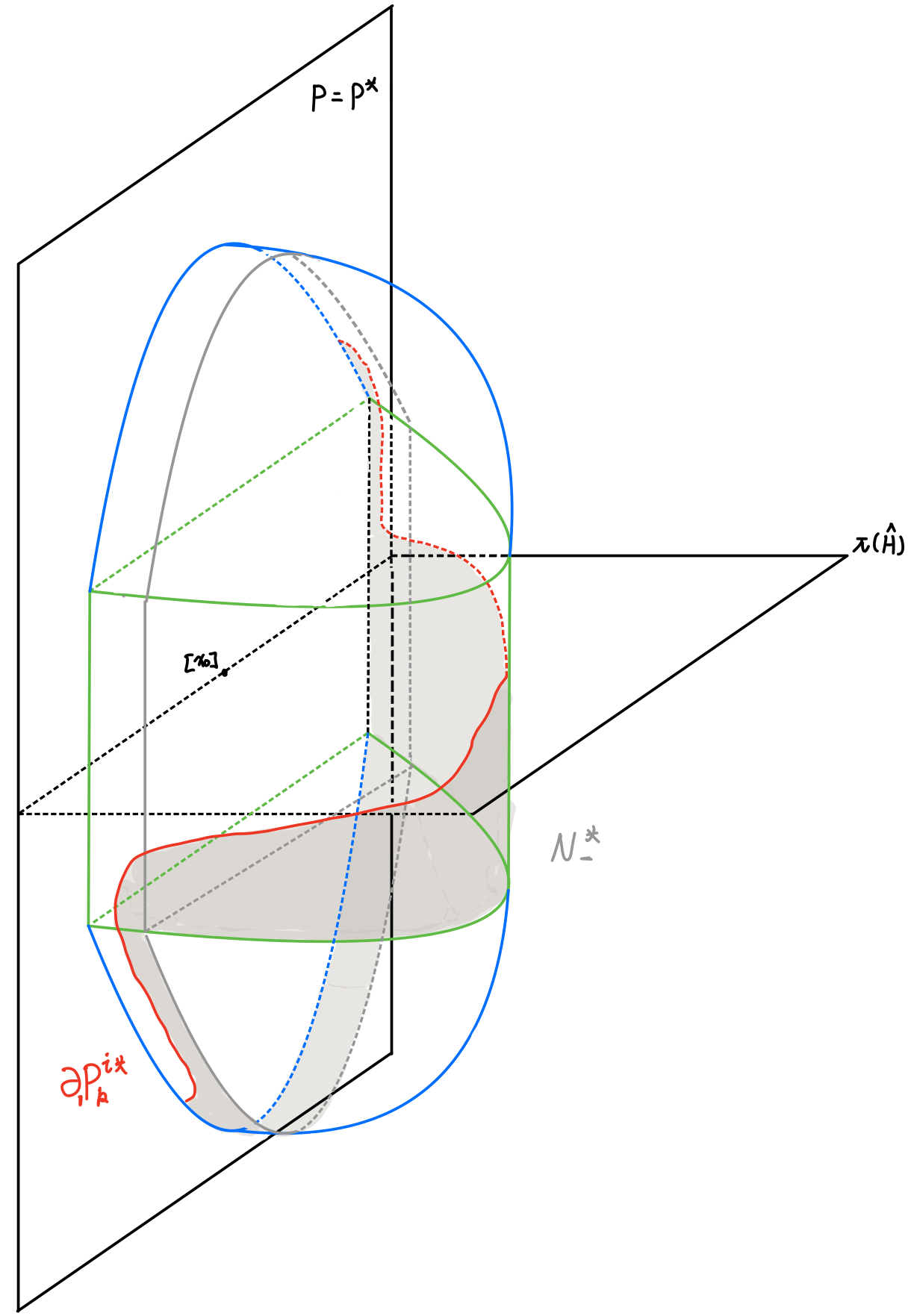} 
            \caption{$N_-^*$}\label{fig: N-}
        \end{subfigure}
        \hspace{1.5cm}
        \begin{subfigure}{0.4\linewidth}
            \centering
            \includegraphics[width=1.6in]{AB.png}
            \caption{$(\mc A^{G,\pm}_{\rho,\sigma,\tau})^*,(\mc B^{G,\pm}_{\rho,\sigma,\tau})^*$}\label{fig: A- and B-}
        \end{subfigure}
        	\caption{}
    \end{figure}
    Hence, we know 
    \begin{itemize}
        \item $\mc H^n(\pi^{-1}(N_-^*))\leq \mc H^n(\cCG_{\rho_k,\sigma_k',\tau_k})\leq C(\sigma_0+\tau^{n-k_0-1})$;
        \item $\mc H^n(\mc B^{G,-}_{\rho_k,\sigma_k',\tau_k})\leq \mc H^n ( G\cdot (\widehat{H}\cap \mb B_{\rho_k}(0))) + \epsilon_1(r)$; 
        \item $\mc H^n(\mc A^{G,-}_{\rho_k,\sigma_k',\tau_k})\leq \epsilon_2(\tau)$ if $\codim(G\cdot x)>\cohom(G)=3$ for $x\in K_1$ (since $\dim(K_1)<n$); 
        \item $\mc A^{G,-}_{\rho_k,\sigma_k',\tau_k}=\emptyset$ if $\codim(G\cdot x)=\cohom(G)=3$ for $x\in K_1$ (since $\tau=\tau_k=0$ by Claim \ref{Claim: first regularity - not on edge}).
    \end{itemize}
    Using \eqref{Eq: first regularity - 5.11}, we can compare the area of $P_k^i$ and $N$ to get \eqref{Eq: first regularity - 5.16}.

    Next, by \eqref{Eq: first regularity - 5.15} and \eqref{Eq: first regularity - 5.12}, we see $R_k$ is uniformly bounded. 
    Hence, we have a positive integer $m$ and a subsequence (without relabeling) so that $R_k\equiv m$, and 
    \begin{itemize}
        \item $\rho_k\to \rho_\infty\in [\sqrt{1-\sigma_0^2},1]\subset [\frac{3}{4},1]$ as $k\to\infty$;
        \item for $i\in\{1,\dots, m\equiv R_k\}$, $\left|\bleta^G_{\rho_k}(P_k^i)\right|$ converges to a $G$-varifold $W_i$ satisfying (by taking $\tau\to 0$ after $k\to\infty$ in \eqref{Eq: first regularity - 5.15}\eqref{Eq: first regularity - 5.16})
        \begin{align}
            \mc H^n ( G\cdot (\widehat{H}\cap \mb B_{\rho}(0))) - \epsilon_1(r)
            &\leq \|W_i\| \left(\cKG_{\rho,1}\right) \qquad\forall \rho\in (0,1] \label{Eq: first regularity - 5.17} 
            \\
            &\leq \|W_i\| \left(\cKG_{1,1}\right) \leq \mc H^n ( G\cdot (\widehat{H}\cap \mb B_{1}(0))) + C\sigma_0 + \epsilon_1(r) , \label{Eq: first regularity - 5.17.2} 
        \end{align}
        \begin{align}\label{Eq: first regularity - 5.18} 
            \spt(\|W_i\|)\subset \cKG_{1,3\sigma_0},
        \end{align}
        \begin{align}\label{Eq: first regularity - 5.19} 
            \left(\bleta^G_{r\rho_\infty\#}W \right)\res G_n(\cKG_{\frac{1}{2},1}) = \sum_{i=1}^{m} W_i\res G_n(\cKG_{\frac{1}{2},1}). 
        \end{align}
    \end{itemize}
    Additionally, by applying the proofs of Claim \ref{Claim: first regularity - 5.2} and Claim \ref{Claim: tangent bundle} with $\bleta^G_{\rho_k}(P^i_k)$ and $W_i$ in place of $\Sigma_k$ and $W$ respectively, each $W_i$ is a stationary rectifiable $G$-varifold in $\mc K^G_{1,3\sigma_0}$ written as $\mf v(G\cdot \widehat W_i, \theta)$ for some $(n-k_0)$-rectifiable $G_{x_0}$-set $\widehat  W_i\subset \whcK_{1,3\sigma_0}$ and $G$-function $\theta$ on $\cKG_{1,3\sigma_0}$. 
    
    The above constructions are analogous to those in \cite{almgren79plateau}*{Theorem 2}, but our pillar region $\whcK_{\rho,\sigma}$ has a spherical cap, whereas in \cite{almgren79plateau}*{p.464}, $K_{\rho,\sigma}$ has a flat cap. 
    This spherical cap is chosen so that $\cKG_{\rho,\sigma}$ is mean convex under the metric $g_{_r}$ for $r>0$ small enough, which is required in Lemma \ref{Lem: convex}. 
    Nevertheless, given $\sigma_0\in (0,\frac{1}{100})$, we can also use a shallower spherical cap by taking 
    \begin{itemize}
        \item $\widehat{K}_{\rho,\sigma}:= \left\{x\in (\whcD_\rho\setminus\bd\whcD_\rho) \times \mb R :  | x - (0,\cdots,0,\pm \frac{1}{\sigma_0}) |^2 < \rho^2+(\sigma+\frac{1}{\sigma_0})^2 \right\}$, for $\rho,\sigma\in (0,1]$.
    \end{itemize}
    Then for $\rho,\sigma\in (0,1]$, $K^G_{\rho,\sigma}:= G\cdot \widehat K_{\rho,\sigma}$ still satisfies the mean convex condition under the metrics $g_{_0}$ and $g_{_r}$ for $r>0$ small enough. 
    Hence, the previous constructions can be modified so that \eqref{Eq: first regularity - 5.17}-\eqref{Eq: first regularity - 5.19} remain valid with $\cKG_{\rho,\sigma}$ replaced by $K^G_{\rho,\sigma}$. 
    Let $B_s^{N(G\cdot x_0)}(G\cdot x_0)$ be the $s$-neighborhood of the zero section of $N(G\cdot x_0)$ for $s>0$. 
    Then, for $s\in [1/2,1]$, by \eqref{Eq: first regularity - 5.17}\eqref{Eq: first regularity - 5.18} (with $\cKG_{\rho,\sigma}$ replaced by $K^G_{\rho,\sigma}$) and the fact that $K^G_{s-\sigma_0,1}\cap K^G_{1,3\sigma_0}\subset K^G_{s-\sigma_0,6\sigma_0}\subset B_s^{N(G\cdot x_0)}(G\cdot x_0)\subset B_1^{N(G\cdot x_0)}(G\cdot x_0)\subset K^G_{1,1}$, we can take a sequence $r_l\to 0^+$ so that   
    \begin{align*}
        \mc H^{k_0}(G\cdot x_0)\omega_{n-k_0}(s-\sigma_0)^{n-k_0} 
        &\leq \lim_{r_l\to 0} \|W_i\|(K^G_{s-\sigma_0,1}) 
        \leq \lim_{r_l\to 0} \|W_i\|(B_s^{N(G\cdot x_0)}(G\cdot x_0))\\
        &\leq \lim_{r_l\to 0} \|W_i\|(B_1^{N(G\cdot x_0)}(G\cdot x_0))
        \leq \mc H^{k_0}(G\cdot x_0)\omega_{n-k_0} + C\sigma_0.
    \end{align*}
    Additionally, $\delta W_i(X^G)=\mc H^{k_0}(G\cdot x_0)\cdot \delta \mf v(\widehat W_i, \theta|_{N_{x_0}(G\cdot x_0)}) (X^G|_{N_{x_0}(G\cdot x_0)})$ under the local-product metric $g_{_0}$, for any $G$-invariant vector field $X^G$ in $B_1^{N(G\cdot x_0)}(G\cdot x_0)$. 
    Hence, by Lemma \ref{Lem: first variation and G-variation}, the monotonicity formula can be applied in the Euclidean slice $N_{x_0}(G\cdot x_0)$ to show the monotonicity of $\lim_{r_l\to 0} \|W_i\|(B_s^{N(G\cdot x_0)}(G\cdot x_0))\cdot s^{-(n-k_0)} $ on $s\in (0,1)$. 
    Combining these with the arbitrariness of $\sigma_0\in (0,\frac{1}{100})$, we see $\omega_{n-k_0} \mc H^{k_0}(G\cdot x_0)=\lim_{\sigma_0^k\to 0}\lim_{r_l\to 0} \frac{\|W_i\|(B^{N(G\cdot x_0)}_{1/2}(G\cdot x_0))}{(1/2)^{(n-k_0)}}$ for a sequence $\sigma_0^k\to 0^+$.  
    Thus, by \eqref{Eq: splitting tangent bundle}\eqref{Eq: first regularity - 5.19} and passing to a diagonal subsequence $\{(\sigma_0^j, r_j)\}_{j\in\mb N}$, 
    \begin{align}\label{Eq: first regularity - density}
        \Theta^n(\|W\|,x_0) &= \Theta^{n-k_0}(\|\widehat{C}\|,0) = 
        \frac{\|\widehat{C}\|(\mb B^{n+1-k_0}_{1/2}(0))}{\omega_{n-k_0}\cdot(1/2)^{n-k_0}} = \lim_{(\sigma_0^j,r_j)\to 0} \frac{\|\bleta_{r\rho_\infty\#}^GW\|(B^{N(G\cdot x_0)}_{1/2}(G\cdot x_0))}{ \omega_{n-k_0}(1/2)^{n-k_0} \mc H^{k_0}(G\cdot x_0)}    \nonumber
        \\
        &= \lim_{(\sigma_0^j,r_j)\to 0}  \sum_{i=1}^m \frac{\|W_i\|(B^{N(G\cdot x_0)}_{1/2}(G\cdot x_0))}{\omega_{n-k_0}(1/2)^{n-k_0} \mc H^{k_0}(G\cdot x_0)}  = m.    
    \end{align}

    {\bf Case II: $\widehat{H}$ contains $P:=T_{x_0} K_0\cap N_{x_0}(G\cdot x_0)$.}
	
	We shall show this case cannot happen as in \cite{jost1986existenceI}*{Theorem 5.1, (5.23)}. 
    To begin with, we notice the unit normal $\nu_{\widehat{H}}$ of $\widehat{H}$ in $N_{x_0}(G\cdot x_0)$ points into $M^{prin}$, which implies $\#(G_{x_0}\cdot p)=\#(G_{x_0}\cdot \nu_{\widehat{H}})=\#\{\pm\nu_{\widehat{H}}\}=2$ for all $p\in S_{\rho_0}(x_0)\cap M^{prin}$. 
    Hence, $G_{x_0}$ acts by the reflection across $P$ in $N_{x_0}(G\cdot x_0)$, $k_0=\dim(G\cdot x_0)= n+1-\cohom(G)=n-2$, and $K_0\cap B$ is a $G$-invariant minimal hypersurface in $B$ (cf. Proposition \ref{Prop: local structure of Sigma/G 1}, \cite{hsiang71cohom}*{Corollary 1.1}). 
    
    Next, we can proceed as in {\bf Case I} within the fundamental domain $\mb R^3_+:=\{x_1\geq 0\}\subset \mb R^3\cong N_{x_0}(G\cdot x_0)$ with $P=\widehat{H}=\{x_1=0\}$. Set $\tau=0$, 
    \begin{itemize}
    	\item $(\cKG_{\rho,\sigma})^*=\{x\in \mb R^3_+: x_2^2+x_3^2<\rho^2, |x|^2<\rho^2+\sigma^2\}$, 
    	\item $(\cSG_{\rho,\sigma,0})^*=(\bd_s\cKG_{\rho,\sigma})^*=\{x\in \mb R^3_+: x_2^2+x_3^2<\rho^2, |x|^2=\rho^2+\sigma^2\}$, 
    	\item $(\cCG_{\rho,\sigma,0})^*=(\bd_c\cKG_{\rho,\sigma})^*=\{x\in \mb R^3_+: x_2^2+x_3^2=\rho^2, |x|^2<\rho^2+\sigma^2\}$. 
    \end{itemize}
    Then we similarly obtain $\{P_k^{i}\}_{i=1}^{R_k'}\subset \mc M^G_0$ satisfying (1)-(4) before Claim \ref{Claim: discard trivial curve} in some $\cKG_{\rho,\sigma}$. 
    However, since the cylindrical region $(\cCG_{\rho,\sigma,0})^*$ deformation retracts onto $\bd_0 (\cCG_{\rho,\sigma,0})^* $, we know every $P_k^i$ satisfies the statement (4) before Claim \ref{Claim: discard trivial curve} (i.e. $R_k=0$). 
    Therefore, it is sufficient to show Claim \ref{Claim: discard trivial curve} in this case, which indicates $x_0\notin\spt(\|W\|)$ and leads to a contradiction. 
    
    If $\bd P_k^i\subset \mc S^G_{\rho,\sigma,0}$, or $\bd_1 P_k^{i*}\subset (\cCG_{\rho,\sigma,0})^*$ has free boundary on $\bd_0 (\cCG_{\rho,\sigma,0})^*$, or $\bd_1 P_k^{i*}$ is a null-homotopic closed curve in $(\cCG_{\rho,\sigma,0})^*$. 
    Then, we can proceed as in Claim \ref{Claim: discard trivial curve} to discard them by the filigree lemma. 
    If $\bd_1 P_k^{i*}\subset (\cCG_{\rho,\sigma,0})^*$ is a non-trivial closed curve in $(\cCG_{\rho,\sigma,0})^*$, then $A_k^{i*}=B_k^{i*}\subset (\cCG_{\rho,\sigma,0})^*$ is the cylindrical region bounded between $\bd_1 P_k^{i*}$ and $\bd_0  (\cCG_{\rho,\sigma,0})^*$. 
    Since $A_k^i\in\mc M^G_0$ and $\bd A_k^i=\bd P_k^i$, we can still apply \eqref{Eq: first regularity - 5.11} and the filigree lemma to discard $P_k^i$ (see \cite{jost1986existenceI}*{(5.23)-(5.26)}). 
    (This is where we require \eqref{Eq: first regularity - 5.11} to hold without the assumption on $\mk b$.) 
    This proves Claim \ref{Claim: discard trivial curve} and gives a contradiction to $x_0\in\spt(\|W\|)$. 
    
    \medskip
    Now, combining the regularity of $W$ in $B\setminus K_0$ with the results in {\bf Case I} and {\bf II}, we know that $\Theta^n(\|W\|,x_0)$ is a positive integer whenever $x_0\in\spt(\|W\|)\cap (B\setminus K_0) $ or $W$ has a tangent varifold at $x_0\in \spt(\|W\|)\cap B\cap K_0$ supported in an $n$-plane. 
    By the rectifiability of $W\res G_n(B)$, such a tangent varifold exists for almost all $x_0\in B$. 
    Thus, $W$ is an integral varifold in $B$. 

    Moreover, applying the above arguments with $\bleta^G_{\rho_k}(P^i_k)$ in place of $\Sigma_k$, we see each $W_i$ (in \eqref{Eq: first regularity - 5.17}-\eqref{Eq: first regularity - 5.19}) is a stationary integral $G$-varifold in $\cKG_{1,1}$. 
    Using \eqref{Eq: first regularity - 5.17.2} and a computation similar to \eqref{Eq: first regularity - density}, we know the Allard regularity theorem \cite{allard1972first}*{\S 8} is applicable for $\sigma_0,r>0$ small enough so that $W_i=|\Sigma|$ near $x_0$ for some smooth minimal hypersurface $\Sigma$. 
    Since $\{P^i_k\}_{i=1,\dots,m}$ are disjoint, the maximum principle of minimal hypersurfaces and \eqref{Eq: first regularity - 5.19} indicate $\bleta^G_{r\rho_\infty\#}W= m |\Sigma|$ near $x_0$. 
    Hence, we get \eqref{Eq: first regularity - regular} by the $G$-invariance of $W$. 
\end{proof}

\begin{remark}
	In the proof above, we intentionally used some subtle constructions, which ensure that, except for the discussion in {\bf Case II}, the rest of the proof remains valid under the weakened condition: $\mc H^n(\Sigma_k)\leq \mc H^n(\Gamma') + \epsilon_k$ for all $\Gamma'\in\mc M^G_0$ with $\bd\Gamma'=\bd\Sigma_k $ and $ \mk b(\Gamma')\leq \mk b(\Sigma_k)$. 
\end{remark}

To show the final regularity of $W$ given in Theorem \ref{Thm: first regularity}, we aim to show the tangent cone of $W$ at $x_0\in\spt(\|W\|)\cap B\cap K_0$ is supported in a plane. 
The proof will be separated into two cases depending on the dimension of $N_{x_0}(G\cdot x_0)$. 
We shall begin with the case that $\dim(N_{x_0}(G\cdot x_0))=3$ using an argument similar to \cite{almgren79plateau}*{Theorem 3, Corollary 2}.

\begin{theorem}\label{Thm: interior regularity k0=n-2}
	Let $K_0,K_1,k_0,\rho_0$ be defined as before, $p_0\in K_0$, and $B:= B_{2\rho_0}(G\cdot p_0)$. 
	Suppose 
	\begin{align}\label{Eq: k0=n-2}
		\cohom(G) = 3 = n+1-k_0 = \codim(G\cdot p_0) .
	\end{align}
    Suppose also $\{\Sigma_k\}_{k\in\mb N}$ and $W$ are as in Theorem \ref{Thm: first regularity} so that $W$ has a tangent varifold 
    \[C=T_{x_0} (G\cdot x_0) \times \widehat{C} \in \mc V^{G_{x_0}}_n(T_{x_0}B) \]
    at $x_0\in B\cap K_0$ with $\widehat{C}\in \mc V^{G_{x_0}}_{2}(N_{x_0}(G\cdot x_0))$ (by Claim \ref{Claim: tangent bundle}) satisfying
    \begin{itemize}
    	\item $\spt(\|\widehat{C}\|)\cap \mb B_{\rho}(0)\subset \cup_{j=1}^R D_j$, and 
    	\item $\spt(\|\widehat{C}\|)\cap  (D_1 \setminus L) \neq \emptyset$,
    \end{itemize}
    where $D_1,\dots, D_R\subset N_{x_0}(G\cdot x_0)$ are distinct open $2$-disks of radius $\rho\in (0,\rho_0)$ that intersect in a common diameter $L$ of the ball $\mb B_{\rho}(0)\subset N_{x_0}(G\cdot x_0)$ and $\cup_{j=1}^R D_j$ is $G_{x_0}$-invariant. 
    Then 
    \begin{align}\label{Eq: interior regularity 1 - classify tangent cone}
    	\whC \res G_2(\mb B_\rho(0)) = m|D_1| 
    \end{align}
    for some $m\in\mb N$, and $D_1$ is orthogonal to $P(x_0)$ (cf. \eqref{Eq: P}). 
    
    As a corollary, $W$ satisfies the regularity \eqref{Eq: first regularity - regular} near $x_0\in \spt(\|W\|)\cap B$ provided \eqref{Eq: k0=n-2}. 
\end{theorem}

\begin{proof}
	To begin with, it follows from \eqref{Eq: k0=n-2} that $N_{x_0}(G\cdot x_0)\cong \mb R^3$. 
	Hence, we can take $P=P(x_0)=\{v\in\mb R^3: v_1=0\}$ and $G_{x_0}\cdot v=\{(v_1,v_2,v_3),(-v_1,v_2,v_3)\}$ for all $v\in N_{x_0}(G\cdot x_0)$. 
	We also borrow the notations from the proof of Theorem \ref{Thm: first regularity} (see Claim \ref{Claim: tangent bundle} and the notations before it) and similarly pull back all the computations to the normal bundle $N(G\cdot x_0)$ by $\exp_{G\cdot x_0}^{\perp,-1}$. 
	Denote by $B^G_{\rho} =G\cdot \mb B_\rho(0)$ the $\rho$ neighborhood of the zero section in $N(G\cdot x_0)$ for $\rho>0$. 
	
	Let $C=\lim_{t_k\to 0} \bleta_{t_k\#}W = T_{x_0} (G\cdot x_0) \times \widehat{C} \in \mc V^{G_{x_0}}_n(T_{x_0}B) $ be given as in the theorem. 
	Then \eqref{Eq: splitting tangent bundle} in Claim \ref{Claim: tangent bundle} implies that (up to a subsequence) $\lim_{t_k\to 0} \bleta^G_{t_k\#}W= G\cdot \whC \in \mc V^{G}_n(N(G\cdot x_0))$. 
	Clearly, we can find a subsequence of $\{\Sigma_k\}_{k\in\mb N}$ (without relabeling) so that $\wti\Sigma_k := \bleta^G_{r_k}(\Sigma_k)$ with $r_k\to 0$ as $k\to\infty$ satisfies 
	\begin{align}\label{Eq: interior regularity 1 - blowup}
		\lim_{k\to\infty} |\wti\Sigma_k| = C^G:=G\cdot \whC \in \mc V^G_n(N(G\cdot x_0))
	\end{align}
	and 
	\begin{align}\label{Eq: interior regularity 1 - almost minimizing}
		\mc H^n(\wti\Sigma_k)\leq \mc H^n(\Gamma')+\tilde{\epsilon}_k,\qquad \forall\Gamma'\in\mc M^G_0 {\rm ~with~}\bd\Gamma'=\bd\wti\Sigma_k,
	\end{align}
	where $\tilde{\epsilon}_k\to 0$ as $k\to\infty$. 
	
	Denote by $D_i^\pm\subset D_i$, $i=1,\dots,R$, the two half-disks with common diameter $L$. 
	Then \eqref{Eq: interior regularity 1 - almost minimizing} indicates that Theorem \ref{Thm: first regularity} is applicable to $G\cdot \whC$ (in place of $W$), and thus 
	\begin{align}\label{Eq: interior regularity 1 - 6.1}
		\whC \res G_2(\mb B_\rho(0)) = \sum_{j=1}^R \left( m_j^+ |D_j^+| + m_j^- |D_j^-| \right),
	\end{align}
	where $m_j^\pm$ are non-negative integers. 
	In addition, by Remark \ref{Rem: regularity in Mprin}, we know $\whC$ is an embedded stable minimal cone in $N_{x_0}(G\cdot x_0)\setminus P\cong \mb R^3\setminus (\{0\}\times \mb R^2)$, which implies \eqref{Eq: interior regularity 1 - classify tangent cone} provided that $L$ is not contained in $P$. 
	Next, we assume $L$ lies in $P$, and consider three cases.

	{\bf Case I: $\sum_{j=1}^R(m_j^++m_j^-)\leq 2$.}
	
	Since $\whC$ is stationary in $N_{x_0}(G\cdot x_0)$, the case $\sum_{j=1}^R(m_j^++m_j^-)= 1$ cannot happen. 
	If $\sum_{j=1}^R(m_j^++m_j^-)= 2$, then the stationarity of $\whC$ together with the assumption $\spt(\|\whC\|)\cap (D_1\setminus L)\neq\emptyset$ indicates $m_1^+=m_1^-=1$ (and $m_j^+=m_j^-=0$ for $j\neq 1$). 
	Hence, \eqref{Eq: interior regularity 1 - classify tangent cone} holds with $m=1$. 
	Moreover, $D_1$ is either orthogonal to $P=P(x_0)$ (cf. \eqref{Eq: P}) or contained in $P\cap\mb B_\rho(0)$ by the $G_{x_0}$-invariance. 
	But the latter case cannot occur by the proof of Theorem \ref{Thm: first regularity} {\bf Case II}.

	{\bf Case II: $\sum_{j=1}^R(m_j^++m_j^-)= 3$.}
	
	In this case, we have $3$ half-disks $H_1,H_2,H_3$ with a common diameter $L\subset P$ so that $\whC\res G_2(\mb B_\rho(0)) = \sum_{i=1}^3|H_i|$. 
	Combining the stationarity of $\whC$ with the $G_{x_0}$-invariance, one can show that one of the half-disks (say $H_1$) must lie in $P$, the other two half-disks (say $H_2,H_3$) are mirror symmetric across $P$, and they meet along $L\subset P$ at angles of $2\pi/3$. 
	By considering the current limit of $\llbracket \wti\Sigma_k \rrbracket$ as in \cite[Remark 5.20]{almgren79plateau}, we have a rectifiable $G$-invariant current $T$ with $\bd T\res B_\rho^G = 0$ and $T\res B_\rho^G = \sum_{i=1}^3 \llbracket G\cdot H_i \rrbracket$, which contradicts $\bd (\sum_{i=1}^3 \llbracket G\cdot H_i\rrbracket)\neq\emptyset$.

	{\bf Case III: $\sum_{j=1}^R(m_j^++m_j^-)\geq 4$.}
	
	In this case, it is sufficient to show that there is $m\in\mb N$ so that 
	\begin{align}\label{Eq: interior regularity 1 - 6.2}
		\whC\res G_2(\mb B_{\rho-\eta}(0)) = m |D_1\cap \mb B_{\rho-\eta}(0)|
	\end{align}
	for some $\eta\in (0,\rho)$. Then the cone structure of $\whC$ gives \eqref{Eq: interior regularity 1 - classify tangent cone}. 
	
	Let $\whU$ be a $C^2$ convex open $G_{x_0}$-set so that $\cup_{j=1}^RD_j \subset \whU\subset \mb B_\rho(0)$, and $U:=G\cdot \whU$ satisfies
	\begin{align}\label{Eq: interior regularity 1 - 6.3}
		\mc H^n(\bd U) = \mc H^{n-2}(G\cdot x_0)\cdot \mc H^2(\bd \whU) = \mc H^{n-2}(G\cdot x_0)\cdot \left( 4\pi\rho^2 - \alpha \right)
	\end{align}
	for some $\alpha>0$, where we have used the co-area formula as in \eqref{Eq: area in slice} under the metric $g_{_0}$ (defined before Claim \ref{Claim: tangent bundle}). 
	After equivariantly modifying $U$ if necessary, we can also suppose $\wti\Sigma_k$ is transversal to $\bd U$. 
	Then we apply Theorem \ref{Thm: replacement} to $\wti\Sigma_k$ in $U$ and obtain pairwise disjoint $P_k^1,\dots,P_k^{T_k}\in\mc M^G_0$ (where $T_k\in\mb N$) so that $P_k^i\setminus\bd P_k^i\subset U$, $\bd P_k^i\subset \bd U$, and 
	\begin{align}\label{Eq: interior regularity 1 - 6.4}
		\mc H^n(P_k^i)\leq \mc H^n(P') + \tilde{\epsilon}_{k,i}, \qquad\forall P'\in\mc M^G_0 {\rm~with~} \bd P'= \bd P_k^i,
	\end{align}
	where $\sum_{i=1}^{T_k}\tilde{\epsilon}_{k,i}\leq \tilde{\epsilon}_k$. 
	Additionally, by Theorem \ref{Thm: replacement}(iv)(vi) and \cite{allard1972first}*{2.6(2)(d)}, we know 
	\begin{align}\label{Eq: interior regularity 1 - 6.5}
		C^G \res G_n(B_{\rho-\eta}^G) = \lim_{k\to\infty} \sum_{i=1}^{T_k} \left| P_k^i\cap B_{\rho-\eta}^G\right|,
	\end{align}
	for any pre-assigned $\eta \in (0,\rho)$. 
	
	By \eqref{Eq: interior regularity 1 - 6.4} and the filigree lemma (Lemma \ref{Lem: filigree}), there exists $\beta>0$ (depending on $\eta$ but independent of $k$) so that $\mc H^n(P_k^i\cap B^G_{\rho-2\eta}) \leq 2\tilde{\epsilon}_{k,i}$ whenever $\mc H^n(P_k^i\cap B^G_{\rho-\eta})<\beta$. 
	(Note that the filigree lemma can be applied to $B^G_{\rho-\eta}$ by an argument similar to Remark \ref{Rem: compare disk instead cylinder}.) 
	Hence, combining this with \eqref{Eq: interior regularity 1 - 6.5} and \cite{allard1972first}*{2.6(2)(d)}, we see 
	\[ \lim_{k\to\infty} \sum_{i\in \mc I_k}  \left| P_k^i\cap B_{\rho-2\eta}^G\right| = C^G \res G_n(B_{\rho-2\eta}^G),\]
	where $\mc I_k\subset \{1,\dots, T_k\}$ is the set of $i$ with $\mc H^n(P_k^i\cap B^G_{\rho-\eta})\geq\beta$. 
	Note $\# \mc I_k$ is uniformly bounded by \eqref{Eq: interior regularity 1 - 6.5}. 
	Hence, up to a subsequence and relabeling, we have $\# \mc I_k\equiv m\in\mb N$ and 
	\begin{align}\label{Eq: interior regularity 1 - 6.5'}
		C^G_i = \lim_{k\to\infty} |P_k^i|, \qquad i\in \{1,\dots,m\}= \mc I_k;
		\\
		C^G \res G_n(B_{\rho-2\eta}^G) = \sum_{i=1}^m C^G_i\res G_n(B_{\rho-2\eta}^G).\nonumber
	\end{align}
	Using \eqref{Eq: interior regularity 1 - 6.4} and Theorem \ref{Thm: first regularity}, one can see that $C^G_i$ is a stationary integral $G$-varifold in $U$ so that $\spt(\|C^G_i\|)\cap B^G_{\rho-\eta}\subset \spt(\|C^G\|)\cap B^G_{\rho-\eta}\subset G\cdot (\cup_{j=1}^RD_j)$. 
	Therefore, Theorem \ref{Thm: first regularity} and the $G$-invariance further indicate that there is an integral stationary $G_{x_0}$-varifold $\whC_i=\mf v(\cup_{j=1}^R D_j, \theta_{\whC_i})\in \mc V^{G_{x_0}}_2(N_{x_0}(G\cdot x_0))$ so that 
	\begin{itemize}
		\item $C^G_i=G\cdot \whC_i := \mf v(G\cdot \cup_{j=1}^R D_j, \theta_{C^G_i})$ with $ \theta_{C^G_i}(g\cdot v) = \theta_{\whC_i}(v)$ for $v\in N_{x_0}(G\cdot x_0), g\in G$;
		\item $\whC \res G_2(\mb B_{\rho-2\eta}(0)) = \sum_{i=1}^m \whC_i\res G_2(\mb B_{\rho-2\eta}(0))$, and 
	\end{itemize}
	\begin{align}\label{Eq: interior regularity 1 - 6.6}
		\whC_i\res G_2(\mb B_{\rho-\eta}(0)) = \sum_{j=1}^R\left(m_{ij}^+|D_j^+\cap \mb B_{\rho-\eta}(0)| + m_{ij}^-|D_j^-\cap \mb B_{\rho-\eta}(0)|\right),
	\end{align}
	where $i=1,\dots,m$, and $m_{ij}^\pm$ are non-negative integers. 
	
	By Lemma \ref{Lem: hypersurface position}, we can take an unrelabelled subsequence of $\{P_k^i\}_{k\in\mb N}$ for each $i$ so that either
		\begin{itemize}
			\item[(a)] $P_k^{i*} \subset K_1^*\cap U^*$ for every $k\in\mb N$; or
			\item[(b)] $\bd_1P_k^{i*}$ is a curve in $\closure(\bd_1U^*)\cong \mb S^2_+$ with free boundary on $\bd_0 U^*\cap\closure(\bd_1U^*)\cong \bd \mb S^2_+$ for every $k\in\mb N$, and thus $\bd_1P_k^{i*}$ separates $\bd_1U^*$ into two half-disks $E_+^*,E_-^*$; or
			\item[(c)] $\bd_1P_k^{i*}$ is a closed curve in $\interior (\bd_1U^*)\cong\mb S^2_+$ for every $k\in\mb N$, and thus $\bd_1P_k^{i*}$ separates $\bd_1U^*$ into a disk $E_s^*$ and a cylinder $E_c^*$.
		\end{itemize}
	
	\begin{claim}\label{Claim: from classical cone to plane}
		Whenever $\whC_i\res G_2(\mb B_{\rho-\eta}(0))\neq 0$, there exists $j_i\in\{1,\dots,R\}$ so that 
		$\whC_i \res G_2(\mb B_{\rho-\eta}(0)) = |D_{j_i}\cap\mb B_{\rho-\eta}(0)|$. 
	\end{claim}
	\begin{proof}[Proof of Claim \ref{Claim: from classical cone to plane}]
		Note that Case (a) cannot happen due to the locally $G$-boundary-type condition. 
		In Case (b), we  use \eqref{Eq: interior regularity 1 - 6.3}\eqref{Eq: interior regularity 1 - 6.4} to compare the area of $P_k^i$ and $E_\pm:=\pi^{-1}(E_\pm^*)$,   
		\begin{align*}
		 	\mc H^n(P_k^i) &\leq \min\{\mc H^n(E_+), \mc H^n(E_-)\} + \tilde{\epsilon}_{k,i} \leq \frac{1}{2} \mc H^n(\bd U) + \tilde{\epsilon}_{k,i}
		 	\\&\leq \mc H^{n-2}(G\cdot x_0)\cdot\left( 2\pi\rho^2 - \frac{\alpha}{2} \right) + \tilde{\epsilon}_{k,i} + \epsilon(r_k),
		\end{align*}
		where $\epsilon(r)\geq 0$ with $\lim_{r\to 0}\epsilon(r)=0$ measures the difference between $g_{_{r}}$ and $g_{_0}$ (see $\vartheta$ in \eqref{Eq: area in slice}). 
		Together with \eqref{Eq: area in slice} under the metric $g_{_0}$, we have 
		\begin{align}\label{Eq: interior regularity 1 - half-sphere half-disk}
			\|\whC_i\|(\mb B_{\rho-\eta}(0)) = \frac{\|C_i^G\|(B^G_{\rho-\eta})}{\mc H^{n-2}(G\cdot x_0)} < 2\pi (\rho-\eta)^2 
		\end{align}
		for $\eta\in (0,\rho)$ small enough, which implies $\sum_{j=1}^R(m_{ij}^++m_{ij}^-)\leq 3$ in \eqref{Eq: interior regularity 1 - 6.6}. 
		Hence, the arguments in {\bf Case I, II} show that $\whC_i \res G_2(\mb B_{\rho-\eta}(0)) = |D_{j_i}\cap\mb B_{\rho-\eta}(0)|$ for some $j_i\in\{1,\dots,R\}$ whenever $\whC_i\res G_2(\mb B_{\rho-\eta}(0))\neq 0$. 
		
		In case (c), noting $E_s,E_c\in \mc M^G_0$ with $\bd E_s=\bd E_c=\bd P_k^i$, we can proceed as in case (b) to show $\sum_{j=1}^R(m_{ij}^++m_{ij}^-)\leq 3$ in \eqref{Eq: interior regularity 1 - 6.6}. 
		Hence, the proof in {\bf Case I, II} gives the desired result.  
	\end{proof}
	
	Since $\spt(\|\whC\|)\cap (D_1\setminus L )\neq\emptyset$, Claim \ref{Claim: from classical cone to plane} implies that there is an $i$, say $i=1$, so that 
	$\whC_1 \res G_2(\mb B_{\rho-\eta}(0)) = |D_{1}\cap\mb B_{\rho-\eta}(0)|$. 
    Additionally, if there are $i_1,i_2$ with $D_{j_{i_1}}\neq D_{j_{i_2}}$, then the arguments in \cite{almgren79plateau}*{Page 472, (6.8)} would carry over to give a contradiction. 
    Hence, \eqref{Eq: interior regularity 1 - 6.2} is valid. 
	Therefore, we have shown \eqref{Eq: interior regularity 1 - classify tangent cone} in all {\bf Case I-III}. 
	
	\medskip
	Finally, for any tangent varifold $C=T_{x_0}(G\cdot x_0)\times \whC$ (cf. \eqref{Eq: splitting tangent bundle} in Claim \ref{Claim: tangent bundle}) of $W$ at $x_0\in \spt(\|W\|)\cap (B\cap K_0)$, we can use \eqref{Eq: interior regularity 1 - almost minimizing} and Theorem \ref{Thm: first regularity} to show that $\whC$ is an integral stationary $G_{x_0}$-varifold. 
	Then, similar to \cite{almgren79plateau}*{Corollary 2}, it follows from \cite{allard1972first}*{6.5}\cite{allard76OneDim} that for each $v\neq 0 \in N_{x_0}(G\cdot x_0)$ there are $\rho>0$ and distinct closed half-disks $H_1,\dots,H_R$ sharing a common diameter $L$ of $\mb B_\rho(v)$ so that 
	\[\spt(\|\whC \|)\cap \mb B_\rho(v) = \cup_{j=1}^R H_j,\]
	which implies  $\whC\res G_2(\mb B_\rho(0)) = m|D|$ for some disk $D$ and $m\in\mb N$ by using \eqref{Eq: interior regularity 1 - classify tangent cone} with $C$ in place of $W$. 
	Therefore, $\whC$ is supported in a plane, and thus $W$ is a smooth embedded minimal $G$-hypersurface in $B$ with integer multiplicity by Theorem \ref{Thm: first regularity}.  
\end{proof}

It remains to consider the case that $\dim(N_{x_0}(G\cdot x_0))\geq 4$. 
We mention that the proof of Theorem \ref{Thm: interior regularity k0=n-2} would {\em not} carry over if $\dim(N_{x_0}(G\cdot x_0))\geq 4$. 
For instance, in the above theorem, we have used the fact that $\mc H^2(\mb S^2_+)/\mc H^2(\mb D^2)=2$ to get  $\sum_{j=1}^R(m_{ij}^++m_{ij}^-)\leq 3$ from \eqref{Eq: interior regularity 1 - half-sphere half-disk}. 
However, for $n-k_0>2$, one verifies that $\mc H^{n-k_0}(\mb S^{n-k_0}_+)/\mc H^{n-k_0}(\mb D^{n-k_0})>2$, which is not enough to deduce {\bf Case III} into {\bf Case I, II} as in Claim \ref{Claim: from classical cone to plane}. 
Additionally, the result in \cite{allard76OneDim} used in the final step only holds for $1$-dimensional varifolds in $\bd \mb B_{|v|}^3(0)$. 

Nevertheless, we can combine Theorem \ref{Thm: first regularity} and Remark \ref{Rem: regularity in Mprin} with a dimension reduction argument to show the regularity in this situation. 

\begin{theorem}\label{Thm: interior regularity k_0<n-2}
	Let $K_0,k_0,\rho_0$ be defined as before, $p_0\in K_0$, and $B:= B_{2\rho_0}(G\cdot p_0)$. 
	Suppose 
	\begin{align}\label{Eq: k0<n-2}
		\cohom(G) = 3 < n+1- k_0 = \codim(G\cdot p_0) \leq 7.
	\end{align}
    Suppose also that $\{\Sigma_k\}_{k\in\mb N}$ and $W$ are given in Theorem \ref{Thm: first regularity}, and $W$ has a tangent varifold 
    \[C=T_{x_0} (G\cdot x_0) \times \widehat{C} \in \mc V^{G_{x_0}}_n(T_{x_0}B) \]
    at $x_0\in B\cap K_0$ with $\widehat{C}\in \mc V^{G_{x_0}}_{n-k_0}(N_{x_0}(G\cdot x_0))$ (by Claim \ref{Claim: tangent bundle}). 
    Then $\whC$ is a $G_{x_0}$-invariant $(n-k_0)$-plane (with integer multiplicity) that is orthogonal to $P(x_0)$ (cf. \eqref{Eq: P}). 
    
    In particular, $W$ satisfies the regularity \eqref{Eq: first regularity - regular} near $x_0\in \spt(\|W\|)\cap B$ provided \eqref{Eq: k0<n-2}. 
\end{theorem}

\begin{proof}
	Borrow the notations from the proof of Theorem \ref{Thm: first regularity} (see Claim \ref{Claim: tangent bundle} and the notations before it) and similarly pull back all the computations to the normal bundle $N(G\cdot x_0)$ by $\exp_{G\cdot x_0}^{\perp,-1}$. 
	
	Firstly, note that \eqref{Eq: interior regularity 1 - blowup}\eqref{Eq: interior regularity 1 - almost minimizing} are also valid. 
	Together with Theorem \ref{Thm: first regularity} and Claim \ref{Claim: tangent bundle}, we see $C^G:=G\cdot \whC$ and $\whC$ are stationary integral varifolds in $N(G\cdot x_0)$ and $N_{x_0}(G\cdot x_0)$ respectively. 
	In addition, combining \cite{simon1983lectures}*{Theorem 19.3} with Remark \ref{Rem: regularity in Mprin} and the compactness theorem for stable minimal hypersurfaces, we know $\whC \subset N_{x_0}(G\cdot x_0)$ is a stationary cone that is smoothly embedded and stable in $N_{x_0}(G\cdot x_0)\setminus P$, where $P=P(x_0)$ is given in \eqref{Eq: P}. 
	Hence, if $\spt(\|\whC\|)\cap P=\{0\}$, then $\whC$ must be a hyperplane, which implies $\dim(\spt(\|\whC\|)\cap P)\geq 1$, yielding a contradiction. 
	
	Next, for any $v\neq 0 \in \spt(\|\whC\|)\cap P$, let $\whD$ be a tangent varifold of $\whC$ at $v$, i.e. $\whD=\lim_{j\to \infty}(\bleta_{v,\lambda_j})_\#\whC$ for some $\lambda_j\to 0$ as $j\to\infty$, and take $D^G:=\lim_{j\to\infty}(\bleta^G_{G\cdot v, \lambda_j})_\# C^G$. 
	Note that the definitions of $\bleta_{v,\lambda_j}, \bleta^G_{G\cdot v, \lambda_j}$ can be easily generalized from the original definitions of $\bleta_{r},\bleta^G_r$ before Claim \ref{Claim: tangent bundle}. 
	In addition, since $v\in P$ satisfies $G_{x_0}\cdot v=\{v\}$, one easily checks that the maps $\bleta_{v,\lambda_j}$ and $\bleta^G_{G\cdot v, \lambda_j}$ preserve the $G_{x_0}$-invariance and $G$-invariance, respectively. 
	
	Moreover, by an iterated construction, $D^G$ also satisfies \eqref{Eq: interior regularity 1 - blowup} and \eqref{Eq: interior regularity 1 - almost minimizing} for some sequence $\widetilde{\Sigma}_k= \bleta^G_{G\cdot v, r_k}(\Sigma_k)$ in $\mc M^G_0$. 
	Therefore, Theorem \ref{Thm: first regularity} indicates $D^G$ is an integral stationary $G$-varifold in $N(G\cdot v)$, which also implies $D^G=G\cdot \whD$ by Claim \ref{Claim: tangent bundle}. 
	Again, it follows from \cite{simon1983lectures}*{Theorem 19.3}, Remark \ref{Rem: regularity in Mprin}, and the compactness theorem for stable minimal hypersurfaces that $\whD$ is an $(n-k_0)$-dimensional stationary rectifiable cone in $N_{v}(G\cdot v)$ which is smoothly embedded and stable in $N_{v}(G\cdot v)\setminus P$ (note that $\bleta_{v,\lambda}(P)=P$). 
	
	One also notices that the cone $\whD$ has a spine $L=\lim_{j\to\infty} \bleta_{v,\lambda_j}(\{tv: t\in[0,\infty)\})\cong \mb R \subset P\subset  N_{v}(G\cdot v)$, i.e. $\whD$ is translation invariant along $L$. 
	Hence, for the codimension-one $G_{v}$-invariant $(n-k_0)$-plane $P'\subset N_{v}(G\cdot v)$ orthogonal to $L$ at $0$, we know $\spt(\|\whD\|)\cap P'$ is an $(n-k_0-1)$-dimensional stationary $G_v$-invariant cone. 
	Combined with the fact that $P^{\perp*}\cong [0,\infty)$, we see
	\begin{align}\label{Eq: interior regularity 2 - smooth cone}
		\spt(\|\whD\|)\cap P'\cap P = \{0\}.
	\end{align}
	Indeed, if this is not true, then $\spt(\|\whD\|)\cap P'\cap P$ contains at least a ray (say $\{tu: t\in [0,\infty)\}$). 
	Consider the tangent varifold $\whD'$ of $\whD$ at $u$. 
	One can repeat the above procedure to show $\whD'$ is a stationary cone in $N_u(G\cdot u)$ which is invariant under the $G_u=G_v=G_{x_0}$ action and also has a $2$-dimensional spine given by $P$. 
	Thus, $\spt(\|\whD'\|)\cap P^\perp$ is a $G_u$-invariant $(n-k_0-2)$-cone. 
	However, this contradicts the fact that the only two $G_u$-invariant cones in $P^\perp$ are $\{0\}$ and $P^\perp$ (of dimension $0$ and $(n-k_0-1)$ respectively), since any ray at $\{0\}$ in $P^\perp$ is a fundamental domain of $P^\perp$ under the $G_u$-action. 
	
	To proceed, we notice that $\whD \setminus P$ is a smooth stable minimal cone and $P'$ is orthogonal to the spine $L$ of $\whD$. 
	Hence, combined with \eqref{Eq: interior regularity 2 - smooth cone}, we conclude that $\whD\cap P'$ is a stable minimal cone of dimension $n-k_0-1 \in [2,5]$ (by \eqref{Eq: k0<n-2}), and $\whD\cap P'$ is smooth except possibly at the origin. 
	Therefore, both $\whD\cap P'$ and $\whD$ are planes. (Indeed, $(\whD)^* = L\times P^\perp$ by Lemma \ref{Lem: hypersurface position}.)
	
	By the definition of $\whD$, we can then apply Theorem \ref{Thm: first regularity} to $\whC$ at any $v\neq 0 \in \spt(\|\whC\|)\cap P$, and see $\whC$ is a smoothly embedded cone in $N_{x_0}(G\cdot x_0)$ except possibly at the origin.  
	In addition, the classification of $\whD$ implies $\whC$ meets $P\setminus\{0\}$ orthogonally, which indicates that $\whC$ is stable by an argument similar to Remark \ref{Rem: regularity in Mprin} using the simple connectedness of $A_{s,t}(G\cdot x_0)/G$. 
	Hence, as $\dim(\whC) = n-k_0\in [3,6]$ by \eqref{Eq: k0<n-2}, we conclude that the stable minimal cone $\whC$ is an $(n-k_0)$-plane in $N_{x_0}(G\cdot x_0)$, which gives the desired regularity of $W$ by Theorem \ref{Thm: first regularity}.  
\end{proof}

As a direct corollary from Theorem \ref{Thm: interior regularity k0=n-2} and \ref{Thm: interior regularity k_0<n-2}, we have the following result: 
\begin{theorem}\label{Thm: regularity part1}
	Suppose \eqref{Eq: cohomogeneity assumption} is satisfied.
	Let $\{\Sigma_k\}_{k\in\mb N}$ and $V\in\mc V^G_n(M)$ be given as in Theorem \ref{Thm: plateau problem}. 
	Then for any $p_0\in \spt(\|V\|)\setminus (M^{prin}\cup\Gamma)$ with $\dim(P(p_0))=2$ (defined in \eqref{Eq: P}), we have $V\llcorner B_{\rho}(G\cdot p_0)=m|\Sigma|$ for some embedded $G$-hypersurface $\Sigma$ with integer multiplicity $m\in\mb N$ and some $\rho>0$ so that $\Sigma$ meets $M_{(G_{p_0})}$ orthogonally at $G\cdot p_0$. 
\end{theorem}

\subsection{Regularity near $p\in M\setminus M^{prin}$ with $\dim(P(p))=1$}\label{Subsec: plateau regularity 2}

Fix a non-principal orbit type stratum $M_{(H)}$ and a compact subset $K_0\subset M_{(H)}$ with 
\begin{align}\label{Eq: Case dimP=1 assumption}
	\dim(K_0^*)=1.
\end{align}
Set $k_0:=\dim(G\cdot p)$ for $p\in K_0$, and let $\rho_0>0$ be given as in Lemma \ref{Lem: uniform constants}.

\begin{theorem}\label{Thm: regularity part2}
	Suppose \eqref{Eq: cohomogeneity assumption} is satisfied. 
	Let $\{\Sigma_k\}_{k\in\mb N}$ and $V\in\mc V^G_n(M)$ be given as in Theorem \ref{Thm: plateau problem}. 
	Then for any $p_0\in \spt(\|V\|)\setminus (M^{prin}\cup\Gamma)$ with $\dim(P(p_0))=1$ (defined in \eqref{Eq: P}), we have $V\llcorner B_{\rho}(G\cdot p_0)=m|\Sigma|$ for some embedded $G$-hypersurface $\Sigma$ with integer multiplicity $m\in\mb N$ and some $\rho>0$ so that $\Sigma$ meets $M_{(G_{p_0})}$ orthogonally at $G\cdot p_0$. 
\end{theorem}

\begin{proof}
	Since the proof for this case is similar to the previous subsection, we only point out the necessary modifications.

	By Proposition \ref{Prop: local structure of Sigma/G 2}, a $G$-hypersurface $\Sigma$ has four types of local structure near $G\cdot p\subset K_0$. 
	If $\Sigma$ is in Proposition \ref{Prop: local structure of Sigma/G 2}(1)(2), then $B_{\rho_0}(G\cdot p)$ is half-ball-type, and $(B_{\rho_0}(G\cdot p)\setminus M^{prin})^*$ represents the boundary $\bd_0B_{\rho_0}^*([p])$ of the relative open wedge $B_{\rho_0}^*([p])$.
	Hence, we can define 
	\[K_1:= (M\setminus M^{prin}) \cap B_{\rho_0}(K_0) \quad \mbox{and} \quad\mk b(\Sigma):=\#\{\mbox{$G$-connected components of $\Sigma\cap K_1$}\},\]
	and the proof in the previous subsection carries over almost verbatim. 
	Note that $K_1$ is a union of submanifolds (non-principal orbit type strata), and $T_pK_1$ is the union of some subspaces of $T_pM$ so that $T_pK_1\cap N_p(G\cdot p)$ is the union of non-principal $G_p$-orbits in $N_p(G\cdot p)$. 
	If $\Sigma$ is in Proposition \ref{Prop: local structure of Sigma/G 2}(3)(4), then $B_{\rho_0}(G\cdot p)$ is a ball-type $G$-neighborhood and $\Sigma$ cannot be contained in $B_{\rho_0}(G\cdot p)\setminus M^{prin}$, which implies some assumptions in the previous subsection will be unnecessary (cf. Remark \ref{Rem: filigree lemma statements}(iii) and $\mk b$ in \eqref{Eq: boundary components number}). 
	Hence, in this case, we take $K_1=\emptyset$.  
	
	\medskip
	{\bf Step 1.} {\it Generalize preliminary lemmas.}
	
	To be exact, we first see that Lemma \ref{Lem: switch lemma} remains valid near $G\cdot p_0\subset K_0$. 
	Indeed, due to the locally $G$-boundary-type assumption, $\Sigma_i\subset B$ in Lemma \ref{Lem: switch lemma} must satisfy either Proposition \ref{Prop: local structure of Sigma/G 2}(1) or (3). 
	Then, the proof can be repeated almost verbatim. 
	In particular, we mention that $(\bd\Sigma_i)^*$ can only be a closed curve in $\bd_1B^*$ if $\Sigma_i$ satisfies Proposition \ref{Prop: local structure of Sigma/G 2}(3). 

	Next, for $G$-hypersurfaces satisfying Proposition \ref{Prop: local structure of Sigma/G 2}(1), Theorem \ref{Thm: replacement} and Lemma \ref{Lem: filigree} remain valid by the same proof using the new definitions of $K_0, K_1$. 
	For $G$-hypersurfaces satisfying Proposition \ref{Prop: local structure of Sigma/G 2}(3), we remove all the terms of $\mk b$ and remove the assumption \eqref{Eq: compare to disk}; then the proof of Theorem \ref{Thm: replacement} and Lemma \ref{Lem: filigree} would also carry over with $\mc D^G$ in place of $\mc M^G_0$. 
	We mention that, in this case, $[p]$ is an interior point of the topological $3$-manifold $M^*\setminus \mc S^*$. 
	Additionally, due to the locally $G$-boundary-type assumption on $\Sigma\in\mc D^G$, $\Sigma$ cannot satisfy Proposition \ref{Prop: local structure of Sigma/G 2}(4). 
	Hence, for any ball-type $G$-neighborhood $U$ with $\bd U$ transversal to $\Sigma$, we know $(\Sigma\cap \bd U)^*$ must be a union of closed curves in $\bd_1 U^*$, each of which separates $\bd U^*$ into two disks. 
	This is important in the area comparison (e.g., to apply \eqref{Eq: filigree - 4.4}) because if $\Sigma$ satisfies Proposition \ref{Prop: local structure of Sigma/G 2}(4), then $\Sigma$ does not separate $\bd U$ into two $G$-components.

	\medskip
	{\bf Step 2.} {\it Generalize the first regularity theorem.}
	
	Consider the First Regularity Theorem \ref{Thm: first regularity} near $x_0\in K_0$ with the same notations. 

	For the case that $B_{2\rho_0}(G\cdot p)$ satisfies Proposition \ref{Prop: local structure of M/G}(2.a), either $H$ is orthogonal to $T_{x_0}K_0$ or $T_{x_0}K_0\subset H\subset T_{x_0}K_1$. 
	If $H$ is orthogonal to $T_{x_0}K_0$, then the proof of Theorem \ref{Thm: first regularity} {\bf Case I} carries over almost verbatim with our new notations of $K_0,K_1$. 
	If $T_{x_0}K_0\subset H$, then we know $H$ satisfies the same property as the tangent plane in Proposition \ref{Prop: local structure of Sigma/G 2}(2).  
    One can similarly obtain a contradiction as in Theorem \ref{Thm: first regularity} {\bf Case II}. 
	Specifically, let $\mb R^3_{++}:=\{x_1,x_2\geq 0\}\subset \mb R^3$ be the wedge $N_{x_0}(G\cdot x_0)/G_{x_0}$, take $P=\{x_1=x_2=0\}$, and modify the notations by 
    \begin{itemize}
    	\item $\widehat{H}^*=(H\cap N_{x_0}(G\cdot x_0))^*:=\{x\in \mb R^3_{++}: x_1=0\}$; 
    	\item $(\cKG_{\rho,\sigma})^*:=\{x\in \mb R^3_{++}: x_2^2+x_3^2<\rho^2, |x|^2<\rho^2+\sigma^2\}$; 
    	\item $\tau=0$ if $\dim(N_{x_0}(G\cdot x_0))=3$, otherwise $\tau\in (0,\sigma)$;  
    	\item $(\cSG_{\rho,\sigma,\tau})^* :=\{x\in \mb R^3_{++}: x_2^2+x_3^2<\rho^2, |x|^2=\rho^2+\sigma^2, x_2\geq \tau \}$ the spherical part of $(\bd \cKG_{\rho,\sigma})^*$ with $(B_\tau(T_{x_0}K_1))^*$ removed; 
    	\item $(\cCG_{\rho,\sigma,\tau})^*:=\{x\in \mb R^3_{++}: x_2^2+x_3^2=\rho^2, |x|^2<\rho^2+\sigma^2\}\cup \{x\in \mb R^3_{++}: x_2^2+x_3^2<\rho^2, |x|^2=\rho^2+\sigma^2, x_2< \tau\}$ the cylindrical part of $(\bd \cKG_{\rho,\sigma})^*$ with $(\bd \cKG_{\rho,\sigma})^*\cap(B_\tau(T_pK_1))^*$ added; 
    	\item $(\cBG_{\rho,\sigma,\tau})^*:=\{x\in \mb R^3_{++}: x_1=\sigma, x_2^2+x_3^2<\rho^2, x_2\geq \tau \}$;
    	\item $(\cAG_{\rho,\sigma,\tau})^*:=\{x\in \mb R^3_{++}: x_1>\sigma, |x|^2<\rho^2+\sigma^2, x_2= \tau \}$. 
    \end{itemize}
    Note that $(\cCG_{\rho,\sigma,\tau})^*$ deformation retracts onto $\bd_0(\cCG_{\rho,\sigma,\tau})^*$. 
    Hence, we proceed as in Theorem \ref{Thm: first regularity} {\bf Case I} to obtain $\{P_k^i\}_{i=1}^{R_k'}$ so that the statements (1)(2)(4) before Claim \ref{Claim: discard trivial curve} are satisfied with $R_k=0$, i.e. none of $P_k^i$ satisfies (3) before Claim \ref{Claim: discard trivial curve}. 
    As in Theorem \ref{Thm: first regularity} {\bf Case II}, the proof of Claim \ref{Claim: discard trivial curve} would carry over to discard all the $\{P_k^i\}_{i=1}^{R_k'}$, which contradicts $x_0\in\spt(\|W\|)$. 

	For the case that $B_{2\rho_0}(G\cdot p)$ satisfies Proposition \ref{Prop: local structure of M/G}(2.b), $B_{\rho}(G\cdot x)$ is a ball-type neighborhood near $x\in K_0$, and we only need to consider the closed curves on $\bd_1(\cKG_{\rho,\sigma})^*$ in the topological arguments (due to the locally $G$-boundary-type assumption). 
	If $H$ is orthogonal to $T_{x_0}K_0$, we can take $\tau=0$ and apply the proof of Theorem \ref{Thm: first regularity} {\bf Case I} directly with $\mc D^G$ in place of $\mc M^G_0$. 
	(The proof is even easier since $(\bd P_k^i)^*$ are all closed curves that do not intersect $K_0^*$.) 
	If $T_{x_0}K_0\subset H$, then by Proposition \ref{Prop: local structure of Sigma/G 2}(4), we know $\codim(G\cdot p)=3$ for $p\in K_0$, and the $G_p$-action on $N_p(G\cdot p)\cong \mb R^3$ satisfies 
	\[G_p\cdot (x_1,x_2,x_3)=\{(x_1,x_2,x_3), (-x_1,-x_2,x_3)\}.\]
	In particular, $(H\cap N_{x_0}(G\cdot x_0))^*$ is a half-plane with boundary $T_{x_0}K_0\cap N_{x_0}(G\cdot x_0)\cong \mb R$. 
	We can then exclude this case as in Theorem \ref{Thm: first regularity} {\bf Case II}. 
	Specifically, let $N_{x_0}(G\cdot x_0)=\mb R^3$ with $P=P(x_0)=\{x_1=x_2=0\}=T_{x_0}K_0\cap N_{x_0}(G\cdot x_0)$, and modify the notations by 
	\begin{itemize}
    	\item $\widehat{H}=H\cap N_{x_0}(G\cdot x_0):=\{x\in\mb R^3: x_1=0\}$; 
    	\item $\whcK_{\rho,\sigma}:=\{x\in \mb R^3: x_2^2+x_3^2<\rho^2, |x|^2<\rho^2+\sigma^2\}$ and $ \cKG_{\rho,\sigma}=G\cdot \whcK_{\rho,\sigma}$; 
    	\item $\tau=0$; 
    	\item $\whcS_{\rho,\sigma,0} :=\{x\in \mb R^3: x_2^2+x_3^2<\rho^2, |x|^2=\rho^2+\sigma^2\}$ the spherical part of $\bd \whcK_{\rho,\sigma}$; 
    	\item $\whcC_{\rho,\sigma,0}:=\{x\in \mb R^3: x_2^2+x_3^2=\rho^2, |x|^2<\rho^2+\sigma^2\}$ the cylindrical part of $\bd \whcK_{\rho,\sigma}$; 
    	\item $\whcB_{\rho,\sigma,0}:=\whcB_{\rho,\sigma,0}^+\cup \whcB_{\rho,\sigma,0}^-$ where $\whcB_{\rho,\sigma,0}^{\pm}:=\{x\in \mb R^3: x_1=\pm\sigma, x_2^2+x_3^2<\rho^2\}$;
    	\item $\widehat \Sigma:=\Sigma\cap N_{x_0}(G\cdot x_0)$ for any $G$-hypersurface $\Sigma\subset N(G\cdot x_0)$. 
	\end{itemize}
	Let $\cSG_{\rho,\sigma,0}:=G\cdot \whcS_{\rho,\sigma,0}$, $\cCG_{\rho,\sigma,0}:=G\cdot \whcC_{\rho,\sigma,0}$, and $\cBG_{\rho,\sigma,0}:=G\cdot \whcB_{\rho,\sigma,0}$, which are $G$-hypersurfaces with slices in $N_{x_0}(G\cdot x_0)$ given by $\whcS_{\rho,\sigma,0},\whcC_{\rho,\sigma,0},\whcB_{\rho,\sigma,0}$ respectively. 
	By arguments similar to those in Theorem \ref{Thm: first regularity} {\bf Case I}, we obtain $\{P_k^i\}_{i=1}^{R_k+R_k'}\subset \mc D^G$ satisfying (1)(2) before Claim \ref{Claim: discard trivial curve} with $\mc D^G$ in place of $\mc M^G_0$, and 
	\begin{itemize}
		\item[(3')] for each $i\in\{1,\dots,R_k\}$, the curve $\bd \widehat P_k^{i}$ is non-trivial in the homotopy group $\pi_1(\whcC_{\rho_k,\sigma_k',0} )$, i.e. every component of $\bd \widehat P_k^{i}$ cannot be null-homotopic in $\whcC_{\rho_k,\sigma_k',0}$; note also that $\bd \widehat P_k^{i}\subset (N_{x_0}(G\cdot x_0))^{prin}$ does not intersect $K_0$; 
       \item[(4')] for each $i\in \{R_k+1,\dots,R_k+R_k'\}$, every component of the curve $\bd \widehat P_k^{i}$ is either 
        		\begin{itemize}
        			\item closed and null-homotopic in $\whcS_{\rho_k,\sigma_k',0}$, or
        			\item closed and null-homotopic in $\whcC_{\rho_k,\sigma_k',0}$. 
        		\end{itemize}
	\end{itemize}
	Similar to Claim \ref{Claim: discard trivial curve}, we can discard $\{P_k^i\}_{i=R_k+1}^{R_k+R_k'}$ so that \eqref{Eq: first regularity - 5.13} remains valid. 
	Then for any $1\leq i\leq R_k$, we see from the representation of the $G_{x_0}$-action that $\bd \widehat P_k^i$ is the disjoint union of two non-trivial closed curves with $G_{x_0}$ permuting them. 
	Combined with the fact that $ P_k^i\in \mc D^G$, we also know that $\widehat P_k^i$ is the disjoint union of two disks $ \widehat P_k^{i,\pm}$ with $G_{x_0}$ permuting them. 
	In particular, $\whcC_{\rho_k,\sigma_k',0}\cup \whcB_{\rho_k,\sigma_k',0}\cong \mb S^2$, and $(\whcC_{\rho_k,\sigma_k',0}\cup \whcB_{\rho_k,\sigma_k',0})\setminus \bd \widehat P_k^i$ has two $G_{x_0}$-components $\widehat E_s, \widehat E_c$ so that $\widehat E_s$ is a union of two disks, and $ \widehat E_c$ is a cylinder contained in $\whcC_{\rho_k,\sigma_k',0}$. 
	Clearly, $E_s:=G\cdot \widehat E_s$ is a disk-type $G$-hypersurface. 
	Moreover, as an important observation, we mention that $ \whcC_{\rho_k,\sigma_k',0}/G_{x_0}$ and $ \widehat E_c/G_{x_0}$ are also disks in the orbit space. 
	Therefore, by \eqref{Eq: first regularity - 5.11}, we have 
	\[ \mc H^n( P_{k}^i )\leq \mc H^n(E_c)+\epsilon_{k,i}\leq \mc H^n(\mc C^G_{\rho_k,\sigma_k',0})+\epsilon_{k,i},\]
	where $E_c:=G\cdot \widehat E_c$. 
	Similar to Claim \ref{Claim: discard trivial curve}, we can apply the filigree lemma to further discard $\{P_k^i\}_{i=1}^{R_k}$ in \eqref{Eq: first regularity - 5.13}, which contradicts $x_0\in \spt(\|W\|)$. 
	This generalizes Theorem \ref{Thm: first regularity}. 

	\medskip
	{\bf Step 3.} {\it Generalize the main regularity theorem.}
	
	For the regularity theory (Theorem \ref{Thm: interior regularity k0=n-2}) near $G\cdot x_0$ with $\codim(G\cdot x_0)=3$, we first notice that the stationary rectifiable tangent cone $\whC\subset N_{x_0}(G\cdot x_0)$ is smooth away from the $1$-dimensional subspace $P=T_{x_0}K_0\cap N_{x_0}(G\cdot x_0)$ by the regularity results in the previous subsection. 
    Additionally, $\widehat C$ is $G_{x_0}$-stable by \eqref{Eq: interior regularity 1 - almost minimizing} and Claim \ref{Claim: tangent bundle}. 
    Noting that the cone $\spt(\|\widehat C\|)/G_{x_0}$ is $2$-sided in $(N_{x_0}(G\cdot x_0)\setminus P)/G_{x_0}$, we know $\spt(\|\widehat C\|)\setminus P$ admits a $G_{x_0}$-invariant unit normal, and is stable in $N_{x_0}(G\cdot x_0)\setminus P $ (Lemma \ref{Lem: stability and G-stability}). 
    If $\spt(\|\widehat C\|)\cap P=\{0\}$, then $\widehat C$ is a plane that is orthogonal to $P$, and the proof is finished. 
    If $\spt(\|\widehat C\|)\cap P\neq \{0\}$, then one can follow the proof of Theorem \ref{Thm: interior regularity k0=n-2} to show that $\spt(\|\widehat C\|)$ is a plane containing $P$, which contradicts the orthogonal intersection property in Theorem \ref{Thm: first regularity} {\bf Case II} with modifications in the previous step. 
    We point out some modifications to apply the arguments in Theorem \ref{Thm: interior regularity k0=n-2}. 
	If $B_{\rho_0}(G\cdot x_0)$ is in the case of Proposition \ref{Prop: local structure of M/G}(2.a), then the proof of Theorem \ref{Thm: interior regularity k0=n-2} would carry over. 
	If $B_{\rho_0}(G\cdot x_0)$ is in the case of Proposition \ref{Prop: local structure of M/G}(2.b), then each $\bd_1 P_k^{i*}=(\bd P_k^i)^*$ in the proof of Theorem \ref{Thm: interior regularity k0=n-2} {\bf Case III} is always a closed curve in $\bd_1 U^*\cong \mb S^2$ due to the locally $G$-boundary-type assumption. 
	Similar to Theorem \ref{Thm: interior regularity k0=n-2} {\bf Case III}, we have an unrelabelled subsequence $\{P_k^i\}_{k\in \mb N}$ for each $i$ so that for every $k\in\mb N$, $\bd_1 P_k^{i*}$ is a closed curve separating $\closure(\bd_1 U^*)\cong \mb S^2$ into two disks $E_\pm^*$. 
	We mention that if we lift $E_\pm^*$ to the slice $N_{x_0}(G\cdot x_0)$, one of  the lifts may be a cylinder (see {\bf Step 2}), although both of $E_\pm^*$ are disks in the orbit space. 
	Hence, one can easily modify the proof of Claim \ref{Claim: from classical cone to plane} and show that $\whC_i \res G_2(\mb B_{\rho-\eta}(0))=|D_{j_i}\cap\mb B_{\rho-\eta}(0)|$ remains valid. 
	The rest of Theorem \ref{Thm: interior regularity k0=n-2} would then carry over with these modifications. 

	Finally, for the regularity theory (Theorem \ref{Thm: interior regularity k_0<n-2}) near $G\cdot x_0$ with $\codim(G\cdot x_0)>3$, we can show as before that the stationary rectifiable tangent cone $\whC$ in $N_{x_0}(G\cdot x_0)$ is smooth and stable away from the $1$-dimensional subspace $P=P(x_0)$ (cf. \eqref{Eq: P}). 
	If $\spt(\|\whC\|)\cap P=\{0\}$, then $\whC$ is a hyperplane, and the proof is finished. 
	If $v\neq 0\in \spt(\|\whC\|)\cap P$, then there is a contradiction. 
    Specifically, consider the tangent varifold $\whD$ of $\whC$ at $v$, which is also a stationary cone that is smooth and stable away from its spine $P$ (see the proof of Theorem \ref{Thm: interior regularity k_0<n-2}). 
	In particular, $\spt(\|\whD\|)\cap P^\perp$ is a hyperplane in $P^\perp$,  and thus $\whC$ is smooth away from $\{0\}$ by the modified Theorem \ref{Thm: first regularity}.
	Since $P\cap \widehat C$ has codimension at least $2$ in $\whC$, the stability of $\whC$ extends to $N_{x_0}(G\cdot x_0)\setminus \{0\}$, which implies $\whC$ is a hyperplane in $N_{x_0}(G\cdot x_0)$. 
    Since $v\neq 0 \in \spt(\|\whC\|)\cap P$, we see from the proof of Lemma \ref{Lem: hypersurface position} that $P\subset \widehat C$. 
    However, by the proof of Theorem \ref{Thm: first regularity} {\bf Case II}, one can get a contradiction using modifications in {\bf Step 2}. 
\end{proof}

\subsection{Proof of Theorem \ref{Thm: plateau problem}}\label{Subsec: plateau}

\begin{proof}[Proof of Theorem \ref{Thm: plateau problem}]
    The proof is divided into three steps according to the orbit type of $G\cdot p$. 
    
    {\bf Step 1.} 
    The regularity near $p_0\in  M^{prin}$ follows directly from \cite{almgren79plateau} and \eqref{Eq: area in orbit space}. 
    Combining Theorem \ref{Thm: regularity part1} and \ref{Thm: regularity part2}, we obtain the regularity result and the orthogonal intersection in Theorem \ref{Thm: plateau problem} near any $p_0\in \spt(\|V\|)\setminus (M^{prin }\cup \mc S\cup\Gamma)$. 

    {\bf Step 2.} Consider $p_0\in \mc S\setminus (\mc S_{n.m.}\cup\Gamma)$. 
    
    Since $G\cdot p_0$ is an isolated non-principal orbit, there exists $\rho>0$ so that $B_\rho(G\cdot p_0)\cap \mc S = G\cdot p_0$, which implies that $\spt(\|V\|)$ is an embedded minimal $G$-hypersurface $\Sigma$ in $B_\rho(G\cdot p_0)\setminus G\cdot p_0$. 
    By Sard's theorem, we can set $0<s<t<\rho$ so that $\Sigma$ is transversal to $\bd A_{s,t}(G\cdot p_0)$. 
    Recall that in Proposition \ref{Prop: local structure of Sigma/G 2}(4), a surface $\wti\Sigma^*$ may have boundary points in the interior of $M^*$. 
    However, by the orthogonal intersection property in Theorem \ref{Thm: regularity part1} and \ref{Thm: regularity part2}, none of the cases in Proposition \ref{Prop: local structure of Sigma/G 1}(ii) or \ref{Prop: local structure of Sigma/G 2}(2)(4) can occur on $\Sigma$, and thus $\Sigma^*\llcorner (A_{s,t}(G\cdot p_0))^*$ is a surface with boundary on $ \bd (A_{s,t}(G\cdot p_0))^*$. 
    It follows from the $C^0$ $3$-manifold structure of $M^*$ at $[p_0]$ that $(A_{s,t}(G\cdot p_0))^*$ is simply connected, and every component $\Sigma_i^*$ of $\Sigma^*$ separates $(A_{s,t}(G\cdot p_0))^*$ into two components, which implies $ A_{s,t}(G\cdot p_0)\setminus \Sigma_i$ has two $G$-components. 
    Therefore, $\Sigma$ admits a $G$-invariant unit normal. 
    Since $\Sigma$ is $G$-stable by the area minimizing construction, we conclude from Lemma \ref{Lem: stability and G-stability} that $\Sigma$ is stable in $A_{s,t}(G\cdot p_0)$. 
    After taking $s\to 0$, we see $V$ is stable in $B_t(G\cdot p_0)\setminus G\cdot p_0$. 
    As $3\leq \codim(G\cdot p_0)\leq 7$ in $M$, the regularity of $V$ at $G\cdot p_0$ follows from \cite{schoen1981regularity}.

    {\bf Step 3.} Regularity at $p_0\in \mc S_{n.m.}\setminus\Gamma$ and the proof of Remark \ref{Rem: local G-boundary for plateau}. 

    The first statement in Remark \ref{Rem: local G-boundary for plateau} follows from the orthogonal intersection property in {\bf Step 1} and the arguments in {\bf Step 2} since $B_{\inj(G\cdot p)}^*([p])$ is simply connected for $p\in M\setminus\mc S_{n.m.}$. 

    Recall that $\Sigma_k\subset B_{r_0}(G\cdot a)$ and $A=B_{r}(G\cdot a)$ with $0<r<r_0$ satisfying Lemma \ref{Lem: uniform constants}. 
    Hence, $\spt(\|V\|)\subset \closure(B_{r_0}(G\cdot a))$. 
    Noting $(\closure(B_{r_0}(G\cdot a))\setminus G\cdot a )\cap \mc S=\emptyset$, we have $\spt(\|V\|)\subset \closure(A)$ by the regularity of $V$ in {\bf Step 1} and the maximum principle. 
    Additionally, using  \eqref{Eq: area in orbit space}, the boundary regularity of $V$ at $\Gamma\cap M^{prin}$ follows from Almgren-Simon \cite{almgren79plateau}*{Theorem 5}. 
    
    If $G\cdot a \not\subset \mc S_{n.m.}$, then $\spt(\|V\|)$ is smoothly embedded in $ A$ by {\bf Step 1, 2}. 
    If $G\cdot a \subset \mc S_{n.m.}$, then 
    %
    $\Sigma:=\spt(\|V\|)\setminus (\Gamma\cup G\cdot a)$ is an embedded $G$-stable minimal $G$-hypersurface. 
    We now extend the regularity of $\Sigma$ to $G\cdot a$. 
    
    By Proposition \ref{Prop: local structure of M/G}(3), $B_{r}^*([a])$ is a {\em closed} cone over $RP^2$. 
    Thus, whenever $G\cdot x\subset A\setminus(G\cdot a)$ is a non-principal orbit, $G\cdot x$ must satisfy Proposition \ref{Prop: local structure of M/G}(2.b). 
    Recall that $\Gamma\subset \bd A$ is an embedded $G$-connected $G$-submanifold that separates $\bd A$ into two $G$-components. 
    Thus, we have \[\Gamma\subset \bd A\cap M^{prin}.\] 
    Indeed, the $G$-separating property of $\Gamma\subset \bd A$ indicates that the $G$-hypersurface $C(\Gamma):=\{\exp_{G\cdot a}^\perp(tv): t\in (1/2,2), \exp_{G\cdot a}^\perp(v)\in\Gamma\}$ separates $A_{r/2,2r}(G\cdot a)$ into two $G$-components. 
    If $x=\exp_{G\cdot a}^\perp(v)\in \Gamma\cap(\bd A\setminus M^{prin})$, then $\exp_{G\cdot a}^\perp(tv)\in C(\Gamma)\cap M_{(G_x)}$ for all $t\in (1/2,2)$. 
    Hence, $C(\Gamma)$ satisfies Proposition \ref{Prop: local structure of Sigma/G 2}(4) at $G\cdot x$, which implies $(C(\Gamma))^*$ has a boundary point $[x]$ in $\interior(A_{r/2,2r}^*([a]))$ and thus cannot separate $A_{r/2,2r}^*([a])$, yielding a contradiction. 

    Now, we know the boundary regularity of $V$ is valid at $\Gamma\subset \bd A\cap M^{prin}$, i.e. $\bd \Sigma\llcorner \bd A=\Gamma$. 
    Take any small $s\in (0,r)$ so that $\bd B_s(G\cdot a)$ is transversal to $\Sigma$. 
    Then $\Sigma_s:=\Sigma\setminus B_s(G\cdot a)$ is an embedded $G$-hypersurface with smooth boundary. 
    In particular, for every $G$-component $\Sigma_s^i$ of $\Sigma_s$, either $\bd\Sigma_s^i\llcorner\bd A=\Gamma$, or $\bd\Sigma_s^i\cap \bd A=\emptyset$ and $\bd \Sigma_s^i\subset \bd B_s(G\cdot a)$. 
    In the latter case, we can shrink $r$ to some $t\in (s,r]$ so that $\bd B_{t}(G\cdot a)$ touches $\interior(\Sigma^i_s)$ for the first time, which contradicts the maximum principle. 
    Hence, $\Gamma=\bd\Sigma_s^i\llcorner \bd A$ is non-empty.
    Consider the mod $2$ $n$-current limit $T$ of the minimizing sequence $\Sigma_k$. 
    Since the current limit preserves the (non-trivial) boundary, we know $\Sigma_s^i$ is contained in $\spt(T)\subset \spt(\|V\|)$. 
    Additionally, the locally $G$-boundary-type condition on $\Sigma_k$ implies that $T$ is also induced by the boundary of a $G$-invariant Caccioppoli set, and thus $\Sigma_s^i$ admits a $G$-invariant unit normal. 
    By Lemma \ref{Lem: stability and G-stability}, $\Sigma_s$ is stable in $A\setminus B_s(G\cdot a)$. 
    Since $s>0$ can be arbitrarily small, the regularity of $\Sigma$ extends to $G\cdot a$ by \cite{schoen1981regularity} and \eqref{Eq: cohomogeneity assumption}. 
\end{proof}
\section{$G$-isotopy minimizing problem}\label{Sec: G-isotopy minimizing}

In this section, we show the existence and regularity of the $G$-isotopy area minimizing $G$-hypersurface, which generalizes \cite{meeks82exotic}*{Theorem 1}. 

\begin{definition}\label{Def: G-isotopy minimizing}
	Let $U\subset M$ be an open $G$-set, and $\Sigma\in  \mc {LB}^G$ be an embedded closed $G$-hypersurface (not necessarily $G$-connected) in $M$ intersecting $\bd U$ transversally. 
	For a sequence $\{\{\varphi^k_t\}_{t\in [0,1]} \}_{k\in\mb N}\subset \mk {Is}^G(U)$ of $G$-isotopies in $U$, we say $\Sigma_k:=\varphi^k_1(\Sigma)$ is a {\em minimizing sequence} for the minimization problem $(\Sigma, \mk {Is}^G(U))$ if 
	\[ \lim_{k\to\infty} \mc H^n(\varphi^k_1(\Sigma)) = \inf_{\{\phi_t\}_{t\in [0,1]}\in \mk {Is}^G(U)} \mc H^n(\phi_1(\Sigma)).  \]
\end{definition}

The main result of this section is the following theorem:
\begin{theorem}\label{Thm: G-isotopy minimizer}
    Let $U\subset M$ be an open $G$-set. 
	Suppose \eqref{Eq: cohomogeneity assumption} is satisfied, $M$ admits no special exceptional orbit, and $\{\Sigma_k\}_{k\in \mb N}$ is a minimizing sequence for the minimization problem $(\Sigma, \mk {Is}^G(U))$ (Definition \ref{Def: G-isotopy minimizing}). 
	Then, up to a subsequence, $\Sigma_k$ converges to a $G$-varifold $V\in \mc V^G_n(M)$ so that 
	\begin{align}\label{Eq: MSY - isotopy regularity - regularity}
		V\llcorner U=\sum_{j=1}^R m_j |\Sigma^{(j)}|, 
	\end{align}
	where $R, m_1,\dots,m_R\in\mb N$, and $\Sigma^{(j)}\subset U$ are pairwise disjoint smoothly embedded $G$-connected minimal $G$-hypersurfaces so that $\cup_{j=1}^R\Sigma^{(j)}$ is $G$-stable in $U$. 
	
	Additionally, if $\{\Sigma_k\}_{k\in \mb N}$ is a minimizing sequence for the minimization problem $(\Sigma, \mk {Is}^G(M))$ so that the initial $\Sigma^*$ and every $\Sigma_k^*$ are orientable, 
	then for all $k$ sufficiently large,
	\begin{align}\label{Eq: MSY - isotopy regularity - genus}
        \sum_{j\in\mc O} m_j \mk g(\Sigma^{(j)})  + \sum_{j\in\mc U} \frac{1}{2}m_j(\mk g(\Sigma^{(j)})-1)\leq \genus(\Sigma_k^*),
	\end{align}
	where $\mc O$ (resp. $\mc U$) is the set of $j\in \{1,\dots, R\}$ so that $(\Sigma^{(j)})^*$ is orientable (resp. non-orientable), and $\mk g(\Sigma^{(j)}):=\genus((\Sigma^{(j)}\cap M^{prin})^*)$. 
	Moreover, $m_j$ is even whenever $j\in \mc U$. 
\end{theorem}

\begin{remark}
	If $\Sigma$ and every $\Sigma_k$ in the minimizing sequence all admit $G$-invariant unit normals, then $m_j$ is even whenever $\Sigma^{(j)}$ admits no $G$-invariant unit normal.  
\end{remark}

\begin{remark}\label{Rem: isotopy minimizing - locally boundary type}
    In the above theorem,  $\cup_{j=1}^R\Sigma^{(j)}$ can only meet $N$ orthogonally, where $N\subset M\setminus (M^{prin}\cup\mc S)$ is any $G$-connected component of a non-principal orbit type stratum. 
\end{remark}

We will first show the regularity result in $M\setminus \mc S$, where $\mc S$ is the union of isolated non-principal orbits (Definition \ref{Def: isolated orbits}). 
Then, using \cite{meeks82exotic}*{Remark 3.27}, we can show the $G$-isotopy minimizer $V$ is a stable minimal $G$-hypersurface near any $G\cdot p\subset \mc S$. 
Hence, the singularity at the center orbit $G\cdot p$ can be removed by the regularity theory for stable minimal hypersurfaces \cite{schoen1981regularity}. 

\begin{proof}[Proof of Theorem \ref{Thm: G-isotopy minimizer}]
	Let $\{\Sigma_k\}_{k\in \mb N}$ be the minimizing sequence for the minimization problem $(\Sigma, \mk {Is}^G(U))$ so that $V=\lim_{k\to\infty}|\Sigma_k|$. 
	By a similar construction in \cite{meeks82exotic}*{Remark 3.14 and P.634 (3.22)-(3.26)}, we can perform $G$-equivariant surgeries on $\Sigma_k$ (Definition \ref{Def: gamma reduction}, \ref{Def: gamma reduction strong}) without changing notations so that for all $k\in\mb N$, $(\Sigma_k)_0=\emptyset$, $V=\lim_{k\to\infty}|\Sigma_k|$, and $\Sigma_k$ is strongly $(G,\gamma)$-irreducible in $U$ for some $0<\gamma<(\delta\rho_0)^n/9$, where $(\Sigma_k)_0$ and $\delta,\rho_0\in (0,1)$ are given in \eqref{Eq: Sigma0} and Lemma \ref{Lem: isoperimetric lemma in MSY} respectively. 
	Additionally, we also generalize \cite{meeks82exotic}*{(3.26)} by
	\begin{align}\label{Eq: G-isotopy - minimizing}
		\mc H^n(\Sigma_k) \leq \inf_{\Sigma'\in J_U^G(\Sigma_k)} \mc H^n(\Sigma') + \epsilon_k,
	\end{align}
	where $J_U^G(\Sigma_k)$ is defined as in \eqref{Eq: area error}, and $\epsilon_k\to 0_+$ as $k\to\infty$. 
	By virtue of \eqref{Eq: G-isotopy - minimizing}, one easily checks that $V$ is $G$-stationary and $G$-stable in $U$, which implies $V$ is indeed stationary in $U$ by Lemma \ref{Lem: first variation and G-variation}. 
	We claim the following equivariantly reduced monotonicity formula for $V$. 
	\begin{claim}[Equivariantly reduced monotonicity formula]\label{Claim: reduced monotonicity formula}
		For a $G$-varifold $V\in \mc V^G_n(M)$ that is stationary in an open $G$-set $U\subset M$ and $x\in U$, we have
		\begin{align}\label{Eq: monotonicity in slice}
		\frac{\|V\|(B_\sigma(G\cdot x))}{\sigma^{n-\dim(G\cdot x)}} \leq C_{mono} \frac{\|V\|(B_\rho(G\cdot x))}{\rho^{n-\dim(G\cdot x)}}
	\end{align}
	for any $0<2\sigma\leq \rho\leq \rho_0$, where $\rho_0=\rho_0(M,U,G\cdot x)>0$ and $C_{mono}=C_{mono}(M,\dim(G\cdot x))>1$. 
	\end{claim}
	\begin{proof}
		Firstly, the standard monotonicity formula can be applied to $V$, i.e. for any $x\in U$
		\[ \frac{\|V\|(B_\sigma(x))}{\sigma^n} \leq c_{mono} \frac{\|V\|(B_\rho(x))}{\rho^n}\] 
		for any $0<\sigma\leq\rho\leq\rho_0$, where $\rho_0=\rho_0(M,U,x)>0$ and $c_{mono}=c_{mono}(M)>1$. 
		Recall that $T(x,s,t)$ is the part of the tube $B_t(G\cdot x)$ centered at the (geodesic) $s$-neighborhood $B^{G\cdot x}_s(x)$ of $x$ in $G\cdot x$ (as in \eqref{Eq: part of tube}). 
	By shrinking $\rho_0>0$ (depending on $G\cdot x$), we have $B_r(x)\subset T(x,r,r) \subset B_{2r}(x)$ for all $0<r<\rho_0$. 
	Combined with the monotonicity formula, we see
	\[ \frac{\|V\|(T(x,\sigma,\sigma))}{\sigma^n} \leq \frac{\|V\|(B_{2\sigma}(x))}{\sigma^n} \leq 2^nc_{mono} \frac{\|V\|(B_\rho(x))}{\rho^n} \leq  2^nc_{mono} \frac{\|V\|(T(x,\rho,\rho))}{\rho^n}\]
	for any $0<2\sigma\leq \rho \leq \rho_0$. 
	By the $G$-invariance of $V$, we conclude that for any $s,t>0$, 
	\begin{align}\label{Eq: area on part of tube}
		\|V\|(T(x,s,t) ) = \frac{\mc H^{\dim(G\cdot x)}(B^{G\cdot x}_s(x)) }{\mc H^{\dim(G\cdot x)}(G\cdot x) } \|V\|(B_t(G\cdot x)).
	\end{align}
	Note that we can shrink $\rho_0$ even smaller so that $\frac{\mc H^{\dim(G\cdot x)}(B^{G\cdot x}_r(x)) }{\mc H^{\dim(G\cdot x)}(\mb B^{\dim(G\cdot x)}_r(0)) } \in [1/2,2]$ for all $r\in (0,\rho_0)$. 
	Combining these estimates, we obtain the desired inequality \eqref{Eq: monotonicity in slice}. 
	\end{proof}

	\medskip
	{\bf Step 1.} {\it Regularity in $U\setminus\mc S$.}
	
	For any $x_0\in\spt(\|V\|)\cap (U\setminus\mc S)$, take $\rho_0>0$ small enough to satisfy Lemma \ref{Lem: uniform constants}, Lemma \ref{Lem: isoperimetric lemma in MSY}, \eqref{Eq: monotonicity in slice}, and $B_{\rho_0}(G\cdot x_0)\subset\subset U$. 
	Denote by $k_0:=\dim(G\cdot x_0)$ for simplicity. 
	It follows from the co-area formula that 
	\[ \int_{\rho-\sigma}^\rho \mc H^{n-1}(\Sigma_k\cap \bd B_s(G\cdot x_0)) ds \leq \mc H^n(\Sigma_k\cap (\closure(B_\rho(G\cdot x_0)) \setminus B_{\rho-\sigma}(G\cdot x_0)) ) \]
	for almost all $\rho\in (0,\rho_0)$ and $\sigma\in (0,\rho)$. 
	Setting $\sigma=\rho/2$, the reduced monotonicity formula \eqref{Eq: monotonicity in slice} then implies that for all $k$ large enough,
	\[\int_{\rho/2}^\rho \mc H^{n-1}(\Sigma_k\cap \bd B_s(G\cdot x_0)) ds \leq c\rho^{n-k_0}, \]
	where $c>1$ depends only on $M, k_0$, and $\|V\|(B_{\rho_0}(G\cdot x_0))/\rho_0^{n-k_0}$. 
	Therefore, there is a sequence $\{\rho_k\}_{k\in\mb N}\subset (3\rho/4, \rho)$ so that $\Sigma_k$ is transversal to $\bd B_{\rho_k}(G\cdot x_0)$, and 
	\begin{align}\label{Eq: MSY - isotopy regularity - 5.3}
		\mc H^{n-1}(\Sigma_k\cap \bd B_{\rho_k}(G\cdot x_0)) \leq c\rho^{n-k_0-1} \leq c (\eta\rho_0)^{n-k_0-1}
	\end{align}
	for all $k$ large enough, provided $\rho<\eta\rho_0$, where for the moment $\eta\in (0,1)$ is arbitrary. 
	By taking $\eta>0$ small enough (depending on $\gamma,\rho_0,k_0$ and $\|V\|(M)$), we know Theorem \ref{Thm: MSY thm2 gamma reduction} is applicable for $\Sigma_k$ with $B_{\rho_k}(G\cdot x_0), B_{\rho_0}(G\cdot x_0)$ in place of $B,U$ respectively. 
	Indeed, by Remark \ref{Rem: compare disk instead cylinder}, we know \eqref{Eq: MSY thm2 assumption - isoperimetric} and \eqref{Eq: MSY thm2 assumption - disk instead of cylinder} can be simultaneously satisfied on $\bd B_{\rho_k}(G\cdot x_0)$, and thus Theorem \ref{Thm: MSY thm2 gamma reduction} (iii) and (iv) are satisfied by \eqref{Eq: MSY thm2 assumption - isoperimetric} and \eqref{Eq: MSY - isotopy regularity - 5.3} with $\eta$ small enough. 
	Therefore, for $k$ large enough, there are $G$-connected $G$-hypersurfaces $D_k^{(1)},\dots, D_k^{(q_k)}\in \mc M^G_0$ in $\Sigma_k$ so that $\bd D_k^{(j)}\subset \bd B_{\rho_k}(G\cdot x_0)$, 
	\begin{align}\label{Eq: MSY - isotopy regularity - 5.4 - 1}
		\Sigma_k\cap B_{\rho_k}(G\cdot x_0) = \left( \bigcup_{j=1}^{q_k} D_k^{(j)} \right) \cap B_{\rho_k}(G\cdot x_0),
	\end{align}
	and 
	\begin{align}\label{Eq: MSY - isotopy regularity - 5.4 - 2}
		\sum_{j=1}^{q_k}\mc H^n(  D_k^{(j)} ) \leq c\rho^{n-k_0 } \leq c(\eta\rho_0)^{n-k_0 } ,
	\end{align}
	where $c$ is independent of $k,\eta,\rho$. 
	Additionally, for any $\alpha>0$, there are $G$-isotopies $\{\varphi^{(k)}\}\subset \mk {Is}^G(B_{\rho_0}(G\cdot x_0))$ so that 
	\begin{align}\label{Eq: MSY - isotopy regularity - 5.4 - 3}
		\mc H^n\left( \left( \bigcup_{j=1}^{q_k} \left( \varphi_1^{(k)}(D_k^{(j)})\setminus \bd D_k^{(j)} \right) \right)  \setminus B_{\rho_k}(G\cdot x_0) \right) <\alpha. 
	\end{align}
	Since $\mc H^n(D_k^{(j)}) \leq c(\eta\rho_0)^{n-k_0 } $, we can shrink $\eta>0$ small enough and apply the replacement lemma (Theorem \ref{Thm: replacement}), which holds for all $x_0\in U\setminus \mc S$ by the modifications in Section \ref{Subsec: plateau regularity 2} and \cite{meeks82exotic}. 
	Then, we obtain $G$-connected $G$-hypersurfaces $\wti D_k^{(j)}\in \mc M^G_0$ with 
	\begin{align}\label{Eq: MSY - isotopy regularity - 5.5}
		\bd \wti D_k^{(j)} = \bd D_k^{(j)}, \qquad \wti D_k^{(j)}\setminus \bd \wti D_k^{(j)} \subset B_{\rho_k}(G\cdot x_0),
	\end{align}
	and 
	\begin{align}\label{Eq: MSY - isotopy regularity - 5.6}
		\mc H^n(\wti D_k^{(j)}) \leq \mc H^n( D_k^{(j)}). 
	\end{align}
	Combining \eqref{Eq: G-isotopy - minimizing} and \eqref{Eq: MSY - isotopy regularity - 5.6}, and using Lemma \ref{Lem: switch lemma}, \eqref{Eq: MSY - isotopy regularity - 5.4 - 1}, \eqref{Eq: MSY - isotopy regularity - 5.4 - 3}, and \eqref{Eq: MSY - isotopy regularity - 5.5}, we conclude that $D_k^{(j)}$ is a minimizing sequence among all $G$-hypersurfaces $\Gamma\in \mc M^G_0$ with $\bd\Gamma=\bd D_k^{(j)}$ (without any restriction on $\mk b(\Gamma)$ defined in \eqref{Eq: boundary components number}). 
    This is similar to \cite{meeks82exotic}*{(5.7)} and \cite{li2015general}*{P. 317}. 
    Note that the non-existence of special exceptional orbits is used to ensure that for any two $G$-hypersurfaces $\Gamma_1,\Gamma_2\in \mc M^G_0$ in $B_\rho(G\cdot x_0)$ with $\bd\Gamma_1=\bd\Gamma_2$, $\Gamma_2$ can be approximated by $\Gamma_1$ via $G$-isotopies (see the proof of Lemma \ref{Lem: G-isotopy approximation} and Remark \ref{Rem: no special and approximation}).

	By the modifications in Section \ref{Subsec: plateau regularity 2} and \cite{meeks82exotic}, we know the generalized filigree lemma (Lemma \ref{Lem: filigree}) is valid near any $x_0\in U\setminus \mc S$. 
	Hence, by \eqref{Eq: MSY - isotopy regularity - 5.4 - 1}-\eqref{Eq: MSY - isotopy regularity - 5.4 - 3} and the area-minimizing property of $D_k^{(j)}$, we can apply the filigree lemma (Lemma \ref{Lem: filigree}) for $k$ large enough to find a finite subset $\mc I_k\subset\{1,\dots, q_k\}$ of $l$ elements with $l$ independent of $k$ so that $\mc I_k=\{1,\dots, l\}$ by relabeling, and 
	\begin{align}\label{Eq: MSY - isotopy regularity - 5.9}
		\lim_{k\to\infty} \sum_{j=1}^l \left| D_k^{(j)} \right| \llcorner B_{\rho_0/2}(G\cdot x_0) = V\llcorner B_{\rho_0/2}(G\cdot x_0),
	\end{align}
	and $\lim_{k\to\infty} \sum_{j=l+1}^{q_k} | D_k^{(j)} | \llcorner B_{\rho_0/2}(G\cdot x_0) = 0$. 
	
	Now, up to a subsequence, we can apply the regularity result in Theorem \ref{Thm: plateau problem} to each $\{D^{(j)}_k\}_{k\in\mb N}$, $1\leq j\leq l$, in $B_{\rho_0/2}(G\cdot x_0)$. 
	Combined with the maximum principle and the arbitrariness of $x_0\in U\setminus \mc S$, we conclude the regularity of $V$ in $U\setminus\mc S$. 
	Moreover, since $\|V\|<\infty$, the monotonicity formula of $V$ indicates that the number of $G$-connected components of $\spt(\|V\|)$ is finite. 
	Therefore, \eqref{Eq: MSY - isotopy regularity - regularity} and Remark \ref{Rem: isotopy minimizing - locally boundary type} are valid in $U\setminus\mc S$.

    \medskip
 	{\bf Step 2.} {\it Regularity at $\mc S\setminus \mc S_{n.m.}$.}

    Take any $G\cdot x_0\subset \mc S\setminus \mc S_{n.m.}$ and $\rho>0$ small enough. 
    Since the regularity of $V$ in $ U\setminus \mc S$ follows from the regularity result of the equivariant Plateau problem (Theorem \ref{Thm: plateau problem}), we know $\spt(\|V\|)\cap (B_{\rho}(G\cdot x_0)\setminus G\cdot x_0)$ is orthogonal to $M\setminus M^{prin}$ (Remark \ref{Rem: local G-boundary for plateau}). 
    Hence, one can combine the simple connectivity of $B_\rho^*([x_0])\setminus \{[x_0]\}$ and the $G$-stability of $V$ to show that $V$ is stable in $B_\rho(G\cdot x_0)\setminus G\cdot x_0$ by the proof in Section \ref{Subsec: plateau} {\bf Step 2}. 
    Then the regularity of $\spt(\|V\|)$ extends to $G\cdot x_0$ by \eqref{Eq: cohomogeneity assumption} and the regularity theory for stable minimal hypersurfaces \cite{schoen1981regularity}. 

    \medskip
    {\bf Step 3.} {\it Regularity at $\mc S_{n.m.}$}.

    For $ x_0\in \spt(\|V\|)\cap U\cap \mc S_{n.m.}$ and $\rho>0$ small enough, $\wti\Sigma:= \spt(\|V\|)\cap (B_\rho(G\cdot x_0)\setminus G\cdot x_0)$ is an embedded minimal $G$-hypersurface. 
    For simplicity, we assume that $\wti\Sigma$ is $G$-connected. 
    Then by Sard's theorem, we can choose $0<s<t<\rho$ so that $\bd A_{s,t}(G\cdot x_0)$ is transversal to $\wti \Sigma$. 
    Note that $\wti\Sigma$ is $G$-stable in $A_{s,t}(G\cdot x_0)$. 
    Hence, if $\wti\Sigma$ admits a $G$-invariant unit normal, then $\wti\Sigma$ is stable in $A_{s,t}(G\cdot x_0)$, and thus we can remove the singularity at $G\cdot x_0$ by \cite{schoen1981regularity} as in {\bf Step 2}. 

    If $\wti\Sigma$ does not admit a $G$-invariant unit normal. 
    Then, by equivariantly lifting to double covers if necessary \cite{wang2026multiplicity}*{Proposition 2.4} (cf. \cite{bredon1972introduction}*{Chapter 1, Theorem 9.3}), we can assume without loss of generality that $\wti\Sigma\llcorner A_{s,t}(G\cdot x_0)$ admits a unit normal $\nu$. 
 	Thus, there is a $2$-index Lie subgroup
 	\[G^+:=\{g\in G: dg(\nu)=\nu\},\]
 	so that every $g\in G^-:=G\setminus G^+$ satisfies $dg(\nu)=-\nu$. 
    Let $\wti U\subset\subset A_{s,t}(G\cdot x_0)\cap M^{prin}$ be any smooth open $G$-set whose boundary is transversal to $\wti\Sigma$ and every $\Sigma_k$ in the minimizing sequence. 
    
    We claim that $\wti\Sigma$ is stable in $\wti U$, and thus stable in $A_{s,t}(G\cdot x_0)\cap M^{prin}$. 
 	Indeed, since $\wti U\subset\subset M^{prin}$, $\{\Sigma_k^*\}$ is a minimizing sequence that converges to $m|\wti\Sigma^*|$ in $\wti U^*$ under the rescaled metric $\wti g_{_{M^*}}$ by \eqref{Eq: area in orbit space}, where $m\in\mb N$. 
    Then, one can apply the proof of \cite{meeks82exotic}*{Remark 3.27} in $(\wti U^*,\wti g_{_{M^*}})$ to show that for every $k$ large enough, the surgery in $U^*$ can equivariantly deform $\Sigma_k$ to $\wti\Sigma_k\in\mc {LB}^G$ so that $V=\lim_{k\to\infty}|\wti\Sigma_k|$, \eqref{Eq: G-isotopy - minimizing} holds with $\wti\Sigma_k$ in place of $\Sigma_k$, and
    \[\wti\Sigma_k\llcorner \wti U = \varphi_1(S_k\cup S_k^{0})\cap \wti U\] 
    for some $\{\varphi_t\} \in \mk {Is}^G(\wti U)$, where $S_k$ is a $G$-hypersurface and $S_k^{0}$ is a $G$-set with $S_k^{0}\cap S_k=\emptyset$ so that for any compact $G$-set $K\subset\subset \wti U$, $\lim_{k\to\infty}\mc H^n(S_k^{0}\cap K)=0$, and 
    \begin{align}\label{Eq: MSY remark3.27 - 3.29}
 	  S_k = \left\{
 		\begin{array}{ll}
 			\bigcup_{r=1}^{l} \bd  B_{r/k}(\wti\Sigma) & \mbox{if $m=2l$ is even}
 			\\
 			\wti\Sigma\cup \left(\bigcup_{r=1}^{l} \bd  B_{r/k}(\wti\Sigma) \right) & \mbox{if $m=2l+1$ is odd}
 		\end{array}
 		\right. 
 		\qquad \mbox{in $K$}. 
 	\end{align} 
    Recall that $\nu$ is not $G$-invariant; conversely, $\wti\Sigma_k\llcorner (B_{\rho}(G\cdot x_0)\setminus G\cdot x_0)$ does have a $G$-invariant unit normal $\nu_k$ due to the locally $G$-boundary-type condition. 
 	Hence, we conclude from \eqref{Eq: MSY remark3.27 - 3.29} that the multiplicity $m=2l$ is even. 
 	Additionally, in a neighborhood of $\wti\Sigma$, let $d$ be the signed geodesic distance to $\wti\Sigma$ so that $\nabla d= \nu$ on $\wti\Sigma\llcorner \wti U$. 
 	Then by \eqref{Eq: MSY remark3.27 - 3.29}, for $k$ large enough, we can separate $S_k\cap \wti U$ into two parts $S_k^{+}\sqcup S_k^{-}$ by 
 	\[S_k^{\pm}:= \{x\in S_k\cap \wti U : \pm d(x)>0 \}\] 
 	so that $S_k^{\pm}$ are $G^+$-invariant and $G^-\cdot S_k^{\pm}=S_k^{\mp}$. 
    In particular, this implies that $\wti\Sigma_k\llcorner \wti U$ also has two disjoint $G^+$-components $\wti\Sigma_k^\pm$ with $G^-$ permuting them. 
    We next show the $G^+$-stability (and thus the stability) of $\wti\Sigma$ in $\wti U$ by considering the convergence $|\wti\Sigma_k^\pm|\to l|\wti\Sigma|$.

	Consider now the first eigenfunction $u_1$ of the Jacobi operator of $\wti\Sigma\llcorner \wti U$. 
	By Lemma \ref{Lem: stability and G-stability}, $u_1>0$ is a smooth $G^+$-invariant function on $\wti\Sigma\llcorner \wti U$, and $u_1\nu$ can be extended to a $G^+$-invariant vector field $X_1$ in $\wti U$ with $\langle X_1,\nabla d\rangle\geq 0$. 
	Since $\wti\Sigma_k^+\cap \wti\Sigma_k^-=\emptyset$, one can construct a $G$-invariant vector field $Y_k\in \mk X^{G}(\wti U)$ so that $Y_k=X_1$ (resp. $Y_k=-X_1$) on $\wti\Sigma_k^+$ (resp. $\wti\Sigma_k^-$). 
 	Noting that $\mc H^n(\wti\Sigma_k\cap \wti U)=2\mc H^n(\wti\Sigma_k^\pm)$ and $\delta^2\wti\Sigma_k(Y_k)=2\delta^2\wti\Sigma_k^\pm(X_1)$, it follows from \eqref{Eq: G-isotopy - minimizing} with $\wti\Sigma_k$ in place of $\Sigma_k$ that 
    $\mc H^n(\wti\Sigma_k^+) \leq \mc H^n(\varphi_t(\wti\Sigma_k^+)) +  \epsilon_k/2$, where $\{\varphi_t\}_{t\in [0, \epsilon)}$ is the flow generated by $X_1$. 
    Hence, taking $k\to\infty$, we get $\|V\|(M)\leq \|(\varphi_t)_\#V\|(M)$, and thus, $\wti\Sigma\llcorner \wti U$ is stable. 

    By the arbitrariness of $\wti U$, we conclude that $\wti\Sigma$ is stable in $M^{prin}\cap A_{s,t}(G\cdot x_0)$. 
    Moreover, we can further show that $\wti\Sigma$ is stable in $A_{s,t}(G\cdot x_0)$. 
    Indeed, since $x_0\in \mc S_{n.m.}$, we see from Proposition \ref{Prop: local structure of M/G}(3) that $B_\rho^*([x_0])$ is a closed cone over $\mb {RP}^2$. 
    Hence, there is no boundary point of the topological manifold $M^*\setminus \mc S_{n.m.}^*$ lying in $A_{s,t}^*([x_0])$.
    Namely, if $x\in A_{s,t}(G\cdot x_0)\setminus M^{prin}$, then $x$ satisfies Proposition \ref{Prop: local structure of M/G}(2.b). 
    As in {\bf Step 2}, we can then conclude from the orthogonal intersection property (the last statement in Theorem \ref{Thm: plateau problem}) that every $x\in \wti\Sigma\cap (A_{s,t}(G\cdot x_0)\setminus M^{prin})$ satisfies Proposition \ref{Prop: local structure of Sigma/G 2}(3), which implies $\wti\Sigma \cap (A_{s,t}(G\cdot x_0)\setminus M^{prin})$ is the union of finitely many orbits. 
    By the codimension assumption \eqref{Eq: cohomogeneity assumption}, the standard logarithmic cut-off trick extends the stability of $\wti\Sigma$ in $M^{prin}\cap A_{s,t}(G\cdot x_0)$ to the stability in $A_{s,t}(G\cdot x_0)$. 

    Finally, since $s>0$ can be arbitrarily small, we know $\wti\Sigma$ is stable in $B_t(G\cdot x_0)\setminus G\cdot x_0$, and the regularity of $\wti\Sigma$ extends to $G\cdot x_0$ by \cite{schoen1981regularity} and \eqref{Eq: cohomogeneity assumption}. 
    
 	\medskip
 	{\bf Step 4.} {\it The weighted genus upper bound.}

    Let $\{\Sigma_k\}_{k\in \mb N}$ be a minimizing sequence for the minimization problem $(\Sigma, \mk {Is}^G(M))$ so that $V=\lim_{k\to\infty}|\Sigma_k|$. 
    Let $\Sigma^{(j)}$ and $m_j$ be given as in \eqref{Eq: MSY - isotopy regularity - regularity} for $1\leq j\leq R$. 
    

    Take any smooth open $G$-set $U\subset\subset M^{prin}$ so that $\bd U$ is transversal to $\spt(\|V\|)$ and $\{\Sigma_k\}_{k\in\mb N}$. 
    It is sufficient to show the weighted genus upper bound \eqref{Eq: MSY - isotopy regularity - genus} with $\genus((\Sigma^{(j)}\cap U)^*)$ in place of $\mk g(\Sigma^{(j)})$. 
    By Lemma \ref{Lem: minimal in M^prin} and \eqref{Eq: area in orbit space}, $\cup_{j=1}^R(\Sigma^{(j)}\cap U)^*$ is an isotopy minimizer in $U^*$ under the rescaled metric $\wti g_{_{M^*}}$ (cf.  \eqref{Eq: weighted metric in M/G}). 
    Hence, one can follow the proof of \cite{meeks1982plateau}*{(1.4), Remark 3.27} in $(U^*, \wti g_{_{M^*}})$ to obtain the desired weighted genus upper bound.  
\end{proof}

Before we end this section, we also introduce the following definition and corollary for the stability of locally $G$-isotopy area minimizers, which will be used in Section \ref{Sec: G-min-max}. 

\begin{definition}\label{Def: simple G-set}
	We say an open $G$-set $U\subset M$ is {\em simply $G$-connected} if $U^*$ is simply connected. 
\end{definition}

\begin{remark}\label{Rem: simply G-connected}
    For any $G\cdot p\subset M\setminus \mc S_{n.m.}$ and $0<s<t<\inj(G\cdot p)$, we see from \eqref{Eq: cohomogeneity assumption} and the $C^0$ manifold structure of $M^*\setminus\mc S_{n.m.}^*$ that $B_t(G\cdot p)$ and $A_{s,t}(G\cdot p)$ are simply $G$-connected.
\end{remark}

\begin{corollary}\label{Cor: stability of minimizers}
	Let $U\subset M$ be an open $G$-set containing no special exceptional orbit. 
	Suppose $\Sigma\subset U$ with $\bd\Sigma\cap U=\emptyset$ is a $G$-stable minimal $G$-hypersurface that is locally contained in the solution of a $G$-isotopy area minimizing problem. 
	Namely, for any $p\in U$, there exist $B_r(G\cdot p)\subset\subset U$ and $\wti \Sigma\in \mc {LB}^G$ so that $\Sigma\llcorner B_r(G\cdot p)$ is contained in the support of the minimizer $V\in \mc V^G_n(M)$ given by Theorem \ref{Thm: G-isotopy minimizer} with respect to $(\wti \Sigma, \mk {Is}^G(B_r(G\cdot p)))$. 
    
	Then $\Sigma$ is stable in any simply $G$-connected open $G$-set $\wti U\subset\subset U$. 
\end{corollary}
\begin{proof}
	By slightly shrinking $\wti U$, we assume $\Sigma$ is transversal to $\bd \wti U$. 
	Then $\wti U^*\subset M^*$ is a simply connected (relative) open set. 
    By Remark \ref{Rem: isotopy minimizing - locally boundary type}, we know $\Sigma\llcorner (\wti U\setminus\mc S)$ can only meet $\wti U\setminus M^{prin}$ orthogonally. 
	Combined with Proposition \ref{Prop: local structure of M/G}(3), we know the topological surface $\Sigma^*$ has no boundary point in $\interior(\wti U^*)$.
	Hence, every component $\Sigma^*_i$ of $\Sigma^*\cap \wti U^*$ separates the simply connected $ \wti U^*$ into two components, which implies $\wti U\setminus\Sigma_i$ has two $G$-components. 
	Thus, $\Sigma \cap \wti U$ has a $G$-invariant unit normal and is stable by Lemma \ref{Lem: stability and G-stability}.
\end{proof}

\section{Equivariant min-max theory}\label{Sec: G-min-max}

In this section, we establish an equivariant min-max theory for minimal $G$-hypersurfaces in the generalized setting of Simon-Smith (cf. \cite{colding2002minmax}).

To begin with, we introduce the generalized family of $G$-hypersurfaces with mild singularities. 
Let $I(1,k)$ be the cubical complex on $I=[0,1]$ with $0$-cells $\{[0],[3^{-k}],\cdots,[1-3^{-k}],[1]\}$ and $1$-cells $\{[0,3^{-k}],[3^{-k}, 2\cdot 3^{-k}],\cdots, [1-3^{-k},1] \}$. 
Then for $m,k\in\mb N$, $I(m,k)=I(1,k)\otimes \cdots\otimes I(1,k)$ ($m$-times) is the $m$-dimensional cubical complex with $p$-cells ($0\leq p\leq m$) written as $\sigma_1\otimes\cdots\otimes\sigma_m$, where each $\sigma_i$ is a cell of $I(1,k)$ with $\sum_{i=1}^m\dim(\sigma_i)=p$ (see \cite{marques2017existence}*{\S 2.1}). 
In what follows, we always assume that the parameter space $X$ is a cubical subcomplex of some $I(m,k)$. 
\begin{definition}\label{Def: G-sweepout}
	Let $X$ be a cubical complex, and $Z\subset X$ be a subcomplex. 
	A family $\{\Sigma_x\}_{x\in X}$ of compact $G$-hypersurfaces in $M$ is said to be a {\em generalized $(X,Z)$-family of $G$-hypersurfaces} if there exist a finite set $T\subset X$ and a finite union of orbits $P\subset M$ so that 
	\begin{enumerate}
		\item for $t\notin T\cup Z$, $\Sigma_t\in \mc {LB}^G$ is a smooth embedded closed $G$-hypersurface;
		\item for $t\in T\setminus Z$, $\Sigma_t\setminus P\in \mc {LB}^G$ is a smooth embedded $G$-hypersurface;
		\item $x\mapsto \Sigma_x$ is continuous in the varifold (weak) topology and the Hausdorff topology.
	\end{enumerate}
	If, in addition to (1)-(3) above, 
	\begin{itemize}
		\item[(4)] $x\mapsto \Sigma_x$ is continuous in the smooth topology for $x\in X\setminus (T\cup Z)$;
		\item[(5)] for $x\in T\setminus Z$, $\Sigma_y\to\Sigma_x$ smoothly in $M\setminus P$ as $y\to x$;
	\end{itemize}
	then we say the generalized $(X,Z)$-family is {\em smooth}. 
	For simplicity, we also denote by 
	\[ \mf \Sigma:=\{\Sigma_x\}_{x\in X}, \qquad \mf \Sigma|_Z:=\{\Sigma_x\}_{x\in Z}, \]
	and a smooth generalized $(X,Z)$-family $\mf \Sigma$ is also called a {\em ($G$-equivariant) $(X,Z)$-sweepout}.
\end{definition}

\begin{remark}\label{Rem: relax finiteness of T}
    The above definition is generalized from \cite{colding2002minmax}*{Definition 1.1, 1.2}. 
    Note that the finiteness condition on $T$ can be relaxed by assuming that every $\Sigma_t$, $t\in X\setminus Z$, has a (possibly empty) singular set $P_t=\cup_{i=1}^{l_t}G\cdot p_i^t$ so that $\sup_{t\in X\setminus Z} l_t<\infty$. 
    We can also define $\genus(\Sigma_t^*):=\lim_{r_i\to 0} \genus((\Sigma_t\setminus B_{r_i}(P_t))^*)\in \mb N\cup\{\infty\}$ for some $r_i\to 0$ with $\bd B_{r_i}(P_t)$ transversal to $\Sigma_t$. 
\end{remark}

Given a $G$-equivariant $(X,Z)$-sweepout $\{\Sigma_x\}_{x\in X}$, we can construct new $G$-equivariant $(X,Z)$-sweepouts by the following procedure. 
Let $\psi:X\times M\to M$ be a smooth map so that 
\begin{itemize}
	\item for any $x\in X$, $\psi(x,\cdot)=\varphi_1^{(x)}$ for some $\{\varphi_t^{(x)}\}_{t\in [0,1]}\in\mk {Is}^G(M)$;
	\item for any $x\in Z$, $\psi(x,\cdot )= id_M$.  
\end{itemize}
Then the new family $\{\Sigma_x'\}_{x\in X}$ given by $\Sigma_x':=\psi(x,\Sigma_x)$ is also a $G$-equivariant $(X,Z)$-sweepout so that $\Sigma_x'=\Sigma_x$ for all $x\in Z$. 
In particular, the map $x\mapsto \Sigma_x'$ is homotopic to $ x\mapsto \Sigma_x$ relative to $x\in Z\mapsto \Sigma_x$. 
Moreover, the singular set $P'$ of $\{\Sigma_x'\}$ in Definition \ref{Def: G-sweepout} has a uniform upper bound on $\# (P'/G)$ independent of $\psi$, which ensures a technical assumption in \cite{colding2002minmax}*{Remark 1.3}. 

\begin{definition}\label{Def: G-sweepout homotopy class}
	Given a $G$-equivariant $(X,Z)$-sweepout $\mf \Sigma=\{\Sigma_x\}_{x\in X}$, the set $\Pi$ containing all the $G$-equivariant $(X,Z)$-sweepouts obtained by the above constructions is called the ($G$-equivariant) {\em $(X,Z)$-homotopy class of $\mf \Sigma$}. 
	We also say that $\Pi$ is relative to $\mf \Sigma|_Z$. 
\end{definition}

Note that an $(X,Z)$-homotopy class $\Pi$ is {\em saturated} in the sense that it is closed under the above relative $G$-equivariant deformations of sweepouts. 

We now set up the relative equivariant min-max theory, which extends \cite{colding2002minmax}. 
\begin{definition}
	Given a $G$-equivariant $(X,Z)$-homotopy class $\Pi$, define the {\em width} of $\Pi$ by
	\[\mf L(\Pi) := \inf_{\mf \Sigma \in \Pi} \sup_{x\in X} \mc H^n(\Sigma_x). \]
	A sequence $\{\mf \Sigma^i:=\{\Sigma_x^i\}_{x\in X}\}_{i\in\mb N}\subset \Pi$ is called a {\em minimizing sequence} if 
	\[ \mf L(\mf \Sigma^i) := \sup_{x\in X} \mc H^n(\Sigma_x^i) ~\to~  \mf L(\Pi)\qquad \mbox{as $i\to\infty$}.  \]
	Given a minimizing sequence $\{\mf \Sigma^i\}_{i\in\mb N}\subset \Pi$, a subsequence $\{\Sigma_{x_j}^{i_j}\}_{j\in\mb N}$ with $x_j\in X$ is called a {\em min-max sequence} if 
	\[ \mc H^n(\Sigma_{x_j}^{i_j}) ~\to~ \mf L(\Pi) \qquad \mbox{as $j\to\infty$} .\]
	The critical set of a minimizing sequence $\{\mf \Sigma^i\}_{i\in\mb N}\subset \Pi$ is defined by 
	\[
	\mathbf{C}(\left\{\mf \Sigma^i\right\}_{i\in \mb N})=\left\{
	\begin{array}{l|l} 
		V \in \mc V^G_n(M) & 
		\begin{array}{l}
			\exists \text { a min-max subsequence }\{\Sigma_{x_j}^{i_j}\}_{j\in\mb N} \\ \text {such that } \mf F(|\Sigma_{x_j}^{i_j}|, V) \rightarrow 0 \text { as } j \rightarrow \infty
		\end{array}
	\end{array}\right\} .
	\]
\end{definition}

Note that the existence of a minimizing sequence $\{\mf \Sigma^i\}_{i\in\mb N}\subset \Pi$ follows easily from a diagonal argument.
The main results of this section are the following theorems. 
\begin{theorem}[Equivariant Min-max Theorem]\label{Thm: G-min-max}
	Let $\Pi$ be a $G$-equivariant $(X,Z)$-homotopy class relative to $\mf\Sigma^0|_Z$.  
	Suppose \eqref{Eq: cohomogeneity assumption} is satisfied, $M$ has no special exceptional orbit, and
	\begin{align}\label{Eq: G-min-max - nontrivial width}
		\mf L(\Pi) > \max\left\{ \sup_{x\in Z}\mc H^n(\Sigma^0_x), 0 \right\}.
	\end{align}
	Then there exist a minimizing sequence $\{\mf \Sigma^i\}_{i\in\mb N}\subset \Pi$ and a sequence $\{x_i\}_{i\in \mb N}\subset X\setminus Z$ so that $\{\Sigma^i_{x_i}\}_{i\in\mb N}$ is a min-max sequence converging in the sense of varifolds to $V\in \mc V^G_n(M)$ satisfying
	\begin{itemize}
		\item[(i)] $\|V\|(M)= \mf L(\Pi)$;
		\item[(ii)] there are disjoint, smoothly embedded, closed, $G$-connected, minimal $G$-hypersurfaces $\{\Gamma_j\}_{j=1}^R$ in $M$ 
        so that 
		\[V=\sum_{j=1}^R m_j|\Gamma_j|,\]
	\end{itemize}
	where $R,m_1,\dots,m_R\in\mb N$. 
\end{theorem}

The proof of our equivariant min-max theory contains three parts: a pull-tight procedure, the existence of almost minimizing varifolds, and the regularity theory. 
Moreover, we also have the reduced topological control for the min-max $G$-hypersurface in Theorem \ref{Thm: G-min-max}. 

\begin{theorem}\label{Thm: G-min-max - topological control}
	Let $\Pi, \{\mf \Sigma^i\}_{i\in\mb N}, \{\Sigma^i_{x_i}\}_{i\in\mb N}$ and $V=\sum_{j=1}^R m_j\Gamma_j$ be given as in Theorem \ref{Thm: G-min-max}. 
	Suppose further that every $(\Sigma^i_x)^*$, ($i\in\mb N, x\in X\setminus Z$), is orientable. 
	Then, 
	\begin{align}\label{Eq: G-min-max - genus}
        \sum_{j\in\mc O}  m_j\mk g(\Gamma_j) + \sum_{j\in\mc U} \frac{1}{2}m_j(\mk g(\Gamma_j)-1)   \leq \liminf_{i\to\infty}\liminf_{x\to x_i}\genus((\Sigma^i_x)^*),
	\end{align}
	where $\mc O$ (resp. $\mc U$) is the set of $j\in \{1,\dots, R\}$ so that $\Gamma_j^*$ is orientable (resp. non-orientable), and $\mk g(\Gamma_j):= \genus((\Gamma_j\cap M^{prin})^*)$. 
\end{theorem}

Since the boundary of $M^*$ may have a degenerate metric, we do not have any control over the number of boundary components of $\Gamma_j^*$.  
We next prove Theorem \ref{Thm: G-min-max} and \ref{Thm: G-min-max - topological control}.

\subsection{Pull-tight procedure}

In this section, we will find a nice minimizing sequence in $\Pi$ such that every $G$-varifold in its critical set is stationary in $M$ by a pull-tight procedure. 
The construction is similar to \cite{colding2002minmax}*{\S 4} (see also \cite{wang2023free}*{Proposition 4.15} for an equivariant version). 

\begin{proposition}[Pull-tight]\label{Prop: pull-tight}
	Let $\Pi$ be a $G$-equivariant $(X,Z)$-homotopy class relative to $\mf \Sigma^0|_Z$. 
	Given a minimizing sequence $\{\mf {\wti \Sigma}^i\}\subset \Pi$,  there exists another minimizing sequence $\{\mf \Sigma^i\}_{i\in\mb N}\subset \Pi$ so that $\mf C(\{\mf \Sigma^i\}) \subset \mf C(\{\mf {\wti\Sigma}^i\})$, and every $V\in \mf C(\{\mf \Sigma^i\})$ is either stationary in $M$ or belongs to $B:=\{|\Sigma_x^0|: x\in Z\}$. 
\end{proposition}
\begin{proof}
	Set $L:= \mf L(\Pi)+1$, and define the following compact subsets in $\mc V^G_n(M)$:
    \[A^L:=\{V\in\mc V^G_n(M): \|V\|(M)\leq L\}, \quad \mbox{and}\quad A^L_0:=\{V\in A^L: V \mbox{ is stationary in $M$}\}.\]
	Consider the concentric annuli around $A_0:=A^L_0\cup B$ for $j\in\{2,3,\dots\}$,
	\[ A_1:=\{V\in A^L: \mf F(V,A_0)\geq \frac{1}{2}\},\quad A_j:=\{V\in A^L: \frac{1}{2^j}\leq \mf F(V,A_0)\leq \frac{1}{2^{j-1}} \}.  \]
	By Lemma \ref{Lem: first variation and G-variation}, we only need to consider $Y\in \mk X^G(M)$ in the first variation of $V\in \mc V^G_n(M)$. 
	Therefore, using the above modifications, we can follow the constructions in \cite{colding2002minmax}*{Proposition 4.1} to get a continuous map $\mc X: A^L\to \mk X^G(M)$ in the $C^1$-topology, and continuous functions $T: A^L\to [0,\infty)$, $\mc L: (0,\infty )\to (0,\infty)$ with $\lim_{t\to 0}\mc L(t)=0$, so that  
	\begin{itemize}
		\item $\mc X(V)=0$ and $T(V)=0$ if $V\in A_0=A^L_0\cup B$;
		\item $\delta V(\mc X(V))<0$ and $T(V)>0$ if $V\in A^L\setminus A_0$;
		\item for any $V\in A^L\setminus A_0$, $\| (f_{T(V)}^{\mc X(V)})_\# V  \|(M) \leq  \| V\|(M) - \mc L(\mf F(V,A_0)) $, 
	\end{itemize}
	where $\{f_{t}^{\mc X(V)}\}\in \mk {Is}^G(M)$ is the flow generated by $\mc X(V)$. 
	
	Without loss of generality, we assume $\mf L(\{\mf {\wti\Sigma}^i\}) \leq L$ for all $i\in\mb N$. 
	Then each $\mf {\wti\Sigma}^i$ is associated with a family of $G$-invariant vector fields: 
	\[ \mc X_i:X\to \mk X^G(M) , \qquad \mc X_i(x):= \mc X(|\wti\Sigma^i_x|) ,\]
	which is continuous in the $C^1$-topology, and $\mc X_i|_Z=0$. 
	Hence, $\widehat\Sigma^i_x:=f_{T(|\wti\Sigma^i_x|)}^{\mc X(|\wti\Sigma^i_x|)}(\wti\Sigma^i_x)$ 
	satisfies $\widehat\Sigma^i_x=\wti\Sigma^i_x$ for all $x\in Z$, and $\mc H^n(\widehat\Sigma^i_x ) - \mc H^n (\wti\Sigma^i_x ) \leq -\mc L(\mf F(|\wti\Sigma^i_x|, A_0)) $ for all $x\in X$. 
	However, since $\mc X_i$ is only continuous in the $C^1$-topology, we cannot directly say $\{\widehat\Sigma^i_x \}_{x\in X}\in \Pi$. 
	
	Nevertheless, for each $i\in\mb N$, we can smooth out $\mc X_i$ equivariantly to some $\mc Y_i: X\to \mk X^G(M)$ that is continuous in the smooth topology so that $\mc Y_i|_Z=0$ and $\|\mc X_i-\mc Y_i\|_{C^1}\leq 1/i$. 
	Note that 
	\[ \left | \delta V(\mc X) - \delta V(\mc Y)   \right| \leq C\|V\|(M)\cdot \|\mc X-\mc Y\|_{C^1} \]
	for some constant $C>0$ independent of $V\in A^L$. 
	Hence, $\psi^i_x(t,\cdot ):=\{ f^{\mc Y_i(x)}_{T(|\wti\Sigma^i_x|)t} \}_{t\in [0,1]} \in \mk {Is}^G(M)$ is a smooth $X$-family of $G$-isotopies with $\psi^i_x(t,\cdot)\equiv id$ for all $t\in [0,1]$ and $x\in Z$. 
	Let $\mf \Sigma^i=\{ \Sigma^i_x:= \psi^i_x(1,\wti\Sigma^i_x ) : x\in X\}$ for each $i\in\mb N$. 
	Then $\mf \Sigma^i\in \Pi$ and 
	\[ \mc H^n(\Sigma^i_x ) -  \mc H^n( \wti \Sigma^i_x) \leq -\mc L(\mf F(|\wti \Sigma^i_x|,A_0)) + \frac{C'}{i},  \]
	for some universal constant $C'>0$. 
	The desired result then follows easily. 
\end{proof}

\subsection{Existence of $G$-almost minimizing varifolds}

In this subsection, we introduce the $G$-almost minimizing varifolds and show the existence result. 
To begin with, the following definition is adapted from \cite{colding2002minmax}*{Definition 3.2}. 

\begin{definition}\label{Def: G-almost minimizing}
	Given $\epsilon,\delta>0$, an open $G$-set $U\subset M$, and a $G$-varifold $V\in\mc V^G_n(M)$, we say $V$ is {\em $(G,\epsilon,\delta)$-almost minimizing in $U$} if there is no $G$-isotopy $\{\varphi_t\}_{t\in[0,1]}\in\mk {Is}^G(U)$ satisfying
	\begin{itemize}
		\item[(i)] $\|(\varphi_t)_\#V\|(M) \leq \|V\|(M) +\delta$ for all $t\in [0,1]$, and 
		\item[(ii)] $\|(\varphi_1)_\#V\|(M) \leq \|V\|(M) -\epsilon$. 
	\end{itemize}
	Additionally, given a sequence of closed $G$-hypersurfaces $\{\Sigma_j\}_{j\in\mb N}\subset \mc {LB}^G$, we say $V\in \mc V^G_n(M)$ is {\em $G$-almost minimizing in $U$} (w.r.t. $\{\Sigma_j\}_{j\in\mb N}$) if there exist $\epsilon_j\to 0$ and $\delta_j\to 0$ so that $\lim_{j\to\infty} \mf F(|\Sigma_j|,V)=0$ and $|\Sigma_j|$ is $(G,\epsilon_j,\delta_j)$-almost minimizing in $U$. 
\end{definition}

We mention that $\Sigma_j$ is assumed to be locally $G$-boundary-type in the above definition, which is required in the regularity theory. 

\begin{lemma}\label{Lem: amv implies stationary}
	If $V\in\mc V^G_n(M)$ is $G$-almost minimizing in $U$, then $V$ is stationary and $G$-stable in $U$.
\end{lemma}
\begin{proof}
	Using $G$-isotopies, the proof in \cite{li2015general}*{Proposition 3.8} and \cite{wangzc2023four-sphere}*{Lemma 3.3} can be easily adapted to our $G$-equivariant setting. 
	Then by Lemma \ref{Lem: first variation and G-variation}, $V$ is also stationary in $U$. 
\end{proof}

We also adapt the following definition from \cite{colding2018heegaard}*{Appendix}. 
\begin{definition}\label{Def: admissible G-annuli}
	Given $L\in \mb N$ and $p\in M$, a collection of $G$-annuli centered at $G\cdot p$
	\[ \mc A:=\{A_i:=A_{r_i,s_i}(G\cdot p):  1\leq i \leq L, 0<r_i<s_i\} \]
	is said to be {\em $L$-admissible} if $2s_{i}< r_{i+1}$ for all $i=1,\dots, L-1$. 
	
	Additionally, $V\in \mc V^G_n(M)$ is said to be {\em $G$-almost minimizing in $\mc A$ (w.r.t. $\{\Sigma_j\}\subset \mc {LB}^G$)} if there exist $\epsilon_j\to 0$ and $\delta_j\to 0$ so that 
	\begin{itemize}
		\item $\lim_{j\to\infty} \mf F(V,|\Sigma_j|)=0$;
		\item for each $j\in\mb N$, $|\Sigma_j|$ is $(G,\epsilon_j,\delta_j)$-almost minimizing in at least one $A_{i_j}\in \mc A$. 
	\end{itemize}
\end{definition}

\begin{lemma}\label{Lem: existence of amv in L-admissible annuli}
	Let $\{\mf \Sigma^i\}_{i\in\mb N}=\{\{\Sigma^i_x\}_{x\in X}\}_{i\in\mb N}$ be the pull-tight minimizing sequence in Proposition \ref{Prop: pull-tight}, and suppose $\Pi$ satisfies \eqref{Eq: G-min-max - nontrivial width}. 
	Then there exist an integer $L=L(m)>0$ (depending on the dimension of $I(m,k)$ into which the parameter cubical complex $X$ is embedded), and a min-max sequence $\{\Sigma^{i_j}_{x_j}\}_{j\in\mb N}$ converging to a stationary $G$-varifold $V\in \mf C(\{\mf \Sigma^i\})$ so that $V$ is $G$-almost minimizing in every $L$-admissible family of $G$-annuli w.r.t. $\{\Sigma^{i_j}_{x_j}\}$.
\end{lemma}
\begin{proof}
	The proof is essentially the same as in \cite{colding2018heegaard}*{Lemma A.1}, and we only introduce some necessary details. 
	Let $L:=(3^m)^{3^m}$. 
	We only need to show that for any $\epsilon>0$, there exist $\delta>0$, $i>1$, and $x\in X$ with $\mc H^n(\Sigma^i_x)\geq \mf L(\Pi)-\epsilon$, so that for any $L$-admissible set of $G$-annuli $\mc A$, $|\Sigma^i_x|$ is $(G,\epsilon,\delta)$-almost minimizing in at least one $G$-annulus in $\mc A$. 
	Suppose it fails. Then we have $\epsilon_0>0$ so that for any $\delta>0$, $i>1$, and $x\in Y_i:=\{x\in X: \mc H^n(\Sigma^i_x)\geq \mf L(\Pi)-\epsilon_0\}$, there exists an $L$-admissible set of $G$-annuli $\mc A_{i,x}$ so that $\Sigma^i_x$ is not $(G,\epsilon_0,\delta)$-almost minimizing in any $A^l_{i,x}\in\mc A_{i,x}$, $1\leq l\leq L$. 
	In the following, we fix $\delta>0$ small enough, and fix an $i>1$. 
	
	By Definition \ref{Def: G-almost minimizing} and the continuity of $x\mapsto |\Sigma^i_x|$, we can associate any $x\in Y_i$ with an open set $U_{i,x}\subset X$ and a $G$-isotopy $\{\psi^l_{i,x}(t,\cdot)\}_{t\in [0,1]}\in \mk {Is}^G(A^l_{i,x})$ for each $A^l_{i,x}\in\mc A_{i,x}$ so that 
	\begin{itemize}
		\item[(1)] $\mc H^n(\psi^l_{i,x}(t, \Sigma^i_y)) \leq \mc H^n(\Sigma^i_y) + 2\delta$ for all $t\in [0,1]$ and $y\in U_{i,x}$;
		\item[(2)] $\mc H^n(\psi^l_{i,x}(1, \Sigma^i_y)) \leq \mc H^n(\Sigma^i_y) - \epsilon_0/2$ for all $y\in U_{i,x}$. 
	\end{itemize}
	Thus, $\{U_{i,x}\}_{x\in Y_i}$ is an open cover of the compact set $Y_i$. 
	By adapting the proof in \cite{colding2018heegaard}*{Lemma A.1} (\cite{pitts2014existence}*{Chapter 4}) to $G$-annuli, we obtain a finite open cover $\{\mc U_i^j\}$ of $Y_i$ so that 
	\begin{itemize}
		\item each $\mc U_i^j$ is contained in some $U_{i,x}$ and is associated with a $G$-annulus $A_i^j:=A_{i,x}^{l_j}\in \mc A_{i,x}$, $1\leq l_j\leq L$; hence, (1) and (2) are satisfied for $l:=l_j$, $\psi_i^j:=\psi_{i,x}^{l_j}$, and $y\in \mc U_i^j$;
		\item each $\mc U_i^j$ can intersect at most $d(m)\in\mb N$ other elements in $\{\mc U_i^j\}$, and $A_i^{j_1}\cap A_i^{j_2}=\emptyset$ whenever $\mc U_i^{j_1}\cap\mc U_i^{j_2}\neq \emptyset$. 
	\end{itemize}
	Next, we can take smooth functions $\phi_i^j:X\to [0,1]$ with $\spt(\phi_i^j)\subset \mc U_i^j$ for each $j$ so that for any $x\in Y_i$, there exists $\mc U_i^j$ with $\phi_i^j(x)=1$. 
	Now, for any $x\in X$, we define a $G$-equivariant diffeomorphism $\Phi_{i,x}$ by 
    $\Phi_{i,x}(p):=\psi_{i}^j(\phi_i^j(x), p)$ if $p\in A_i^j$ for some $j$ with $x\in \mc U_i^j$, otherwise $\Phi_{i,x}(p):=p$, which is well-defined. 
	By \eqref{Eq: G-min-max - nontrivial width}, we can further make $\mc U_i^j\cap Z=\emptyset$. 
	Therefore, we obtain $\{\wti\Sigma^i_x:=\Phi_{i,x}(\Sigma^i_x)\}_{x\in X}\in \Pi$. 
	Note that for any $x\in Y_i$, $\{\phi_i^j(x)\}_{j}$ contains at most $d(m)+1$ non-zero elements, and at least one of them is $1$. 
	Therefore, for $x\in Y_i$,
	\[ \mc H^n(\wti\Sigma^i_x) \leq \mc H^n(\Sigma^i_x) - \frac{\epsilon_0}{2} + d(m)\cdot 2\delta, \]
	and for $x\notin Y_i$, $\mc H^n(\wti\Sigma^i_x) \leq \mc H^n(\Sigma^i_x) + (d(m)+1)\cdot 2\delta< \mf L(\Pi) - \epsilon_0+(d(m)+1)\cdot 2\delta$. 
	For $\delta>0$ small enough and $i>1$ large enough, we obtain $\sup_{x\in X}\mc H^n(\wti\Sigma^i_x ) < \mf L(\Pi)$, a contradiction. 
\end{proof}

By \cite{colding2018heegaard}*{Lemma A.3} (\cite{wangzc2023four-sphere}*{Appendix D}), the above lemma gives us the following theorem. 

\begin{theorem}[Existence of $G$-almost minimizing $G$-varifolds]\label{Thm: existence amv in annuli}
	Suppose $\Pi$ is given in Theorem \ref{Thm: G-min-max} satisfying \eqref{Eq: G-min-max - nontrivial width}. 
	Let $\{\mf \Sigma^i\}_{i\in\mb N}\subset\Pi$ be the pull-tight minimizing sequence in Proposition \ref{Prop: pull-tight}. 
	Then there is a min-max sequence $\{\Sigma^{i_j}_{x_j}\}_{j\in\mb N}$ converging to a stationary $G$-varifold $V\in \mf C(\{\mf \Sigma^i\})$ so that $V$ is {\em $G$-almost minimizing in small $G$-annuli w.r.t. $\{\Sigma^{i_j}_{x_j}\}$} in the following sense: for any $p\in M$, there exists $r_{am}(G\cdot p)>0$ so that $V$ is $G$-almost minimizing w.r.t. $\{\Sigma^{i_j}_{x_j}\}$ in $A_{s,t}(G\cdot p)$ for all $0<s<t<r_{am}(G\cdot p)$.
\end{theorem}
\begin{proof}
	Using $G$-annuli $A_{s,t}(G\cdot p)$ in place of annuli centered at a point, the proof of \cite{colding2018heegaard}*{Lemma A.3} (\cite{wangzc2023four-sphere}*{Appendix D}) can be taken verbatim. 
	Then, the theorem follows from Lemma \ref{Lem: existence of amv in L-admissible annuli}. 
\end{proof}

\subsection{Regularity for equivariant min-max varifolds}\label{Subsec: regularity min-max}

In this subsection, we introduce the $G$-replacement of a stationary $G$-varifold and show the regularity of the min-max $G$-varifolds.

\subsubsection{$G$-replacement chain property}
We begin with the definition of $G$-replacement. 
\begin{definition}\label{Def: G-replacement}
	Let $V\in \mc V^G_n(M)$ be a stationary $G$-varifold and $U\subset M$ be an open $G$-set. 
	Then $V'\in\mc V^G_n(M)$ is said to be a {\em $G$-replacement of $V$ in $U$} if 
	\begin{enumerate}
		\item $V'$ is stationary in $U$;
		\item $V'=V$ outside $\closure(U)$, and $\|V'\|(M)=\|V\|(M)$;
		\item $V'\llcorner U$ is an integer multiple of a smooth embedded minimal $G$-hypersurface that is $G$-stable in $U$ and stable in any simply $G$-connected (Definition \ref{Def: simple G-set}) open $G$-set $U'\subset\subset U$. 
	\end{enumerate}
	We also say $V$ has {\em (weak) good $G$-replacement property in $U$} if for any $p\in U$, there is $r_{gr}=r_{gr}(G\cdot p)>0$ so that $V$ has a $G$-replacement in any $G$-annulus $A_{s,t}(G\cdot p)$ with $0<s<t<r_{gr}$. 
\end{definition}

\begin{remark}\label{Rem: G-replacement}
    The above definition mainly differs from \cite{wang2024index}*{Definition 4.14} in the simply $G$-connected assumption in (3). 
    In particular, by Remark \ref{Rem: simply G-connected}, the above definition coincides with \cite{wang2024index}*{Definition 4.14} if $U=A_{s,t}(G\cdot p)$ for $0<s<t<\inj(G\cdot p)$ and $p\in M\setminus \mc S_{n.m.}$. 
\end{remark}

We emphasize that the stability of the $G$-replacement $V'$ in $U'$ is necessary since we do not have curvature estimates or regularity theory (\cite{schoen1981regularity}) for a general $G$-stable minimal $G$-hypersurface. 
We also need the following stronger $G$-replacement property, which is analogous to \cite{wang2024index}*{Definition 1.4} and \cite{wangzc2023four-sphere}*{Definition 3.6}.

\begin{definition}[$G$-replacement chain property]\label{Def: G-replacement chain}
	Let $U\subset M$ be an open $G$-set, and $V\in\mc V^G_n(M)$ be stationary in $U$. 
	We say $V$ has {\em $G$-replacement chain property in $U$} if for any finite sequence of open $G$-sets $\{B_j\}_{j=1}^k\subset\subset U$, there exists a sequence $V=V_0,V_1,\dots,V_k\in \mc V^G_n(M)$ of $G$-varifolds so that $V_j$ is stationary in $U$, and is a $G$-replacement of $V_{j-1}$ in $B_j$ for $j=1,\dots, k$. 
    %
\end{definition}

One verifies that the $G$-replacement chain property (in $G$-annuli) implies the (weak) good $G$-replacement property. 
Next, we show the rectifiability of a stationary $G$-varifold with (weak) good $G$-replacement property, and the classification of tangent cones away from $\mc S_{n.m.}$. 
\begin{proposition}\label{Prop: rectifiability and tangent cone for amv}
	Let $V\in\mc V^G_n(M)$ be a stationary varifold, and $U\subset M$ be an open $G$-set. 
	Suppose $V$ has (weak) good $G$-replacement property in $U$. 
	Then $V$ is integer rectifiable in $U$, and for any $p\in U\setminus \mc S_{n.m.}$, the tangent varifold $C$ of $V$ at $p$ is an integer multiple of a $G_p$-invariant hyperplane $T_p(G\cdot p)\times \widehat H$ in $T_pM=T_p(G\cdot p)\times N_p(G\cdot p)$ for some hyperplane $\widehat H\subset N_p(G\cdot p)$. 
\end{proposition}
\begin{proof}
    The rectifiability of $V$ in $U$ follows from \cite{wang2023free}*{Lemma 5.6} (see also \cite{wang2022min}*{Lemma 11}). 
	If $p\in U\setminus\mc S_{n.m.}$, then the classification of the tangent cone $C$ is similar to \cite{colding2002minmax}*{Lemma 6.4}. 
	Indeed, by the simple $G$-connectedness of $A_{s,t}(G\cdot p)$ for $0<s<t<\inj(G\cdot p)$ (Remark \ref{Rem: simply G-connected}), any $G$-replacement $V'$ of $V$ in $A_{s,t}(G\cdot p)$ is stable in $A_{s,t}(G\cdot p)$ by Definition \ref{Def: G-replacement}(3).  Hence, one can directly apply the proof of \cite{wang2023free}*{Lemma 5.7, Proposition 5.10} (or \cite{wang2022min}*{Proposition 5}). 
\end{proof}

Using the above result, we further have the following regularity theorem for $G$-varifolds with $G$-replacement chain property. 
One can also directly refer to the regularity result \cite{wang2024index}*{Theorem 4.18}, where the $G$-replacement chain property is defined similarly in \cite{wang2024index}*{Definition 4.14}. 
\begin{theorem}\label{Thm: regularity for V of G-replacement property}
	Suppose \eqref{Eq: cohomogeneity assumption} is satisfied, and $V\in \mc V^G_n(M)$ is a stationary $G$-varifold in $M$ with $G$-replacement chain property in an open $G$-set $U\subset M$. 
	Then $V\llcorner (U\setminus \mc S_{n.m.})$ is induced by a smoothly embedded minimal $G$-hypersurface with integer multiplicity. 
\end{theorem}
\begin{proof}
	The above theorem follows from \cite{colding2002minmax}*{\S 6} with some equivariant modifications as in \cite{wang2022min}\cite{wang2023free}. 
    Indeed, by Remark \ref{Rem: simply G-connected}, for any $G\cdot p\subset U\setminus \mc S_{n.m.}$ and $0<s<t<\inj(G\cdot p)$, the $G$-annulus $A_{s,t}(G\cdot p)$ is simply $G$-connected.  
	Thus, by Remark \ref{Rem: G-replacement}, the proof of the regularity result \cite{wang2024index}*{Theorem 4.18} is also valid near any $G\cdot p\subset U\setminus \mc S_{n.m.}$. 
\end{proof}

\begin{remark}\label{Rem: improved G-replacement}
    In the proof of Proposition \ref{Prop: rectifiability and tangent cone for amv} and Theorem \ref{Thm: regularity for V of G-replacement property}, one only needs to use the $G$-replacement chain property of $V$ in small $G$-annuli $A_{s,t}(G\cdot p)\subset\subset U$, and use Definition \ref{Def: G-replacement chain} to take the $G$-replacement chain $\{V_i\}$ in concentric $G$-annuli $\{A_{s_i,t_i}(G\cdot p)\subset\subset A_{s,t}(G\cdot p)\}$. 
    Note also that the results in Proposition \ref{Prop: rectifiability and tangent cone for amv} and Theorem \ref{Thm: regularity for V of G-replacement property} are valid only near $p\in U\setminus\mc S_{n.m.}$ because the stability of our $G$-replacement is valid only in simply $G$-connected open $G$-sets (cf. Definition \ref{Def: G-replacement}(3)). 
    Nevertheless, if one can improve the stability of $G$-replacements in Definition \ref{Def: G-replacement}(3) by removing the simply $G$-connected condition, then the above results are also valid at $U\cap \mc S_{n.m.}$ by directly applying the proof of \cite{wang2024index}*{Theorem 4.18}. 
\end{remark}

\subsubsection{Existence of $G$-replacements}

To show the regularity of our $G$-equivariant min-max varifold $V$ in Theorem \ref{Thm: existence amv in annuli}, we still need to prove an improved $G$-replacement chain property of $V$ in small $G$-annuli $A_{s,t}(G\cdot p)$ (cf. Remark \ref{Rem: improved G-replacement}). 
For this purpose, we first show a regularity result for the constrained $G$-isotopy area minimizing problem.

Given any open $G$-set $U$, a closed $G$-hypersurface $\Sigma\in\mc {LB}^G$, and $\epsilon, \delta>0$, define 
	\[ \mk {Is}^G_\delta(U):=\left \{ \{\varphi_t\}_{t\in [0,1]}\in \mk {Is}^G(U): ~\mc H^n(\varphi_t(\Sigma))\leq \mc H^n(\Sigma)+\delta,~\forall t\in [0,1] \right\}.  \]
	Set 
	\[ m_\delta:=\inf_{\{\psi_t\}_{t\in[0,1]}\in\mk {Is}^G_\delta(U)} \mc H^n(\psi_1(\Sigma)) . \]
	For a sequence $\{\{\varphi^k_t\}_{t\in [0,1]}\}_{k\in\mb N}\subset \mk {Is}^G_\delta(U)$ of $G$-isotopies in $U$, we say $\Sigma_k:=\varphi^k_1(\Sigma)$ is a {\em minimizing sequence} for the constrained area minimization problem $(\Sigma, \mk {Is}^G_\delta(U))$ if 
	\[ \mc H^n(\Sigma)\geq\mc H^n(\Sigma_k)\to m_\delta \qquad\mbox{as $k\to\infty$}. \]
	Then, we have the following result generalizing \cite{colding2002minmax}*{Lemma 7.6}.
	\begin{lemma}\label{Lem: squeezing G-isotopy}
		Suppose $\{\Sigma_k\}_{k\in\mb N}$ is a minimizing sequence for the constrained area minimization problem $(\Sigma, \mk {Is}^G_\delta(U))$. 
		Then for any $p\in U'\subset\subset U$, there exists $\rho=\rho(\mc H^n(\Sigma),U',M,\delta,G\cdot p)>0$ with $B_{2\rho}(G\cdot p)\subset U'$ so that for any $k$ sufficiently large and $\{\varphi_t\}_{t\in [0,1]}\in \mk {Is}^G(B_\rho(G\cdot p))$ with $\mc H^n(\varphi_1(\Sigma_k))\leq \mc H^n(\Sigma_k)$, there exists $\{\psi_t\}_{t\in [0,1]}\in \mk {Is}^G(B_{\rho}(G\cdot p))$ satisfying $\varphi_1=\psi_1$, and 
		\[ \mc H^n(\psi_t(\Sigma_k)) \leq \mc H^n(\Sigma_k) + \delta,\qquad \forall t\in [0,1]. \]
	\end{lemma}
	\begin{proof}
		Note that the monotonicity formula for stationary $G$-varifolds also holds in $G$-tubes by Claim \ref{Claim: reduced monotonicity formula}. 
		Hence, using $A_{s,t}(G\cdot p), B_r(G\cdot p)$ and $\exp_{G\cdot p}^\perp$ in place of ${\rm An}(p,s,t), B_r(p)$ and $\exp_p$ respectively, the rescaling arguments in \cite{colding2002minmax}*{\S 7.4} can be taken almost verbatim. 
	\end{proof}
	
Using the above lemma, we show the regularity of the constrained area minimizer.
\begin{proposition}[Regularity of constrained minimizers]\label{Prop: constrained minimizer}
	Suppose \eqref{Eq: cohomogeneity assumption} is satisfied, and $M$ contains no special exceptional orbit. 
    Given a $G$-annulus $U:= A_{s,t}(G\cdot p)$ with $G\cdot p\subset M$, $0<s<t<\inj(G\cdot p)$, and $t>0$ sufficiently small, let $\Sigma\in\mc {LB}^G$ be a closed $G$-hypersurface that is $(G,\epsilon,\delta)$-almost minimizing in $U$ and transversal to $\bd U$. 
	Let $\{\Sigma_k\}_{k\in\mb N}\subset \mc {LB}^G$ be the minimizing sequence in the constrained area minimization problem $(\Sigma, \mk {Is}^G_\delta(U))$. 
	Then, $\Sigma_k$ converges (up to a subsequence) to some $V\in \mc V^G_n(M)$ so that 
    \begin{itemize}
        \item $V\llcorner U$ is induced by an integer multiple of a minimal $G$-hypersurface $\wti\Sigma$;
        \item $\wti\Sigma$ is not only $G$-stable, but also stable in $U$; and 
    \end{itemize}
	\begin{align}\label{Eq: constrained minimizer - area bound}
		\mc H^n(\Sigma)-\epsilon \leq \|V\|(M) \leq \mc H^n(\Sigma). 
	\end{align}
\end{proposition}
\begin{proof}
		By Definition \ref{Def: G-almost minimizing}, we directly obtain \eqref{Eq: constrained minimizer - area bound}. 
		Additionally, the minimizer $V$ is clearly stationary (by Lemma \ref{Lem: first variation and G-variation}) and $G$-stable in $U$. 

    \medskip
    {\bf Step 1.} We make the following claim for the regularity of constrained area minimizers. 
		
	\begin{claim}\label{Claim: constrained minimizer - replacement chain}
			For any $x\in U$, there exists $r>0$ so that $V$ has the $G$-replacement chain property in $B_r(G\cdot x)\subset \subset U$. 
	\end{claim}
	\begin{proof}[Proof of Claim \ref{Claim: constrained minimizer - replacement chain}]
			For any $x\in U$, take $0<r< \min \{\rho, \dist(G\cdot x,\bd U)/4\}$, where $\rho$ is given by Lemma \ref{Lem: squeezing G-isotopy} with respect to some $G$-neighborhood $U'$ of $G\cdot x$ in $U$. 
			Then for any open $G$-set $B\subset\subset B_{r}(G\cdot x)$, let $\{\Sigma_{k,l}\}_{l\in\mb N}\subset\mc {LB}^G$ be minimizing in the $G$-isotopy area minimization problem $(\Sigma_k, \mk {Is}^G(B))$. 
			By Theorem \ref{Thm: G-isotopy minimizer} and the constructions, as $l\to\infty$, $|\Sigma_{k,l}|$ converges (up to a subsequence) to a $G$-varifold $V_k'\in \mc V^G_n(M)$ with $V=V_k'$ outside $\closure(B)$ so that $V_k'\llcorner B$ is induced by a $G$-stable minimal $G$-hypersurface. 
			Up to a subsequence, there also exist $V'\in \mc V^G_n(M)$ and $\{l_k\}_{k\in\mb N}$ so that $V'=\lim_{k\to\infty} V_k'=\lim_{k\to\infty}|\Sigma_{k,l_k}|$. 
			By Corollary \ref{Cor: stability of minimizers} and the compactness theorem for stable minimal hypersurfaces (see \cite{schoen1981regularity}), we obtain the regularity of $V'$ and also the stability of $V'$ in any simply $G$-connected open $G$-set contained in $B$. 
			
			Using Lemma \ref{Lem: squeezing G-isotopy}, one easily verifies that $\{\Sigma_{k,l_k}\}_{k\in \mb N}$ is also a minimizing sequence for the constrained area minimization problem $(\Sigma, \mk {Is}^G_\delta(U))$.  
			Hence, $V'$ is also stationary in $U$ and $\|V\|(M)=\|V'\|(M)$. 
			Together, we see $V'$ is a $G$-replacement of $V$ in $B$. 
			Using $\{\Sigma_{k,l_k}\},V'$ in place of $\{\Sigma_k\},V$ respectively, we can repeat this procedure, and thus finish the proof.
	\end{proof}
    Since $U=A_{s,t}(G\cdot p)$ contains no isolated non-principal orbit (Definition \ref{Def: isolated orbits}), we see from the above claim and Theorem \ref{Thm: regularity for V of G-replacement property} that $V\llcorner U$ is induced by a minimal $G$-hypersurface $\wti\Sigma$. 

    \medskip
    {\bf Step 2.} We now show the stability of $V$ in simply $G$-connected $U'\subset\subset U$.

    Recall that $U= A_{s,t}(G\cdot p)$ with $0<s<t<\inj(G\cdot p)$. 
    Let $N$ be a $G$-component of any non-principal orbit type stratum $M_{(G_x)}\cap \closure(A_{s,t}(G\cdot p))\subset M\setminus M^{prin}$. 
    Note that $\exp_{G\cdot p}^\perp(p,v)$ and $\exp_{G\cdot p}^\perp(p, \lambda v)$ have the same orbit type for $|v|,\lambda|v|\in (0,\inj(G\cdot p))$. Hence, $\emptyset\neq  N\cap \bd B_t(G\cdot p)$.

    We will next show that $\wti\Sigma:=\spt(\|V\|)\cap U$ can only intersect $N$ orthogonally, which implies that $\wti\Sigma\llcorner U'$ admits a $G$-invariant unit normal for any simply $G$-connected open $G$-set $U'\subset\subset U$ (cf. the proof of Corollary \ref{Cor: stability of minimizers}). Thus, $\wti\Sigma$ is stable in $U'$ by Lemma \ref{Lem: stability and G-stability}. 

    Indeed, take any $x\in N\cap \wti\Sigma$. 
    Since $U=A_{s,t}(G\cdot p)$, there is no isolated non-principal orbit (Definition \ref{Def: isolated orbits}) in $U$. 
    Hence, there are three possibilities at $G\cdot x$:
    \begin{itemize}
        \item[(A)] $M$ satisfies Proposition \ref{Prop: local structure of M/G}(1), and $\wti\Sigma$ satisfies Proposition \ref{Prop: local structure of Sigma/G 1}(i) or (ii); 
        \item[(B)] $M$ satisfies Proposition \ref{Prop: local structure of M/G}(2.a), and $\wti\Sigma$ satisfies Proposition \ref{Prop: local structure of Sigma/G 2}(1) or (2); 
        \item[(C)] $M$ satisfies Proposition \ref{Prop: local structure of M/G}(2.b), and $\wti\Sigma$ satisfies Proposition \ref{Prop: local structure of Sigma/G 2}(3) or (4).
    \end{itemize}
    Since there is no special exceptional orbit in $M$, we know $\dim(M\setminus M^{prin})\leq n-1$, which implies Proposition \ref{Prop: local structure of Sigma/G 1}(ii) and Proposition \ref{Prop: local structure of Sigma/G 2}(2) cannot occur at $G\cdot x$. 
    Hence, in Case (A) and (B), $G\cdot x$ must be a singular orbit (i.e. $\dim(G\cdot x)<n+1-\cohom(G)$), and $\wti\Sigma$ must be orthogonal to $N$ at $G\cdot x$ by Proposition \ref{Prop: local structure of Sigma/G 1}(i), \ref{Prop: local structure of Sigma/G 2}(1) and Lemma \ref{Lem: hypersurface position}.  

    
    In Case (C), it follows from Proposition \ref{Prop: local structure of M/G}(2.b) that $G\cdot x$ is an exceptional orbit, and 
    there exists $r>0$ so that $(B_r(N)\setminus N) \cap U \subset M^{prin}$. 
    Additionally, combining this with the locally $G$-boundary-type assumption on the initial $G$-hypersurface $\Sigma$, we know the transversal intersection $\Gamma=\Sigma\cap\bd U$ satisfies $\Gamma\cap N = \emptyset$. 
    (This is because if $y\in \Gamma\cap N \neq \emptyset$, then $\Sigma$ satisfies Proposition \ref{Prop: local structure of Sigma/G 2}(4) at $G\cdot y$, which contradicts the locally $G$-boundary-type condition.)
    In particular, we have $\Gamma\llcorner B \subset M^{prin}\cap \bd U$, where $B:=B_r(N\cap \bd U)$. 
    Denote by $S_t^G:=\bd B_t(G\cdot p)$ the outer boundary of $U$. 
    For $t>0$ small enough, we have the convexity of $B_t(G\cdot p)$ (Lemma \ref{Lem: uniform constants}(i)). 
    Hence, using \eqref{Eq: area in orbit space} and Lemma \ref{Lem: minimal in M^prin}, one can apply the boundary regularity of De Lellis-Pellandini \cite{delellis2010genus}*{Lemma 8.1} in $U^*$ near $(\Gamma\llcorner(B\cap S_t^G))^*$ to show that $\bd \wti\Sigma\llcorner ( B \cap S_t^G) = \Gamma  \llcorner ( B \cap S_t^G)\subset M^{prin}$. 
    This implies that Proposition \ref{Prop: local structure of Sigma/G 2}(4) cannot occur on $\wti\Sigma$ at $G\cdot x$.
    (Otherwise, $N\subset \wti \Sigma$, and $  \bd\wti\Sigma\llcorner ( B \cap S_t^G)$ contains $\emptyset\neq N\cap S_t^G \subset M\setminus M^{prin}$ yielding a contradiction.)
    Therefore, at $G\cdot x$, $\wti\Sigma$ satisfies Proposition \ref{Prop: local structure of Sigma/G 2}(3) and is orthogonal to $N$. 

    {\bf Step 3.} Stability of $V$ in $U=A_{s,t}(G\cdot p)$. 
    
    Firstly, we conclude from Remark \ref{Rem: simply G-connected} that $U= A_{s,t}(G\cdot p)$ is simply $G$-connected whenever $p\in M\setminus \mc S_{n.m.}$, which implies the stability of $\wti \Sigma:=\spt(\|V\|)\cap U$ in $U$ by {\bf Step 2}. 
    
    Suppose next $G\cdot p\subset \mc S_{n.m.}$. 
    Then by Proposition \ref{Prop: local structure of M/G}(3), $B_t^*([p])$ is a closed cone over $RP^2$, and $\bd B_t(G\cdot p)\cap \mc S=\emptyset$. 
    Since $G\cdot x$ in Proposition \ref{Prop: local structure of M/G}(1) or (2.a) corresponds to a boundary point $[x]$ in $M^*$, we know every non-principal orbit $G\cdot x\subset U$ satisfies Proposition \ref{Prop: local structure of M/G}(2.b). 
    Hence, as in {\bf Step 2}, using the locally $G$-boundary-type condition of $\Sigma$, we know $\Gamma:=\Sigma\cap\bd U$ is an $(n-1)$-dimensional $G$-submanifold with $\Gamma\subset M^{prin}$. 
    One can then apply the boundary regularity of De Lellis-Pellandini \cite{delellis2010genus}*{Lemma 8.1} in $U^*$ near the outer boundary $(\Gamma\llcorner \bd B_t(G\cdot p))^*$ to show that $\bd \wti\Sigma \llcorner \bd B_t(G\cdot p)= \Gamma \llcorner \bd B_t(G\cdot p) \subset M^{prin}$. 
    Take $s'\in (s,t)$ close to $s$ so that $\bd B_{s'}(G\cdot p)$ is transversal to $\wti\Sigma$. 
    Then $\wti\Sigma_{s'}:= \wti\Sigma\llcorner A_{s',t}(G\cdot p)$ is an embedded $G$-hypersurface with smooth boundary. 
    As in \cite{delellis2010genus}*{Corollary 8.2}, every $G$-component $\wti\Sigma^i_{s'}$ of $\wti\Sigma_{s'}$ either 
    \begin{itemize}
        \item has non-empty boundary on $\bd B_t(G\cdot p)$ (i.e. $\emptyset\neq \bd \wti\Sigma^i_{s'}\cap \bd B_t(G\cdot p) \subset \Gamma$); or 
        \item has (possibly empty) boundary $\bd \wti\Sigma^i_{s'} $ contained in $ \bd B_{s'}(G\cdot p)$ (i.e. $ \bd \wti\Sigma^i_{s'} \subset \bd B_{s'}(G\cdot p)$). 
    \end{itemize}
    In the latter case, we can shrink $t$ to $t'\in (s',t)$ so that $\bd B_{t'}(G\cdot p)$ touches $\interior(\wti\Sigma^i_{s'})$ for the first time, which contradicts the maximum principle. 
    Next, consider $\wti\Sigma^i_{s'}$ with $\Gamma_i:=\bd \wti\Sigma^i_{s'}\cap \bd B_t(G\cdot p) \neq \emptyset $. 
    The non-trivial fixed boundary $\Gamma_i\subset \Gamma$ implies that $\wti\Sigma^i_{s'}$ is also contained in $\spt (T)$, where $T$ is the limit of $\Sigma_k$ as mod $2$ $n$-currents. 
    By the locally $G$-boundary-type condition, $\Sigma\llcorner U$ and $\Sigma_k\llcorner U$ are {\em $G$-boundaries}. Namely, they are contained in the boundary of some $G$-invariant Caccioppoli sets. 
    Hence, the mod $2$ $n$-current limit $T\llcorner U$ is also induced by the boundary of a certain $G$-invariant Caccioppoli set, which implies $\wti\Sigma^i_{s'}\subset \spt(T)\cap A_{s',t}(G\cdot p)$ admits a $G$-invariant unit normal and is stable by Lemma \ref{Lem: stability and G-stability}. 
    By taking $s'\to s$, we see $\wti\Sigma$ is stable in $U=A_{s,t}(G\cdot p)$.
\end{proof}

Now, we can show the {\em improved} $G$-replacement chain property of min-max $G$-varifolds.


\begin{proposition}\label{Prop: amv and replacement}
	Suppose \eqref{Eq: cohomogeneity assumption} is satisfied, and $M$ contains no special exceptional orbit. 
    Let $U:=A_{s,t}(G\cdot p)$ be a $G$-annulus with $G\cdot p\subset M$, $0<s<t<\inj(G\cdot p)$, and $t>0$ small enough. 
    If $V\in \mc V^G_n(M)$ is $G$-almost minimizing in $U$, 
    then for any finite sequence of concentric $G$-annuli $\{A_{s_j,t_j}(G\cdot p)\}_{j=1}^k\subset\subset U$, there exists a sequence $V=V_0,V_1,\dots,V_k\in \mc V^G_n(M)$ of $G$-varifolds so that for each $j=1,\dots, k$, 
    \begin{itemize}
        \item $V_j$ is stationary in $U$, and is a $G$-replacement of $V_{j-1}$ in $A_{s_j,t_j}(G\cdot p)$;
        \item $V_j\llcorner A_{s_j,t_j}(G\cdot p)$ is not only $G$-stable, but also stable in $A_{s_j,t_j}(G\cdot p)$.  
    \end{itemize} 
\end{proposition}
\begin{proof}	
	By Lemma \ref{Lem: amv implies stationary}, $V$ is stationary in $U=A_{s,t}(G\cdot p)$. 
	Let $\{\Sigma_j\}_{j\in\mb N}\subset \mc {LB}^G$ be a sequence of closed $G$-hypersurfaces that are $(G,\epsilon_j,\delta_j)$-almost minimizing in $U$ with $V=\lim_{j\to\infty}|\Sigma_j|$, where $\epsilon_j,\delta_j\to 0$ as $j\to\infty$. 
	For any concentric $G$-annulus $A_1:=A_{s_1,t_1}(G\cdot p)\subset \subset U$, we can slightly shrink $A_1$ to $A_1^l:=A_{s_1^l,t_1^l}(G\cdot p)\subset A_1$ so that $\bd A_1^l$ is transversal to every $\Sigma_j$ and $s_1^l\to s_1^+, t_1^l\to t_1^-$ as $l\to\infty$. 
    Consider the minimizing sequence $\{\Sigma_{j,k}^l\}_{k=1}^{\infty}$ for the constrained area minimization problem $(\Sigma_j, \mk {Is}^G_{\delta_j}(A_1^l))$ so that $W_j^l=\lim_{k\to\infty} |\Sigma_{j,k}^l|\in \mc V^G_n(M)$. 
    By \eqref{Eq: constrained minimizer - area bound}, there exist $W_j\in \mc V^G_n(M)$, an unrelabelled subsequence of $\{W_j^l\}_{l\in\mb N}$, and $\{k_l\}_{l\in\mb N}$ so that $W_j=\lim_{l\to\infty}W_j^l=\lim_{l\to\infty}|\Sigma_{j,k_l}^l|$. 
    By Proposition \ref{Prop: constrained minimizer} and the compactness theorem for stable minimal hypersurfaces, we know $W_j\llcorner A_1$ is an integer multiple of a smoothly embedded stable minimal $G$-hypersurface in $A_1$ so that $\mc H^n(\Sigma_j)-\epsilon_j\leq \|W_j\|(M)\leq \mc H^n(\Sigma_j)$. 
    Hence, we also have subsequences $\{W_j\}_{j\in\mb N}$ and $\{l_j\}_{j\in\mb N}$ so that there exists $V_1\in\mc V^G_n(M)$ with $V_1=\lim_{j\to\infty}W_j=\lim_{j\to\infty}|\Sigma_{j,k_{l_j}}^{l_j}|$. 
    One easily verifies that $V_1$ is a $G$-replacement of $V$ in $A_1$ that is stable in $A_1$. 
    %
    %
	
	Note that $\Sigma_{j,k_{l_j}}^{l_j}$ is $(G,\epsilon_{j},\delta_{j})$-almost minimizing in $U$ by the constrained area minimizing construction. 
	Thus, we can repeat this procedure to show the desired result. 
\end{proof}

Finally, we combine the above results and prove the $G$-equivariant min-max theorem \ref{Thm: G-min-max}. 

\begin{proof}[Proof of Theorem \ref{Thm: G-min-max}]
	Let $V\in \mf C(\{\mf \Sigma^i\})$ be the stationary min-max $G$-varifold given in Theorem \ref{Thm: existence amv in annuli}.
    For any $p\in M$, it now follows from Theorem \ref{Thm: existence amv in annuli} and Proposition \ref{Prop: amv and replacement} that there exists $r=r(G\cdot p)>0$ so that $V$ has the {\em improved} $G$-replacement chain property (as in Proposition \ref{Prop: amv and replacement}) within every $G$-annulus $U:=A_{s,t}(G\cdot p)$ with $0<s<t<r$. 
    Then, as we mentioned in Remark \ref{Rem: improved G-replacement}, one can apply \cite{wang2024index}*{Theorem 4.18} to obtain the desired regularity result for $V$. 
\end{proof}

\subsection{Genus upper bound}


\begin{proof}[Proof of Theorem \ref{Thm: G-min-max - topological control}]
    Take any smooth open $G$-set $U\subset\subset M^{prin}$ so that $\bd U$ is transversal to $\cup_{j=1}^R\Gamma_j$ and every $\Sigma^i_{x_i}$ in the min-max sequence. 
    Then, it is sufficient to show the weighted genus upper bound \eqref{Eq: G-min-max - genus} with $\genus((\Gamma_j\cap U)^*)$ in place of $\mk g(\Gamma_j)$. 
    By Lemma \ref{Lem: minimal in M^prin} and \eqref{Eq: area in orbit space}, all the $G$-equivariant objects (e.g., minimal $G$-hypersurface, $G$-almost minimizing varifolds, $G$-replacement, ....) in $U$ can be reduced via the submersion $\pi$ to the non-equivariant classical ones in $(U^*, \wti g_{_{M^*}})$. 
	Hence, the proof in \cite{ketover2019genus} would carry over for $\Gamma_j^*$ in $(U^*, \wti g_{_{M^*}})$. 
\end{proof}

\section{Spherical Bernstein problem in $\mb S^4$}\label{Sec: spherical Bernstein problem}

In this section, we apply the equivariant min-max theory to provide a new resolution of Chern's spherical Bernstein problem in $\mb S^4$. 
We begin with our equivariant setup. 

\subsection{$S^1$-action on $\mb S^4$}
Let $M=\mb S^4= \mb S^4_1(0)$ be the unit sphere in $\mb R^5$ with coordinates $(x, z_1, z_2)$, where $x\in \mb R$, $z_1, z_2\in \mb C$ satisfy $x^2+|z_1|^2+|z_2|^2=1$. 
Then, $e^{i\theta}\in S^1$ acts on $\mb S^4$ by 
\begin{align}\label{Eq: suspended S1-action}
    e^{i\theta}\cdot (x,z_1,z_2)= R_\theta(x,z_1,z_2):= (x, e^{i\theta} z_1, e^{i\theta}z_2). 
\end{align}
In particular, we can regard $M=\mb S^4$ as $\{(\cos(\alpha), p): p\in \mb S^3_{\sin(\alpha)}(0), \alpha\in [0,\pi]\}$ the suspension of $\mb S^3_1(0)$, and $S^1$ acts on each slice $\mb S^3_{\sin(\alpha)}(0)$ by the standard Hopf action so that $\{(\pm 1, 0,0,0,0)\}= M\setminus M^{prin}$ are the $S^1$-fixed points. 
Hence, the orbit space $M/S^1$ is a suspension of $S^2$, and thus is a topological $S^3$ with two singular points $\{(\pm 1,0,0,0)\}$. 

The round metric on $M=\mb S^4$ takes the standard warped-product form: 
\[ g_{_M}:=d\alpha^2 + \sin^2(\alpha)g_{_{\mb S^3_1}},\]
where $g_{_{\mb S^3_1}}$ denotes the round metric on $\mb S^3_1(0)$. 
Note that the Hopf fibration provides a Riemannian submersion $h: \mb S^3_1 \to \mb S^2_{1/2}(0)$ under the corresponding round metrics. 
Hence, the quotient map $\pi : M\to M/S^1$ induces a Riemannian metric $g_{_{M/S^1}}$ on $M^{prin}/S^1$: 
\begin{align}\label{Eq: suspension orbit space}
    M^{prin}/S^1:= \left\{ (\cos(\alpha), p): p\in \mb S^2_{\frac{\sin(\alpha)}{2}}, \alpha\in (0,\pi) \right\}\quad {\rm and}  \quad  g_{_{M/S^1}} := d\alpha^2+ \frac{\sin^2(\alpha)}{4} g_{_{\mb S^2_1}}, 
\end{align}
so that $\pi\llcorner M^{prin}= \pi\llcorner (\mb S^4_1(0)\setminus\{(\pm 1,0,0,0,0)\})$ is a Riemannian submersion. 
In \eqref{Eq: suspension orbit space}, although $M/S^1$ is represented by an ellipsoid in $\mb R^4$, the metric $g_{_{M/S^1}}$ is not induced from $\mb R^4$. 
For any $(\cos(\alpha), p)\in M^{prin}/S^1$, $\pi ^{-1}((\cos(\alpha), p))$ is a great circle in $\{\cos(\alpha)\}\times \mb S^3_{\sin(\alpha)}$, and thus
\[ V((\cos(\alpha), p)) := \mc H^1(\pi^{-1}((\cos(\alpha), p)))=2\pi\sin(\alpha). \]
Continuously extend $V$ by $0$ to the singular set $(\pm 1,(0,0,0,0))\in M/S^1$ as in \eqref{Eq: orbits volume}, and define the weighted metric $\wti g_{_{M/S^1}}$ as in \eqref{Eq: weighted metric in M/G} by 
\[ \wti g_{_{M/S^1}} := V\cdot g_{_{M/S^1}}. \]

\begin{lemma}\label{Lem: unique equivariant equator}
	Let the $S^1$-action on $\mb S^4$ be given as above. 
	Then, there exists only one equatorial hypersphere $E_0=\{0\}\times \mb S^3$ in $\mb S^4$ that is $S^1$-invariant. 
\end{lemma}
\begin{proof}
	Every equatorial hypersphere can be written as $S_v:=\{p\in\mb S^4: p\cdot v=0\}$ for some $v\in \mb S^4$. 
	If $S_v$ is $S^1$-invariant, then the line $\{tv:t\in\mb R\}$ is also $S^1$-invariant. 
	Thus, the only possible choice of $\{tv:t\in\mb R\}$ is the line $\mb R\times \{(0,0,0,0)\}$ passing through the fixed points.  
\end{proof}

\begin{lemma}\label{Lem: positive Ricci M/S1}
	Using the above notations, $(M^{prin}/S^1, \wti g_{_{M/S^1}})$ has positive Ricci curvature.
\end{lemma}
\begin{proof}
	Under the parametrization $(\cos(\alpha), \mb S^2_{\sin(\alpha)/2})$ of $M^{prin}/S^1$ with $\alpha\in (0,\pi)$, we have that $\wti g_{_{M/S^1}}$ is a conformal warped-product metric:
	\[ \wti g_{_{M/S^1}}= (2\pi\sin(\alpha))d\alpha^2 + \frac{\pi}{2}\sin^3(\alpha) \cdot g_{_{\mb S^2_1}}.  \]
	For the Ricci curvature $\Ric(\cdot,\cdot)$ under the warped metric $g_{_{M/S^1}}=d\alpha^2+ f^2(\alpha)g_{_{\mb S^2_{1}}}$ with $f(\alpha)=\sin(\alpha)/2$, the standard formula shows
	\[ \Ric(\bd_\alpha,\bd_\alpha)=-2\frac{f''}{f}=2, \quad\Ric(\bd_\alpha,e)=0,\quad \Ric(e,e)=\frac{1}{f^2} -\left(\frac{f''}{f} + \frac{(f')^2}{f^2}\right)=\frac{5-2\cos^2(\alpha)}{\sin^2(\alpha)}, \]
	where $e$ is any $g_{_{M/S^1}}$-unit vector orthogonal to $\bd_\alpha$. 
	For the conformal factor $2\pi\sin(\alpha)= e^{2u}$ with $u=\frac{1}{2}\ln(2\pi\sin(\alpha))$, we have $u'=\cot(\alpha)/2$, $u''=-1/(2\sin^2(\alpha))$, and 
	\[\Delta u= u''+2\frac{f'}{f}u'=\frac{1}{2\sin^2(\alpha)} -1. \]
	Hence, under the conformal change, the Ricci curvature $\wti\Ric(\cdot,\cdot)$ in $\wti g_{_{M/S^1}}$ is expressed by $\wti \Ric = \Ric -\nabla^2 u + du\otimes du -(\Delta u+|\nabla u|^2)g_{_{M/S^1}}$, which implies $\wti \Ric(\bd_\alpha, e)=0$,
	\[ \wti \Ric(\bd_\alpha,\bd_\alpha) = \Ric(\bd_\alpha,\bd_\alpha)- u''+ (u')^2 - (\frac{1}{2\sin^2(\alpha)} -1+ \frac{\cot^2(\alpha)}{4} ) = 3\]
	and 
	\begin{align*}
		\wti \Ric(e,e) &= \Ric(e,e)-\nabla^2u(e,e) - (\Delta u+|\nabla u|^2) \\
		&= \Ric(e,e)-  \frac{f'}{f} u' - \left(\frac{1}{2\sin^2(\alpha)} -1+ \frac{\cot^2(\alpha)}{4} \right)
		\\
		&= \frac{22-15\cos^2(\alpha)}{4\sin^2(\alpha)}.
	\end{align*}
	After normalization by $1/(2\pi \sin(\alpha))$, we know $\wti\Ric$ is positive  since $\alpha\in (0,\pi)$.  
\end{proof}

In the following lemma, we show that any embedded $S^1$-hypersurface (e.g., equivariant min-max minimal $S^1$-hypersurfaces) contains no fixed point $\{(\pm 1, 0,0,0,0)\}$. 
\begin{lemma}\label{Lem: tangent cone at fixed point}
	Using the above notations, let $x_0\in \{(\pm 1, 0,0,0,0)\}$, and $C$ be an $S^1$-invariant regular minimal cone in $ N_{x_0}(S^1\cdot x_0)=T_{x_0}\mb S^4$ with vertex at the origin. 
	Then the link of $C$ is a Clifford torus, and thus $C$ is unstable. 
    Moreover, there is no $S^1$-invariant hyperplane in $T_{x_0}\mb S^4$. 
    
    In particular, every embedded $S^1$-invariant hypersurface is contained in $M^{prin}$. 
\end{lemma} 
\begin{proof}
	In $T_{x_0}S^4=\mb R^4$, $G_{x_0}=S^1$ acts on each $\mb S^3_r(0)$, $r>0$, so that the $S^1$-orbits are the fibers in the Hopf fibration of $\mb S^3_r(0)$. 
	Hence, the link $C\cap \mb S^3_1(0)$ is a minimal $S^1$-surface in $\mb S^3_1(0)$. 
	Since every $S^1$-orbit in $\mb S^3_1(0)$ has the same length and $\mb S^3_1(0)/ S^1 = \mb S^2_{1/2}(0)$, we see from Lemma \ref{Lem: minimal in M^prin} that $(C\cap \mb S^3_1(0))/ S^1$ is a closed geodesic in a Riemannian $S^2$ with constant curvature, which must be a great circle. 
	Hence, $C\cap \mb S^3_1(0)$ is a Clifford torus, and $C$ is unstable.
    
    Moreover, every hyperplane in $T_{x_0}\mb S^4$ has a link $\mb S^2\subset \mb S^3_1(0)$, which cannot be $S^1$-invariant. 
    Otherwise, the link $\mb S^2$ admits a free $S^1$-action, and then $X(p):=\frac{d}{dt}\vert_{t=0} e^{i t}\cdot p$, ($p\in \mb S^2$), gives a non-vanishing smooth vector field on $\mb S^2$, which is a contradiction. 
    
    Finally, if $\Sigma\subset M$ is an embedded $S^1$-hypersurface containing $x_0$,  
    then $T_{x_0}\Sigma$ is an $S^1$-invariant hyperplane in $N_{x_0}(S^1\cdot x_0)=T_{x_0}M$, which is a contradiction. 
\end{proof}

\subsection{Equivariant min-max for minimal hyperspheres}\label{Subsec: minimal hypersphere}
Next, we notice that the orbit space $M/S^1 \overset{\rm homeo}{\cong} S^3$ admits an isometric antipodal map, i.e. a $\mb Z_2$-action. 
The following lemma indicates that the $S^1$-action can be extended by this $\mb Z_2$-action. 
Namely, the $\mb Z_2$-action on $M/S^1$ can be lifted to the ambient manifold $M$. 
\begin{lemma}\label{Lem: extend action}
	There exists an extension $G=(S^1 \rtimes \mb Z_4)/\mb Z_2$ of $\mb Z_2$ by $S^1$ acting on $M=\mb S^4$ by isometries so that some $\tilde\sigma\in G$ induces the antipodal map $A:(\cos(\alpha),p)\mapsto (-\cos(\alpha), -p)$ on $\pi(M)= M/S^1= \{ (\cos(\alpha), p): p\in \mb S^2_{\sin(\alpha)/2}, \alpha\in [0,\pi] \}$, i.e. $\pi\circ \tilde \sigma =A\circ \pi$. 
\end{lemma}
\begin{proof}
	Recall that $R_\theta: (x,z_1,z_2)\mapsto (x, e^{i\theta} z_1, e^{i\theta}z_2)$ is the $S^1=\{e^{i\theta}\}$ action on $M=\{(x,z_1,z_2)\in \mb R\times \mb C^2:|x|^2+|z_1|^2+|z_2|^2=1\}$. 
	Define then an isometry $\sigma: M\to M$ by 
    \begin{equation}\label{eq:sigma}
        \sigma(x,z_1,z_2):= (-x, -\bar{z}_2, \bar{z}_1), \qquad\forall (x,z_1,z_2)\in\mb S^4,
    \end{equation}
	where $\bar{z}$ is the conjugate complex of $z\in\mb C$. 
	Since the Hopf map $h: \mb S^3_1\to \mb S^2_{1/2}$ is given by 
	\begin{align}\label{Eq: Hopf map}
        (z_1,z_2)\in\mb S^3_1 ~\mapsto~ ({\rm Re}(z_1\bar{z}_2), {\rm Im}(z_1\bar{z}_2), (|z_1|^2-|z_2|^2)/2)\in\mb S^2_{1/2},
    \end{align}
	one easily verifies that $\pi\circ \sigma = A\circ \pi$, where $\pi: M\to M/S^1$ is the quotient map, and $A$ is the antipodal map on $M/S^1$. 
	Note that we cannot take $G=S^1\times \mb Z_2$ since $\sigma$ generates a $\mb Z_4$-action on $M$ that does not commute with the $S^1$-action. 
	Indeed, we have 
	\[\sigma^2= R_\pi \qquad {\rm  and}\qquad \sigma\circ R_\theta= R_{-\theta}\circ \sigma,\quad\forall \theta\in [0,2\pi].\] 
	Let $\varphi: \mb Z_4\to {\rm Aut}(S^1)$ be given by $\varphi(id)=\varphi(\sigma^2) := id$ and $\varphi(\sigma)=\varphi(\sigma^3):= \iota$, where $\iota: S^1\to S^1$ is given by $\iota(e^{i\theta})=e^{-i\theta}$. 
	Define then the semidirect product $\wti G$ by 
	\[ \wti G:= S^1 \rtimes_\varphi \mb Z_4 \]
	with the group operation 
	\[(e^{i\theta_1},\sigma^{k_1})\cdot (e^{i\theta_2},\sigma^{k_2}) := (e^{i\theta_1}\cdot \varphi(\sigma^{k_1})(e^{i\theta_2}), \sigma^{k_1}\cdot \sigma^{k_2})=(e^{i(\theta_1+(-1)^{k_1}\theta_2)}, \sigma^{k_1+k_2\mod 4}).\] 
	Now, $\wti G$ acts isometrically on $M=\mb S^4$ by 
	\[(e^{i\theta},\sigma^k)\cdot (x,z_1,z_2):= R_\theta\circ \sigma^k(x,z_1,z_2).\] 
	This group action is well-defined since for any $(e^{i\theta_1},\sigma^{k_1}), (e^{i\theta_2},\sigma^{k_2})\in \wti G$,
	\begin{align*}
		(e^{i\theta_1},\sigma^{k_1})\cdot \left( (e^{i\theta_2},\sigma^{k_2})\cdot (x,z_1,z_2)\right) &= R_{\theta_1}\circ \sigma^{k_1}\circ R_{\theta_2}\circ \sigma^{k_2}(x,z_1,z_2)
	\\&= R_{\theta_1}\circ R_{(-1)^{k_1}\theta_2} \circ \sigma^{k_1+k_2}(x,z_1,z_2) 
	\\&= \left( (e^{i\theta_1},\sigma^{k_1})\cdot (e^{i\theta_2},\sigma^{k_2})\right) \cdot (x,z_1,z_2). 
	\end{align*}
	Additionally, note that $(e^{i\pi}, \sigma^2)\in \wti G$ also acts by the identity on $M$, i.e. the $\wti G$-action is not effective. 
	One verifies that $H:=\{(e^{i0}, \sigma^0 ), (e^{i\pi}, \sigma^2)\}\cong \mb Z_2$ is a normal subgroup of $\wti G$. 
	Hence, the $\wti G$-action on $M$ descends to an action of the quotient group $G:= \wti G/H$, which has $S^1$ as its $2$-index subgroup. 
    In particular, $\tilde\sigma:=[(e^{i0},\sigma)]\in G$ induces the antipodal map on $M/S^1$. 
\end{proof}

\begin{remark}\label{Rem: pin group} 
    By identifying $(z_1,z_2)\in \mb C^2$ with the quaternion $z_1+z_2j\in \mb H = \{a + bi + cj + dk: a, b, c, d\in \mb R, i^2 = j^2 = k^2 =-1, ij = -ji = k, jk = -kj = i, ki = -ik = j\}$, the suspended Hopf $S^1$-action on $\mb S^4=\{(x,q)\in\mb R\times \mb H: x^2+|q|^2=1\}$ can be expressed as $e^{i\theta}\cdot (x,q)= (x, e^{i\theta}\cdot q)$ for any $e^{i\theta}\in S^1\subset \mb C$, and the lift $\sigma$ of the antipodal map is defined by $\sigma(x,q)=(-x,jq)$.   
    In particular, the group $G=(S^1\rtimes\mb Z_4)/\mb Z_2$ is isomorphic to the pin group ${\rm Pin}_-(2):=S^1\cup S^1\cdot j$ via the map $[(e^{i\theta}, \sigma^k)]\in G \mapsto e^{i\theta}j^k \in {\rm Pin}_-(2)$, where $S^1:=\{e^{i\theta}\in \mb C\}\subset \mb H$ and $\sigma$ is given by \eqref{eq:sigma};
    see, for instance, \cite{manolescu2016pin}*{Section 2.1}.
\end{remark}

Note that the group $G$ in Lemma \ref{Lem: extend action} is not the direct product $S^1\times \mb Z_2$, but we still have that $M/G$ is homeomorphic to $S^3/\mb Z_2=RP^3$, where $\mb Z_2=\{id, A\}$ and $A$ is the antipodal map on $M/S^1$. 
In particular, $M/S^1$ is an antipodally symmetric $S^3$ with two singular points and positive Ricci curvature (Lemma \ref{Lem: positive Ricci M/S1}) in its regular part $M^{prin}/S^1$. 
Additionally, $M$ has no exceptional $G$-orbit, $\mc S_{n.m.}=\emptyset$, and every embedded $G$-hypersurface in $M$ is locally $G$-boundary-type since $M/G$ is a closed topological manifold. 

Consider the following space of $G$-hypersurfaces with smooth topology:
\begin{align}\label{Eq: space of RP2}
	\mc X^G:=\{\wti\Sigma \subset M^{prin} \mbox{ is a closed embedded $G$-hypersurface with $\wti\Sigma/G \cong RP^2$} \},
\end{align}
and take
\[\mc Y^G:=\{E_0\cup\gamma: \gamma/G \mbox{ is a curve or a point}\},\]
where $E_0=\{0\}\times \mb S^3$ is the only $G$-invariant hypersphere (Lemma \ref{Lem: unique equivariant equator}). 
Let 
\[\overline {\mc X}^G:=\mc X^G\cup\mc Y^G.\]
Note that the equatorial hypersphere $E_0$ has the least area among all the minimal hypersurfaces in $\mb S^4$.
Hence, for any $G$-isotopy class within $\mc X^G$, Theorem \ref{Thm: G-isotopy minimizer} gives a $G$-isotopy area minimizing minimal $G$-hypersurface with area at least $2\pi^2=\mc H^3(E_0)$, which implies  
\begin{align}\label{Eq: infimum area of equator}
	\inf_{\Sigma\in \overline{\mc X}^G } \mc H^3(\Sigma)=\mc H^3(E_0) = 2\pi^2 .
\end{align}

Next, by Lemma \ref{Lem: positive Ricci M/S1}, we can follow the constructions in \cite{haslhofer2019sphere}*{Theorem 10.1} to find a non-trivial family of antipodally symmetric $2$-spheres $\{\Sigma_t\}_{t\in [0,1]}$ in $M/S^1$ with $\wti\Sigma_t:=\pi^{-1}(\Sigma_t)$ so that
\begin{enumerate}
	\item $\wti\Sigma_0=E_0$, $\wti\Sigma_1\in\mc Y^G $, and $\wti \Sigma_t\in \mc X^G$ for $t\in (0,1)$;
	\item[(2)] $\sup_{t\in [0,1]} \mc H^3_{g_{_M}}(\wti\Sigma_t)=\sup_{t\in [0,1]} \mc H^2_{\wti g_{_{M/S^1}}}(\Sigma_t) < 3 \mc H^2_{\wti g_{_{M/S^1}}}(E_0/S^1)=3\mc H^3_{g_{_{M}}}(E_0)=6\pi^2$; 
	\item[(3)] $\{\wti\Sigma_t\}_{t\in [0,1]}$ is a $G$-equivariant $(X,Z)$-sweepout (Definition \ref{Def: G-sweepout}) with $X=[0,1]$, $Z=\{0,1\}$;
	\item[(4)] $\Phi:[0,1]\to \mc Z_3(M;\mb Z_2) $ given by $\Phi(t):=\llbracket \wti\Sigma_t\rrbracket$ is an Almgren-Pitts $1$-sweepout (cf. \cite{marques2017existence}); namely, there exists a continuous family of Caccioppoli sets $\{\Omega_t\}_{t\in [0,1]}$ with $\bd\Omega_t=\wti\Sigma_t$ so that $\Omega_1=M\setminus\Omega_0$. 
\end{enumerate}
Indeed, for $t\in (0,1)$, let 
\[\wti\Gamma_t:=\left\{x=(x_1,\dots,x_5)\in\mb S^4_1: x_1=\cos((1\pm t)\frac{\pi}{2})\right\}\] 
form a foliation of $M^{prin}\setminus E_0$ so that $\{\Gamma_t:=\wti\Gamma_t/S^1\}_{t\in (0,1)}$ is a foliation of the (punctured) $3$-sphere $(M^{prin}\setminus E_0)/S^1$ by antipodally symmetric $2$-spheres. 
Then for each $t$, we can antipodal-symmetrically connect $\Gamma_t$ and $E_0/S^1$ by a sufficiently thin tube $T_t$ with $T_t$ tending (as $t\to 1$) to two curves $\gamma_0/S^1$ that connect the two singular points $(M\setminus M^{prin})/S^1$ to the equator $E_0/S^1$. 
In particular, $\lim_{t\to 1} \mc H^2_{\wti g_{M/S^1}}(T_t)=0$. 
Hence, the connected sum $\Sigma_t:=(E_0/S^1) \#_{T_t}\Gamma_t$, $t\in (0,1)$, is a family of antipodally symmetric $2$-spheres in $M^{prin}/S^1$, and $\wti\Sigma_t=\pi^{-1}(\Sigma_t)\in \mc X^G$. 
As $t\to 1$, $\wti\Sigma_t$ approaches $E_0\cup\gamma_0\in \mc Y^G$ in the Hausdorff topology. 
Additionally, for $t$ near $0$, it follows from the catenoid estimates that (\cite{haslhofer2019sphere}*{(10.4)})
\begin{align}\label{Eq: catenoid estimates for one neck}
	\mc H^3(\wti\Sigma_t)<3\mc H^3(E_0)-Ct^2.
\end{align}
Suppose the above connected sum is performed antipodal-symmetrically via necks over $\pm p\in E_0/S^1$ respectively. 
Then there exists an antipodal equivariant retraction $R$ that retracts $(E_0/S^1)\setminus \{\pm p\}$ onto a circle $\beta$, which can be extended to a neighborhood of $E_0/S^1$. 
Specifically, we can write $E_0/S^1 $ as $\mb S^2$ with a neighborhood $B_a(E_0/S^1)=(-a,a)\times \mb S^2$ so that 
\begin{itemize}
    \item $\pm p=(\pm 1, 0,0)\in E_0/S^1=\mb S^2$, and $\beta=\{(0,x_2,x_3)\in\mb S^2\}$;
    \item $R(s,(r,x_1,x_2,x_3)):= (r, \frac{((1-s)x_1,x_2,x_3)}{|((1-s)x_1,x_2,x_3)|})$ for all $s\in [0,1]$, $r\in (-a,a)$, and $(x_1,x_2,x_3)\in \mb S^2\setminus\{\pm p\}$. 
\end{itemize}
For $t$ near $0$, we use $R$ to open the neck $T_t$ so that $\Sigma_t$ becomes the connected union of two half-spheres in $\Gamma_t$ and a cylinder over $\beta$. 
Then, as $t\to 0$, the amended $\Sigma_t$ tends to $E_0/S^1$ smoothly through a compression. 
As in \cite{haslhofer2019sphere}*{(10.5)}, the amended family $\{\Sigma_t\}_{t\in [0,1]}$ satisfies (1)-(4).

Consider the $G$-equivariant $(X,Z)$-homotopy class $\Pi$ of $\{\wti\Sigma_t\}_{t\in [0,1]}$ in $(\overline{\mc X}^G,\mc Y^G )$. 
Then every $\{\Gamma_t\}_{t\in [0,1]}\in \Pi$ satisfies $\Gamma_t\in \mc X^G$ for $t\in (0,1)$ and $\{\Gamma_0=\wti\Sigma_0,\Gamma_1=\wti\Sigma_1\}\subset\mc Y^G$.  

\begin{lemma}\label{Lem: width lower bound}
	Using the above notations, $\mf L(\Pi)>\mc H^3(E_0)>0$, and $\mf L(\Pi)$ is realized by the area of a $G$-invariant minimal hypersphere $\Gamma$ in $M=\mb S^4$ with integer multiplicity $m=1$ or $2$. 
\end{lemma}
\begin{proof}
	Since $\inf_{\Sigma\in \overline{\mc X}^G } \mc H^3(\Sigma)=\mc H^3(E_0)$, we have $\mf L(\Pi)\geq \mc H^3(E_0)$. 
	Suppose $\mf L(\Pi)= \mc H^3(E_0)$. Then there exists a minimizing sequence $\{\{\Gamma^i_t\}_{t\in [0,1]}\}_{i\in\mb N}\subset \Pi$ with 
	\[ \mc H^3(E_0)\leq  \sup_{t\in [0,1]} \mc H^3(\Gamma^i_t)\leq \mf L(\Pi) + \frac{1}{i} = \mc H^3(E_0) + \frac{1}{i}.  \]
	In particular, for any sequence $t_i\in [0,1]$, $\Gamma^i_{t_i}$ is an area minimizing sequence in $\overline{\mc X}^G$. 
	By Lemma \ref{Lem: unique equivariant equator} and Theorem \ref{Thm: G-isotopy minimizer}, we know $|\Gamma^i_{t_i}|$ converges to $|E_0|$ up to a subsequence. 
    Consider also the mod $2$ current limit $T_0\in\mc Z_3(M;\mb Z_2)$ of $\Gamma^i_{t_i}$, which is supported in $E_0$ with mass $\mf M(T_0)\leq \mc H^3(E_0)$. 
    Since $\Gamma^i_{t_i}/G\cong RP^2$ is non-orientable in $M/G\cong RP^3$, we have $T_0\neq 0$ by the isoperimetric lemma \cite{almgren1962homotopy}*{Corollary 1.14}. 
	Thus, the Constancy Theorem implies $T_0=\llbracket E_0 \rrbracket$, and the convergence of $\Gamma^i_{t_i}$ to $E_0$ is also in the mod $2$ current sense.  
	Hence, it follows from a simple contradiction argument that for any $\epsilon>0$, there exists $i_0>1$ so that $\mf F(\llbracket \Gamma^i_t\rrbracket, \llbracket E_0\rrbracket)<\epsilon$ for all $i\geq i_0$ and $t\in [0,1]$, where $\mf F$ is the metric on the mod $2$ cycle space defined in \cite{pitts2014existence}. 
    However, by the statement (4) before this lemma and \cite{marques2017existence}*{Proposition 3.3}, $\Phi_{i_0}(t):=\llbracket \Gamma^{i_0}_t \rrbracket$, ($t\in [0,1]$), is a non-trivial Almgren-Pitts $1$-sweepout, which cannot be contained in an $\epsilon$-neighborhood of $\llbracket E_0\rrbracket$ for $\epsilon>0$ small enough (\cite{marques2017existence}*{Corollary 3.4}).
    This gives a contradiction. 
	
	Hence, $\mf L(\Pi)>\mc H^3(E_0)$, and Theorem \ref{Thm: G-min-max} gives an embedded minimal $G$-hypersurface $\Gamma$ with integer multiplicity $m$. 
	By Lemma \ref{Lem: tangent cone at fixed point} and the embeddedness of $\Gamma$, we know $\Gamma\subset M^{prin}$. 
	Additionally,  $\Gamma$ is connected by the Frankel property as $\mb S^4$ has positive Ricci curvature. 
    
    Next, we claim that $\Gamma\in \mc X^G$ is a hypersphere. 
    Indeed, although $\wti \Sigma_t/G$ are non-orientable $RP^2$ in $M/G=RP^3$, we can use the technique in \cite{li2025RP2}*{Theorem 4.3} and apply the proof of the reduced genus upper bound (Theorem \ref{Thm: G-min-max - topological control}) in $M/S^1\cong S^3$ under $\mb Z_2$-symmetric constraints to show that $\genus(\Gamma/ S^1)=0$. 
    Thus, $\Gamma/ S^1$ is a connected antipodally symmetric $S^2$. 
	The antipodal symmetry of $\Gamma/ S^1$ implies it is homotopic to $E_0/S^1$ in $M^{prin}/S^1\cong S^3\setminus\{(\pm 1,0,0,0)\} \cong S^2\times \mb R$. 
	Thus, by \cite{laudenbach1974topologie}*{P.64, Theorem 1.3}, $\Gamma/S^1$ is also isotopic to $E_0/S^1$ in $M^{prin}/S^1$, which can be lifted to an $S^1$-isotopy in $M^{prin}$ by \cite{schwarz1980lifting}. 
	In particular, $\Gamma$ is a hypersphere in $M= \mb S^4$. 
	
	Finally, since $m\mc H^3(\Gamma)=\mf L(\Pi)\leq \sup_{t\in [0,1]}\mc H^3(\wti\Sigma_t) < 3\mc H^3(E_0)$ and $ \mc H^3(E_0) \leq \mc H^3(\Gamma)$, we conclude that the multiplicity $1\leq m< 3$. 
\end{proof}

Now, we are ready to solve the spherical Bernstein problem in $M=\mb S^4$.
\begin{theorem}\label{Thm: spherical Bernstein in S4}
	There exists an embedded non-equatorial minimal hypersphere in $\mb S^4$. 
\end{theorem}
\begin{proof}
	Suppose for contradiction that every minimal hypersphere in $\mb S^4$ is equatorial. 
	Then the $G$-invariant minimal hypersphere $\Gamma$ in Lemma \ref{Lem: width lower bound} must coincide with the $G$-invariant equator $E_0$ by Lemma \ref{Lem: unique equivariant equator}. 
	Since $\mf L(\Pi)=m\mc H^3(\Gamma)=m\mc H^3(E_0)> \mc H^3(E_0)$ and $m<3$ in Lemma \ref{Lem: width lower bound}, we must have $m=2$. 
	Consider the min-max sequence $\{\Gamma^i_{t_i}\}_{i\in\mb N}$ of $\Pi$ converging to $2|E_0|\in \mc V^G_3(M)$. 
	Note that each $\Gamma^i_{t_i}/S^1$ separates $M/S^1$ into two regions interchanged under the antipodal map. 
	Hence, $M\setminus \Gamma^i_{t_i}=\Omega_i^+\sqcup\Omega_i^-$ for two $S^1$-invariant open sets $\Omega_i^\pm$ that are interchanged under the action of $\tilde\sigma=[(e^{i0},\sigma)]\in G$. 
	Up to a subsequence, we also have that $\Omega_i^\pm$ converge to $\Omega^\pm$ as $S^1$-invariant Caccioppoli sets so that $\tilde{\sigma}(\Omega^\pm)=\Omega^\mp$. 
    In particular, $\Omega^\pm$ are half-volume subsets of $M$ so that $\Gamma^i_{t_i}=\bd \Omega_i^\pm$ converges to $\bd\Omega^\pm\neq 0$ in the sense of mod $2$ currents. 
    Therefore, $\bd\Omega^\pm$ is supported in $\Gamma=E_0$ with mass $\mf M(\bd\Omega^\pm)\leq 2\mc H^3(E_0)$, which implies $\bd\Omega^\pm = \llbracket E_0 \rrbracket$ by the Constancy Theorem and the mod $2$ coefficients. 
    However, since the mod $2$ current limit $\lim_{i\to\infty}\llbracket \Gamma^i_{t_i} \rrbracket = \llbracket \Gamma \rrbracket $ has multiplicity one, it now follows from the proof of \cite{zhou2015positiveRic}*{Proposition 6.1} (with \cite{wang2024index}*{Proposition 4.19} in place of \cite{zhou2015positiveRic}*{Lemma 6.11}) that the multiplicity $m$ of the varifold convergence $|\Gamma^i_{t_i}|\to m|\Gamma|$ must be odd, which contradicts $m=2$.
\end{proof}

\subsection{Minimal hypertori in $\mb S^4$}

Let $M=\mb S^4_1(0)$ and the $G$-action be given as before. 
Fix any embedded $G$-hypersurface $\Sigma_0\subset M$ so that $\Sigma_0/G$ is separating in $M/G$. 
Let  
\[\ms X^G(\Sigma_0):=\{\Sigma=\varphi_1(\Sigma_0): \{\varphi_t\}_{t\in[0,1]}\in \mk {Is}^G(M)\}\]
be the collection of embedded $G$-hypersurfaces with fixed topological type as $\Sigma_0$. 
Define 
\[\ms Y^G(\Sigma_0):=\{\phi(\Sigma_0): \phi:\Sigma_0\to M \mbox{ is any $G$-equivariant smooth map with } \mc H^n(\phi(\Sigma_0))=0 \},\]
and $\overline{\ms X}^G(\Sigma_0):= \ms X^G(\Sigma_0)\cup\ms Y^G(\Sigma_0)$. 

For any cubical complex $X$ with a subcomplex $Z$, consider a $G$-equivariant $(X,Z)$-sweepout $\{\Sigma_x\}_{x\in X}$ (Definition \ref{Def: G-sweepout}) with $\Sigma_x\in {\ms X}^G(\Sigma_0)$ for all $x\in X\setminus Z$, and $\Sigma_x\in {\ms Y}^G(\Sigma_0)$ for all $x\in Z$. 
Let $\Pi$ be the $(X,Z)$-homotopy class of $\{\Sigma_x\}_{x\in X}$. 

Suppose $\mf L(\Pi)> 0$. 
Then by our equivariant min-max theorem (Theorem \ref{Thm: G-min-max}), there exists $V\in \mc V^G_n(M)$ with $\|V\|(M)=\mf L(\Pi)$ that is induced by an integer multiple of a minimal $G$-hypersurface in $M$. 
By the last statement in Lemma \ref{Lem: tangent cone at fixed point}, we see that $\spt(\|V\|)\subset\subset M^{prin}$. 
It then follows from the embedded Frankel property and Lemma \ref{Lem: positive Ricci M/S1} that $V=m|\Gamma|$ for some connected minimal $G$-hypersurface $\Gamma\subset M$ and $m\in\mb N$. 
Hence, for $\epsilon>0$ small enough, \[\Gamma\subset\subset U:=M\setminus B_\epsilon((\pm 1,0,0,0,0))\subset M^{prin}.\] 
In addition, by the compactness theorem for the $G$-equivariant min-max hypersurfaces (\cite{wang2024index}*{Theorem 5.3}), we can shrink $\epsilon>0$ so that every minimal $G$-hypersurface produced by Theorem \ref{Thm: G-min-max} with respect to $\Pi$ is contained in $U$. 

Moreover, the multiplicity one result (\cite{wangzc2023four-sphere}\cite{li2025RP2}) 
is valid for $\Gamma/S^1$, and thus  
\begin{align}\label{Eq: multi 1 and weighted genus in S4}
    m=1\qquad{\rm and}\qquad \genus(\Gamma/S^1)\leq  \genus(\Sigma_0/S^1). 
\end{align}
Indeed, for the multiplicity one result, we can use \eqref{Eq: area in orbit space} to reduce our $G$-equivariant min-max in $M$ to a $\mb Z_2$-equivariant min-max in $M/S^1$. 
Since $\Gamma\subset\subset U\subset M^{prin}$, we can restrict the proof of \cite{wangzc2023four-sphere}*{Theorem 7.3} (see also \cite{li2025RP2} for a $\mb Z_2$-equivariant version) in $U/S^1$ under the weighted metric $\wti g_{_{M/S^1}}$. 
Namely, the solution of PMC min-max given by \cite{wangzc2023four-sphere}*{Theorem 2.4} has the regularity and compactness results valid in $U/S^1$. 
Hence, the approximation construction in \cite{wangzc2023four-sphere}*{\S 6, \S 7} would carry over in $U/S^1$ to show $m=1$ provided that $\Gamma/S^1$ is $2$-sided (ensured by $M/S^1=S^3$) and unstable (ensured by Lemma \ref{Lem: positive Ricci M/S1}). 
Besides, note that one only needs $\Ric_M>0$ (instead of Lemma \ref{Lem: positive Ricci M/S1}) to show $\Gamma/S^1$ is unstable. This is because $\Gamma/S^1$ must enclose a region in the simply connected $M/S^1=S^3$, and thus $\Gamma$ admits an $S^1$-invariant unit normal, which implies that the stability of $\Gamma$  is equivalent to the stability of $\Gamma/S^1$ in $(M/S^1, \wti g_{M/S^1})$ by Lemma \ref{Lem: stability and G-stability} and \eqref{Eq: area in orbit space}. 
Finally, the (weighted) genus upper bound in \eqref{Eq: multi 1 and weighted genus in S4} follows from the proof of Theorem \ref{Thm: G-min-max - topological control}.


Now, we can show the existence of a minimal hypertorus in $\mb S^4$. 
The same result was also obtained in \cite{carlotto2023hypertori} using an equivariant dynamic method similar to Hsiang's method \cite{hsiang1983sphericalI}. 

\begin{theorem}\label{Thm: minimal hypertori in S4}
    There exists an embedded minimal $T^3\cong S^1\times S^1\times S^1$ in $\mb S^4_1(0)$. 
\end{theorem}
\begin{proof}
    By the suspension construction \eqref{Eq: suspension orbit space}, the coordinates of $M/S^1$ can be changed to 
    \[ M/S^1=\{(x_1,x_2,x_3,x_4)\in\mb R^4: x_1^2+(x_2^2+x_3^2+x_4^2)=1 \} = \mb S^3_1 \]
    so that the two singular points $(M\setminus M^{prin})/S^1$ are $(\pm 1,0,0,0)$. 
    Let $T_0\subset M^{prin}$ be a $G$-hypersurface so that $T_0/S^1=\{x\in M/S^1: x_1^2+x_2^2=x_3^2+x_4^2\}$ is an antipodally symmetric (Clifford) torus in $M/S^1$. 
    Consider the family of $G$-hypersurfaces $\{\wti\Sigma_t\}_{t\in [0,\pi/2]}$ with 
    \[ \wti\Sigma_t/S^1=\{x\in M/S^1: x_1^2+x_2^2=\cos^2(t) \}, \]
    which degenerates to a curve at $t\in\{0,\pi/2\}$, and only $\wti\Sigma_0$ contains the singular points $M\setminus M^{prin}$. 
    In particular, this gives a $G$-equivariant $(X,Z)$-sweepout with $X=[0,\pi/2], Z=\bd X$ so that $\{\wti\Sigma_t/G\}$ is a family of tori in $RP^3$. 
    Then the above constructions provide a connected $G$-invariant minimal hypersurface $\Gamma$ with multiplicity $m=1$ so that $\genus(\Gamma/S^1)\leq 1$ (see \eqref{Eq: multi 1 and weighted genus in S4}). 
    If $\genus(\Gamma/S^1)=0$, then $\Gamma/S^1$ is an antipodally symmetric $2$-sphere, and thus $\Gamma/G$ is a non-orientable $RP^2$. 
    However, since $\Sigma_t/G$ are orientable, it follows from  \cite{zhou2015positiveRic}*{Proposition 6.1} (with \cite{wang2024index}*{Proposition 4.19} in place of \cite{zhou2015positiveRic}*{Lemma 6.11}) that the multiplicity $m$ must be even, which contradicts $m=1$ in \eqref{Eq: multi 1 and weighted genus in S4}. 
    Therefore, we know $\Gamma/S^1$ and $\Gamma/G$ are both tori. 

    Finally, we claim that $\Gamma$ is diffeomorphic to $T^3$. 
    Indeed, $\Gamma/S^1$ separates the $3$-sphere $M/S^1$ into two $\mb Z_2$-invariant solid tori $\Omega_1,\Omega_2$ so that the two singular points $(M\setminus M^{prin})/S^1 $ form a $\mb Z_2$-orbit contained in one of $\{\Omega_i\}$ (say $ \Omega_2$). 
    In particular, $\Omega_1$ is a solid torus in $ M^{prin}/S^1$ with $\bd\Omega_1=\Gamma/S^1$. 
    Then, $\Gamma/S^1$ can be deformation retracted in $\Omega_1$ into a circle $\gamma$. 
    Note that $\pi^{-1}(\gamma)\subset M^{prin}$ is a principal $S^1$-bundle over $\gamma\cong S^1$, which must be a trivial $S^1$-bundle, i.e. $\pi^{-1}(\gamma)\cong S^1\times S^1$. 
    Therefore, $\Gamma\cong S^1\times S^1\times S^1$ is an embedded minimal hypertorus. 
\end{proof}

\begin{remark}\label{Rem: valid for other metric}
    Although we considered separating $\Sigma_0/G$ in this subsection, the multiplicity one result \eqref{Eq: multi 1 and weighted genus in S4} is also true for the relative min-max in $(\overline{\mc X}^G, \mc Y^G)$ (cf. \eqref{Eq: space of RP2}) by the technique in \cite{li2025RP2}*{Theorem 1.3}. 
    Note also that the last statement in Lemma \ref{Lem: tangent cone at fixed point} always holds independently of the metrics since there is no free $S^1$-action on $S^2$. 
    Hence, the above existence results remain valid for any other $G$-invariant metric $g_{_M}$ on $M=S^4$ as long as $(S^4, g_{_M})$ has positive Ricci curvature, which gives two distinct minimal hyperspheres and a minimal hypertorus by \eqref{Eq: multi 1 and weighted genus in S4}. 
    Moreover, using Remark \ref{Rem: relax finiteness of T} and \eqref{Eq: multi 1 and weighted genus in S4}, one can indeed apply the proof of \cite{li2025RP2}*{Theorem 1.2, 1.6} to obtain $4$ distinct minimal hyperspheres and $4$ distinct minimal hypertori in any $G$-invariant Riemannian $S^4$ of positive Ricci curvature. 
    However, in the round $\mb S^4$, some of these minimal hypersurfaces may be congruent. 
\end{remark}

\subsection{Minimal $S^1\times S^2$ in $\mb S^4$}

By the above constructions, we see that the topology of an embedded $S^1$-hypersurface $\Sigma$ not only depends on $\genus(\Sigma/S^1)$, but also relies on whether $\Sigma/S^1$ separates the two fixed points $(\pm 1, 0,0,0)\in \mb S^4/S^1$. 
Hence, the $\mb Z_2$-action on $\mb S^4/S^1$ generated by the antipodal map plays an important role in determining the topology of the min-max minimal $S^1$-hypersurfaces.
In this subsection, we apply another $\mb Z_2$-action on $\mb S^4/S^1$ generated by a $\pi$-rotation to construct an embedded minimal $S^1\times S^2$ in $\mb S^4$. 

Specifically, recall that $\mb S^4/S^1 $ is parameterized by the ellipsoid $\{(x,p_1,p_2,p_3)\in \mb R^4: x^2+4(p_1^2+p_2^2+p_3^2)=1\}$ with the metrics $g_{M/S^1}$ and $\wti g_{M/S^1}$ (see \eqref{Eq: suspension orbit space}).  
Then, one can define a $\pi$-rotation map on $\mb S^4/S^1$ by
\begin{align}\label{Eq: double reflection}
	s: \mb S^4/S^1 \to \mb S^4/S^1 ,\qquad s(x,p_1,p_2,p_3) := (-x, p_1, -p_2, p_3), 
\end{align}
which generates a $\mb Z_2:=\langle s\rangle$ action on $\mb S^4/S^1$ by orientation-preserving isometries (under the metrics $g_{M/S^1}$ and $\wti g_{M/S^1}$). 
Let $\tilde s$ be an isometry on $\mb S^4=\{(x,z_1,z_2)\in \mb R\times \mb C\times \mb C: x^2+|z_1|^2+|z_2|^2=1\}$ defined by 
\begin{align}\label{Eq: lifted double reflection}
	\tilde s(x,z_1,z_2) := (-x, \bar z_1, \bar z_2 ). 
\end{align}
Using the Hopf map $h$ in \eqref{Eq: Hopf map} on every slice $\{(x,z_1,z_2)\in\mb S^4: x=\cos(\alpha)\}$, one can easily verify that $\tilde s$ is a lift of $s$, i.e. $ s\circ \pi = \pi\circ \tilde s$. 
Additionally, note that 
\[ \tilde s^2=id\qquad {\rm and}\qquad \tilde s\circ R_\theta = R_{-\theta}\circ \tilde s, \quad \forall \theta\in [0,2\pi],\]
where $\{R_{\theta}\}$ is the $S^1$-action on $\mb S^4$ defined in \eqref{Eq: suspended S1-action}.  
Hence, by taking $\varphi: \mb Z_2= \langle\tilde s\rangle \to {\rm Aut}(S^1)$ as $\varphi(id)(e^{i\theta}):=e^{i\theta}$ and $\varphi(\tilde s) (e^{i\theta}):= e^{-i\theta}$, we have a semidirect product group $G'$ given by 
\begin{align}\label{Eq: extension of S1 by double reflection}
	G':= S^1\rtimes_\varphi \mb Z_2
\end{align}
with the group operation
\[ (e^{i\theta_1}, \tilde s^{k_1})\cdot (e^{i\theta_2}, \tilde s^{k_2}):= (e^{i\theta_1}\cdot \varphi(\tilde s^{k_1})(e^{i\theta_2}), \tilde s^{(k_1+k_2~~ {\rm mod~}2)}) .\]
Similar to Lemma \ref{Lem: extend action}, $G'$ acts isometrically on $\mb S^4$ by 
\[ (e^{i\theta}, \tilde s^k)\cdot (x,z_1,z_2):= R_\theta\circ \tilde s^k (x,z_1,z_2).  \]
In particular, $\mb S^4/G'= (\mb S^4/S^1)/\mb Z_2$ is topologically a $3$-sphere without boundary, and thus there is no special exceptional $G'$-orbit on $\mb S^4$. 

Next, we will construct a family of (possibly degenerate) $G'$-equivariant hypersurfaces $\{\wti\Sigma_t\}_{t\in [0,1]}$ in $\mb S^4$ with $\Sigma_t:=\wti\Sigma_t/S^1$ so that 
\begin{itemize}
	\item[(a)] $\Sigma_0= (0,0,0,1/2)\in \mb S^4/S^1$ and $\Sigma_1=\{(\cos(\alpha),0,0,-\sin(\alpha)/2): \alpha\in [0,\pi]\}$;
	\item[(b)] for $t\in (0,1)$, $\Sigma_t\subset (\mb S^4/S^1)\setminus\{(\pm1, 0,0,0)\}$ is an embedded $\mb Z_2$-invariant $2$-sphere; 
	\item[(c)] $\{\wti\Sigma_t\}_{t\in [0,1]}$ is a $G'$-equivariant $([0,1],\{0,1\})$-sweepout (Definition \ref{Def: G-sweepout}); 
	\item[(d)] there is a continuous family of $\mb Z_2$-invariant Caccioppoli sets $\{\Omega_t\}_{t\in [0,1]}$ in $\mb S^4/S^1$ so that $\bd\Omega_t=\Sigma_t$ for $t\in (0,1)$, $\Omega_0=\emptyset$, and $\Omega_1=\mb S^4/S^1$. 
\end{itemize}
Indeed, let $\{\gamma_t\}_{t\in [0,1]}$ be the continuous family of latitude circles on $\mb S^2_{1/2}$ given by 
\[ \gamma_t:=\{(p_1, p_2 ,\cos(t\pi)/2) \in \mb R^3: p_1^2+p_2^2=\sin^2(t\pi)/4 \} \subset \mb S^2_{1/2}, \]
each of which encloses a closed disk in $\mb S^2_{1/2}$:
\[D_t:=\{(p_1, p_2 ,p_3) \in \mb S^2_{1/2}: p_3\geq \cos(t\pi)/2 \}. \]
Using the suspension structure of $\mb S^4/S^1$, we define $\Gamma_t=\Sigma \gamma_t$ and $\Sigma D_t$, $t\in [0,1]$, as the suspensions of $\gamma_t$ and $D_t$, respectively, i.e.
\[ \Gamma_t:= \{(\cos(\alpha), \sin(\alpha)p)\in \mb S^4/S^1:  p\in \gamma_t, \alpha\in [0,\pi]\}, \]
and 
\[\Sigma D_t:= \{(\cos(\alpha), \sin(\alpha)p)\in \mb S^4/S^1:  p\in D_t, \alpha\in [0,\pi]\}. \]
In particular, for any $t\in (0,1)$, $\Sigma D_t$ is a $\mb Z_2$-invariant closed $3$-ball, and $\Gamma_t=\bd \Sigma D_t$ is a $\mb Z_2$-invariant $2$-sphere with two singular points at the north/south poles $(\pm1,0,0,0)$. 
Additionally, $\Gamma_0, \Gamma_1$ are two curves connecting $(\pm1,0,0,0)$. 
Denote by 
\[B_t:= \{ (\cos(\alpha), p)\in \mb S^4/S^1: |\cos(\alpha)|> t \}\]
the $\mb Z_2$-invariant union of two open $3$-balls in $\mb S^4/S^1$ containing $(\pm 1,0,0,0)$. 
By the constructions, $\bd B_t$ is transversal to $\bd \Sigma D_t = \Gamma_t$ for $t\in (0,1)$. 
Then, we can remove $(\pm 1,0,0,0)$ from $\Sigma D_t$ by deleting $B_t$ to have a continuous family of $\mb Z_2$-invariant $3$-balls:
\[ \Omega_t:= \Sigma D_t\setminus B_t \subset \mb S^4/S^1. \]
Define then $\Sigma_0= \Gamma_0\setminus B_0 = (0,0,0, 1/2)$, $\Sigma_1=\Gamma_1\setminus B_1=\Gamma_1$, and 
\[\Sigma_t:=\bd\Omega_t \qquad\forall t\in (0,1) .\]
By smoothing and taking $\wti\Sigma_t:=\pi^{-1}(\Sigma_t)$, we obtain the desired $G'$-equivariant sweepout in $\mb S^4$. 
Note that $\wti\Sigma_t$ is of (locally) $G'$-boundary-type since $\wti\Sigma_t=\bd \pi^{-1}(\Omega_t)$.

\begin{theorem}\label{Thm: minimal S1 times S2}
    Using the above notations, let $\Pi$ be the $G'$-equivariant $([0,1],\{0,1\})$-homotopy class of $\{\wti\Sigma_t\}_{t\in [0,1]}$ in $\mb S^4$. 
    Then there exists an embedded $G'$-invariant minimal hypersurface $\Gamma\subset \mb S^4$ so that $\mc H^3(\Gamma) = \mf L(\Pi)>0$ and $\Gamma$ is diffeomorphic to $S^1\times S^2$. 
\end{theorem}

\begin{remark}\label{Rem: valid of S1 times S2 for other metrics}
    The proof of Theorem \ref{Thm: minimal S1 times S2} only requires the $G'$-invariance of $\mb S^4$ and $\Ric_{\mb S^4}>0$. 
    Hence, there always exists an embedded minimal $S^1\times S^2$ in any $S^4$ with a $G'$-invariant metric and positive Ricci curvature. 
\end{remark}

\begin{proof}
    By (d), $\{\wti\Sigma_t\}_{t\in [0,1]}$ induces a $1$-sweepout in the sense of Almgren-Pitts, and thus $\mf L(\Pi)>0$. 
    By Theorem \ref{Thm: G-min-max}, we have a $G'$-invariant minimal hypersurface $\Gamma\subset \mb S^4$ with $m\in\mb N$ so that $\mf L(\Pi)=m \mc H^3(\Gamma)$. 

    Firstly, we claim that $m=1$. 
    Indeed, the proof is very similar to the argument leading to \eqref{Eq: multi 1 and weighted genus in S4} except that the action of $\mb Z_2=\langle s\rangle $ is not free on $U/S^1$, but has a $1$-dimensional fixed-point set $\xi\subset U/S^1$, where $U:=M\setminus B_\epsilon((\pm 1,0,0,0,0))\subset M^{prin}$. 
    Nevertheless, following the proof of \cite{wangzc2023four-sphere}*{Theorem 7.3}, we can take the prescribed mean curvature function $h\in C^\infty(U/S^1)$ to be $\mb Z_2$-invariant with $\spt(h)\subset\subset (U/S^1)\setminus\xi$ (since $\dim(\xi)<2$). 
    Because the $\mb Z_2$-action is free on $\spt(h)$, one can combine \cite{wangzc2023four-sphere}*{Theorem 2.4} and Theorem \ref{Thm: G-min-max} to obtain the regularity of the solutions of $\epsilon_k h$-PMC min-max in $U/S^1$, where $\epsilon_k\to 0$. 
    Then, by the approximation construction in \cite{wangzc2023four-sphere}*{\S 6, \S 7}, we have $m=1$ provided that $\Gamma/S^1$ is $2$-sided and unstable in $U/S^1$. 
    Since $M/S^1=S^3$ is simply connected, $\Gamma/S^1$ is $2$-sided and separating. 
    Hence, $\Gamma$ admits an $S^1$-invariant unit normal, and is unstable under $S^1$-equivariant variations by $\Ric_M>0$ and Lemma \ref{Lem: stability and G-stability}. 
    Using \eqref{Eq: area in orbit space}, we see $\Gamma/S^1$ is unstable, and thus $m=1$. 

    Next, we show that there is a $G'$-invariant open set $\wti\Omega$ with $\Gamma=\bd \wti \Omega$. 
    To be exact, let $\{\Gamma^i_{t_i}\}_{i\in\mb N}$ be the min-max sequence converging to $\Gamma$ (with multiplicity $m=1$) in the sense of varifolds. 
    Then, by (d), there are open $G'$-sets $\wti\Omega_i\subset \mb S^4$ so that $\Gamma^i_{t_i}=\bd\wti\Omega_i$ for all $i\in\mb N$. 
    Up to a subsequence, we have a limit Caccioppoli $G'$-set $\wti\Omega=\lim_{i\to\infty}\wti\Omega_i$ so that $\bd\wti\Omega$ is supported in $\Gamma$ with mass $\leq \mc H^3(\Gamma)$. 
    This implies that $\bd\wti\Omega$ is either $0$ or $\llbracket \Gamma\rrbracket$ due to the Constancy Theorem. 
    It then follows from the proof of \cite{zhou2015positiveRic}*{Proposition 6.1} that $\bd\wti\Omega=\llbracket \Gamma\rrbracket$ since otherwise, $\bd\wti\Omega=0$ and the convergence $|\Gamma^i_{t_i}|\to m|\Gamma|$ has even multiplicity $m$, which contradicts $m=1$.  
    
    Note that $\mb S^4/G'$ is homeomorphic to a $3$-sphere with an isolated singular point $\{(\pm 1,0,0,0)\}/\mb Z_2$ and a $1$-dimensional singular curve $\xi/\mb Z_2$, where $\xi\subset M/S^1$ is the fixed-point set of the $\mb Z_2$-action. 
    Since $\Gamma$ encloses an open $G'$-set $\wti\Omega\subset \mb S^4$, we have $\Gamma/G' = \bd (\wti\Omega/G')$, which implies that $\Gamma/G'$ is a closed surface without boundary. 
    In particular, by Lemma \ref{Lem: hypersurface position} and Proposition \ref{Prop: local structure of Sigma/G 2}(3)(4), $\Gamma/S^1$ can only meet $\xi\subset M/S^1$ orthogonally at a finite number of points. 
    Therefore, one can follow the proof of Theorem \ref{Thm: G-min-max - topological control} to show that $\genus(\Gamma/S^1)=0$ since $\genus(\wti\Sigma_t/S^1) =0$ for all $t\in (0,1)$ by (b). 

    Finally, since $\wti\Omega$ is $G'$-invariant, we may assume that the two singular points $(\pm 1,0,0,0)\in M/S^1$ are both contained in $ (M\setminus\wti\Omega)/S^1$. 
    Hence, $\wti\Omega/S^1$ is a $3$-ball enclosed by the $2$-sphere $\Gamma/S^1$ and is contractible in $(M/S^1)\setminus\{(\pm 1, 0,0,0)\}\cong  S^2\times (-1,1)$. 
    In particular, $H^2(\wti\Omega/S^1;\mb Z)=0$ and the principal $S^1$-bundle $\wti\Omega\to\wti\Omega/S^1$ has Euler class $0$, which implies that the principal $S^1$-bundle is trivial. 
    Therefore, $\Gamma=\partial(S^1\times (\wti\Omega/S^1))\cong S^1\times\bd B^3=S^1\times S^2$. 
\end{proof}

\bibliographystyle{abbrv}

\bibliography{reference.bib}   
\end{document}